\documentclass[reqno, 10pt]{amsart}
\usepackage{fullpage,graphicx}
\usepackage{amsmath,amsthm}
\usepackage{amsfonts,amssymb}
\usepackage{newlfont,mathrsfs,mathtools}
\usepackage{txfonts}
\usepackage{etoolbox}
\usepackage{enumitem}

\setlist[enumerate]{leftmargin=14pt, labelwidth=0cm, label=\arabic*.}

\makeatletter
\let\old@tocline\@tocline
\let\section@tocline\@tocline
\newcommand{\subsection@dotsep}{4.5}
\newcommand{\subsubsection@dotsep}{4.5}
\patchcmd{\@tocline}
{\hfil}
{\nobreak
	\leaders\hbox{$\m@th
		\mkern \subsection@dotsep mu\hbox{.}\mkern \subsection@dotsep mu$}\hfill
	\nobreak}{}{}
\let\subsection@tocline\@tocline
\let\@tocline\old@tocline

\patchcmd{\@tocline}
{\hfil}
{\nobreak
	\leaders\hbox{$\m@th
		\mkern \subsubsection@dotsep mu\hbox{.}\mkern \subsubsection@dotsep mu$}\hfill
	\nobreak}{}{}
\let\subsubsection@tocline\@tocline
\let\@tocline\old@tocline

\let\old@l@subsection\l@subsection
\let\old@l@subsubsection\l@subsubsection

\def\@tocwriteb#1#2#3{%
	\begingroup
	\@xp\def\csname #2@tocline\endcsname##1##2##3##4##5##6{%
		\ifnum##1>\c@tocdepth
		\else \sbox\z@{##5\let\indentlabel\@tochangmeasure##6}\fi}%
	\csname l@#2\endcsname{#1{\csname#2name\endcsname}{\@secnumber}{}}%
	\endgroup
	\addcontentsline{toc}{#2}%
	{\protect#1{\csname#2name\endcsname}{\@secnumber}{#3}}}%

\newlength{\@tocsectionindent}
\newlength{\@tocsubsectionindent}
\newlength{\@tocsubsubsectionindent}
\newlength{\@tocsectionnumwidth}
\newlength{\@tocsubsectionnumwidth}
\newlength{\@tocsubsubsectionnumwidth}
\newcommand{\settocsectionnumwidth}[1]{\setlength{\@tocsectionnumwidth}{#1}}
\newcommand{\settocsubsectionnumwidth}[1]{\setlength{\@tocsubsectionnumwidth}{#1}}
\newcommand{\settocsubsubsectionnumwidth}[1]{\setlength{\@tocsubsubsectionnumwidth}{#1}}
\newcommand{\settocsectionindent}[1]{\setlength{\@tocsectionindent}{#1}}
\newcommand{\settocsubsectionindent}[1]{\setlength{\@tocsubsectionindent}{#1}}
\newcommand{\settocsubsubsectionindent}[1]{\setlength{\@tocsubsubsectionindent}{#1}}

\renewcommand{\l@section}{\section@tocline{1}{\@tocsectionvskip}{\@tocsectionindent}{}{\@tocsectionformat}}%
\renewcommand{\l@subsection}{\subsection@tocline{2}{\@tocsubsectionvskip}{\@tocsubsectionindent}{}{\@tocsubsectionformat}}%
\renewcommand{\l@subsubsection}{\subsubsection@tocline{3}{\@tocsubsubsectionvskip}{\@tocsubsubsectionindent}{}{\@tocsubsubsectionformat}}%
\newcommand{\@tocsectionformat}{}
\newcommand{\@tocsubsectionformat}{}
\newcommand{\@tocsubsubsectionformat}{}
\expandafter\def\csname toc@1format\endcsname{\@tocsectionformat}
\expandafter\def\csname toc@2format\endcsname{\@tocsubsectionformat}
\expandafter\def\csname toc@3format\endcsname{\@tocsubsubsectionformat}
\newcommand{\settocsectionformat}[1]{\renewcommand{\@tocsectionformat}{#1}}
\newcommand{\settocsubsectionformat}[1]{\renewcommand{\@tocsubsectionformat}{#1}}
\newcommand{\settocsubsubsectionformat}[1]{\renewcommand{\@tocsubsubsectionformat}{#1}}
\newlength{\@tocsectionvskip}
\newcommand{\settocsectionvskip}[1]{\setlength{\@tocsectionvskip}{#1}}
\newlength{\@tocsubsectionvskip}
\newcommand{\settocsubsectionvskip}[1]{\setlength{\@tocsubsectionvskip}{#1}}
\newlength{\@tocsubsubsectionvskip}
\newcommand{\settocsubsubsectionvskip}[1]{\setlength{\@tocsubsubsectionvskip}{#1}}

\patchcmd{\tocsection}{\indentlabel}{\makebox[\@tocsectionnumwidth][l]}{}{}
\patchcmd{\tocsubsection}{\indentlabel}{\makebox[\@tocsubsectionnumwidth][l]}{}{}
\patchcmd{\tocsubsubsection}{\indentlabel}{\makebox[\@tocsubsubsectionnumwidth][l]}{}{}

\newcommand{\@sectypepnumformat}{}
\renewcommand{\contentsline}[1]{%
	\expandafter\let\expandafter\@sectypepnumformat\csname @toc#1pnumformat\endcsname%
	\csname l@#1\endcsname}
\newcommand{\@tocsectionpnumformat}{}
\newcommand{\@tocsubsectionpnumformat}{}
\newcommand{\@tocsubsubsectionpnumformat}{}
\newcommand{\setsectionpnumformat}[1]{\renewcommand{\@tocsectionpnumformat}{#1}}
\newcommand{\setsubsectionpnumformat}[1]{\renewcommand{\@tocsubsectionpnumformat}{#1}}
\newcommand{\setsubsubsectionpnumformat}[1]{\renewcommand{\@tocsubsubsectionpnumformat}{#1}}
\renewcommand{\@tocpagenum}[1]{%
	\hfill {\mdseries\@sectypepnumformat #1}}

\let\oldappendix\appendix
\renewcommand{\appendix}{%
	\leavevmode\oldappendix%
	\addtocontents{toc}{%
		\protect\settowidth{\protect\@tocsectionnumwidth}{\protect\@tocsectionformat\sectionname\space}%
		\protect\addtolength{\protect\@tocsectionnumwidth}{2em}}%
}
\makeatother

\makeatletter
\settocsectionnumwidth{2em}
\settocsubsectionnumwidth{2.5em}
\settocsubsubsectionnumwidth{3em}
\settocsectionindent{1pc}%
\settocsubsectionindent{\dimexpr\@tocsectionindent+\@tocsectionnumwidth}%
\settocsubsubsectionindent{\dimexpr\@tocsubsectionindent+\@tocsubsectionnumwidth}%
\makeatother

\settocsectionvskip{0pt}
\settocsubsectionvskip{0pt}
\settocsubsubsectionvskip{0pt}

\settocsectionformat{\bfseries}
\settocsubsectionformat{\mdseries}
\settocsubsubsectionformat{\mdseries}
\setsectionpnumformat{\bfseries}
\setsubsectionpnumformat{\mdseries}
\setsubsubsectionpnumformat{\mdseries}

\let\oldtableofcontents\tableofcontents
\renewcommand{\tableofcontents}{%
	\vspace*{-\linespacing}
	\oldtableofcontents}

\allowdisplaybreaks[1]

\usepackage[colorlinks,
linkcolor=blue,
anchorcolor=blue,
citecolor=blue
]{hyperref}

\numberwithin{equation}{section}
\newtheorem{thm}{Theorem}[section]
\newtheorem{pp}[thm]{Proposition}
\newtheorem{lm}[thm]{Lemma}
\newtheorem{df}[thm]{Definition}

\newtheorem{Rmk}[thm]{Remark}
\newcommand{\R}{\mathbb{R}}
\newcommand{\rd}{\mathrm{d}}

\newcommand{\mcI}{\mathcal{I}}

\newcommand{\mS}{\mathbb{S}}
\newcommand{\T}{\mathbb{T}}
\newcommand{\N}{\mathbb{N}}
\newcommand{\Z}{\mathbb{Z}}

\newcommand{\cF}{\cal{F}}

\newcommand{\dlt}{\delta}

\newcommand{\Id}{\operatorname{Id}}
\newcommand{\Div}{\operatorname{div}}
\newcommand{\tr}{\operatorname{tr}}
\newcommand{\supp}{\operatorname{supp}}

\newcommand{\oR}{\overline{R}}
\newcommand{\tR}{\tilde{R}^{(u)}}

\newcommand{\uR}{\mathcal{R}^{(u_{q,i+1})}}

\newcommand{\tj}{\widetilde{j}}

\newcommand{\tu}{\widetilde{u}}
\newcommand{\pttu}{\widetilde{\pa_tu}}
\newcommand{\tV}{\widetilde{V}}
\newcommand{\tG}{\widetilde{G}}
\newcommand{\tC}{\widetilde{C}}
\newcommand{\tD}{\widetilde{D}}
\newcommand{\ty}{\tilde{y}}
\newcommand{\tx}{\tilde{x}}
\newcommand{\tU}{\tilde{U}}
\newcommand{\tlm}{\tilde{\lambda}}
\newcommand{\tlA}{\tilde{\lambda}_A}

\newcommand{\tup}{\upsilon'}
\newcommand{\ttup}{\upsilon''}
\newcommand{\tPi}{\widetilde{\Pi}}
\newcommand{\mPO}{\mathbb{P}_{=0}}
\newcommand{\mPG}{\mathbb{P}_{\gtrsim\lambda_{q,i+1}}}
\newcommand{\pa}{\partial}
\newcommand{\pt}{\pa_t}

\newcommand{\tmu}{\tilde{\mu}}

\newcommand{\tL}{\widetilde{\Lambda}}
\newcommand{\tl}{\tilde{\lambda}}
\newcommand{\te}{\tilde{e}}
\newcommand{\tlu}{\tilde{\lambda}_u}
\newcommand{\tlU}{\tilde{\lambda}_U}

\newcommand{\ou}{\overline{u}}
\newcommand{\oM}{\overline{M}}
\newcommand{\ow}{\overline{w}}
\newcommand{\oW}{\overline{W}}
\newcommand{\oy}{\overline{y}}

\newcommand{\tS}{\widetilde{\Sgm}}

\newcommand{\Rl}{R_{\ell,i}}

\newcommand{\Cl}{C_{\ell,i}}
\newcommand{\Dl}{D_{\ell,i}}

\newcommand{\ul}{u_{\ell,i}}

\newcommand{\Gl}{G_{\ell,i}}

\newcommand{\tA}{\mathring{A}}
\newcommand{\tB}{\tilde{B}}
\newcommand{\tE}{\tilde{E}}
\newcommand{\tH}{\tilde{H}}

\newcommand{\tsgm}{\widetilde{\sgm}}

\newcommand{\PL}{P_{\leqslant\ell^{-1}}}

\newcommand{\PLN}[1]{P_{\leqslant\ell_{q,#1}^{-1}}}
\newcommand{\ULN}[1]{U_{\leqslant\ell_{q,#1}^{-1}}}
\newcommand{\UL}{U_{\leqslant\ell^{-1}}}

\newcommand{\PLq}{P_{\leqslant\ell_{q,i}^{-1}}}
\newcommand{\PGq}{P_{>\ell_{q,i}^{-1}}}

\newcommand{\ULq}{U_{\leqslant\ell_{q,i}^{-1}}}
\newcommand{\UGq}{U_{>\ell_{q,i}^{-1}}}

\newcommand{\nb}{\nabla}
\newcommand{\Sgm}{\Sigma}
\newcommand{\sgm}{\sigma}
\newcommand{\nrm}[1]{\Vert#1\Vert}
\newcommand{\Nrm}[1]{\left\Vert#1\right\Vert}
\newcommand{\nRM}[2]{\Vert#1\Vert_{#2}}

\newcommand{\nRMr}[3]{\Vert\pt^{#3}(#1)\Vert_{#2}}

\newcommand{\nrmrN}[1]{\Vert\pt^r(#1)\Vert_N}

\newcommand{\brk}[1]{\left\langle#1\right\rangle}
\newcommand{\bra}[1]{\left\lbrace#1\right\rbrace}
\newcommand{\set}[1]{\{#1\}}
\newcommand{\abs}[1]{\left\vert#1\right\vert}


\begin{document}
	\title[] {The null condition in elastodynamics leads to non-uniqueness}
	
	\author{Shunkai Mao}
	\address{School of Mathematics and Statistics, Donghua University, China.}
	\email[Shunkai Mao]{skmao@dhu.edu.cn}
	
	\author{Peng Qu}
	\address{School of Mathematical Sciences $\&$ Shanghai Key Laboratory for Contemporary Applied Mathematics, Fudan University, China.}
	\email[Peng Qu]{pqu@fudan.edu.cn}
	\thanks{}
	
	\begin{abstract}
	We consider the Cauchy problem for the system of elastodynamic equations in two dimensions. Specifically, we focus on  materials characterized by a null condition imposed on the quadratic part of the nonlinearity. We can construct  non-zero weak solutions $u \in C^1([0, T] \times \T^2)$ that emanate from zero initial data.	 The proof relies on the convex integration scheme. By exploiting the characteristic double wave speeds of the equations, we construct a new class of building blocks. This work extends the application of convex integration techniques to hyperbolic systems with a null condition and reveals the rich solution structure in nonlinear elastodynamics.
	\end{abstract}
	\maketitle
	
	{
		\tableofcontents
	}

\section{Introduction}

In this paper, we consider the equations of motion for the displacement of an isotropic, homogeneous, hyperelastic material filling space $\T^2=[-\pi,\pi]^2$. This corresponds to the Cauchy problem  for the system of elastodynamic equations on $[0,T] \times \T^2$, with $T > 0$.
\begin{equation}\label{system}
	\left\{
	\begin{aligned}
		&{\rho_0}\pa_{tt}y-\Div P=0,&& (t,x)\in (0,T)\times\T^2,  \\
		&y(0,x)=y_0(x),\quad \pa_ty(0,x)=y_1(x),&& x\in \T^2,
	\end{aligned}
	\right.
\end{equation}
where $y:[0,T]\times\T^2\rightarrow\R^2$ is a vector-valued function representing the deformed position, $\rho_0$ is the reference density, and $P$ is the first Piola--Kirchhoff stress tensor. 

\subsection{Basic concepts and notation in nonlinear elasticity}\par
The unknown variable in this problem is $y(t, x)$, which represents a $C^1$ deformation of the material evolving over time. Then, we denote the  deformation gradient by  $F=\nb y$ and its determinant by $J=\det F$. The density $\rho_0$ in the reference configuration  is taken to be $1$. From the hyperelastic assumption of the material, we consider a stored energy function $\sgm=\sgm(F)$, which depends on $F$ through the principal invariants of the (left) Cauchy-Green strain tensor $B=FF^{\top}$. \par 
It is convenient to introduce the displacement $u(t,x)=y(t,x)-x$ and its gradient $G=F-\Id=\nb u$. Here, $u$ measures the displacement from the reference configuration. In two-dimensional elasticity, the stored-energy function $\sgm=\sgm(j_1,j_2)$ depends on the principal invariants $j_1$ and $j_2$ of the matrix $C = B-\Id = FF^{\top} - \Id=G+G^{\top}+GG^{\top}$. In more detail, we consider a case with a simple stored-energy function:
\begin{align}\label{stored-energy function}
	\sgm(j_1,j_2)=\sum_{r=1}^2\sgm_rj_r+\sum_{r=1}^2\sum_{s=1}^2\frac{\sgm_{rs}}{2}j_rj_s+\frac{\sgm_{111}}{6}j_1^3,
\end{align}
where $\sgm_r$, $\sgm_{rs}$, and $\sgm_{111}$ are constants satisfying
$\sgm_{rs}=\sgm_{sr}$ for $1\leqslant r,s\leqslant2$. Moreover, we consider the total mechanical energy function \begin{equation}
	\int_{\T^2}
	\left(\frac12|\pa_tu|^2+\sgm(j_1,j_2)\right)\,\rd x.
\end{equation}

We impose the condition $\sgm_1 = 0$, indicating that the reference configuration is a stress-free state. The Lam{\'e} constants  $\lambda = 4(\sgm_{11} + \sgm_2)$ and $\mu = -2\sgm_2$ are assumed to satisfy $\mu>0$ and $\lambda>0$. This makes the linear part of the equation hyperbolic. We also assume $\sgm_{11}$, $\sgm_{22}>0$, and $\sgm_{11}\sgm_{22}-\sgm_{12}^2>0$, indicating that the stored energy function $\sgm(j_1,j_2)$ is a convex function near the origin. Moreover, we consider a stored-energy function satisfying the following part of the well-known null condition, see for instance \cite{Sid96}:
\begin{align}\label{the null condition}
	3\sgm_{11}+2\sgm_{111}=0, \quad\sgm^* := 2\sgm_{11} + 8\sgm_{12} + 2\sgm_2 + 4\sgm_{111} \neq 0.
\end{align}
\begin{Rmk}
For a general stored energy function $\overline{\sgm}=\overline{\sgm}(j_1,j_2)$, it is still possible to use the method provided in this paper to show the non-uniqueness results, but the process of the corresponding proofs  would become more complex. Since we wish to elucidate the basic idea in the proof of our non-uniqueness results, we would focus on this simpler case under the assumption \eqref{stored-energy function} in this paper.  In fact, for a general
$\overline{\sgm}=\overline{\sgm}(j_1,j_2)$,
we could denote by $\overline{\sgm}_{r,0}=\frac{\pa\overline{\sgm}}{\pa j_r}(0,0)$, $\overline{\sgm}_{rs,0}=\frac{\pa^2\overline{\sgm}}{\pa j_r\pa j_s}(0,0)$, and $\overline{\sgm}_{rsm,0}=\frac{\pa^3\overline{\sgm}}{\pa j_r\pa j_s\pa j_m}(0,0)$, and then, for sufficiently small $\varepsilon$ such that $|G|< \varepsilon$, we have
\begin{align*}
	\overline{\sgm}(j_1,j_2)=\sum_{r=1}^2\overline{\sgm}_{r,0}j_r+\sum_{r=1}^2\sum_{s=1}^2\frac{\overline{\sgm}_{rs,0}}{2}j_rj_s+\sum_{r=1}^2\sum_{s=1}^2\sum_{m=1}^2\frac{\overline{\sgm}_{rsm,0}}{6}j_rj_sj_m+o\!\left((|j_1|+|j_2|)^3\right), 
\end{align*}
as  $(j_1,j_2)\to(0,0)$. For terms of third order and higher, the analysis in our proof is straightforward. Therefore, to simplify our proof, we use a simple stored energy function that reveals the core structure. This allows for a clearer understanding of the source of non-uniqueness.
\end{Rmk}
The principal invariants of $C$ can be expressed by $G$ as:
\begin{align}
	j_1 &= \text{tr}C = 2\tr G+\tr(GG^{\top}), \label{def of j_1}\\
	j_2 &=\det(C)= \frac{1}{2} \left( (\tr C)^2 - \tr(C^2) \right) \nonumber	\\
	&= 2 (\tr G)^2+2(\tr G)\tr (GG^{\top})+\frac{1}{2}(\tr (GG^{\top}))^2-\tr\left(G^2+GG^{\top}+2G^2G^{\top}+\frac{1}{2}(GG^{\top})^2\right).\label{def of j_2}
\end{align}
Then, we can represent the first Piola--Kirchhoff stress tensor $P$ as:
\begin{equation}\label{the first Piola stress tensor}
	\begin{aligned}
		P&=\frac{\pa \sgm}{\pa F}=\frac{\pa \sgm}{\pa G}
		=\sum_{i=1}^2\frac{\pa \sgm}{\pa j_i}\frac{\pa j_i}{\pa G}=\left(\sgm_{11}j_1+\sgm_{12}j_2+\frac{\sgm_{111}}{2}j_1^2\right)\frac{\pa j_1}{\pa G}+\left(\sgm_2+\sgm_{12}j_1+\sgm_{22}j_2\right)\frac{\pa j_2}{\pa G},
	\end{aligned} 
\end{equation}
where 
\begin{align*}
	\frac{\pa j_1}{\pa G}&=2(\Id+G),\quad \frac{\pa j_2}{\pa G}=2\tr C(\Id+G)-2C(\Id+G).
\end{align*}
We also consider the second Piola--Kirchhoff stress tensor $\Sgm_G=F^{-1}P$, which is a symmetric tensor. It can be written as follows:
\begin{equation}\label{the second Piola stress tensor}
	\begin{aligned}
		\Sgm_G=F^{-1}P&=\left(\sgm_{11}j_1+\sgm_{12}j_2+\frac{\sgm_{111}}{2}j_1^2\right)F^{-1}\frac{\pa j_1}{\pa G}+\left(\sgm_2+\sgm_{12}j_1+\sgm_{22}j_2\right)F^{-1}\frac{\pa j_2}{\pa G}\\
		&=\underbrace{2\left(\sgm_{11}j_1+\sgm_{12}j_2+\frac{\sgm_{111}}{2}j_1^2\right)}_{\Lambda_1}\Id+\underbrace{2\left(\sgm_2+\sgm_{12}j_1+\sgm_{22}j_2\right)}_{\Lambda_2}(\tr C\Id-D)\\
		&=\Lambda_1\Id+\Lambda_2(\tr C\Id-D).
	\end{aligned} 
\end{equation}
where $D$ is the symmetric matrix defined by
\begin{align*}
	D=G+G^{\top}+G^{\top}G.
\end{align*}
At this point, we can rewrite the original equation as:
\begin{equation}\label{system 1}
	\left\{
	\begin{aligned}
		&\pa_{tt}u-\Div ((\Id+G)\Sgm_G)=0,&& (t,x)\in (0,T)\times\T^2,  \\
		&u(0,x)=u_0(x),\quad \pa_tu(0,x)=u_1(x),&& x\in \T^2.
	\end{aligned}
	\right.
\end{equation}
\begin{Rmk}
From \eqref{the second Piola stress tensor}, we know that the definition of $\Sgm_G$ depends on $G$. The subscript $G$ is used to distinguish the second Piola--Kirchhoff stresses corresponding to different $G$. For example, in the following text, we might use $\Sgm_{G_{q}}$ and $\Sgm_{G_{q,i}}$, which represent the second Piola--Kirchhoff stresses corresponding to $G_{q}$ and $G_{q,i}$, respectively. Here, we avoid using the notation $\Sgm(G)$ to prevent confusion with matrix multiplication.
\end{Rmk}
We next introduce  the definition of a weak solution of this system.
\begin{df}[Weak solution]\label{Weak Solution}
	By a weak solution of \eqref{system 1} on $[0,T]$, we mean a
	function
	$$
	u\in C([0,T];H^1(\T^2))
	\cap C^1([0,T];L^2(\T^2)),
	\qquad
	\pa_{tt}u\in L^2(0,T;H^{-1}(\T^2)),
	$$
	with $G:=\nabla u$, such that
	$
	(\Id+G)\Sgm_G
	\in L^2((0,T)\times\T^2),
	$
	$u(0,\cdot)=u_0$, $\pa_tu(0,\cdot)=u_1$, and
	\begin{align}
		\int_0^T\int_{\T^2}
		\left(
		\pa_tu\cdot\pa_t\eta
		-(\Id+G)\Sgm_G:\nabla\eta
		\right)\rd x\rd t
		=
		-\int_{\T^2}
		u_1(x)\cdot\eta(0,x)\rd x
	\end{align}
	for every
	$\eta\in C_c^\infty([0,T)\times\T^2;\R^2)$.
\end{df}
In this paper, we aim to demonstrate the non-uniqueness of weak solutions that belong to $C^{1}([0,T] \times \T^2)$ of elastodynamic equations as described in \eqref{system 1}. \par

The wave-type equations with null condition have been widely studied by many researchers.  Christodoulou \cite{Chr86} and Klainerman \cite{Kla86} independently proved that, in three space dimensions, nonlinear wave equations satisfying the null condition admit global classical solutions for
sufficiently small smooth initial data. In two space dimensions, Alinhac \cite{Ali01} established
almost-global existence under the quadratic null condition and
global existence when both the quadratic and cubic null conditions
are satisfied. Zha \cite{Zha19} later extended these results to a more general
class of two-dimensional quasilinear wave equations in a unified
framework.

There has also been important work
\cite{John88,Lei15,Lei16,Sid96,Sid00,ST05}
on elastodynamic equations. John \cite{John88} proved almost-global existence for elastic waves
of finite amplitude arising from small initial disturbances. Subsequently, Sideris \cite{Sid96} proved global existence for
three-dimensional isotropic nonlinear elastic waves under a null
condition. In \cite{Sid00}, he established global classical solutions for small
perturbations of prestressed homogeneous dilations under a
nonresonance or null condition.  Lei \cite{Lei16} introduced the strong null condition and proved
that, for sufficiently small initial displacements in a weighted
Sobolev space, the two-dimensional incompressible isotropic
elastodynamic system admits a unique global classical solution. Recently, Zhang and Zhou \cite{ZZ24} identified the wave-maps-type
nonlinearities of incompressible Hookean elastodynamics in
Lagrangian coordinates and iterated them in adapted $U^2$-type
spaces to prove small-data global well-posedness in the critical
Besov space
$\dot{B}^{\frac{n}{2}+1}_{2,1}(\mathbb{R}^n)
\times\dot{B}^{\frac{n}{2}}_{2,1}(\mathbb{R}^n)$ for $n\geq 2$.

Research on low-regularity well-posedness and ill-posedness for
quasilinear wave equations and elastic wave systems has yielded a
series of results, including
\cite{ACY23,ACY24,Lei08,Lin98,ST05A,Wang17,ZhaH20,Zhang24}. For general quasilinear wave equations in three space dimensions,
Smith--Tataru \cite{ST05A} and Wang \cite{Wang17} proved local
well-posedness in
$H^s(\R^3)\times H^{s-1}(\R^3)$ for $s>2$, while Lindblad's
counterexamples \cite{Lin98} show that this threshold is sharp in
general. Zha--Hidano \cite{ZhaH20} considered the Cauchy problem for systems
of three-dimensional quasilinear wave equations satisfying the null
condition with low-regularity initial data. In the radially symmetric case, they proved global existence for
sufficiently small data in $H^3(\R^3)\times H^2(\R^3)$ with a low
weight and applied the result to three-dimensional nonlinear elastic
waves. An--Chen--Yin \cite{ACY23} extended Lindblad's classical
counterexamples \cite{Lin98} from scalar quasilinear wave equations
to the three-dimensional elastic wave system with multiple wave
speeds. They
showed that the Cauchy problem is ill-posed in $H^3(\R^3)$ and that
this ill-posedness is caused by instantaneous shock formation.   Recently, An--Chen--Yu \cite{ACY24} proved low-regularity local
well-posedness for the elastic wave system in admissible harmonic
elastic materials, with the divergence part controlled in $H^{3+}$
and the curl part in $H^{4+}$. They also showed that the $H^{3+}$
assumption is optimal for the divergence part.

It is well known that the Cauchy problem \eqref{system 1}  with small initial data can be essentially reduced to the Cauchy problem of a system of quasilinear hyperbolic equations.  The classical theory of hyperbolic conservation laws
\cite{Daf16,Lax57} demonstrates the importance of entropy
admissibility conditions in selecting physically relevant weak
solutions.  More recently, Chen--Vasseur--Yu \cite{CVY24} showed that a geometric condition
on $2\times2$ systems of conservation laws leads to non-uniqueness
in the class of one-dimensional continuous solutions, demonstrating
that the Liu entropy condition alone is insufficient to guarantee
uniqueness.

Our proof builds on the convex integration method starting from De Lellis--Sz{\'e}kelyhidi \cite{DS09, DS13}.  In \cite{DS09}, they reformulated the incompressible Euler equations
as a differential inclusion and constructed bounded non-unique weak
solutions. In \cite{DS10}, they showed that, for some bounded
compactly supported initial data, the energy admissibility criteria
considered there do not single out a unique weak solution. Following a series of significant advancements \cite{Buc15,BDIS15,BDS16,DS17,DS13,DS14,Isett16}, the well-known Onsager's conjecture, which states that the exponent $\alpha=1/3$ marks the threshold for conservation of energy for weak solutions in the class $C^\alpha$ for the incompressible Euler equations, was resolved by Isett \cite{Isett18}, see Buckmaster--De Lellis--Sz{\'e}kelyhidi--Vicol \cite{BDSV19} for the dissipative case. Recent years have seen significant progress on the non-uniqueness
of dissipative weak solutions to the incompressible Euler equations
and entropy solutions to the compressible Euler equations in
H\"older classes. Among them, De Lellis--Kwon \cite{DK22} constructed $C^\alpha$ weak solutions
of the three-dimensional incompressible Euler equations for every
$\alpha<\frac{1}{7}$ that satisfy the local energy inequality and
strictly dissipate the total kinetic energy. Giri--Kwon \cite{GK22} subsequently established non-uniqueness for
continuous entropy solutions to the isentropic compressible Euler
equations. Convex integration has also led to important developments for the
Navier--Stokes equations, Euler--Maxwell equations, MHD systems,
the Prandtl system, and elastodynamic systems; see
\cite{BBV20,BV19,CHL22,CL22,LQZZ22,LZZ22,LQ20,LX20,MQ24,MY22}.

A series of works
\cite{AKKMM20,BKM21,CDK15Re,CK18,CKMS21,KKMM20}
investigated non-uniqueness of admissible weak solutions to the
compressible Euler equations. 
Chiodaroli--De Lellis--Kreml \cite{CDK15Re} constructed classical
Riemann data for the two-dimensional isentropic compressible Euler
system with pressure law $p(\rho)=\rho^2$ that admit infinitely many
bounded admissible weak solutions. Chiodaroli--Kreml \cite{CK18}
subsequently extended this non-uniqueness framework to further
classes of Riemann data. For the Chaplygin gas,
B\v{r}ezina--Kreml--M\'acha \cite{BKM21} proved non-uniqueness for
Riemann data associated with delta shocks and for a class of data
whose one-dimensional solution consists of two contact
discontinuities.
 
 \subsection{Structure assumptions and the main result}
 We develop a convex integration framework for elastodynamic equations
 in Lagrangian coordinates and use it to prove non-uniqueness of weak
 solutions. The mathematical structure of the elastodynamic system
 differs substantially from that of the Euler and Navier--Stokes
 equations, leading to several new difficulties. We first introduce
 the null condition that underlies the construction and then state
 the main result.
 
\begin{df}[Null condition]\label{Null condition}
	If the stored-energy function $\sgm$ in
	\eqref{stored-energy function} satisfies
	\begin{equation}\label{null condition}
		3\sgm_{11}+2\sgm_{111}=0,
		\qquad
		\sgm^*
		:=2\sgm_{11}+8\sgm_{12}
		+2\sgm_2+4\sgm_{111}\neq0,
	\end{equation}
	then, in this paper, we say that $\sgm$ satisfies the
	null condition.
\end{df}
\begin{Rmk}\label{Comparison with the classical null condition}
	Sideris \cite{Sid96} considered the following null
	condition for the stored-energy function:
	\begin{equation}\label{Sideris null condition}
		\sgm_{11}=\sgm_{12},
		\qquad
		3\sgm_{12}+2\sgm_{111}=0.
	\end{equation}
	Under the assumption $\lambda+\mu\neq0$,
	\eqref{Sideris null condition} implies
	\eqref{null condition}. Indeed,
	$$
	3\sgm_{11}+2\sgm_{111}=0,
	\qquad
	\sgm^*
	=
	2\sgm_{11}+8\sgm_{12}+2\sgm_2+4\sgm_{111}
	=
	4\sgm_{11}+2\sgm_2
	=
	\lambda+\mu
	\neq0.
	$$
	Thus, the condition used in the present paper does not require
	the additional identity $\sgm_{11}=\sgm_{12}$.
	The first identity in \eqref{null condition} provides the
	quadratic cancellation used in the construction of the building
	blocks, whereas the condition $\sgm^*\neq0$ is the corresponding
	nondegeneracy requirement.
	
	In the classical small-data theory of nonlinear elastic waves, null
	conditions cancel the leading quadratic self-interactions along the
	characteristic cones and play a central role in global smooth
	existence. The role of \eqref{null condition} in the present
	periodic setting is different: the first identity cancels the leading
	quadratic self-interaction of a longitudinal wave in one fixed
	direction, which is crucial for the construction of our convex
	integration building blocks.
\end{Rmk}
To sum up, in this paper, all the structure assumptions we consider are given on the stored-energy function as follows
\begin{equation}\label{eq:structural-assumptions}
	\begin{aligned}
		&\sgm_1=0,
		\qquad
		\mu=-2\sgm_2>0,
		\qquad
		\lambda=4\sgm_{11}+4\sgm_2>0,
		\qquad
		\sgm_{11}\sgm_{22}-\sgm_{12}^2>0,
		\\
		&3\sgm_{11}+2\sgm_{111}=0,
		\qquad
		\sgm^*
		=
		2\sgm_{11}+8\sgm_{12}
		+2\sgm_2+4\sgm_{111}
		\neq0.
	\end{aligned}
\end{equation}
\begin{Rmk}
	Based on the building blocks constructed in
	Lemma \ref{construction of building blocks}, we can actually achieve the case where the parameter condition is weaker, namely $\mu = -2\sgm_{2} > 0$, $\lambda + 2\mu = 4\sgm_{11} > 0$, $\lambda + \mu = 4\sgm_{11} + 2\sgm_{2} \neq 0$, and the null condition \eqref{the null condition} holds.
\end{Rmk}

Now we can state our main theorem that implies the non-uniqueness of weak solutions of \eqref{system 1}.
\begin{thm}\label{Thm 1}
	Let $T>0$. Assume that the stored-energy function $\sgm$
	given by \eqref{stored-energy function} satisfies
	\eqref{eq:structural-assumptions}. Then there exists a
	nontrivial weak solution
	$u\in C^1([0,T]\times\T^2)$ of \eqref{system 1} emanating
	from the zero initial data.
\end{thm}

\subsection{Difficulties and strategies}\label{Motivation and difficulties}
In the study of the non-uniqueness of a series of fluid dynamics equations using the convex integration method, it is well known that for the Euler equations, non-uniqueness arises from the second-order nonlinearity in the convection term of the momentum equation. Similarly, the non-uniqueness of elastodynamics in Lagrangian coordinates studied in this paper naturally originates from the nonlinear structure of the equations.  However, in this elastodynamic system, the nonlinear structure is highly complex, incorporating second-order nonlinearities and higher-order nonlinearities.

In the convex integration method, our goal is to construct a sequence of approximate solutions $ u_q $ that satisfy the approximate system and converge to a weak solution $ u $ of \eqref{system 1}.  After adding a perturbation
$(\tu_{q+1},\tG_{q+1}=\nb\tu_{q+1})$ to $(u_q,G_q)$, we set
$
u_{q+1}=u_q+\tu_{q+1},
\
G_{q+1}=G_q+\tG_{q+1}.
$ The corresponding increment of the left-hand side of
\eqref{system 1} can then be decomposed as
\begin{align}
&\pa_{tt}\tu_{q+1}-\Div\left(\tG_{q+1}\Sgm_{G_{q}}+(\Id+G_{q+1})\tS_{q+1}\right)\nonumber\\
=&\underbrace{\pa_{tt}\tu_{q+1}-\Div\left(\tG_{q+1}\Sgm_{G_{q}}+(\Id+G_{q})\tS_{q+1}^{(1)}\right)}_{\text{linear}}-\Div\Big(\underbrace{\tG_{q+1}\tS_{q+1}^{(1)}+(\Id+G_{q})\tS_{q+1}^{(2)}}_{\text{second order}}+\underbrace{\tG_{q+1}\tS_{q+1}^{(2)}+(\Id+G_{q+1})\tS_{q+1}^{(\geqslant 3)}}_{\text{higher order}}\Big),\label{decomposition of left}
\end{align}
where $\tS_{q+1}
:=
\Sgm_{G_{q+1}}-\Sgm_{G_q}
=
\tS_{q+1}^{(1)}
+\tS_{q+1}^{(2)}
+\tS_{q+1}^{(\geqslant3)}$, and  the superscript $ (i) $ indicates that the term is approximately of the form $ G_q^k \tG_{q+1}^i $ and has a magnitude approximately $ \mathcal{O}(|G_q|^k |\tG_{q+1}|^i) $.

Unlike the Euler equations, where the linear part takes the form of a transport equation, the linear part here takes the form of a  linear wave equation. 
The low-frequency components of the main part of second-order terms are still used to cancel out the Reynolds error. For the higher-order terms, it is easy to see that their magnitude is small (of amplitude smaller than $\delta_{q+1}^{\frac{3}{2}}$). Although these terms are smaller, some of the nonprincipal
quadratic and higher-order contributions cannot be directly
represented in the form
$
(\Id+G_{q+1})R_{q+1}
$
with $R_{q+1}$ symmetric. This breaks the iterative framework used in the Euler equations. Therefore, we need a new iterative framework and a redefined inverse of the divergence operator. Below, we will discuss the challenges posed by the above description in detail:

As shown in our related work \cite{MQ25}, if higher-order nonlinearities (third-order and above) are absent, it is possible to achieve non-uniqueness for $C^{1,\alpha}(\alpha>0)$ continuous solutions, which is difficult to achieve for elastodynamic systems. This is primarily due to our treatment of the non-principal part of the second-order terms and third-order and higher nonlinear terms: perturbation terms similar to $ G_q \tG_{q+1}^2 $ are directly absorbed into $ R_{q+1} $. This absorption provides only an $\varepsilon$-level benefit, which limits our choice of $ \delta_q $ in \eqref{def of parameter} that decays exponentially as $q\rightarrow\infty$. In particular, for the linear components satisfying the linear wave equation, we have specific estimates:
\begin{align*}
	\nrm{R_{linear}}_{0}\lesssim\frac{\delta_{q+1}^{\frac{1}{2}}}{\lambda_{q+1}\mu_{q}}.
\end{align*}
For the right-hand side terms of the equation to be smaller than $\delta_{q+2}$, the choice of $\lambda_q$ becomes restricted. This limitation prevents us from achieving the non-uniqueness of weak solutions in $C^{1,\alpha}$ spaces with $\alpha>0$.

The construction of the perturbation is also closely related to the form of the nonlinear terms. The nonlinear terms in \eqref{system 1} differ from those in the Euler equations. In \eqref{system 1}, the nonlinear term $(\Id + G) \Sgm_G$ satisfies \eqref{eq:structural-assumptions} and is not composed solely of second-order nonlinear terms. Higher-order nonlinear terms are also present.  \eqref{eq:structural-assumptions} guarantees a certain geometric structure so that by a corresponding geometric lemma (Lemma \ref{Geometric Lemma}), we can use the low frequency part of $(\tG_{q+1}\tS_{q+1}^{(1)})^{(m)}+\tS_{q+1}^{(2),m}$ to eliminate the Reynolds error $R_q$. Here, the superscripts $(m)$ and $m$ indicate the principal parts. More generally, a term of order $i$ is approximately of the form $\tG_{q+1}^i$ and has magnitude $\mathcal O(|\tG_{q+1}|^i)$. Moreover, it ensures that longitudinal waves are solutions to the equation when only second-order nonlinear terms are present, allowing us to handle low-high frequency terms.  However, the linear combination of longitudinal waves in different directions is not a solution, which produces low-high frequency terms that cannot be eliminated and leads to poor estimates. Therefore, we may only add perturbations in one direction at a time. Consequently, in the 2D case, at each step of the iteration, high-frequency waves need to be added to the approximate solution in three distinct directions (see Section \ref{Outline of the induction scheme}). This process is based on the method first used by Luo and Xin in \cite{LX20}.

Our goal is to use the low-frequency principal part of $(\tG_{q+1}\tS_{q+1}^{(1)})^{(m)}+\tS_{q+1}^{(2),m}$ to cancel $ R_q $, thereby obtaining a smaller $ R_{q+1} $.
According to the geometric lemma, the quadratic low-frequency terms can only represent positive definite matrices. However, we cannot guarantee that $ R_q $ is positive definite at each step. Unlike the incompressible Euler equations, we do not have a pressure term to absorb the residual low-frequency components that cannot be eliminated.  In this paper, we have constructed a suitable sequence of temporal functions $c_q(t)$ (see \eqref{def of c_q} and \eqref{def of te}) in order to ensure that the normalized matrix
$
\Id+\frac{R_q}{\sgm^*\delta_{q+1}}
$
is sufficiently close to $\Id$ and hence positive definite.
The sign of the target matrix is carried by $\sgm^*$.  Although $c_q$ is only a continuous function over the entire interval $[0, T]$, we can ensure that our iterative scheme can proceed by making $c_{q+1}$ constant on the support of $u_q$.

The choice of building blocks is also a challenging task. For the Euler equations, we typically aim to construct perturbations whose principal part satisfies the transport equation to obtain good estimates on the transport error. However, we find that the perturbations used in this paper need to approximately satisfy specific wave equations. Moreover, based on the truncation technique, at the $(q+1)$-st stage, we can treat  $u_{q}$ and its derivatives as constants in the truncation region, so that the equation to be satisfied by the perturbation can be reduced to a linear wave equation with constant coefficients (see Section \ref{Definition of the perturbation}).  Based on this observation, we propose Lemma \ref{construction of building blocks} to construct new building blocks  consisting of a longitudinal wave of size $O(1)$ and a transverse wave of size $O(\varepsilon)$. Notably, we can construct such building blocks only if the difference of the squared wave speeds $\lambda+\mu\neq0$, highlighting the importance of the double wave speeds property.

Directly absorbing these nonprincipal quadratic terms and higher-order terms in \eqref{decomposition of left} into the new Reynolds error at the next iteration would result in $R_{q+1}$ no longer being a symmetric matrix. Consequently, a new iterative framework and a redefined divergence operator inverse are required.
In \cite{Isett16}, Isett and Oh showed that a compactly supported vector
field $U:\R^d\to\R^d$ can be represented as the divergence of a
compactly supported symmetric matrix, provided that its linear and
angular moments vanish:
\begin{align}
	\int_{\R^d} U_l(y) \, \rd y = 0, \quad  \int_{\R^d}  (y_k U_l(y) - y_l U_k(y)) \, \rd y = 0, \label{law of conservation}
\end{align}
then one can formally represent $ U $ as the divergence of some symmetric matrix $R$ by solving the following underdetermined elliptic equation, known as the symmetric divergence equation:
\begin{equation}
	\pa_{y_j} R_{lj} = U_l.\label{symmetric div eq}
\end{equation}

In this paper, we consider two-dimensional elastodynamics in Lagrangian coordinates and derive the corresponding conservation laws for momentum and angular momentum:
\begin{align}
	\int_{\R^2} \rho_0U_l(t,x) \, \rd x = 0, \quad \int_{\R^2} (\rho_0 y_j U_l(t,x) - \rho_0y_l U_j(t,x)) \, \rd x = 0, \label{law of conservation in Lc}
\end{align}
for some vector field $U=U(t,x)\in C_c^{\infty}(\R\times\R^2)$, where $x$ and $y=y(t,x)$ denote the Lagrangian coordinates and the deformation with small displacements $u(t,x)$, respectively. 
Then,  in Section \ref{Defining a Local Inverse of Divergence by Solving the Asymmetric Divergence Equation}, we  prove that under \eqref{law of conservation in Lc}, there exists a symmetric matrix $\tR[U]$ such that
\begin{align*}
	\Div(F\tR[U])=\Div((\Id+\nabla u)\tR[U])=\rho_0U.
\end{align*} 
That is, we can define a special local right inverse of the divergence operator. 
At this point, we present the equation \eqref{approximate  elstro dynamic} satisfied by the approximate solution $u_q$. Moreover,  since the matrices  $\tS_{q+1}^{(2)}$ and  $\tS_{q+1}^{(\geqslant3)}$ are  symmetric,   $(\Id+G_{q+1})(\tS_{q+1}^{(2)}-\tS_{q+1}^{(2),m})$ and the higher-order terms $(\Id+G_{q+1})\tS_{q+1}^{(\geqslant 3)}$ can be easily absorbed into the new Reynolds error $(\Id+G_{q+1})R_{q+1}$. Thus, we only need to focus on and handle the linear and main part of the quadratic terms.

In fact, since we use a local right inverse of the divergence, we need to decompose the new Reynolds error  into a sum of parts with compact support in space, see \eqref{divide of Reynold error 1}. Moreover, to ensure that  $ R_{q+1} $ can be made a symmetric matrix (see Section \ref{Definition of the perturbation}), we need to introduce a perturbation $\tu_{ac}$ such that the new Reynolds error, when decomposed into parts with spatial compact support, satisfies \eqref{law of conservation in Lc}.


We also hope to apply an approach similar to that developed in this paper to prove the non-uniqueness of the compressible elastodynamic equations in the Eulerian coordinate system. Due to the presence of the density term and the more complex relationship between the deformation gradient tensor and the velocity, we anticipate that some difficulties will persist. Addressing these challenges will be the focus of our future work.

\subsection{Organization of the paper}\label{Outline of the proof}
Section \ref{Outline of the induction scheme} introduces the
approximate system, the parameter hierarchy, and the main inductive
proposition. Section \ref{Mollification and cutoffs} describes
the mollification and space--time localization procedures. In
Section
\ref{Defining a Local Inverse of Divergence by Solving the Asymmetric Divergence Equation},
we construct a deformation-dependent local right inverse for the
asymmetric divergence operator. Section \ref{Some notations}
introduces the notation used in the perturbation construction. In
Section \ref{Definition of the perturbation}, we construct the
elastic-wave building blocks and define the perturbations used at
each stage of the iteration. Section
\ref{Estimates on the new Reynolds error and energy} defines and
estimates the new Reynolds error and proves the energy increment.
Finally, Section \ref{Proof of the main theorem} proves the inductive
proposition and the main theorem. The appendices collect the analytic
estimates used throughout the paper.

\section{Outline of the induction scheme}\label{Outline of the induction scheme}
In this paper, we will construct a series of approximate solutions $u_q$ which satisfy the following approximate system and converge to a weak solution $u$ of  \eqref{system 1}. We first give the definition of  the approximate solution.
\begin{df}\label{def of of approximate solution}
	A tuple of tensors $(u,c,R)$ is called an approximate solution tensor of the elastodynamic equations \eqref{system 1}, if it solves the following system in the sense of distributions,
	\begin{equation}
		\pa_{tt}u-\Div((\Id+G)\Sgm_G)=\Div((\Id+G)(c+R)), \label{approximate  elstro dynamic}
	\end{equation}
	where $G=\nb u,\ \Sgm_G$  is defined in \eqref{the second Piola stress tensor}, $c$ and the Reynolds error $R$ are $2\times2$ symmetric matrices.
\end{df}
We denote by $\R^{2\times2}$ the space of $2\times2$ matrices and by $\mS^{2\times2}$ the corresponding subspace of symmetric matrices. Moreover, we set $\nrm{R}:=\max_{i,j}|R_{ij}|$.   For $r>0$ and $K_0\in\mS^{2\times2}$, we define
$
B_r(K_0):=\left\{K\in\mS^{2\times2}:\nrm{K-K_0}<r\right\}.
$
In this paper, we will use the following geometric lemma, which is a two-dimensional version of the one proposed in \cite{DS13}.  
\begin{lm}[Geometric Lemma]\label{Geometric Lemma}
	We can choose a finite subset $\Lambda_*$ of $\Z^2$  with three elements and  the following property. There exist a geometric constant $\tilde{c}_0\leqslant|\sgm^*|$ and smooth functions
	\begin{align*}
		\Gamma_f\in C^\infty(B_{\tilde{c}_0}(\Id)),\quad f\in\Lambda_*,
	\end{align*}
	such that, for every $K\in B_{\tilde{c}_0}(\Id)$,
	\begin{align}
		K=\sum_{i=1}^3(\Gamma_{f_i}(K))^2\left(\Id-\frac{f_i}{|f_i|}\otimes\frac{f_i}{|f_i|}\right).\label{Geometric lemma 1}
	\end{align}
\end{lm}
\begin{proof}
	We choose $\Lambda_*=\{f_1,f_2,f_3\}$, where
	$
	f_1=(4,3),\ f_2=(4,-3),\  f_3=(0,5).
	$
	Notice that
	\begin{align*}
		&\Id-\frac{f_1}{|f_1|}\otimes \frac{f_1}{|f_1|}=\left(\begin{matrix}
			\frac{9}{25}&-\frac{12}{25}\\
			-\frac{12}{25}&\frac{16}{25}\\
		\end{matrix}\right),&&
		\Id-\frac{f_2}{|f_2|}\otimes \frac{f_2}{|f_2|}=\left(\begin{matrix}
			\frac{9}{25}&\frac{12}{25}\\
			\frac{12}{25}&\frac{16}{25}\\
		\end{matrix}\right), &&
		\Id-\frac{f_3}{|f_3|}\otimes \frac{f_3}{|f_3|}=\left(\begin{matrix}
			1&0\\
			0&0\\
		\end{matrix}\right),
	\end{align*}
	and
	\begin{align*}
		\frac{25}{32}\left(\Id-\frac{f_1}{|f_1|}\otimes \frac{f_1}{|f_1|}\right)+\frac{25}{32}\left(\Id-\frac{f_2}{|f_2|}\otimes \frac{f_2}{|f_2|}\right)+\frac{7}{16}\left(\Id-\frac{f_3}{|f_3|}\otimes \frac{f_3}{|f_3|}\right)=\Id.
	\end{align*}
	For $\tilde{c}_0=\min\{\frac{1}{100},|\sgm^*|\}$ and $K\in B_{\tilde{c}_0}(\Id)$, the quantities
	$
	\frac{25}{32}K_{22}-\frac{25}{24}K_{12},
	$
	$
	\frac{25}{32}K_{22}+\frac{25}{24}K_{12},
	$
	and
	$
	K_{11}-\frac{9}{16}K_{22}
	$
	are uniformly positive. We therefore define
	\begin{align*}
		\Gamma_{f_1}(K)=\left(\frac{25}{32}K_{22}-\frac{25}{24}K_{12}\right)^{\frac{1}{2}},\quad
		\Gamma_{f_2}(K)=\left(\frac{25}{32}K_{22}+\frac{25}{24}K_{12}\right)^{\frac{1}{2}},\quad
		\Gamma_{f_3}(K)=\left(K_{11}-\frac{9}{16}K_{22}\right)^{\frac{1}{2}}.
	\end{align*}
	Hence, $\Gamma_{f_i}\in C^\infty(B_{\tilde{c}_0}(\Id))$ for $i=1,2,3$. For $i=1,2,3$, define
	$
	K^{<i>}:=
	\Gamma_{f_i}^2(K)
	\left(
	\Id-\frac{f_i}{|f_i|}
	\otimes
	\frac{f_i}{|f_i|}
	\right).
	$
	Then
	\begin{equation*}
		K=\sum_{i=1}^3K^{<i>}
		=\sum_{i=1}^3\Gamma_{f_i}^2(K)
		\left(
		\Id-\frac{f_i}{|f_i|}
		\otimes
		\frac{f_i}{|f_i|}
		\right).
		\qedhere
	\end{equation*}
\end{proof}
\begin{Rmk}
	Note that for a general symmetric matrix $K$, we still have:
	\begin{align*}
		K= \left( \frac{25}{32} K_{22} - \frac{25}{24} K_{12} \right)\left(\Id-\frac{f_1}{|f_1|}\otimes \frac{f_1}{|f_1|}\right)+ \left( \frac{25}{32} K_{22} + \frac{25}{24} K_{12} \right)\left(\Id-\frac{f_2}{|f_2|}\otimes \frac{f_2}{|f_2|}\right) + \left( K_{11} - \frac{9}{16} K_{22} \right) \left(\Id-\frac{f_3}{|f_3|}\otimes \frac{f_3}{|f_3|}\right).
	\end{align*}
	In applications such as \eqref{def of c_q}, the coefficients in the above decomposition need not be positive, and hence suitable real-valued functions $\Gamma_{f_i}(K)$ may not exist. Nevertheless, the linear decomposition remains valid. We continue to write $K=\sum_{i=1}^3K^{<i>}$, where each $K^{<i>}$ depends linearly on $K$.
\end{Rmk}
For convenience, we introduce the following notation:
\begin{itemize}
	\item $\mcI_{\ell}=\mcI + \ell$ is the concentric enlarged interval $(a-\ell, d+\ell)$ when $\mcI = [a,d]$. For a general $ \mcI \subseteq \mathbb{R} $, we define
	\begin{align*}
	\mcI_{\ell} := \left\{ t' \in \mathbb{R} \mid \exists\ t \in \mcI, \text{ such that } |t - t'| < \ell \right\}.
	\end{align*}
	 Similarly, we use the following notations for sets in $\T^2$
	\begin{align*}
	\Omega_\ell&=\Omega+\ell:=\left\lbrace \tx\in \T^2| \exists\ x \in \Omega, \text{ such that }\  |\tx-x|<\ell\right\rbrace.
	\end{align*}
	 $Q(x,r):=\left\lbrace \tx\in\T^2:|\tx-x|_\infty<r\right\rbrace$  and $B(x,r):=\left\lbrace \tx\in\T^2:|\tx-x|<r\right\rbrace$ denote the square and disk, respectively.
	\item Furthermore, for the sake of convenience, in what follows, we use the notation $A \lesssim_{\kappa} B$ to mean $A \leqslant CB$, where $C > 0$ may depend on some fixed constants or functions $\kappa$. 
	Especially, if we use $A \lesssim_{\sgm} B$, we mean  $C > 0$ may depend on the coefficients of $\sgm$ in \eqref{stored-energy function}.
	Moreover, we use the notation $A\lesssim_{N,r} B$ without pointing out the dependence of the implicit constant $C$, and $N,r\in\N$ can be chosen to have $N\leqslant N^*, r\leqslant r^*$ for some positive integers $N^*,r^*$.    We will not repeatedly specify this.
	\item We use the following notation for the skew-symmetric product
	and curl operator in $\R^d$. For $a,b\in\R^d$, we define
	\begin{align*}
		(a\times b)_{jl}
		&:=a_jb_l-a_lb_j,
		\qquad 1\leqslant j,l\leqslant d.
	\end{align*}
	For a smooth vector field $a:\R^d\to\R^d$, we define
	\begin{align*}
		(\nabla\times a)_{jl}
		&:=\partial_ja_l-\partial_la_j,
		\qquad 1\leqslant j,l\leqslant d.
	\end{align*}
	Thus, $a\times b$ and $\nabla\times a$ are skew-symmetric matrices.
	When $d=2$, we identify them with their $(1,2)$-components:
	\begin{align*}
		a\times b
		&=a_1b_2-a_2b_1,
		\qquad
		\nabla\times a
		=\partial_1a_2-\partial_2a_1.
	\end{align*}
		\item Whenever a summation arises from repeated applications of the Leibniz rule, we understand the summation symbol to include the corresponding standard binomial or multinomial coefficients. This convention applies to the distributions of both spatial and time derivatives among the factors and will be used throughout the paper unless otherwise stated.
\end{itemize}

We next introduce parameters controlling the frequencies and amplitudes of the approximate solutions:
\begin{equation}
	\begin{aligned}
		&\lambda_q=\lceil\varepsilon^{-b(q+3)^2/3}\rceil^3, \quad\lambda_{q,i}=\lambda_{q}^{1-\frac{i}{3}}\lambda_{q+1}^{\frac{i}{3}},
		\quad\dlt_q=\varepsilon^{2(q+3)/3},&& 0\leqslant i\leqslant 3,\quad q\in\N,
 \end{aligned}\label{def of parameter}
\end{equation}
where $\lceil x\rceil$ denotes the smallest integer $\tilde{n}\geqslant x$, $0<\varepsilon\ll1$ is a small parameter close to $0$, and $b>0$ is a positive parameter. 

Because of smoothing and truncation, the time domains on which the approximate solutions $u_{q,i}$ are defined change at each stage. We therefore introduce parameters describing the smoothing and truncation scales:
\begin{equation}\label{def of parameter 1}
	\begin{aligned}
	\ell_{q,i}&=\left(\lambda_{q,i}\lambda_{q,i+1}\right)^{-\frac{1}{2}}\dlt_{q}^{-\frac{1}{4}},&& \tau_{q,i}=\left(\lambda_{q,i}\lambda_{q,i+1}\right)^{-\frac{1}{2}}\dlt_{q}^{-\frac{1}{4}},&& \mu_{q,i}=(2\lceil(\lambda_{q,i}\lambda_{q,i+1})^{\frac{1}{2}}\dlt_{q}^{\frac{1}{4}}/2\rceil)^{-1}, && 0\leqslant i\leqslant 2,\quad q\in\N.
	\end{aligned}
\end{equation}
Set $\tau_{0,-1}=10\tau_{0,0}$ and $\tau_{q,-1}=\tau_{q-1,2}$ for $q\geqslant1$. We consider the approximate solution $u_{q,i}$ at stage $q$ on the time interval $\mcI^{q,i-1}=[-\tau_{q,i-1},T+\tau_{q,i-1}]$.
We use the following notation for the support of a function:
\begin{align*}
	&\supp_{t}^{(q,i)} f = \overline{\set{t\in \mcI^{q,i-1}|\exists x\ \text{with}\ f(t,x)\neq0}},&&\supp_{x} f = \overline{\set{x\in \T^2|f(x)\neq 0}}.
\end{align*}
	In what follows, at the $i$-th substep of step $q$, for convenience, and where no confusion arises, we omit the superscript $(q,i)$. 
	Moreover, we denote $\supp_t(u,R)=\supp_tu\bigcup\supp_{t}R$.
\begin{Rmk}
	Our primary focus is on the properties of approximate solutions on $[0, T]$. Therefore, we only need to consider approximate solution $u_{q,i}$ in a small extension  $\mcI^{q,i-1}$ around $[0, T]$.
\end{Rmk}	
	It is easy to find a nonnegative smooth function
	$\theta^*(t)\in C^\infty(\R)$, chosen as the square of a
	smooth cutoff function, such that
	$
	0\leqslant\theta^*\leqslant1,
	\
	(\theta^*)^{\frac12}\in C^\infty(\R),
	$
	and
	\begin{equation}\label{def of theta^*}
	\theta^*(t)=\left\lbrace
	\begin{aligned}
	&1,&&t\in[1,\infty),\\
	&0,&&t\in(-\infty,0].
	\end{aligned}\right.
	\end{equation}
	Then, we will give two important continuous matrix-valued functions $c_q(t)$ and $c_{q,i}(t)$, which will determine the behavior of approximate solutions.
	\begin{align}
		&c_{q}(t)=\sgm^*\sum_{j=q}^{\infty}\sum_{k=0}^{2}\te_{j,k}(t)\Id^{<k+1>},\quad c_{q,i}(t)=c_{q+1}(t)+\sgm^*\sum_{k=i}^{2}\te_{q,k}(t)\Id^{<k+1>}, 
		&& 0\leqslant i\leqslant 2, \label{def of c_q}
	\end{align}
	where
	\begin{equation}\label{def of te}
		\te_{q,k}(t)=\left\lbrace
	\begin{aligned}
		&\delta_{q+1},&&t\in (T/2-T_{q,k}-\ell_{q,k},T+\tau_{0,-1}],\\
		&\delta_{q+1}\theta^*(\lambda_{q}(t-(T/2-T_{q,k}-\ell_{q,k}-\lambda_{q}^{-1}))),&&t\in [T/2-T_{q,k}-\ell_{q,k}-\lambda_{q}^{-1},T/2-T_{q,k}-\ell_{q,k}],\\
		&0,&&t\in[-\tau_{0,-1},T/2-T_{q,k}-\ell_{q,k}-\lambda_{q}^{-1}),
	\end{aligned}\right.
	\end{equation}
	and $T_{q,k}:=\sum_{j=0}^{q-1}\sum_{i=0}^2(3\tau_{j,i}+\ell_{j,i}+\lambda_{j}^{-1})+\sum_{i=0}^{k-1}(3\tau_{q,i}+\ell_{q,i}+\lambda_{q}^{-1})$. For convenience, we set $c_{q,3}:=c_{q+1}$.
	
At each stage, we correct $u_q$ so that the errors $c_q$ and $R_q$ decrease and converge to zero in $C^0$ as $q\to\infty$.
Moreover, we assume the following  inductive estimates on $(u_{q},c_q,R_{q})$ satisfying \eqref{approximate  elstro dynamic}:
\begin{enumerate}
	\item The supports of $u_{q}$ and $R_{q}$ satisfy
	\begin{align}
		&\supp_t(u_q,R_q) \subseteq\supp_{t} (u_{q-1},R_{q-1})+\sum_{k=0}^{2}(3\tau_{q-1,k}+\ell_{q-1,k}+\lambda_{q-1}^{-1})\subseteq[T/2-T_{q,0},T+\tau_{q,-1}],&&q\geqslant 1. \label{pp of supp q}
	\end{align}
	\item The estimates on $u_{q}$ and $R_{q}$:
\begin{align}
		\nrm{\pa_t^ru_q}_{N}&\leqslant 2\varepsilon-\dlt_q^{\frac{1}{2}}, && 0\leqslant N+r\leqslant 1,&&q\in\N,\label{est on u_q low}\\ \nrm{\pa_t^ru_q}_{N}&\leqslant \oM\lambda_{q}^{N+r-1}\dlt_{q}^{\frac{1}{2}}, && 2\leqslant N+r\leqslant 3,&&q\in\N, \label{est on u_q high}\\
		\nrm{\pa_t^rR_q}_{N}&\leqslant \tilde{c}_0^4\lambda_{q}^{N+r}\dlt_{q+1}, && 0\leqslant N+r\leqslant 2,&&q\in\N, \label{est on R_q}
	\end{align}
	where  $\nrm{\cdot}_N=\nrm{\cdot}_{C^0(\mcI^{q,-1};C^N(\T^2))}$, and $\oM>1$  depends on the coefficients of stored-energy function $\sgm$ in \eqref{stored-energy function}.
\end{enumerate}
\begin{Rmk}
	In the following text, when a parameter depends on the function $\sgm$, it refers to properties related to the structure of the function $\sgm$ (such as its coefficients), and not to the function value of $\sgm$. 
\end{Rmk}

Similar to \cite{DK22,GK22}, we give the following inductive proposition.
\begin{pp}[Inductive proposition]\label{Proposition 1}
	Let $\sgm=\sgm(j_1,j_2)$ be a stored-energy function defined by \eqref{stored-energy function} and satisfying \eqref{eq:structural-assumptions}. Then there exist constants $b>5$, $\oM>1$, and $\tilde{c}_0\leqslant\min\{\frac{1}{100},|\sgm^*|\}$, together with a sufficiently small constant $\varepsilon_1=\varepsilon_1(b,\oM)>0$, such that the following holds. Let $0<\varepsilon<\varepsilon_1$ and suppose that $(u_q,c_q,R_q)$ is an approximate solution of \eqref{approximate  elstro dynamic} on $[-\tau_{q,-1},T+\tau_{q,-1}]$ satisfying \eqref{pp of supp q}--\eqref{est on R_q}.
	Then there exists a corrected approximate solution $(u_{q+1},c_{q+1},R_{q+1})$ of \eqref{approximate  elstro dynamic}, defined on $[-\tau_{q+1,-1},T+\tau_{q+1,-1}]$ and satisfying \eqref{pp of supp q}--\eqref{est on R_q} with $q$ replaced by $q+1$. In addition,
	\begin{equation}\label{Induction estimate}
		\begin{aligned}
			\lambda_{q}\nrm{u_{q+1}-u_q}_{C^0([0, T] ; C^0(\T^2))}+\sum_{1\leqslant N+r\leqslant3}\lambda_{q+1}^{1-N-r}\nrm{\pa_{t}^r(u_{q+1}-u_q)}_{C^0([0, T] ; C^N(\T^2))}\leqslant \oM \dlt_{q+1}^{\frac{1}{2}}.
		\end{aligned}
	\end{equation}
	Moreover, for $t\in[T/2-T_{q,0},T+\tau_{q+1,-1}]$, the following energy estimate holds:
	\begin{align}
	3\pi^2\sgm_{11}\delta_{q+1}\leqslant\int_{\T^2}\left(\frac{1}{2}|\pt u_{q+1}|^2+\sgm_{G_{q+1}}\right)\rd x-\int_{\T^2}\left(\frac{1}{2}|\pt u_{q}|^2+\sgm_{G_q}\right)\rd x\leqslant 120\pi^2\sgm_{11}\delta_{q+1}.\label{Induction energy estimate}
	\end{align}
\end{pp}
More precisely, at each iteration stage we add high-frequency waves in three successive substeps. Set $u_{q,0}=u_q$ and $R_{q,0}=R_q$. At the $i$-th substep, $1\leqslant i\leqslant3$, we add the correction $\tu_{q,i}$, with frequency $\lambda_{q,i}$, to $u_{q,i-1}=u_q+\sum_{r=1}^{i-1}\tu_{q,r}$. The new Reynolds stress $R_{q,i}$ then takes the form
\begin{align*}
	R_{q,i}=\sgm^*\te_{q,i-1}\Id^{<i>}+R_{q,i-1}-\sgm^*\left(\te_{q,i-1}^{\frac{1}{2}}(t)\Gamma_{f_i}\left(\Id+\frac{R_{q,i-1}}{\sgm^*\te_{q,i-1}}\right)\right)^2\left(\Id-\frac{f_i}{|f_i|}\otimes \frac{f_i}{|f_i|}\right)+\dlt R_{q,i}, \quad 1\leqslant i\leqslant 3,\quad q\in\N.
\end{align*}
We set $u_{q+1}=u_{q,3}$ and $R_{q+1}=R_{q,3}$. Repeating this construction yields a sequence of approximate solutions whose Reynolds stresses converge to zero.  Moreover, we assume the following  inductive estimates on $(u_{q,i},c_{q,i},R_{q,i})$ :
\begin{enumerate}
	\item The supports of $u_{q,i}$ and $R_{q,i}$ satisfy
	\begin{align}
		\supp_t(u_{q,i},R_{q,i})
		\subseteq \supp_t(u_{q,i-1},R_{q,i-1})
		+3\tau_{q,i-1}+\ell_{q,i-1}+\lambda_q^{-1}
		\subseteq [T/2-T_{q,i},T+\tau_{q,i-1}],
		\qquad 1\leqslant i\leqslant3.
		\label{pp of supp q,i}
	\end{align}
	\item The estimates on $u_{q,i}$ and $R_{q,i}$:
	\begin{align}
		\nrm{\pa_t^ru_{q,i}}_{N}&\leqslant 2\varepsilon-(4-i)\dlt_{q+1}^{\frac{1}{2}},&& 0\leqslant N+r\leqslant 1,&& 1\leqslant i\leqslant 3,&&q\in\N,\label{est on u_qi low}\\ \nrm{\pa_t^ru_{q,i}}_{N}&\leqslant \oM\lambda_{q,i}^{N+r-1}\dlt_{q+1}^{\frac{1}{2}}, && 2\leqslant N+r\leqslant 3, && 1\leqslant i\leqslant 3,&&q\in\N, \label{est on u_qi high}\\
		\nrm{\pa_t^r\dlt R_{q,i}}_{N}&\leqslant \tilde{c}_0^4\lambda_{q,i}^{N+r}\dlt_{q+2}, && 0\leqslant N+r\leqslant 2,&& 1\leqslant i\leqslant 3,&&q\in\N, \label{est on dR_qi}\\
		\nrm{\pa_t^r R_{q,i}}_{N}&\leqslant \frac{i+1}{4} \tilde{c}_0^4\lambda_{q,i}^{N+r}\dlt_{q+1}, && 0\leqslant N+r\leqslant 2,&& 1\leqslant i\leqslant 2,&&q\in\N, \label{est on R_qi}
	\end{align}
	where  $\nrm{\cdot}_N=\nrm{\cdot}_{C^0(\mcI^{q,i-1};C^N(\T^2))}$.
	\item $\te_{q,k}=\te_{q,k}(t)$ satisfies
	\begin{equation}\label{pp of te_q,i}
		\begin{aligned}
			&\supp_t(u_{q,k},R_{q,k})+\ell_{q,k}
			\subseteq\supp_t\te_{q,k}
			\subseteq
			\supp_t(u_{q,k},R_{q,k})
			+\ell_{q,k}+\lambda_q^{-1},
			&&0\leqslant k\leqslant2,
			&&q\in\N,
			\\
			&\te_{q,k}(t)=\delta_{q+1},
			\quad
			\forall t\in
			\supp_t(u_{q,k},R_{q,k})+\ell_{q,k},
			&&0\leqslant k\leqslant2,
			&&q\in\N,
			\\
			&\nrm{\te_{q,k}}_{C_t^r}
			\lesssim_r
			\lambda_q^r\delta_{q+1},
			\qquad r\geqslant0,
			&&0\leqslant k\leqslant2,
			&&q\in\N.
		\end{aligned}
	\end{equation}
\end{enumerate}
The proof of Theorem \ref{Thm 1}  relies on Proposition \ref{Proposition 1} which provides an iterative framework for constructing a sequence of approximate solutions to \eqref{system 1}, all of which emanate from zero initial data. Specifically, it states that, given an initial tuple $(0, c_0,0)$, we can generate a sequence of tuples $(u_q,c_q, R_q)$ that satisfy \eqref{pp of supp q}--\eqref{est on R_q}, as well as  \eqref{Induction estimate}--\eqref{Induction energy estimate} in Proposition \ref{Proposition 1} for all $q \geq 0$. Consequently, we can demonstrate that $u_q$ converges to $u$, which is a weak solution of \eqref{system 1} in $C^{1}([0,T] \times \T^2)$ emanating from zero initial data.

\section{Mollification and cutoffs}\label{Mollification and cutoffs}
To address the loss of temporal and spatial derivatives, in this section we introduce the mollification of $(u_{q,i},c_{q,i},R_{q,i})$ at stage $q$.
Following the mollification method in \cite{Isett17,NV23} and Definition 4.15 in \cite{GKN24}, we state the following lemma without proof.
\begin{lm}\label{lm: Mollification}
There exist smooth functions $\eta^t\in C_c^\infty(\R)$ and $\eta\in C_c^\infty(\R^2)$ such that $\supp\eta^t\subset (-\frac{1}{2},\frac{1}{2}), \supp\eta\subset Q(0,\frac{1}{2})$, and for every $n\in\{1,2,3,4\}$ and every multi-index $\alpha$ with $1\leqslant|\alpha|\leqslant4$,
\begin{align}
	\int_{\R}\eta^t(\tau)\rd\tau&=1, \quad \int_{\R}\eta^t(\tau)\tau^n\rd\tau=0, \quad \forall n=1,2,3,4,\\
	\int_{\R^2}\eta(x)\rd x&=1, \quad \int_{\R^2}\eta(x)x^\alpha\rd x=0, \quad \forall \alpha, 1\leqslant|\alpha|\leqslant 4.
\end{align} 

 Let $\ell>0$ and define $\eta_\ell^t(t)=\frac{1}{\ell}\eta^t\left(\frac{t}{\ell}\right)$, $\eta_\ell(x)=\frac{1}{\ell^2}\eta\left(\frac{x}{\ell}\right)$. If $f$ is a spatially periodic function on $\mcI\times\T^2\subset\R\times\T^2$, $\eta_\ell\ast f$ is also a spatially periodic function on $\mcI\times\T^2$. Moreover, for any smooth function $f \in C^{\infty}(\mcI+\ell;C^{\infty}(\T^2))$,  there exists a constant $C$ such that for any $1\leqslant N\leqslant 5$,
 \begin{align}
 	\nrm{\pa_t^N(f\ast\eta^t_\ell)}_{C^0(\mcI,C^0(\T^2))}&\leqslant C\ell^{-N}\nrm{f}_{C^0(\mcI+\ell,C^0(\T^2))},\\
 	\nrm{f\ast\eta_\ell}_{C^0(\mcI,C^N(\T^2))}&\leqslant C\ell^{-N}\nrm{f}_{C^0(\mcI,C^0(\T^2))},\\
	\nrm{f-f\ast\eta^t_\ell}_{C^0(\mcI,C^0(\T^2))}&\leqslant C\ell^{N}\nrm{\pa_t^Nf}_{C^0(\mcI+\ell,C^0(\T^2))},\\
	\nrm{f-f\ast\eta_\ell}_{C^0(\mcI,C^0(\T^2))}&\leqslant C\ell^{N}\nrm{f}_{C^0(\mcI,C^N(\T^2))}.
 \end{align}
\end{lm}
We now define the spatial and temporal convolution operators by
\begin{align*}
P_{\leqslant\ell^{-1}}f:=\eta_\ell\ast f,\quad U_{\leqslant\ell^{-1}}f:=\eta^t_\ell\ast f,\quad P_{>\ell^{-1}}f:=f-P_{\leqslant\ell^{-1}}f,\quad U_{>\ell^{-1}}f:=f-U_{\leqslant\ell^{-1}}f.
\end{align*}
\begin{Rmk}
If we set $\ell=0$ when using $\PL$ or $\UL$, we mean $\PL f=f$ and $\UL f= f$.
\end{Rmk}
Next, we introduce the mollification length $\ell_{q,i}$ defined in \eqref{def of parameter 1}:
\begin{equation}\label{def of l}
	\ell_{q,i}=\left(\lambda_{q,i}\lambda_{q,i+1}\right)^{-\frac{1}{2}}\dlt_{q}^{-\frac{1}{4}}, \quad 0\leqslant i\leqslant 2,\quad q\in\N,
\end{equation} 
and give the regularized terms as 
\begin{align*}
\ul&=\ULq\PLq u_{q,i},&&\Gl=\ULq\PLq G_{q,i}=\nb \ul,&&  \Rl=\ULq\PLq R_{q,i},  && 0\leqslant i\leqslant 2,\\
\Cl&=\Gl+\Gl^{\top}+\Gl\Gl^{\top},&&\Dl=\Gl+\Gl^{\top}+\Gl^{\top}\Gl,&& && 0\leqslant i\leqslant 2,
\end{align*}
which are well defined on $\mcI^{q,i}+3\ell_{q,i}\subset\mcI^{q,i-1}$ provided that $\varepsilon$ is sufficiently small.  Notice that $c_{q,i}$ is actually a constant matrix in $\supp_t(u_{q,i},R_{q,i})+\ell_{q,i}$. Then, $\ul$ satisfies
\begin{align}
\pa_{tt}\ul-\Div\left((\Id+\Gl)\Sgm_{\Gl}\right)&=\Div((\Id+\Gl)(c_{q,i}+R_{\ell,i})+R_{m,i}),
\end{align}
where
\begin{equation}\label{def of R_mi}
\begin{aligned}
	R_{m,i}:=&\PLq\ULq((\Id+G_{q,i})(\Sgm_{G_{q,i}}+R_{q,i}))-(\Id+\Gl)(\Sgm_{\Gl}+R_{\ell,i}).
\end{aligned}
\end{equation}
The following estimates follow directly.
\begin{pp}\label{pp of Mollification}
For every $b>1$, there exists $\varepsilon_1^*=\varepsilon_1^*(b,\oM)>0$ such that, whenever $0<\varepsilon<\varepsilon_1^*$, the following properties hold for $0\leqslant i\leqslant2$ and $q\in\N$:
\begin{enumerate}
	\item The supports of $\ul$, $R_{m,i}$, and $\Rl$ satisfy
\begin{align} 
	\supp_{t}(\ul,\Rl,R_{m,i})&\subseteq\supp_{t}(u_{q,i},R_{q,i})+\ell_{q,i}.\label{pp of supp l,i}
\end{align}
	\item The estimates on $\ul$ and $\Rl$:
	\begin{align}
		\nrm{\pa_t^r\ul}_N&\leqslant 2\varepsilon,&& N+r\leqslant 1, \label{est on u_li 0}\\
		\nrm{\pa_t^r\ul}_N&\lesssim_{N,r}\ell_{q,i}^{2-N-r}\oM\lambda_{q,i}\dlt_{q,i}^{\frac{1}{2}}\lesssim_{N,r}\ell_{q,i}^{1-N-r}\dlt_{q,i}^{\frac{1}{2}},&& N+r\geqslant2, \label{est on u_li N}\\
		\nrm{\pa_t^rR_{\ell,i}}_N&\leqslant \tilde{c}_0^4\lambda_{q,i}^{N+r}\dlt_{q+1},&& N+r\leqslant2, \label{est on R_li low}\\
		\nrm{\pa_t^rR_{\ell,i}}_N&\lesssim_{N,r}\ell_{q,i}^{2-N-r}\tilde{c}_0^4\lambda_{q,i}^2\dlt_{q+1},&& N+r\geqslant3, \label{est on R_li high}\\
		\nrm{\pa_t^rR_{m,i}}_{N}&\lesssim_{\sgm,N,r}\ell_{q,i}^{2-N-r}\oM^2\lambda_{q,i}^2\dlt_{q,i},&& N+r\geqslant0, \label{est on R_mi}\\
		\nrm{\pa_t^r(u_{q,i}-\ul)}_N&\lesssim\ell_{q,i}^{3-N-r}\oM\lambda_{q,i}^2\dlt_{q,i}^{\frac{1}{2}}\leqslant\ell_{q,i}^{-N-r}\frac{\oM}{120}\lambda_{q,i}^{-1}\dlt_{q,i}^{\frac{1}{2}},&&  N+r\leqslant3,\label{est on u_qi-u_li}\\
		\nrm{\pa_t^r(R_{q,i}-\Rl)}_N&\lesssim\ell_{q,i}^{2-N-r}\tilde{c}_0^4\lambda_{q,i}^{2}\dlt_{q+1},&&  N+r\leqslant2.\label{est on R_qi-R_li}
	\end{align}
	\item The kinetic energy and the stored-energy satisfy
	\begin{align}
		\nrm{\pa_t^r(|\pt u_{q,i}|^2-|\pt\ul|^2)}_{N}&\lesssim_{\oM}\ell_{q,i}^{2-N-r}\lambda_{q,i}^2\dlt_{q,i}^{\frac{1}{2}},&&N+r\leqslant2,\label{est on u_qi^2-u_li^2}\\
		\nrm{\pa_t^r(\sgm_{\Gl}-\sgm_{G_{q,i}})}_{N}&\lesssim_{\sgm,\oM}\ell_{q,i}^{2-N-r}\lambda_{q,i}^{2}\dlt_{q,i}^{\frac{1}{2}},&& N+r\leqslant2.\label{est on sgm_qi-sgm_li}
	\end{align}
	\end{enumerate}
		where   $\nrm{\cdot}_N=\nrm{\cdot}_{C^0(\mcI_{3\ell_{q,i}}^{q,i};C^N(\T^2))}$, and
	\begin{align}
		\delta_{q,i}=\left\lbrace\begin{aligned}
			&\delta_{q},&& i=0,\\
			&\delta_{q+1},&& i=1,2.\\
		\end{aligned}\right.
	\end{align}
\end{pp}
\begin{proof}
Notice that for $b>1$, we find sufficiently small $\varepsilon^*_1=\varepsilon^*_1(b,\oM)$ such that for any $\varepsilon<\varepsilon^*_1$,
\begin{align}
	\lambda_{q}\dlt_{q}^{\frac{1}{2}}\leqslant\lambda_{q,i}\dlt_{q,i}^{\frac{1}{2}}, \quad \tau_{q,i}+5\ell_{q,i}\leqslant \tau_{q,i-1},\quad 120\oM\lambda_{q,i}\leqslant \ell_{q,i}^{-1}=\tau_{q,i}^{-1}\approx\mu_{q,i}^{-1}\leqslant\lambda_{q,i+1}\delta_{q+1}^{\frac{1}{2}}\leqslant\lambda_{q,i+1},\quad 0\leqslant i\leqslant 2.	\label{pp of Lambda1}	
\end{align}
 By using \eqref{est on u_q low}--\eqref{est on R_q}, \eqref{est on u_qi low}--\eqref{est on R_qi},  and  the definition of $\PLq$ and $\ULq$, we obtain \eqref{est on u_li 0}--\eqref{est on R_li high}. Next, we could calculate
$$
\begin{aligned}
	F-\ULq\PLq F=F-\PLq F+\PLq F-\ULq\PLq F=\PGq F+\UGq\PLq F,
\end{aligned}
$$ 
and we could use  Lemma \ref{lm: Mollification} to obtain
\begin{align}
	\nrm{\PGq F}_{0}\lesssim \ell_{q,i}^j\nrm{\nb^{j}F}_{0}, \quad
	\nrm{\UGq F}_{0}\lesssim \ell_{q,i}^j\nrm{\pa_{t}^jF}_{C^0(\mcI^{q,i-1}\times\T^2)}, \label{Bernstein inequality}
\end{align}
for $\forall F\in  C^j(\mcI^{q,i-1}\times\T^2),j=0,1,2,3$. Combining them, we have for $N+r\leqslant k\leqslant 3$,
\begin{equation}\label{est on F-PULF}
	\begin{aligned}
		\nrm{\pa_{t}^r(F-\ULq\PLq F)}_{N}&\lesssim\ell_{q,i}^{k-N-r}\nrm{F}_{C^k(\mcI^{q,i-1}\times\T^2)}.
	\end{aligned}
\end{equation}
We can apply it to $u_{q,i}$ and $R_{q,i}$ to get \eqref{est on u_qi-u_li} and \eqref{est on R_qi-R_li}.
Recalling the definition of $R_{m,i}$ in \eqref{def of R_mi}, we have 
\begin{align*}
	R_{m,i}=&\PLq\ULq((\Id+G_{q,i})(\Sgm_{G_{q,i}}+R_{q,i}))-(\Id+\Gl)(\Sgm_{\Gl}+R_{\ell,i})\\
	=&\PLq\ULq((\Id+G_{q,i})(\Sgm_{G_{q,i}}+R_{q,i}))-\PLq\ULq(\Id+G_{q,i})\PLq\ULq(\Sgm_{G_{q,i}}+R_{q,i})\\
	&+(\Id+\Gl)(\PLq\ULq(\Sgm_{G_{q,i}})-\Sgm_{\Gl}).
\end{align*}
Notice that
\begin{align*}
\Sgm_{G_{q,i}}=&(2(\sgm_{11}+\sgm_2)\tr C_{q,i}+(\sgm_{111}+3\sgm_{12})(\tr C_{q,i})^2+\sgm_{22}(\tr C_{q,i})^3)\Id\\
&-(\sgm_{12}\tr C_{q,i}^2+\sgm_{22}\tr C_{q,i}\tr C_{q,i}^2)\Id-2(\sgm_2+\sgm_{12}\tr C_{q,i}+\frac{\sgm_{22}}{2}((\tr C_{q,i})^2-\tr C_{q,i}^2))D_{q,i},\\
\Sgm_{\Gl}=&(2(\sgm_{11}+\sgm_2)\tr \Cl
+(\sgm_{111}+3\sgm_{12})(\tr \Cl)^2
+\sgm_{22}(\tr \Cl)^3)\Id\\
&-(\sgm_{12}\tr \Cl^2
+\sgm_{22}\tr \Cl\tr \Cl^2)\Id
-2(\sgm_2+\sgm_{12}\tr \Cl
+\frac{\sgm_{22}}{2}((\tr \Cl)^2-\tr \Cl^2))\Dl,
\end{align*}
it is easy to get
\begin{align*}\nonumber
	\nrm{\pt^r\Sgm_{G_{q,i}}}_N\lesssim&\sum_{k=1}^6\sum_{\sum_nr_n=r}\sum_{\sum_nN_n=N}\prod_{n=1}^k\nrm{\pt^{r_n}G_{q,i}}_{N_n}\lesssim\oM\lambda_{q,i}^{N+r}\delta_{q,i}^{\frac{1}{2}},&&1\leqslant N+r\leqslant 2,\\
	\nrm{\pt^r\Sgm_{\Gl}}_N\lesssim&\sum_{k=1}^6\sum_{\sum_nr_n=r}\sum_{\sum_nN_n=N}\prod_{n=1}^k\nrm{\pt^{r_n}\Gl}_{N_n}\lesssim\oM\lambda_{q,i}^{N+r}\delta_{q,i}^{\frac{1}{2}},&&1\leqslant N+r\leqslant 2,\\	\nrm{\pt^r\Sgm_{\Gl}}_N\lesssim&\sum_{k=1}^6\sum_{\sum_nr_n=r}\sum_{\sum_nN_n=N}\prod_{n=1}^k\nrm{\pt^{r_n}\Gl}_{N_n}\lesssim\oM\ell_{q,i}^{2-N-r}\lambda_{q,i}^{2}\delta_{q,i}^{\frac{1}{2}},&&3\leqslant N+r.
\end{align*}
Next, we could use Lemma \ref{lm of commutator} to get 
\begin{align*}
&\quad\Nrm{\pt^r\left(\PLq\ULq((\Id+G_{q,i})(\Sgm_{G_{q,i}}+R_{q,i}))-\PLq\ULq(\Id+G_{q,i})\PLq\ULq(\Sgm_{G_{q,i}}+R_{q,i})\right)}_{N}\\
&\lesssim\ell_{q,i}^{2-N-r}\left(\Nrm{\pt G_{q,i}}_{C^0(\mcI^{q,i-1};C^0(\T^2))}\Nrm{\pt(\Sgm_{G_{q,i}}+R_{q,i})}_{C^0(\mcI^{q,i-1};C^0(\T^2))}+\Nrm{G_{q,i}}_{C^0(\mcI^{q,i-1};C^1(\T^2))}\Nrm{\Sgm_{G_{q,i}}+R_{q,i}}_{C^0(\mcI^{q,i-1};C^1(\T^2))}\right)\\
&\lesssim_{N,r}\ell_{q,i}^{2-N-r}\oM^2\lambda_{q,i}^2\dlt_{q,i}.
\end{align*}
So we only need to get estimate for $(\Id+\Gl)(\PLq\ULq(\Sgm_{G_{q,i}})-\Sgm_{\Gl})$. By using Lemma \ref{lm of commutator}, we could obtain
\begin{equation}\label{est on sgm_q-sgm_l}
	\begin{aligned}
	\Nrm{\pt^r(\PLq\ULq\Sgm_{G_{q,i}}-\Sgm_{\Gl})}_N\lesssim&_{\sgm,N,r}\ell_{q,i}^{2-N-r}\sum_{k=2}^6\sum_{\substack{s_1+\cdots+s_k=2\\s_n<2,n=1,\cdots,k}}\left(\prod_{n=1}^k\|G_{q,i}\|_{C^0(\mcI^{q,i-1};C^{s_n}(\T^2))}+\prod_{n=1}^k\|\pt^{s_n}G_{q,i}\|_{C^0(\mcI^{q,i-1};C^{0}(\T^2))}\right)\\
	\lesssim&_{\sgm,N,r}\ell_{q,i}^{2-N-r}\oM^2\lambda_{q,i}^{2}\delta_{q,i},
	\end{aligned}
\end{equation}
where we have used \eqref{est on u_qi low} and \eqref{est on u_qi high}. 
\begin{Rmk}
Noticing that the linear terms similar to $\PLq\ULq G_{q,i} - \Gl=0$ vanish in $\PLq\ULq \Sgm_{G_{q,i}} - \Sgm_{\Gl}$, the remaining terms are quadratic or higher-order terms of the form $\PLq\ULq (G_{q,i}^k) - \Gl^k, \; k \geqslant 2$.
\end{Rmk}
Then, we have
\begin{align*}
&\quad\Nrm{\pt^r\left((\Id+\Gl)(\PLq\ULq(\Sgm_{G_{q,i}})-\Sgm_{\Gl})\right)}_{N}\\
&\lesssim\sum_{N_1+N_2=N}\sum_{r_1+r_2=r}\Nrm{\pt^{r_1}\left(\Id+\Gl\right)}_{N_1}\Nrm{\pt^{r_2}\left(\PLq\ULq(\Sgm_{G_{q,i}})-\Sgm_{\Gl}\right)}_{N_2}\lesssim_{\sgm,N,r}\ell_{q,i}^{2-N-r}\oM^2\lambda_{q,i}^2\dlt_{q,i}.
\end{align*}
To sum up, we have proved \eqref{est on R_mi}. 

Notice that for any $t\notin\supp_{t}(u_{q,i},R_{q,i})+\ell_{q,i}$, we have $\ul=u_{q,i}=0$ and $\Rl=R_{q,i}=0$, and then $R_{m,i}=0$. So the property of the support \eqref{pp of supp l,i} can be clearly derived from it and Lemma \ref{lm: Mollification}.

Finally, we consider the kinetic energy and stored-energy. We could calculate
\begin{align*}
	|\pt u_{q,i}|^2-|\pt\ul|^2=|\pt u_{q,i}|^2-\PLq\ULq|\pt u_{q,i}|^2+\PLq\ULq|\pt u_{q,i}|^2-|\pt\ul|^2,
\end{align*}
and  then, by using \eqref{est on F-PULF} and Lemma \ref{lm of commutator}, we have  for $0\leqslant N+r\leqslant 2$, 
\begin{align}
&\quad\Nrm{\pt^r(|\pt u_{q,i}|^2-|\pt\ul|^2)}_{C^0(\mcI^{q,i};C^N(\T^2))}\nonumber\\
&\leqslant\Nrm{\pt^r(|\pt u_{q,i}|^2-\PLq\ULq|\pt u_{q,i}|^2)}_{C^0(\mcI^{q,i};C^N(\T^2))}+\Nrm{\pt^r(\PLq\ULq|\pt u_{q,i}|^2-|\pt\ul|^2)}_{C^0(\mcI^{q,i};C^N(\T^2))}\nonumber\\
&\lesssim\ell_{q,i}^{2-N-r}\Nrm{|\pt u_{q,i}|^2}_{C^2(\mcI^{q,i-1}\times\T^2)}+\ell_{q,i}^{2-N-r}\left(\Nrm{\pt ^2u_{q,i}}_{C^0(\mcI^{q,i-1};C^0(\T^2))}^2+\Nrm{\pt u_{q,i}}_{C^0(\mcI^{q,i-1};C^1(\T^2))}^2\right)\nonumber\\
&\lesssim_{\oM,N,r}\ell_{q,i}^{2-N-r}\lambda_{q,i}^2\dlt_{q,i}^{\frac{1}{2}}.\label{est on energy ptu}
\end{align}
Recalling \eqref{stored-energy function} and using $\sgm_1=0$, we compute
\begin{align*}
\sgm_{\Gl}-\sgm_{G_{q,i}}=&\sgm_2(j_2|_{G=\Gl}-j_2|_{G=G_{q,i}})+\sum_{r=1}^2\sum_{s=1}^2\frac{\sgm_{rs}}{2}((j_rj_s)|_{G=\Gl}-(j_rj_s)|_{G=G_{q,i}})+\frac{\sgm_{111}}{6}(j_1^3|_{G=\Gl}-j_1^3|_{G=G_{q,i}}).
\end{align*}
Noting that $\sgm_{\Gl}-\sgm_{G_{q,i}}$ consists of terms such as $(\tr\Gl)^k-(\tr G_{q,i})^k$, $\tr(\Gl^k)-\tr(G_{q,i}^k)$, and other similar terms with $k\geqslant2$, we can argue as in \eqref{est on energy ptu} to obtain
\begin{align*}
	&\quad\Nrm{\pt^r((\tr\Gl)^k-(\tr G_{q,i})^k)}_{C^0(\mcI^{q,i};C^N(\T^2))}\\
	&\leqslant\Nrm{\pt^r(\PLq\ULq(\tr G_{q,i})^k-(\tr G_{q,i})^k)}_{C^0(\mcI^{q,i};C^N(\T^2))}+\Nrm{\pt^r((\tr\Gl)^k-\PLq\ULq(\tr G_{q,i})^k)}_{C^0(\mcI^{q,i};C^N(\T^2))}\\
	&\lesssim\ell_{q,i}^{2-N-r}\Nrm{(\tr G_{q,i})^k}_{C^2(\mcI^{q,i-1}\times\T^2)}+\ell_{q,i}^{2-N-r}\sum_{\substack{s_1+\cdots+s_k=2\\s_n<2,\ n=1,\ldots,k}}\left(\prod_{n=1}^k\Nrm{\tr G_{q,i}}_{C^0(\mcI^{q,i-1};C^{s_n}(\T^2))}+\prod_{n=1}^k\Nrm{\pt^{s_n}(\tr G_{q,i})}_{C^0(\mcI^{q,i-1};C^0(\T^2))}\right)\\
	&\lesssim_{\oM,N,r}\ell_{q,i}^{2-N-r}\lambda_{q,i}^2\dlt_{q,i}^{\frac12}.
\end{align*}
And then, we have
\begin{equation}\label{est on sgm q-l}
	\begin{aligned}
		\nrm{\pa_t^r(\sgm_{\Gl}-\sgm_{G_{q,i}})}_{C^0(\mcI^{q,i};C^N(\T^2))}\lesssim_{\oM,N,r}\ell_{q,i}^{2-N-r}\lambda_{q,i}^2\dlt_{q,i}^{\frac{1}{2}}.
	\end{aligned}
\end{equation}
This concludes the proof.
\end{proof}

Having regularized the background fields, we next localize them in
space and time. The localization allows us to freeze the mollified
coefficients on each localization region and is also needed to
produce compactly supported source terms for the local
inverse-divergence construction. Since
$\tau_{q,i}=\ell_{q,i}$ and $\mu_{q,i}\approx\ell_{q,i}$, the
localization scales are compatible with the mollification scale.
We therefore introduce the following space--time partitions of
unity.

Here, we will give partitions of unity in space $\R^2$ and in time $\R$. We introduce some nonnegative smooth functions $\left\lbrace\chi_\upsilon\right\rbrace_{\upsilon\in\Z^2} $ and $\left\lbrace\theta_s\right\rbrace_{s\in\Z} $ such that
$$\sum_{\upsilon\in\Z^2}\chi_\upsilon^2(x)=1, \quad \forall x\in\R^2; \quad\sum_{s\in\Z}\theta_s^2(t)=1, \quad\forall t\in\R,$$
where $\chi_\upsilon(x)=\chi_0(x-2\pi\upsilon)$ and $\chi_0$ is a nonnegative smooth function supported in $Q(0,9/8\pi)$ satisfying $\chi_0=1$ on $\overline{Q(0,7/8\pi)}$. Similarly, $\theta_s(t)=\theta_0(t-s)$ where $\theta_0\in C_c^\infty(\R)$ satisfies $\theta_0=1$ on $[1/8,7/8]$ and $\theta_0=0$ on $(-1/8,9/8)^c$.  

And then, we  give the cut-off parameters $\tau_{q,i}$ and $\mu_{q,i}^{-1}\in 2\Z$ as \eqref{def of parameter 1}, 
\begin{align}
	\mu_{q,i}^{-1}=2\lceil(\lambda_{q,i}\lambda_{q,i+1})^{\frac{1}{2}}\dlt_{q}^{\frac{1}{4}}/2\rceil, \quad \tau_{q,i}^{-1}=(\lambda_{q,i}\lambda_{q,i+1})^{\frac{1}{2}}\dlt_{q}^{\frac{1}{4}},\quad 0\leqslant i\leqslant 2, \label{def of mu tau}
\end{align}
and introduce the following notations
\begin{align*}
\mathscr{I}&:=\left\lbrace(s, \upsilon):(s, \upsilon)\in\Z\times\Z^2\right\rbrace,
\end{align*}
and for $I=(s, \upsilon)\in \mathscr{I}$,
\begin{equation}\label{def of [I]}
[I]=[s]+\sum_{i=1}^22^i[\upsilon_i]+1, \quad [j]=\left\{\begin{aligned}
&0, && j\ \text{is odd}, \\
&1, && j\ \text{is even}.
\end{aligned}\right.
\end{equation}
So far, we can define the cutoff functions as follows:
\begin{align}
	\chi_I(x)=\chi_\upsilon(\mu_{q,i}^{-1}x), \quad \theta_I(t)=\theta_s(\tau_{q,i}^{-1}t).
\end{align}

\section{Defining a local inverse of divergence by solving the asymmetric divergence equation}\label{Defining a Local Inverse of Divergence by Solving the Asymmetric Divergence Equation}
In order to directly absorb these small parts of second-order terms and higher-order terms in \eqref{decomposition of left} into the new Reynolds error at the next iteration and ensure that $R_{q+1}$  remains symmetric in each iteration, in this section, we will define a local right inverse of the divergence by solving the following asymmetric divergence equation:
\begin{align*}
	\Div(F\tR[U])=\Div((\Id+\nabla u)\tR[U])=\rho_0U, \quad \tR[U]\in\mS^{2\times 2},
\end{align*}
for some vector field $U=U(t,x)\in C_c^{\infty}(\R\times\R^2)$ satisfying the conservation laws for momentum and angular momentum as \eqref{law of conservation in Lc}, small displacements $u\in C^{\infty}(\R\times\R^2)$, and $\rho_0\in C^{\infty}(\R\times\R^2)$. 

In \cite{Isett16}, Isett and Oh show the relationship between Euler--Reynolds flows and the conservation of linear and angular momentum in Eulerian coordinates. Under appropriate decay assumptions, the space of Euler--Reynolds flows on $\mathbb{R} \times \mathbb{R}^3$ can also be viewed as the space of incompressible velocity fields that conserve both linear and angular momentum. Specifically, the usual laws of conservation of linear and angular momentum in $\R^d, d\geq 2$, expressed as \eqref{law of conservation},
can be proven for $U_l(y) = \pa_{y_j} R_{lj}(y)$, under the assumption that $ R_{lj}(y) \in L^1_{t,y}$ is a symmetric matrix for each $y\in\R^d$. 
Conversely, if $ U_l $  conserves both linear and angular momentum, then one can formally represent $ U $ as divergence of some symmetric matrix by solving the symmetric divergence equation \eqref{symmetric div eq}.

In this paper, we consider the equations of 2D elastic dynamics in Lagrangian coordinates, presenting new corresponding conservation laws for momentum and angular momentum as \eqref{law of conservation in Lc}. However, in Eulerian coordinates, the corresponding conservation laws can be expressed as follows:
\begin{align}
	\int_{\R^2} \rho U^*_l(t,y) \, \rd y = 0, \quad \int_{\R^2} \rho(y_j U^*_l(t,y) - y_l U^*_j(t,y)) \, \rd y = \int_{\R^2} \rho_0 (y_j U_l(t,x) - y_l U_j(t,x)) \, \rd x = 0, \label{law of conservation in Ec}
\end{align}
where
$
\rho = \left((\det(\nabla y))^{-1} \rho_0\right)(t, X(t, y)),
$
and $U^*(t, y) = U(t, X(t, y))$, with $X(t, y)$ being the inverse of $y = y(t, x)$.
Then we could observe that $\rho U^*$ satisfies \eqref{law of conservation} in $\R^2$. According to \cite{Isett16}, there exists a symmetric matrix $R=R(y)$ such that 
\begin{align}\label{symmetric divergence equation}
	\pa_{y_j}R_{lj}=\rho U^*_l.
\end{align}
\begin{Rmk}
	To avoid confusion, we denote the same function, $R(t, y)$ and $R(t, y(t, x))$, in different coordinates simply as $R$.
\end{Rmk}
Moreover, if we use 
\begin{align*}
	\frac{\pa R_{lj}}{\pa y_j}=\frac{\pa R_{lj}}{\pa x_r}\frac{\pa x_r}{\pa y_j}=\frac{\pa}{\pa x_r}\left( R_{lj}\frac{\pa x_r}{\pa y_j}\right)-R_{lj}\frac{\pa}{\pa x_r}\left(\frac{\pa x_r}{\pa y_j}\right),
\end{align*} 
and 
\begin{align*}
	\frac{\pa J}{\pa x_i}=\frac{\pa J}{\pa F_{jr}}\frac{\pa F_{jr}}{\pa x_i}=J\frac{\pa x_r}{\pa y_j}\frac{\pa^2 y_{j}}{\pa x_r\pa x_i}=-J\frac{\pa }{\pa x_r}\left(\frac{\pa x_{r}}{\pa y_j}\right)\frac{\pa y_{j}}{\pa x_i},
\end{align*}
we could obtain
\begin{equation}\label{Transport of Div R}
	\begin{aligned}
		J\frac{\pa R_{lj}}{\pa y_j}&=J\frac{\pa}{\pa x_r}\left( R_{lj}\frac{\pa x_r}{\pa y_j}\right)-JR_{lj}\frac{\pa}{\pa x_r}\left(\frac{\pa x_r}{\pa y_j}\right)
		=\frac{\pa}{\pa x_r}\left(J R_{lj}\frac{\pa x_r}{\pa y_j}\right),
	\end{aligned}
\end{equation}
which means
\begin{align}
	\rho_0U_l=J \rho U^*_l=J\pa_{y_j}R_{lj}=\pa_{x_j}(JR F^{-\top})_{lj}=\pa_{x_j}(F\underbrace{JF^{-1}R F^{-\top}}_{\tR[U]})_{lj},\label{asymmetric divergence equation 1}
\end{align}
in the Lagrangian coordinates. Notice that \eqref{asymmetric divergence equation 1} is an asymmetric divergence equation and $\tR=JF^{-1}R F^{-\top}$ is a symmetric matrix. 

Naturally, we summarize  the following principle: Given any vector field $U$ and deformation $y=y(t,x)$ with small displacements $u$ satisfying \eqref{law of conservation in Lc},  we could find a symmetric matrix $\tR[U]$ such that
\begin{align}
	\Div(F\tR[U])=\rho_0U.\label{asymmetric divergence equation}
\end{align} 
Moreover, for any symmetric matrix $K$, we could calculate 
\begin{align*}
	\int_{\R^2}\pa_{x_j}(FK)_{{lj}}\rd x&=0,\\
	\int_{\R^2}y_k\pa_{x_j}(FK)_{{lj}}-y_l\pa_{x_j}(FK)_{{kj}}\rd x&=\int_{\R^2}\pa_{x_j}(y_k(FK)_{{lj}}-y_l(FK)_{{kj}})\rd x-\int_{\R^2}(F_{kj}(FK)_{{lj}}-F_{lj}(FK)_{{kj}})\rd x\\
	&=-\int_{\R^2}((FK F^{\top})_{{lk}}-(FK F^{\top})_{{kl}})\rd x=0.
\end{align*}

In \cite{Isett16}, due to the use of localized waves and a rearrangement of the error terms in the construction, Isett and Oh  solved \eqref{symmetric div eq} with data that satisfies \eqref{law of conservation} while simultaneously remaining localized to a small length scale $\tilde{\rho}$. This smallness of support leads to a gain of a factor $\tilde{\rho}$ for the solution to \eqref{symmetric div eq}, which is an estimate one expects from dimensional analysis:
$$
\nrm{R}_{C^0} \lesssim \tilde{\rho} \nrm{U}_{C^0}.
$$

The gain of this smallness parameter $\tilde{\rho}$ allows us to obtain a $C^0$ estimate for $R$. Moreover, in this paper, we will prove that for items like $U=a(t,x)e^{i\tl(f\cdot x-c(t))}$, with $f\in\Z^2,\tl>0$, and $c\in C(\R)$, we could obtain
$$
\nrm{R}_{C^0}\lesssim \tl^{-1} \nrm{a}_{C^0}.
$$
This differs from the convex integration framework by using a global inverse divergence operator on $\T^d$, such as \cite{DS09,GKN24}, where only a $C^\alpha$ estimate with $\alpha > 0$ is typically achieved, often resulting in 
$$
\nrm{R}_{C^0}\leqslant \nrm{R}_{C^\alpha} \lesssim \tilde{\lambda}^{-1 + \alpha} \nrm{a}_{C^\alpha}.
$$
Such a loss $\tl^{\alpha}$ is unacceptable in the proof of this paper, due to the choice of the parameters $\lambda_q$ and $\dlt_q$ in \eqref{def of parameter}.\par 

Therefore, based on  \cite{Isett16} and our analysis in Section \ref{Motivation and difficulties}, we will extend the conclusions and apply them to the equations of elastic dynamics. Given any vector field $U$ satisfying \eqref{law of conservation in Lc}, how can we find a solution of \eqref{asymmetric divergence equation}? We divide it into two steps:
\begin{enumerate}
	\item Find a symmetric matrix $R$ in the Eulerian coordinates  by solving the symmetric divergence equation \eqref{symmetric divergence equation}.
	\item  Find a symmetric  matrix $\tR =JF^{-1}R F^{-\top}$ in the Lagrangian coordinates which solves the asymmetric divergence equation \eqref{asymmetric divergence equation}.
\end{enumerate}

In the following, we present two lemmas to solve the asymmetric divergence equation \eqref{asymmetric divergence equation}. Throughout this section, we will consistently use the notation $G = \nb u$ and $F = \Id + G = \nb y$.
 \par 
\subsection{Solving the symmetric divergence equation}
In this subsection, we will use the method from \cite{Isett16} to solve the symmetric equation \eqref{symmetric divergence equation}. For the sake of completeness, we provide the proof here. Specifically, for the elastodynamic equations studied in this paper, the estimation of material derivatives is not required, allowing for a suitable simplification.

The following is a main result regarding compactly supported solutions to the symmetric divergence equation on $\R^d$ for any $d\geq2$:
\begin{align}\label{sym div eq}
	\Div R=U.
\end{align}
Although the construction below is presented for any fixed
$d\geqslant2$, all applications in this paper are in dimension
$d=2$. The implicit constants may depend on $d$.
\begin{lm}[Compactly supported solutions to the symmetric divergence equation] \label{lm: sym div eq}
Let $L \geq 1$ be a positive integer, $\nrm{\cdot}_{N}=\nrm{\cdot}_{C^0(\mcI^{q,i};C^N(\R^d))}$, and $\nrm{\cdot}_{C_{t}^r}=\nrm{\cdot}_{C^r(\mcI^{q,i};\R^d)}$. We consider a smooth vector field $U\in C^\infty(\mcI^{q,i}\times\R^d;\R^{d})$ that satisfies  
\begin{equation} \label{pp of supp U in E co}
	\supp_{y} U(t,\cdot) \subseteq  B(y_0(t), \tmu),\quad  \forall t \in \mcI^{q,i},
\end{equation}
for some $y_0 = y_0(t) \in C^\infty(\mcI^{q,i};\R^{d})$, and has zero linear and angular momentum, i.e.,
	\begin{align} 
		\int_{\R^d} U_{l}(t,y) \, \rd y =0,
		\quad \int_{\R^d} (y_{j} U_{l} - y_{l} U_{j})(t,y) \, \rd y = 0, \label{eq of linear and angular momenta in E co}
	\end{align}
	for $j,l=1,\cdots, d$ and all $t$.
	Assume also that for
	$
		A,\tlm>0
	$,
	the vector field $U$ obeys the estimates for $r=0,1,2$,
	\begin{equation} \label{est on y_0 and U} 
		\begin{aligned}
			\nrm{\pa_t^rU}_{N} \leqslant& A\tlm^{N+r},	&& 0\leqslant N+r\leqslant L.
		\end{aligned}
	\end{equation}
	Then there exists a solution $R_{lj}[U]$, with $R_{jl}[U]=R_{lj}[U],\ \forall j,l$, to the symmetric divergence equation \eqref{sym div eq}, depending linearly on $U$, with the following properties:
	\begin{enumerate}
		\item The support of $R_{lj}[U]$ stays in the ball $B(y_0(t), \tmu)$ for all $t\in\mcI^{q,i}$, i.e.,
		\begin{equation} \label{pp of supp oR}
			\supp_{y} R_{lj}[U](t,\cdot) \subseteq B(y_0(t), \tmu), \quad \forall t\in\mcI^{q,i}.
		\end{equation}
		\item The time support of $R_{lj}[U]$ is contained within the time support of $U$, i.e.,
		\begin{equation} \label{pp of time supp oR}
			\supp_{t} R_{lj}[U]\subseteq \supp_{t} U.
		\end{equation}
		
		\item There exists $C_{d,L}=C_{d,L}(d,L) > 0$ such that
		\begin{equation} \label{est of nabla oR}
		\begin{aligned}
			\nrm{R_{lj}[U]}_{N}\leqslant& C_{d,L}A \tmu(\tmu^{-1}+\tlm)^N,&&0\leqslant N\leqslant L,\\
			\nrm{ \pa_tR_{lj}[U]}_{N}
			\leqslant& C_{d,L}A \tmu (\tlm+\tmu^{-1}\nrm{\pt y_0}_{C_{t}^0})(\tmu^{-1}+\tlm)^N,&&0\leqslant N\leqslant L-1,\\
			\nrm{ \pa_t^2R_{lj}[U]}_{N}
			\leqslant& C_{d,L}A \tmu (\tlm^2+\tmu^{-1}\tlm\nrm{\pt y_0}_{C_{t}^0}+\tmu^{-1}\nrm{\pt^2 y_0}_{C_{t}^0}+\tmu^{-2}\nrm{\pt y_0}_{C_{t}^0}^2)(\tmu^{-1}+\tlm)^N,&&0\leqslant N\leqslant L-2.
		\end{aligned}
		\end{equation}
	\end{enumerate}
\end{lm}

\begin{Rmk}\label{Rmk on curl} 
	Note that the second term in condition $\eqref{eq of linear and angular momenta in E co}$ can actually be written as
	\begin{align*} 
		\int_{\R^d} (y\times U)(t,y) \, \rd y = 0.
	\end{align*}
	For convenience, in the following text, we will use the above equality to replace these different equations.
\end{Rmk}

\begin{Rmk}
	This lemma should be compared with Theorem 10.1 in \cite{Isett16}. The major difference is that, since we are considering the elastodynamic equation, which is a quasilinear wave equation, we need the estimation of time derivatives instead of material derivatives.
\end{Rmk}

\subsubsection{Derivation of the solution operator}

The purpose of this subsection is to give a derivation of the solution operator $R_{lj}[U]$ for \eqref{sym div eq} in Lemma \ref{lm: sym div eq}.
In the following part, for simplicity, we consider $U= U(t,y)$ as a smooth vector field supported on some ball, i.e., $\supp_{y} U(t,\cdot) \subseteq B(y_{0},\tmu)$. Moreover, we will omit  the time variable and work entirely on $\R^d$. Especially, we could only consider the case of $y_{0} = 0$. 

First, we will use the Fourier transform to expand $\widehat{U}_{l}(\xi)$ in a Taylor series around the frequency origin $\xi = 0$ in Fourier space. We will then express the terms of this Taylor expansion as the divergence of a symmetric tensor, with some terms evaluated at $\xi = 0$. By translating the resulting formula back to physical space, we will derive a solution operator that retains the desired support property in physical space, despite a mild singularity at $0$.

We first compute
\begin{align*}
	\widehat{U}_{l}(\xi) 
	= & 	\widehat{U}_{l}(0) + \left( \int_{0}^{1} \pa_{k} \widehat{U}_{l} (\sgm \xi) \, \rd \sgm \right) \xi_{k} \\
	= & 	\widehat{U}_{l}(0) + \frac{1}{2}\left( \int_{0}^{1} (\pa_{k} \widehat{U}_{l} + \pa_{l} \widehat{U}_{k}) (\sgm \xi) \, \rd \sgm \right) \xi_{k} 
	+ \frac{1}{2}\left( \int_{0}^{1} (\pa_{k} \widehat{U}_{l} - \pa_{l} \widehat{U}_{k}) (\sgm \xi) \, \rd \sgm \right) \xi_{k} \\
	= & 	\widehat{U}_{l}(0) + \frac{1}{2} (\pa_{k} \widehat{U}_{l} - \pa_{l} \widehat{U}_{k})(0) \xi_{k} 
	+ \frac{1}{2}\left( \int_{0}^{1} (\pa_{k} \widehat{U}_{l} + \pa_{l} \widehat{U}_{k}) (\sgm \xi) \, \rd \sgm \right) \xi_{k} \\
	&+	\frac{1}{2} \left( \int_{0}^{1} (1-\sgm)  (\pa_{j} \pa_{k} \widehat{U}_{l} - \pa_{j} \pa_{l} \widehat{U}_{k})(\sgm \xi) \, \rd \sgm \right) \xi_{j}\xi_{k}\\
	= & 	\widehat{U}_{l}(0) + \frac{1}{2} (\pa_{k} \widehat{U}_{l} - \pa_{l} \widehat{U}_{k})(0) \xi_{k} 
	+ \frac{1}{2}\left( \int_{0}^{1} (\pa_{k} \widehat{U}_{l} + \pa_{l} \widehat{U}_{k}) (\sgm \xi) \, \rd \sgm \right) \xi_{k} \\
	&+	\frac{1}{2} \left( \int_{0}^{1} (1-\sgm) \xi_{k} (\pa_{j} \pa_{k} \widehat{U}_{l} + \pa_{l} \pa_{k} \widehat{U}_{j})(\sgm \xi) \, \rd \sgm \right) \xi_{j}- \left( \int_{0}^{1} (1-\sgm) \xi_{k} (\pa_{l} \pa_{j} \widehat{U}_{k})(\sgm \xi) \, \rd \sgm \right) \xi_{j}.
\end{align*}

Note that  the right-hand side is (formally) the Fourier transform of a divergence of a symmetric tensor. By using  \eqref{pp of supp U in E co} and \eqref{eq of linear and angular momenta in E co}, we have
\begin{equation} \label{eq of momentum}
	\begin{aligned}
		\widehat{U}_{l} (0) &= \int_{\R^d} U_{l}(y)  \rd y= 0, \quad
		\left( \frac{1}{i} \pa_{k} \widehat{U}_{l} - \frac{1}{ i} \pa_{l} \widehat{U}_{k} \right) (0) = \int_{\R^d} (y_{l} U_{k} - y_{k} U_{l})(y)  \rd y = 0,
	\end{aligned}
\end{equation}
and the following formula for $\widehat{U}_{l}(\xi)$ holds:
\begin{align*}
	\widehat{U}_{l} (\xi)
	=&\frac{1}{2}\left( \int_{0}^{1} (\pa_{j} \widehat{U}_{l} + \pa_{l} \widehat{U}_{j}) (\sgm \xi) \, \rd \sgm \right) \xi_{j} 	+ \frac{1}{2} \left( \int_{0}^{1} (1-\sgm) \xi_{k} (\pa_{j} \pa_{k} \widehat{U}_{l} + \pa_{l} \pa_{k} \widehat{U}_{j})(\sgm \xi) \, \rd \sgm \right) \xi_{j} \\
	&	 - \left( \int_{0}^{1} (1-\sgm) \xi_{k} (\pa_{l} \pa_{j} \widehat{U}_{k})(\sgm \xi) \, \rd \sgm \right) \xi_{j}.
\end{align*}

Up to now, we could define $r_{lj}[U] := r_{lj}^{(0)}[U] + r_{lj}^{(1)}[U] + r_{lj}^{(2)}[U]$, where
\begin{align}
	r_{lj}^{(0)}[U]
	= & \cF^{-1} \left[\frac{1}{2}\left( \int_{0}^{1} \left( \frac{1}{i} \pa_{j} \widehat{U}_{l} + \frac{1}{i} \pa_{l}\widehat{U}_{j} \right) (\sgm \xi) \, \rd \sgm \right)\right], \\
	r_{lj}^{(1)}[U] 
	=& \cF^{-1} \left[ \frac{1}{2} \left( \int_{0}^{1} (1-\sgm) (i \xi_{k}) (- \pa_{j} \pa_{k} \widehat{U}_{l} - \pa_{l} \pa_{k} \widehat{U}_{j} )(\sgm \xi) \, \rd \sgm \right)\right], \\
	r_{lj}^{(2)}[U]
	= & \cF^{-1}\left[ - \left( \int_{0}^{1} (1-\sgm) (i \xi_{k}) ( - \pa_{j} \pa_{l} \widehat{U}_{k} )(\sgm \xi) \, \rd \sgm \right) \right]. 
\end{align}
Computing the inverse Fourier transform, we could give the following formula
\begin{align}
	r_{lj}^{(0)}[U]
	=& - \frac{1}{2} \int_{0}^{1} \left( \frac{y_{j}}{\sgm} U_{l}\left(\frac{y}{\sgm}\right) + \frac{y_{l}}{\sgm} U_{j}\left(\frac{y}{\sgm}\right) \right) \, \frac{\rd \sgm}{\sgm^{d}}, \label{eq:rlj0:formal}\\
	r_{lj}^{(1)}[U]
	=& \frac{1}{2}  \frac{\pa}{\pa y_{k}} \int_{0}^{1} (1-\sgm)  \left( \frac{y_{j} y_{k}}{\sgm^{2}} U_{l}\left(\frac{y}{\sgm}\right) + \frac{y_{l} y_{k}}{\sgm^{2}} U_{j}\left(\frac{y}{\sgm}\right) \right) \, \frac{\rd \sgm}{\sgm^{d}}, \label{eq:rlj1:formal}\\
	r_{lj}^{(2)}[U]
	= & -  \frac{\pa}{\pa y_{k}} \int_{0}^{1} (1-\sgm) \left(\frac{y_{l} y_{j}}{\sgm^{2}} U_{k}\left(\frac{y}{\sgm}\right) \right) \, \frac{\rd \sgm}{\sgm^{d}}.\label{eq:rlj2:formal}
\end{align}
In fact, given a test function $\varphi \in C^{\infty}_{c}$, we could define
\begin{align}
	\brk{r_{lj}^{(0)}[U], \varphi} 
	:=& - \lim_{\dlt \to 0+} \frac{1}{2} \int_{\dlt}^{1} \int_{\R^d}\left( \frac{y_{j}}{\sgm} U_{l}\left(\frac{y}{\sgm}\right) + \frac{y_{l}}{\sgm} U_{j}\left(\frac{y}{\sgm}\right) \right) \varphi(y) \, \rd y \frac{\rd \sgm}{\sgm^{d}}, \label{eq:rlj0} \\
	\brk{r_{lj}^{(1)}[U], \varphi}
	:=& - \frac{1}{2} \lim_{\dlt \to 0+} \int_{\dlt}^{1} (1-\sgm)  \int_{\R^d}\left( \frac{y_{j} y_{k}}{\sgm^{2}} U_{l}\left(\frac{y}{\sgm}\right) + \frac{y_{l} y_{k}}{\sgm^{2}} U_{j}\left(\frac{y}{\sgm}\right) \right) \pa_{k} \varphi(y) \, \rd y \frac{\rd \sgm}{\sgm^{d}}, \label{eq:rlj1}\\
	\brk{r_{lj}^{(2)}[U], \varphi}
	:= & \lim_{\dlt \to 0+} \int_{\dlt}^{1} (1-\sgm) \int_{\R^d} \left(\frac{y_{l} y_{j}}{\sgm^{2}} U_{k}\left(\frac{y}{\sgm}\right) \right) \pa_{k} \varphi(y) \, \rd y \frac{\rd \sgm}{\sgm^{d}}.\label{eq:rlj2}
\end{align}

These are well-defined tempered distributions on $\R^{d}$, which satisfy
\begin{align*}
	\left|\brk{r_{lj}^{(0)}[U], \varphi}\right| \leqslant C_{U, \tmu} \nrm{\varphi}_{C^{0}_{y}}, \quad
	\left|\brk{r_{lj}^{(1)}[U], \varphi}\right| + \left|\brk{r_{lj}^{(2)}[U], \varphi}\right| \leqslant C_{U, \tmu} \nrm{\nb \varphi}_{C^{0}_{y}}.
\end{align*}
Therefore, for $a = 0, 1, 2$, the values of $r_{lj}^{(a)}[U]$ at a point $y \in \R^d$ are formally given as weighted integrals of $U$ and $\nabla U$ along the ray emanating from $y$ away from the origin. 
They satisfy the support property 
\begin{equation} \label{pp of supp r}
	\supp_{y} r_{lj} \subseteq B(0,\tmu),
\end{equation}
and are smooth outside $\left\lbrace y=0\right\rbrace$. Moreover, a straightforward computation with distributions shows that 
\begin{equation}
	\pa_{j} r_{lj}[U] = U_{l} - \left( \int_{\R^d} U_{l}(y) \, \rd y \right) \dlt_{0} - \frac{1}{2} \left( \int_{\R^d} (y_{l} U_{j} - y_{j} U_{l})(y) \, \rd y \right) \pa_{j} \dlt_{0}=U_l.
\end{equation}
Thus, under the assumption  \eqref{pp of supp U in E co} and \eqref{eq of linear and angular momenta in E co}, we could achieve that $r_{lj}[U]$ is a distributional solution to \eqref{sym div eq} satisfying \eqref{pp of supp r}.

Unfortunately, $r_{lj}[U]$ as defined above appears to have a singularity at $y=0$. To address this issue, we will employ the method described in \cite{Isett16} which exploits translation invariance of \eqref{sym div eq}. Specifically, we will conjugate $r_{lj}[U]$ by translations and then take a smooth average of the resulting expressions.

Given $\ty \in \R^{d}$, we conjugate the operators $r_{lj}^{(0)}$, $r_{lj}^{(1)}$ and $r_{lj}^{(2)}$ by translation by $\ty$. Then we are led to the conjugated operator ${}^{(\ty)} r_{lj} = {}^{(\ty)} r_{lj}^{(0)} + {}^{(\ty)} r_{lj}^{(1)} + {}^{(\ty)} r_{lj}^{(2)}$, which is formally defined by
\begin{align}
	{}^{(\ty)} r_{lj}^{(0)}[U]
	=& - \frac{1}{2} \int_{0}^{1}  \frac{(y-\ty)_{j}}{\sgm} U_{l}\left(\frac{y-\ty}{\sgm}+\ty\right) + \frac{(y-\ty)_{l}}{\sgm} U_{j}\left(\frac{y-\ty}{\sgm}+\ty\right)  \, \frac{\rd \sgm}{\sgm^{d}}, \label{eq:yrlj0}\\
	{}^{(\ty)} r_{lj}^{(1)}[U]
	=& \frac{1}{2} \frac{\pa}{\pa y_{k}} \int_{0}^{1} (1-\sgm)   \frac{(y-\ty)_{j} (y-\ty)_{k}}{\sgm^{2}} U_{l}\left(\frac{y-\ty}{\sgm}+\ty\right)  \, \frac{\rd \sgm}{\sgm^{d}} \label{eq:yrlj1}\\
	& + \frac{1}{2}  \frac{\pa}{\pa y_{k}} \int_{0}^{1} (1-\sgm)   \frac{(y-\ty)_{l} (y-\ty)_{k}}{\sgm^{2}} U_{j}\left(\frac{y-\ty}{\sgm}+\ty\right)  \, \frac{\rd \sgm}{\sgm^{d}}, \notag \\
	{}^{(\ty) }r_{lj}^{(2)}[U]
	= & - \frac{\pa}{\pa y_{k}} \int_{0}^{1} (1-\sgm)  \frac{(y-\ty)_{l} (y-\ty)_{j}}{\sgm^{2}} U_{k}\left(\frac{y-\ty}{\sgm}+\ty\right)  \, \frac{\rd \sgm}{\sgm^{d}}. \label{eq:yrlj2}
\end{align}

These are to be interpreted as in \eqref{eq:rlj0}--\eqref{eq:rlj2} as distributions. The distribution ${}^{(\ty)} r_{lj}[U]$ satisfies the desired support property \eqref{pp of supp r} as long as $\ty \in B(0,\tmu)$. Motivated by this consideration, we could take a smooth function $\zeta(\ty)$ which is supported in $B(0,\tmu)$ and satisfies 
$
	\int_{\R^d} \zeta(\ty) \rd\ty = 1 .
$
Up to now, we could define the solution operator ${}^{(\zeta)} \oR_{lj}[U]$ by averaging ${}^{(\ty)} r_{lj}[U]$ against $\zeta$, i.e.,
\begin{equation} 
	{}^{(\zeta)} \oR_{lj}[U](y) = \int_{\R^d} {}^{(\ty)} r_{lj}(y) \zeta(\ty) \, \rd \ty.
\end{equation}

From the discussion above, we see that ${}^{(\zeta)} \oR_{lj}[U]$ inherits the desirable properties of $r_{lj}[U]$. Indeed, assuming  \eqref{pp of supp U in E co} and \eqref{eq of linear and angular momenta in E co}, ${}^{(\zeta)} \oR_{lj}[U]$ is a (distributional) solution to \eqref{sym div eq} satisfying the support property
\begin{equation} \label{pp of supp tilde Rlj}
	\supp_{y} {}^{(\zeta)} \oR_{lj}[U] \subseteq B(0,\tmu).
\end{equation}

Thanks to averaging with respect to $\zeta$, ${}^{(\zeta)} \oR_{lj}[U]$ will moreover turn out to be smooth in the spatial variables provided $U$ is smooth as well.

\subsubsection{Formula for $R_{lj}[U]$ and basic properties}
Let $U$ be a vector field satisfying the hypotheses in Lemma \ref{lm: sym div eq}. We choose a function $\zeta=\zeta(y)$ satisfying
\begin{align}
	\int_{\R^d}\zeta(y)\rd y=1,
\end{align}
and the support property
\begin{equation} \label{eq of supp zeta}
	\supp_y\zeta \subseteq B(0,\tmu).
\end{equation}
Moreover, the following estimates hold for $\zeta$:
\begin{equation} \label{eq of zeta c0}
	\nrm{\nb_y^{\beta} \zeta}_{C_y^0} \leqslant C_{\beta} \tmu^{-d-|\beta|}, \qquad \hbox{ for all } |\beta| \geq 0.
\end{equation}

Let $\zeta^{y_0}(y)=\zeta(y-y_0)$.  We are now ready to define the solution operator $R_{lj}[U]$. Let $R_{lj}[U] := R_{lj}^{(0)}[U] + R_{lj}^{(1)}[U] + R_{lj}^{(2)}[U]$, where
\begin{align} 
	R_{lj}^{(0)}[U] 
	=& 	-\frac{d}{2} \int_{0}^{1} \int_{\R^d}\zeta^{y_0}(\ty) \frac{(y-\ty)_{j}}{\sgm} U_{l}\left(t, \frac{y-\ty}{\sgm} + \ty\right) \, \frac{\rd \ty}{\sgm^{d}} \rd \sgm \notag\\
	&	-\frac{d}{2} \int_{0}^{1} \int_{\R^d}\zeta^{y_0}(\ty) \frac{(y-\ty)_{l}}{\sgm} U_{j}\left(t, \frac{y-\ty}{\sgm} + \ty\right) \, \frac{\rd \ty}{\sgm^{d}} \rd \sgm, \label{def of Rlj 0} \\
	R_{lj}^{(1)}[U]
	=&	\frac{1}{2} \int_{0}^{1} \int_{\R^d}(\pa_{k} \zeta^{y_0})(\ty) \frac{(y-\ty)_{j} (y-\ty)_{k}}{\sgm^{2}} U_{l}\left(t,\frac{y-\ty}{\sgm} + \ty\right) \, \frac{\rd \ty}{\sgm^{d}} \rd \sgm \notag \\
	& + \frac{1}{2} \int_{0}^{1} \int_{\R^d}(\pa_{k} \zeta^{y_0})(\ty) \frac{(y-\ty)_{l} (y-\ty)_{k}}{\sgm^{2}} U_{j}\left(t,\frac{y-\ty}{\sgm} + \ty\right) \, \frac{\rd \ty}{\sgm^{d}} \rd \sgm,  \label{def of Rlj 1}\\
	R_{lj}^{(2)}[U]
	=&	- \int_{0}^{1} \int_{\R^d}(\pa_{k} \zeta^{y_0})(\ty) \frac{(y-\ty)_{j} (y-\ty)_{l}}{\sgm^{2}} U_{k}\left(t,\frac{y-\ty}{\sgm} + \ty\right) \, \frac{\rd \ty}{\sgm^{d}} \rd \sgm.  \label{def of Rlj 2}
\end{align}

The following Proposition summarizes the basic properties of the operator $R_{lj}[U]$.
\begin{pp} \label{pp of R_lj}
	Let $U$ be a vector field on $\mcI^{q,i} \times \R^{d}$ satisfying the hypotheses \eqref{pp of supp U in E co} and \eqref{eq of linear and angular momenta in E co}  of Lemma  \ref{lm: sym div eq}. Define $R_{lj}[U] := R_{lj}^{(0)}[U] + R_{lj}^{(1)}[U] + R_{lj}^{(2)}[U]$ by \eqref{def of Rlj 0}--\eqref{def of Rlj 2}. Then $R_{lj}[U]$ possesses the following properties:
	\begin{enumerate}
		\item $R_{lj}[U]$ is symmetric in $j, l$ and depends linearly on $U$.
		\item $R_{lj}[U]$ solves the symmetric divergence equation, i.e.,
		\begin{align*}
			\pa_{j} R_{lj}[U] = U_{l}.
		\end{align*}
		\item $R_{lj}[U]$ has the support property
		\begin{equation} \label{pp of supp oRlj}
		\supp_{y} R_{lj}[U](t,\cdot) \subseteq B(y_0(t),\tmu),\quad \forall t\in\mcI^{q,i}.
		\end{equation}
		\item The time support of $R_{lj}[U]$ is contained within the time support of $U$, i.e.,
		\begin{equation} \label{pp of time supp oRlj}
			\supp_{t} R_{lj}[U]\subseteq \supp_{t} U.
		\end{equation}
		\item The following differentiation formulae hold for $R_{lj}^{(a)}[U]$ $(a=0,1,2)$:
		\begin{align} 
			\pt^r\nb^{\beta} R_{lj}^{(0)}[U] 
			=& 	-\frac{d}{2} \sum_{\beta_{1} + \beta_{2} = \beta}\sum_{r_{1} + r_{2} = r} \int_{0}^{1} \int_{\R^d}(\pt^{r_1}\nb^{\beta_{1}}\zeta^{y_0})(\ty) \frac{(y-\ty)_{j}}{\sgm} (\pt^{r_2}\nb^{\beta_{2}} U_{l})\left(t, \frac{y-\ty}{\sgm} + \ty\right) \, \frac{\rd \ty}{\sgm^{d}} \rd \sgm \notag\\
			&	-\frac{d}{2} \sum_{\beta_{1} + \beta_{2} = \beta}\sum_{r_{1} + r_{2} = r} \int_{0}^{1} \int_{\R^d}(\pt^{r_1}\nb^{\beta_{1}}\zeta^{y_0})(\ty) \frac{(y-\ty)_{l}}{\sgm} (\pt^{r_2}\nb^{\beta_{2}} U_{j})\left(t, \frac{y-\ty}{\sgm} + \ty\right) \, \frac{\rd \ty}{\sgm^{d}} \rd \sgm, \label{eq of pt nabla Rlj 0} \\
			\pt^r\nb^{\beta} R_{lj}^{(1)}[U]
			=&	\frac{1}{2} \sum_{\beta_{1} + \beta_{2} = \beta}\sum_{r_{1} + r_{2} = r} \int_{0}^{1} \int_{\R^d}(\pt^{r_1}\nb^{\beta_{1}} \pa_{k} \zeta^{y_0})(\ty) \frac{(y-\ty)_{j} (y-\ty)_{k}}{\sgm^{2}} (\pt^{r_2}\nb^{\beta_{2}} U_{l})\left(t,\frac{y-\ty}{\sgm} + \ty\right) \, \frac{\rd \ty}{\sgm^{d}} \rd \sgm \notag\\
			& + \frac{1}{2} \sum_{\beta_{1} + \beta_{2} = \beta}\sum_{r_{1} + r_{2} = r} \int_{0}^{1} \int_{\R^d}(\pt^{r_1}\nb^{\beta_{1}} \pa_{k} \zeta^{y_0})(\ty) \frac{(y-\ty)_{l} (y-\ty)_{k}}{\sgm^{2}} (\pt^{r_2}\nb^{\beta_{2}} U_{j})\left(t,\frac{y-\ty}{\sgm} + \ty\right) \, \frac{\rd \ty}{\sgm^{d}} \rd \sgm, \label{eq of pt nabla Rlj 1}  \\
			\pt^r\nb^{\beta} R_{lj}^{(2)}[U]
			=&	- \sum_{\beta_{1} + \beta_{2} = \beta}\sum_{r_{1} + r_{2} = r} \int_{0}^{1} \int_{\R^d}(\pt^{r_1}\nb^{\beta_{1}} \pa_{k} \zeta^{y_0})(\ty) \frac{(y-\ty)_{j} (y-\ty)_{l}}{\sgm^{2}} (\pt^{r_2}\nb^{\beta_{2}} U_{k})\left(t,\frac{y-\ty}{\sgm} + \ty\right) \, \frac{\rd \ty}{\sgm^{d}} \rd \sgm. \label{eq of pt nabla Rlj 2} 
		\end{align}
		where the summations are over all pairs of multi-indices such that $\beta_{1} + \beta_{2} = \beta$ and $0\leqslant|\beta|+r\leqslant L$.
	\end{enumerate}
\end{pp}

\begin{proof} 
The symmetry in $j$ and $l$, the linear dependence on $U$, and the support properties \eqref{pp of supp oRlj}  and \eqref{pp of time supp oRlj} are evident from the definitions \eqref{def of Rlj 0}--\eqref{def of Rlj 2}. Next, we will prove the differentiation formulae \eqref{eq of pt nabla Rlj 0}--\eqref{eq of pt nabla Rlj 2}.

To justify the various calculations to follow, such as differentiating under the integral sign, we will use the following lemma, the proof of which will be provided in Lemma \ref{lm: Lp est on f}:
	\begin{lm} \label{lm: integrable}
		Given $\tmu > 0, y_1\in\R^d$, let $\tilde{\zeta}$ be a non-negative smooth function with $\supp_{y} \tilde{\zeta} \subseteq B(y_{1},\tmu)$ such that 
		\begin{equation} \label{est on zeta}
			\nrm{\tilde{\zeta}}_{C^{0}_{y}} \leqslant \tilde{C} \tmu^{-d},
		\end{equation}
		for some $\tilde{C} > 0$. Then for any $k \geq 0$ and $f \in L^{\infty}_{y}$ supported in $B(y_1,\tmu)$, we have
		\begin{equation} \label{eq of integrable}
			\sup_{y \in \R^{d}, \sgm \in [0,1]} \abs{\int_{\R^d} \tilde{\zeta}(\ty) \left( \frac{|y-\ty|}{\sgm} \right)^{k} f\left(\frac{y-\ty}{\sgm} + \ty\right) \frac{\rd\ty}{\sgm^{d}}} \leqslant C_{k} \tilde{C} \tmu^{k} \nrm{f}_{L^{\infty}_{y}},
		\end{equation}
		for some constant $C_k>0$.
	\end{lm}
	
	In order to get \eqref{eq of pt nabla Rlj 0}, we should only prove the case $|\beta| =1$, i.e.,
	\begin{equation} \label{eq of pm Rlj}
		\begin{aligned}
			\pa_{y_m} R_{lj}^{(0)}[U](t,y) 
			=& - \frac{d}{2} \int_{0}^{1} \int_{\R^d}(\pa_{m} \zeta^{y_0})(\ty) \frac{(y-\ty)_{j}}{\sgm} U_{l}\left(t, \frac{y-\ty}{\sgm} + \ty\right) \, \frac{\rd\ty}{\sgm^d} \rd \sgm \\
			&- \frac{d}{2} \int_{0}^{1} \int_{\R^d}\zeta^{y_0}(\ty) \frac{(y-\ty)_{j}}{\sgm} (\pa_{m} U_{l})\left(t, \frac{y-\ty}{\sgm} + \ty\right) \, \frac{\rd\ty}{\sgm^d} \rd \sgm \\
			& + \hbox{(Symmetric terms in $j, l$)}.
	\end{aligned}\end{equation}

	The case of $|\beta| >1$ will follow from an induction argument, using similar ideas. To prove \eqref{eq of pm Rlj}, we first proceed as follows:
	\begin{align}
		\label{eq of pm Rlj 1}
		\begin{split}
			\frac{\pa}{\pa y_{m}} R_{lj}^{(0)}[U](t,y)
			= & \frac{\pa}{\pa z_{m}} R_{lj}^{(0)}[U](t,y+z) \Big\vert_{z=0} \\
			= & -\frac{d}{2} \frac{\pa}{\pa z_{m}} \Big\vert_{z=0}  \int_{0}^{1} \int_{\R^d}\zeta^{y_0}(\ty) \frac{(y+z-\ty)_{j}}{\sgm} U_{l}\left(t, \frac{y+z-\ty}{\sgm} + \ty\right) \, \frac{\rd\ty}{\sgm^{d}} \rd \sgm \\
			&	-\frac{d}{2} \frac{\pa}{\pa z_{m}} \Big\vert_{z=0} \int_{0}^{1} \int_{\R^d}\zeta^{y_0}(\ty) \frac{(y+z-\ty)_{l}}{\sgm} U_{j}\left(t, \frac{y+z-\ty}{\sgm} + \ty\right) \, \frac{\rd\ty}{\sgm^{d}} \rd \sgm \\
			= & -\frac{d}{2} \frac{\pa}{\pa z_{m}} \Big\vert_{z=0}  \int_{0}^{1} \int_{\R^d}\zeta^{y_0}(\oy+z) \frac{(y-\oy)_{j}}{\sgm} U_{l}\left(t, \frac{y-\oy}{\sgm} + \oy+z\right) \, \frac{\rd \oy}{\sgm^{d}} \rd \sgm \\
			&	-\frac{d}{2} \frac{\pa}{\pa z_{m}} \Big\vert_{z=0} \int_{0}^{1} \int_{\R^d}\zeta^{y_0}(\oy+z) \frac{(y-\oy)_{l}}{\sgm} U_{j}\left(t, \frac{y-\oy}{\sgm} + \oy+z\right) \, \frac{\rd\oy}{\sgm^{d}} \rd \sgm,  
		\end{split}
	\end{align}
	where we obtain  the last equality by making the change of variable $\oy=\ty-z$.\par 
	Next, we differentiate under the integral sign, which is justified by \eqref{eq of supp zeta}, \eqref{eq of zeta c0}, Lemma \ref{lm: integrable} and the smoothness of $U$, we get the desired formula \eqref{eq of pm Rlj}. The time differentiation is straightforward, and thus we obtain \eqref{eq of pt nabla Rlj 0}.
	The proofs of \eqref{eq of pt nabla Rlj 1} and \eqref{eq of pt nabla Rlj 2} are similar and thus omitted.
	
	Now, it only remains to prove that $R_{lj}[U]$ is a (distributional) solution to \eqref{sym div eq} under the assumptions  \eqref{pp of supp U in E co} and \eqref{eq of linear and angular momenta in E co}. For this purpose, it suffices to show that
	\begin{align*}
		R_{lj}[U](t,y) = {}^{(\zeta^{y_0}(\cdot))} \oR_{lj}[U(t, \cdot)](y) .
	\end{align*}
	To arrive at the formulae \eqref{def of Rlj 0}--\eqref{def of Rlj 2}, we need to integrate by parts the derivative $\pa_{k}$ on the outside of \eqref{eq:yrlj1} and \eqref{eq:yrlj2} after averaging against $\zeta^{y_0}(\ty)$. More precisely, consider the expression
	\begin{align*}
		\oR_{lj}^{(2)}[U](t,y) := \int_{\R^d}\zeta^{y_0}(\ty) {}^{(\ty)} r_{lj}^{(2)}[U(t, \cdot)](y) \, \rd\ty.
	\end{align*}
	
	Using \eqref{eq of supp zeta}, \eqref{eq of zeta c0}, Lemma \ref{lm: integrable} and the differentiation formulae that we established, it is not difficult to get
	\begin{align}
		\label{eq of tRlj 2}
		\begin{split}
			\oR_{lj}^{(2)}[U](t,y)
			=& -\frac{\pa}{\pa y_{k}} \int_{0}^{1} \int_{\R^d}(1-\sgm) \zeta^{y_0}(\ty) \frac{(y-\ty)_{l} (y-\ty)_{j}}{\sgm^{2}} U_{k}\left(t, \frac{y-\ty}{\sgm}+\ty\right)  \, \frac{\rd\ty}{\sgm^{d}} \rd \sgm \\
			=& -\int_{0}^{1} \int_{\R^d}(1-\sgm) (\pa_{k} \zeta^{y_0})(\ty) \frac{(y-\ty)_{l} (y-\ty)_{j}}{\sgm^{2}} U_{k}\left(t, \frac{y-\ty}{\sgm}+\ty\right)  \, \frac{\rd\ty}{\sgm^{d}} \rd \sgm \\
			& - \int_{0}^{1} \int_{\R^d}(1-\sgm) \zeta^{y_0}(\ty) \frac{(y-\ty)_{l} (y-\ty)_{j}}{\sgm^{2}} (\pa_{k} U_{k})\left(t, \frac{y-\ty}{\sgm}+\ty\right)  \, \frac{\rd\ty}{\sgm^{d}} \rd \sgm.
		\end{split}
	\end{align}
	
	Moreover, we could use
	\begin{align*}
		(\pa_{k} U_{k})\left(t, \frac{y-\ty}{\sgm}+\ty\right) = - \frac{\sgm}{1-\sgm} \frac{\pa}{\pa \ty_{k}} \Big[ U_{k}\left(t, \frac{y-\ty}{\sgm} + \ty\right) \Big],
	\end{align*}
	to obtain
	\begin{align*}
		\oR_{lj}^{(2)}[U](t,y)
		=& - \int_{0}^{1} \int_{\R^d}(\pa_{k} \zeta^{y_0})(\ty) \frac{(y-\ty)_{l} (y-\ty)_{j}}{\sgm^{2}} U_{k}\left(t, \frac{y-\ty}{\sgm}+\ty\right)  \, \frac{\rd\ty}{\sgm^{d}} \rd \sgm \\
		& + \int_{0}^{1} \int_{\R^d}\zeta^{y_0}(\ty) \frac{(y-\ty)_{l}}{\sgm} U_{j}\left(t, \frac{y-\ty}{\sgm}+\ty\right)  \, \frac{\rd\ty}{\sgm^{d}} \rd \sgm \\
		&+ \int_{0}^{1} \int_{\R^d}\zeta^{y_0}(\ty) \frac{(y-\ty)_{j}}{\sgm} U_{l}\left(t, \frac{y-\ty}{\sgm}+\ty\right)  \, \frac{\rd\ty}{\sgm^{d}} \rd \sgm.
	\end{align*}
	
	Similarly, we compute
	\begin{align*}
		\oR_{lj}^{(1)}[U](t,y)
		:=& 	\int_{\R^d}\zeta^{y_0}(\ty) {}^{(\ty)} r_{lj}^{(1)}[U(t, \cdot)](y) \, \rd\ty \\
		=&	\int_{0}^{1} \int_{\R^d}(\pa_{k} \zeta^{y_0})(\ty)  \frac{(y-\ty)_{j} (y-\ty)_{k}}{\sgm^{2}} U_{l}\left(t, \frac{y-\ty}{\sgm}+\ty\right)  \, \frac{\rd\ty}{\sgm^{d}} \rd \sgm \\
		&	- \frac{d+1}{2} \int_{0}^{1} \int_{\R^d}\zeta^{y_0}(\ty) \frac{(y-\ty)_{l}}{\sgm} U_{j}\left(t, \frac{y-\ty}{\sgm}+\ty\right)  \, \frac{\rd\ty}{\sgm^{d}} \rd \sgm \\
		& + \hbox{(Symmetric terms in $j,  l$)},
	\end{align*}
	and
	\begin{align*}
		\oR_{lj}^{(0)}[U](t,y)
		:=& \int_{\R^d}\zeta^{y_0}(\ty) {}^{(\ty)} r_{lj}^{(0)}[U(t, \cdot)](y) \, \rd\ty \\
		=& - \frac{1}{2}\int_{0}^{1} \int_{\R^d}\zeta^{y_0}(\ty) \frac{(y-\ty)_{l}}{\sgm} U_{j}\left(t, \frac{y-\ty}{\sgm}+\ty\right)  \, \frac{\rd\ty}{\sgm^{d}} \rd \sgm \\
		&- \frac{1}{2}\int_{0}^{1} \int_{\R^d}\zeta^{y_0}(\ty) \frac{(y-\ty)_{j}}{\sgm} U_{l}\left(t, \frac{y-\ty}{\sgm}+\ty\right)  \, \frac{\rd\ty}{\sgm^{d}} \rd \sgm.
	\end{align*}
	
	It therefore follows that 
	\begin{equation*}
		{}^{(\zeta^{y_0}(\cdot))} \oR_{lj}[U](t,y) 
		= \oR_{lj}^{(0)}[U](t,y) + \oR_{lj}^{(1)}[U](t,y) + \oR_{lj}^{(2)}[U](t,y)
		= R_{lj}[U](t,y).\qedhere
	\end{equation*}
\end{proof}

\subsubsection{Estimates for the solution operator and proof of Lemma \ref{lm: sym div eq}}
In this subsection, we  derive a key technical lemma (Lemma \ref{lm: Lp est on f}) that provides $ L^{\infty} $ estimates for the operator $ R_{lj}[U] $ (Lemma \ref{lm: Lp est on R}). Finally, we apply the results obtained to prove Lemma \ref{lm: sym div eq}.
\begin{lm} \label{lm: Lp est on f}
	Given $\tmu > 0, y_1\in\R^d$, let $\tilde{\zeta}$ be a non-negative smooth function with $\supp_{y} \tilde{\zeta} \subseteq B(y_1,\tmu)$ such that 
	\begin{equation} \label{eq:keyLp:hyp}
		\nrm{\tilde{\zeta}}_{C^{0}_{y}} \leqslant \tilde{C} \tmu^{-d},
	\end{equation}
	for some $\tilde{C}> 0$. Then the following statements hold:
	
	\begin{enumerate}
				\item For any $k\geqslant 0$ and $f\in L^\infty(\R^d)$ supported in $B(y_1,\tmu)$, we have
		\begin{equation} \label{eq of Lp est on f 1}
			\sup_{y\in\R^d,\sgm\in[0,1]}\abs{\int_{\R^d}\tilde{\zeta}(\ty)\left(\frac{|y-\ty|}{\sgm}\right)^kf\left(\frac{y-\ty}{\sgm}+\ty\right)\frac{\rd\ty}{\sgm^d}}\leqslant C_k\tilde{C}\tmu^k\nrm{f}_{L^\infty(\R^d)}.
		\end{equation}
		
		\item Moreover, for any $k\geqslant 0$ and $f\in L^\infty(\R^d)$ supported in $B(y_1,\tmu)$, we have
		\begin{equation} \label{eq of Lp est on f 2}
			\Nrm{\int_0^1\int_{\R^d}\tilde{\zeta}(\ty)\left(\frac{|y-\ty|}{\sgm}\right)^kf\left(\frac{y-\ty}{\sgm}+\ty\right)\frac{\rd\ty}{\sgm^d}\rd\sgm}_{L_y^\infty}\leqslant C_k\tilde{C}\tmu^k\nrm{f}_{L^\infty(\R^d)}.
		\end{equation}
	\end{enumerate}
\end{lm}
\begin{proof} 
	It is evident that \eqref{eq of Lp est on f 2} can be derived from \eqref{eq of Lp est on f 1}. We start by reducing \eqref{eq of Lp est on f 1} to the case $ k = 0 $. Assuming that the case $ k = 0 $ holds for \eqref{eq of Lp est on f 1}, we can then apply the triangle inequality,
	\begin{align*}
		\frac{|y-\ty|}{\sgm} \leqslant \abs{\frac{y-\ty}{\sgm} + \ty - y_{1}} + \abs{\ty - y_{1}}.
	\end{align*}
	
	Observe that within the integral, the first and second terms on the right-hand side are each $\leqslant \tmu$, due to the support properties of $f$ and $\tilde{\zeta}$, respectively. This implies
	\begin{align*}
		\left( \frac{|y-\ty|}{\sgm} \right)^{k} \leqslant 2^{k} \tmu^{k},
	\end{align*}
	which implies the $k > 0$ case of \eqref{eq of Lp est on f 1} and \eqref{eq of Lp est on f 2}, respectively.
	
	Next, we could prove \eqref{eq of Lp est on f 1} in the case $k=0$. We first calculate
	\begin{equation} \label{eq of Lp est 1}
		\sup_{y} \abs{\int_{\R^d} \tilde{\zeta}(\ty)  f\left(\frac{y-\ty}{\sgm} + \ty\right)\, \frac{\rd \ty}{\sgm^{d}}} 
		\leqslant  C \tilde{C}\sgm^{-d} \nrm{f} _{L^{\infty}_{y}}.
	\end{equation}
	Note that this estimate becomes less effective as $\sgm \to 0$. However, by making the change of variables
	$
		z = \frac{y-\ty}{\sgm} + \ty,
	$
	we have
	\begin{equation} \label{eq of Lp est 3}
		\begin{aligned}
			\sup_{y} \abs{\int_{\R^d} \tilde{\zeta}(\ty)  f\left(\frac{y-\ty}{\sgm} + \ty\right)\, \frac{\rd \ty}{\sgm^{d}}} 
			= \sup_{y} \abs{\int_{\R^d} \tilde{\zeta} \left( \frac{1}{1-\sgm} y - \frac{\sgm}{1-\sgm} z \right)  f(z)  \, \frac{\rd z}{(1-\sgm)^{d}}} \leqslant  C \tilde{C} (1-\sgm)^{-d} \nrm{f} _{L^{\infty}_{y}}.
		\end{aligned}
	\end{equation}

	Combining \eqref{eq of Lp est 1} and \eqref{eq of Lp est 3}, we obtain \eqref{eq of Lp est on f 1}.
\end{proof}

As a consequence of the previous lemma and the differentiation formulae \eqref{eq of pt nabla Rlj 0}--\eqref{eq of pt nabla Rlj 2}, we obtain the following $C^0$ estimates for $R_{lj}$ and the commutator between $\nb^{\beta}$ and $R_{lj}$.
\begin{lm} [$C^0$ bounds for $R_{lj}$] \label{lm: Lp est on R}
	Let $U$ be a smooth vector field on $\mcI^{q,i} \times \R^{d}$ satisfying the hypotheses  \eqref{pp of supp U in E co} and \eqref{eq of linear and angular momenta in E co} in  Lemma \ref{lm: sym div eq}. Define $R_{lj}[U] := R_{lj}^{(0)}[U] + R_{lj}^{(1)}[U] + R_{lj}^{(2)}[U]$ by \eqref{def of Rlj 0}--\eqref{def of Rlj 2}. Then we have
	\begin{equation} \label{eq of R c^j}
		\nrm{R_{lj}[U]}_{0} \leqslant C \tmu \nrm{U}_{0}.
	\end{equation}
\end{lm}
\begin{lm}[Commutator between $\nb^{\beta}$ and $R_{lj}$] \label{lm: commutator Rlj}
	Let $U_{l}$ and $R_{lj}[U]$ be as in the hypotheses of Lemma \ref{lm: Lp est on R}. 
	Then for every multi-index $\beta$, we have
	\begin{equation} \label{eq of commutator Rlj}
		\nrm{[\nb^{\beta}, R_{lj}][U]}_{0} \leqslant C_{\beta} \tmu \sum_{\beta_{1}+\beta_{2} = \beta : \beta_{2} \neq \beta} \tmu^{-\abs{\beta_{1}}} \nrm{\nb^{\beta_{2}} U}_{0}.
	\end{equation}
\end{lm}
These lemmas follow directly from applying Lemma \ref{lm: Lp est on f} to the differentiation formulae \eqref{eq of pt nabla Rlj 0}--\eqref{eq of pt nabla Rlj 2}, taking into consideration the properties \eqref{eq of supp zeta} and \eqref{eq of zeta c0} of $\zeta^{y_0}$.
\begin{proof}[Proof of Lemma \ref{lm: sym div eq}]
Considering Proposition \ref{pp of R_lj} and Lemma \ref{lm: integrable}--\ref{lm: commutator Rlj}, it remains to establish estimate \eqref{est of nabla oR}. By using \eqref{eq of pt nabla Rlj 0}--\eqref{eq of pt nabla Rlj 2},  and \eqref{eq of Lp est on f 2}, we have for $r = 0, 1, 2$ and $0\leqslant N+r\leqslant L$,
\begin{align*}
\nrm{R_{lj}[U]}_{N}\lesssim&_L \sum_{N_1+N_2=N}(\tmu\nrm{\zeta^{y_0}}_{N_1}+\tmu^2\nrm{\zeta^{y_0}}_{N_1+1})\nrm{U}_{N_2} 
\lesssim_{L}A \tmu \sum_{m=0}^N\tmu^{-N+m}\tlm^{m}\lesssim_{L}A \tmu (\tmu^{-1}+\tlm)^N,\\
\nrm{ \pa_tR_{lj}[U]}_{N}\lesssim&_L \sum_{N_1+N_2=N}(\tmu\nrm{\pt \zeta^{y_0}}_{N_1}+\tmu^2\nrm{\pt \zeta^{y_0}}_{N_1+1})  \nrm{U}_{N_2}+\sum_{N_1+N_2=N}(\tmu\nrm{\zeta^{y_0}}_{N_1}+\tmu^2\nrm{\zeta^{y_0}}_{N_1+1}) \nrm{\pt U}_{N_2}\\
\lesssim&_L \tmu \sum_{N_1+N_2=N}  \tmu^{-N_1-1}\nrm{\pt y_0}_{C_{t}^0} A\tlm^{N_2}+ \tmu \sum_{N_1+N_2=N}  \tmu^{-N_1} A\tlm^{N_2+1} \\
\lesssim&_{L}A \tmu (\tlm+\tmu^{-1}\nrm{\pt y_0}_{C_{t}^0})\sum_{m=0}^N\tmu^{-N+m}\tlm^{m}\\
\lesssim&_{L}A \tmu (\tlm+\tmu^{-1}\nrm{\pt y_0}_{C_{t}^0})(\tmu^{-1}+\tlm)^N,\\
\nrm{ \pa_t^2R_{lj}[U]}_{N}\lesssim&_L \sum_{N_1+N_2=N}\sum_{r_1+r_2=2}(\tmu\nrm{\pt^{r_1}\zeta^{y_0}}_{N_1}+\tmu^2\nrm{\pt^{r_1}\zeta^{y_0}}_{N_1+1}) \nrm{\pt^{r_2} U}_{N_2}\\
\lesssim&_L \tmu \sum_{N_1+N_2=N}  \tmu^{-N_1-1}\nrm{\pt^2 y_0}_{C_{t}^0} A\tlm^{N_2}+ \tmu \sum_{N_1+N_2=N}  \tmu^{-N_1-2}\nrm{\pt y_0}_{C_{t}^0}^2 A\tlm^{N_2}\\
&+ \tmu \sum_{N_1+N_2=N}  \tmu^{-N_1-1}\nrm{\pt y_0}_{C_{t}^0} A\tlm^{N_2+1}+ \tmu \sum_{N_1+N_2=N}  \tmu^{-N_1} A\tlm^{N_2+2} \\
\lesssim&_{L}A \tmu(\tlm^2+\tmu^{-1}\tlm\nrm{\pt y_0}_{C_{t}^0}+\tmu^{-1}\nrm{\pt^2 y_0}_{C_{t}^0}+\tmu^{-2}\nrm{\pt y_0}_{C_{t}^0}^2)\sum_{m=0}^N\tmu^{-N+m}\tlm^{m}\\
\lesssim&_{L}A \tmu(\tlm^2+\tmu^{-1}\tlm\nrm{\pt y_0}_{C_{t}^0}+\tmu^{-1}\nrm{\pt^2 y_0}_{C_{t}^0}+\tmu^{-2}\nrm{\pt y_0}_{C_{t}^0}^2)(\tmu^{-1}+\tlm)^N,
\end{align*}
where we have used
\begin{equation*}
\pt\zeta^{y_0}(y)=-(\pt y_0\cdot\nabla)\zeta^{y_0},\quad\pt^2\zeta^{y_0}(y)=-(\pt^2 y_0\cdot\nabla)\zeta^{y_0}+(\pt y_0)^{\top}\nabla^2\zeta^{y_0}(\pt y_0).\qedhere
\end{equation*}
\end{proof}
\subsection{Solving the asymmetric divergence equation}
In this subsection, we will use the conclusion in Lemma \ref{lm: sym div eq} to solve the asymmetric equation \eqref{asymmetric divergence equation}. The following is a main result regarding compactly supported solutions to the asymmetric divergence equation on $\R^d$ for any $d\geq2$:
\begin{align}\label{asym div eq}
	\Div (F\tR)=U.
\end{align}

\begin{pp}[Compactly supported solutions to the asymmetric divergence equation] \label{lm: asym div eq}
	Let $d\geqslant2$ and $L\geqslant 1$ be integers, and set $\nrm{\cdot}_{N}=\nrm{\cdot}_{C^0(\mcI^{q,i};C^N(\R^d))}$. There exists a sufficiently small constant $\varepsilon_0=\varepsilon_0(d)<1/2$ such that the following holds. Assume that
	\begin{align*}
	A>0,\quad \tmu>0,\quad \tlu\geqslant1
	,\quad 0<\varepsilon\leqslant\varepsilon_0,
	\end{align*} 
	and consider a smooth deformation $y(t,x)=x+u(t,x)\in C^\infty(\mcI^{q,i}\times \R^d;\R^d)$ with $F=\nb_xy$ satisfying
	\begin{align}
		\nrm{u}_{0}\leqslant 2\varepsilon,\qquad\nrm{\nabla u}_{0}, \nrm{\pa_t u}_{0}\leqslant \varepsilon,\quad \nrm{\pa_t^r u}_{N}\leqslant \varepsilon\tlu^{N+r-1},\quad 2\leqslant N+r\leqslant L+2,\quad N,r\geqslant0.
	\end{align}
	
	We consider a smooth vector field  $U\in C^\infty(\mcI^{q,i}\times\R^d;\R^d)$ which satisfies 
	\begin{align}
		\supp_{x} U(t,\cdot) \subseteq  B(x_0,\tmu),\quad t \in \mcI^{q,i},\label{pp of supp U in L co}
	\end{align}
	for some $x_0\in\R^{d}$, and has vanishing linear and angular moments in Lagrangian coordinates, i.e., for $j,l=1,\ldots,d$ and all $t$,
	\begin{align} 
		&\int_{\R^d} U_{l}(t,x) \, \rd x =0,\quad \int_{\R^d} (y_{j} U_{l} - y_{l} U_{j})(t,x) \, \rd x =0. \label{eq of linear and angular momenta in L co}
	\end{align}
	Moreover, assume that, for some constant $\tlU>0$, the vector field $U$ satisfies the following estimates for $r=0,1,2$:
	\begin{equation} \label{est on U in L co} 
		\begin{aligned}
			\nrm{\pt^r U}_{N} \leqslant& A\tlU^{N+r},	\qquad 0\leqslant N+r\leqslant L.
		\end{aligned}
	\end{equation}
	Then there exists a symmetric tensor field $\tR[U]=(\tR_{lj}[U])_{1\leqslant l,j\leqslant d}$ solving the asymmetric divergence equation \eqref{asym div eq}. The solution depends linearly on $U$ for fixed $y$ and has the following properties:
	\begin{enumerate}
		\item There exists a constant $C_d>0$, depending only on $d$, such that the support of $\tR_{lj}[U]$ is contained in the ball $B(x_0,(1+C_d\varepsilon)\tmu)$ for all $t\in\mcI^{q,i}$, i.e.,
		\begin{equation} \label{pp of supp oR in L co}
			\supp_{x} \tR_{lj}[U](t,\cdot) \subseteq B(x_0,(1+C_d\varepsilon)\tmu),\quad \forall t\in\mcI^{q,i}.
		\end{equation}
		\item The time support of $\tR_{lj}[U]$ is contained in the time support of $U$, i.e.,
		\begin{equation} \label{pp of time supp tRlj}
			\supp_{t} \tR_{lj}[U]\subseteq \supp_{t} U.
		\end{equation}
		\item There exists a constant $C_{d,L}^*>0$, depending only on $d$ and $L$, such that for $r=0,1,2$,
		\begin{equation} \label{est of nabla oR in L co}
			\nrm{\pt^r\tR_{lj}[U]}_{N} \leqslant C_{d,L}^*A \tmu \left(\max\bra{\varepsilon\tmu^{-1},\tlU,\tlu}\right)^r\left(\max\bra{\tmu^{-1},\tlU,\tlu}\right)^{N}, \quad 0\leqslant N\leqslant L-r.
		\end{equation}
	\end{enumerate}
	
\end{pp}
\begin{Rmk}
	As in Remark \ref{Rmk on curl}, the second condition in \eqref{eq of linear and angular momenta in L co} can be written as
	\begin{align*} 
		\int_{\R^d} (y\times U)(t,x) \, \rd x = 0.
	\end{align*}
\end{Rmk}
\begin{proof}
	Throughout the proof, let $c_d>0$ be a sufficiently large constant depending only on $d$ such that all the dimensional estimates below hold, and set $C_d:=2(c_d+1)$. We choose $\varepsilon_0(d)>0$ sufficiently small so that $C_d\varepsilon_0(d)\leqslant\frac12$. By the bound $\nrm{\nabla u}_{0}\leqslant\varepsilon$, we have
	$$
	|u(t,x_1)-u(t,x_2)|\leqslant c_d\varepsilon|x_1-x_2|,
	\qquad
	\nrm{\det(\Id+\nabla u)-1}_{0}\leqslant c_d\varepsilon.
	$$
	Consequently,
	$$
	|y(t,x_1)-y(t,x_2)|\geqslant(1-c_d\varepsilon)|x_1-x_2|,
	\qquad
	\frac12\leqslant J(t,x)\leqslant\frac32,
	$$
	where $F=\nabla_xy=\Id+\nabla u$ and $J=\det F$. Furthermore, for every $\ty\in\R^d$, the map $x\mapsto\ty-u(t,x)$ is a contraction on $\R^d$. Hence, $y(t,\cdot)$ is a global smooth diffeomorphism of $\R^d$. Denoting its inverse by $X(t,\cdot)$, we have
	$
	y(t,X(t,\ty))=\ty,
	\
	X(t,y(t,\tx))=\tx
	$
	for all $\tx,\ty\in\R^d$. We then define
	$$
	U^*(t,\ty)=U(t,X(t,\ty)),
	\qquad
	\rho^*(t,\ty)=J^{-1}(t,X(t,\ty)).
	$$
	
	As in \eqref{law of conservation in Ec}, the change of variables $\ty=y(t,x)$, together with \eqref{pp of supp U in L co} and \eqref{eq of linear and angular momenta in L co}, shows that $\rho^*U^*$ satisfies \eqref{eq of linear and angular momenta in E co}, namely,
	\begin{align*}
		\int_{\R^d}(\rho^*U^*_{l})(t,\ty)\rd\ty&=\int_{\R^d}U_l(t,x)\rd x=0,\\
		\int_{\R^d}\rho^*(\ty_jU^*_l-\ty_lU^*_j)(t,\ty)\rd\ty&=\int_{\R^d}(y_jU_l-y_lU_j)(t,x)\rd x=0.
	\end{align*}
	For any $\tx\in\supp_xU(t,\cdot)\subseteq B(x_0,\tmu)$ and $t\in\mcI^{q,i}$, setting $y_0(t)=y(t,x_0)$, we have
	\begin{align*}
		|y(t,\tx)-y_0(t)|&=|y(t,\tx)-y(t,x_0)|\leqslant|u(t,\tx)-u(t,x_0)|+|\tx-x_0|
		\leqslant(1+c_d\nrm{\nabla u}_{0})|\tx-x_0|\leqslant(1+c_d\varepsilon)\tmu.
	\end{align*}
	It follows that $\supp_yU^*(t,\cdot)\subseteq B(y_0(t),(1+c_d\varepsilon)\tmu)$.
	
	Consequently, $\rho^*U^*$ satisfies \eqref{pp of supp U in E co} and \eqref{eq of linear and angular momenta in E co} with center $y_0(t)$ and radius $(1+c_d\varepsilon)\tmu$. Applying the solution operator in Lemma \ref{lm: sym div eq}, we obtain a symmetric tensor $R[\rho^*U^*]$ satisfying the support and time-support properties \eqref{pp of supp oR}--\eqref{pp of time supp oR} and
	$$
	\Div_yR[\rho^*U^*]=\rho^*U^*.
	$$
	As in \eqref{Transport of Div R} and \eqref{asymmetric divergence equation 1}, we define
	\begin{align}
		\tR[U](t,x)=JF^{-1}R[\rho^*U^*](t,y(t,x))F^{-\top}.
	\end{align}
	Since $R[\rho^*U^*]$ is symmetric and depends linearly on $\rho^*U^*$, the tensor $\tR[U]$ is symmetric and depends linearly on $U$ for fixed $y$. By the Piola identity, it satisfies
	$$
	\Div_x(F\tR[U])=U.
	$$
	
	Moreover, since $\nrm{\nabla u}_{0}\leqslant\varepsilon\leqslant1/2$, we have
	\begin{align*}
		\nrm{\nabla_yX}_{0}=\nrm{(\nabla y)^{-1}}_{0}=\nrm{(\Id+\nabla u)^{-1}}_{0}=\nrm{\Id-\nabla u+\sum_{N=2}^{\infty}(-\nabla u)^N}_{0}
		\leqslant1+\nrm{\nabla u}_{0}+\frac{\nrm{\nabla u}_{0}^2}{1-\nrm{\nabla u}_{0}}=\frac{1}{1-\nrm{\nabla u}_{0}}\leqslant1+2\varepsilon.
	\end{align*}
	For any $\ty\in\supp_yR[\rho^*U^*](t,\cdot)\subseteq B(y_0(t),(1+c_d\varepsilon)\tmu)$, we have
	\begin{align*}
		|X(t,\ty)-x_0|=|X(t,\ty)-X(t,y_0(t))|\leqslant\nrm{\nabla_yX}_{0}|\ty-y_0(t)|
		\leqslant(1+2\varepsilon)(1+c_d\varepsilon)\tmu\leqslant(1+C_d\varepsilon)\tmu,
	\end{align*}
	where the last inequality follows from $\varepsilon\leqslant1/2$ and $C_d=2(c_d+1)$. Therefore,
	$$
	\supp_x\tR[U](t,\cdot)\subseteq B(x_0,(1+C_d\varepsilon)\tmu).
	$$
	Since the construction is linear and local in time, we also have $\supp_t\tR[U]\subseteq\supp_tU$.
	
	Next, we prove \eqref{est of nabla oR in L co}. We first derive estimates for $F^{-1}$ and $J$. Since
	\begin{equation}\label{der of F^-1}
		\begin{aligned}
			\pt (F^{-1}) =& -F^{-1}\pt F F^{-1},\quad \ \pa_{x_i}(F^{-1})=-F^{-1}\pa_{x_i} F F^{-1},
		\end{aligned}
	\end{equation}
	by induction on $N$ and for $r=0,1,2$, we obtain
	\begin{align*}
		\nrm{\nb F^{-1}}_{N}\lesssim&\sum_{N_1+N_2+N_3=N}\nrm{F^{-1}}_{N_1}\nrm{ F^{-1}}_{N_2}\nrm{\nabla F}_{N_3}\lesssim_{L,d}\varepsilon\tlu^{N+1},&& 0\leqslant N\leqslant L,\\
		\nrm{\pt^{r+1} F^{-1}}_{N}\lesssim&\sum_{N_1+N_2+N_3=N}\sum_{r_1+r_2+r_3=r}\nrm{\pt^{r_1}F^{-1}}_{N_1}\nrm{ \pt^{r_2}F^{-1}}_{N_2}\nrm{\pt^{1+r_3} F}_{N_3}\lesssim_{L,d}\varepsilon\tlu^{N+r+1},&& 0\leqslant N\leqslant L-r.
	\end{align*}
	Moreover, for $m\in\Z\setminus\{0\}$, we have
	\begin{equation}\label{der of J}
		\begin{aligned}
			\pt J^m
			=
			mJ^m\tr(F^{-1}\pt F),
			\qquad
			\pa_{x_i}J^m
			=
			mJ^m\tr(F^{-1}\pa_{x_i}F).
		\end{aligned}
	\end{equation}
	Furthermore,
	\begin{align*}
		\nrm{J^m}_{0}
		&\lesssim_{d,m}1,
		\\
		\nrm{\nabla J^m}_{N}
		&\lesssim_{d,m}
		\sum_{N_1+N_2+N_3=N}
		\nrm{J^m}_{N_1}
		\nrm{F^{-1}}_{N_2}
		\nrm{\nabla F}_{N_3}
		\lesssim_{L,d,m}
		\varepsilon\tlu^{N+1},
		&&
		0\leqslant N\leqslant L,
		\\
		\nrm{\pt^{r+1}J^m}_{N}
		&\lesssim_{d,m}
		\sum_{r_1+r_2+r_3=r}
		\sum_{N_1+N_2+N_3=N}
		\nrm{\pt^{r_1}J^m}_{N_1}
		\nrm{\pt^{r_2}F^{-1}}_{N_2}
		\nrm{\pt^{1+r_3}F}_{N_3}
		\lesssim_{L,d,m}
		\varepsilon\tlu^{N+r+1},
		&&
		0\leqslant N\leqslant L-r.
	\end{align*}
	Next, we derive estimates for $X(t,y)$. We calculate
	\begin{align*}
		\nabla_yX(t,y)
		&=
		F^{-1}(t,X(t,y)),
		\quad 
		\pt X(t,y)
		=
		-F^{-1}\pa_ty(t,X(t,y)),
		\\
		\pt^2X(t,y)
		&=
		-F^{-1}
		\left(
		\pa_{tt}y
		-2(\pt F)F^{-1}\pt y
		+\nabla^2y:
		(F^{-1}\pt y)\otimes(F^{-1}\pt y)
		\right)(t,X(t,y)).
	\end{align*}
		Then, for $1\leqslant N\leqslant L$, Lemma \ref{lm: ffhs} and an induction on $N$ yield
		\begin{align*}
		\nrm{\nabla_y X}_{0}\leqslant&\nrm{F^{-1}}_{0}\lesssim_d 1,\quad
		\nrm{\pt X}_{0}\leqslant\nrm{F^{-1}}_{0}\nrm{\pt y}_{0}\lesssim_d\varepsilon,\\
		\nrm{\pt^2X}_{0}\leqslant&\nrm{F^{-1}}_{0}(\nrm{\pa_{tt}y}_{0}+2\nrm{\pt F}_{0}\nrm{F^{-1}}_{0}\nrm{\pt y}_{0}+\nrm{\nabla^2y}_{0}\nrm{F^{-1}\pt y}_{0}\nrm{F^{-1}\pt y}_{0})\lesssim_{L,d}\varepsilon\tlu,\\
		\nrm{\nabla_y X}_{N}\lesssim&\nrm{\nabla_yX}_{0}+\nrm{\nabla F^{-1}}_{0}\nrm{\nabla_yX}_{N-1}+\nrm{\nabla F^{-1}}_{N-1}\nrm{\nabla_yX}_{0}^N\\
		\lesssim&1+\varepsilon\tlu\nrm{\nabla_yX}_{N-1}+\varepsilon\tlu^N\nrm{\nabla_yX}_{0}^N\\
		\lesssim&1+\varepsilon\tlu(1+\varepsilon\tlu^{N-1})+\varepsilon\tlu^N\\
		\lesssim&_{L,d}1+\varepsilon\tlu^{N},\\
		\nrm{\pt X}_{N}\lesssim&\nrm{\pt X}_{0}+\nrm{\nabla(F^{-1}\pt y)}_{0}\nrm{\nabla_yX}_{{N-1}}+\nrm{\nabla(F^{-1}\pt y)}_{N-1}\nrm{\nabla_yX}_{0}^N
		\lesssim_{L,d}\varepsilon\tlu^{N},\\
		\nrm{\pt^2X}_{N}\lesssim&\nrm{\pt^2X}_{0}+\nrm{\nabla(F^{-1}(\pa_{tt}y-2(\pt F)F^{-1}\pt y+\nabla^2y:(F^{-1}\pt y)\otimes (F^{-1}\pt y)))}_{0}\nrm{\nabla_yX}_{{N-1}}\\
		&+\nrm{\nabla(F^{-1}(\pa_{tt}y-2(\pt F)F^{-1}\pt y+\nabla^2y:(F^{-1}\pt y)\otimes (F^{-1}\pt y)))}_{N-1}\nrm{\nabla_yX}_{0}^N\\
		\lesssim&_{L,d}\varepsilon\tlu^{N+1}.
	\end{align*}
	Using Lemma \ref{lm: ffhs}, we obtain
	\begin{align*}
		\nrm{\rho^*U^*}_{0}\leqslant&
		\nrm{J^{-1}}_{0}\nrm{U}_{0}
		\lesssim_d A,
		\\
		\nrm{\pt(\rho^*U^*)}_{0}\leqslant&
		\nrm{\pt(J^{-1}U)}_{0}
		+\nrm{\pt X}_{0}\nrm{\nabla(J^{-1}U)}_{0}
		\lesssim A(\tlU+\varepsilon\tlu)
		\lesssim A\left(\max\bra{\tlU,\tlu}\right),
		\\
		\nrm{\pt^2(\rho^*U^*)}_{0}\leqslant&
		\nrm{\pt^2(J^{-1}U)}_{0}
		+2\nrm{\pt X}_{0}\nrm{\nabla\pt(J^{-1}U)}_{0}
		+\nrm{\pt^2X}_{0}\nrm{\nabla(J^{-1}U)}_{0}+\nrm{\nabla^2(J^{-1}U)}_{0}\nrm{\pt X}_{0}^{2}
		\\
		\lesssim&_{d} A(\tlU^2+\varepsilon\tlu^2)
		\lesssim A\left(\max\bra{\tlU,\tlu}\right)^2,
		\\
		\nrm{\rho^*U^*}_{N}\lesssim&_{L,d}
		\nrm{\nabla(J^{-1}U)}_{0}\nrm{\nabla_yX}_{N-1}
		+\nrm{\nabla(J^{-1}U)}_{N-1}\nrm{\nabla_yX}_{0}^{N}
		\\
		\lesssim&_{L,d}
		A\left(\max\bra{\tlU,\tlu}\right)^{N},
		&& 1\leqslant N\leqslant L,
		\\
		\nrm{\pt(\rho^*U^*)}_{N}\lesssim&_{L,d}
		\nrm{\nabla\pt(J^{-1}U)}_{0}\nrm{\nabla_yX}_{N-1}
		+\nrm{\nabla\pt(J^{-1}U)}_{N-1}\nrm{\nabla_yX}_{0}^{N}+\nrm{\pt X}_{N}\nrm{\nabla(J^{-1}U)}_{0}
		\\
		&+\sum_{\substack{N_1+N_2=N\\N_2\geqslant1}}
		\nrm{\pt X}_{N_1}
		\left(
		\nrm{\nabla^2(J^{-1}U)}_{0}\nrm{\nabla_yX}_{N_2-1}
		+\nrm{\nabla^2(J^{-1}U)}_{N_2-1}
		\nrm{\nabla_yX}_{0}^{N_2}
		\right)
		\\
		\lesssim&_{L,d}
		A\left(\max\bra{\tlU,\tlu}\right)^{N+1},
		&&1\leqslant N\leqslant L-1,
		\\
		\nrm{\pt^2(\rho^*U^*)}_{N}\lesssim&_{L,d}
		\nrm{\nabla\pt^2(J^{-1}U)}_{0}\nrm{\nabla_yX}_{N-1}
		+\nrm{\nabla\pt^2(J^{-1}U)}_{N-1}\nrm{\nabla_yX}_{0}^{N}+\nrm{\pt X}_{N}\nrm{\nabla\pt(J^{-1}U)}_{0}
		\\
		&+\sum_{\substack{N_1+N_2=N\\N_2\geqslant1}}
		\nrm{\pt X}_{N_1}
		\left(
		\nrm{\nabla^2\pt(J^{-1}U)}_{0}\nrm{\nabla_yX}_{N_2-1}
		+\nrm{\nabla^2\pt(J^{-1}U)}_{N_2-1}
		\nrm{\nabla_yX}_{0}^{N_2}
		\right)
		\\
		&+\nrm{\pt^2X}_{N}\nrm{\nabla(J^{-1}U)}_{0}+\sum_{N_1+N_2=N}
		\nrm{\pt X}_{N_1}\nrm{\pt X}_{N_2}
		\nrm{\nabla^2(J^{-1}U)}_{0}
		\\
		&+\sum_{\substack{N_1+N_2=N\\N_2\geqslant1}}
		\nrm{\pt^2X}_{N_1}
		\left(
		\nrm{\nabla^2(J^{-1}U)}_{0}\nrm{\nabla_yX}_{N_2-1}
		+\nrm{\nabla^2(J^{-1}U)}_{N_2-1}
		\nrm{\nabla_yX}_{0}^{N_2}
		\right)
		\\
		&+\sum_{\substack{N_1+N_2+N_3=N\\N_3\geqslant1}}
		\nrm{\pt X}_{N_1}\nrm{\pt X}_{N_2}
		\left(
		\nrm{\nabla^3(J^{-1}U)}_{0}\nrm{\nabla_yX}_{N_3-1}
		+\nrm{\nabla^3(J^{-1}U)}_{N_3-1}
		\nrm{\nabla_yX}_{0}^{N_3}
		\right)
		\\
		\lesssim&_{L,d}
		A\left(\max\bra{\tlU,\tlu}\right)^{N+2},
		&&1\leqslant N\leqslant L-2.
	\end{align*}
	
	Since $y_0(t)=y(t,x_0)$, we have $\nrm{\pt y_0}_{C_t^0}\leqslant\varepsilon$ and $\nrm{\pt^2y_0}_{C_t^0}\leqslant\varepsilon\tlu$. Absorbing the harmless factors $1+C_d\varepsilon$ into the implicit constants and combining Lemma \ref{lm: ffhs} with \eqref{est of nabla oR} in Lemma \ref{lm: sym div eq}, we obtain the following estimates for $R_{lj}[\rho^*U^*]$ for $r=0,1,2$:
	\begin{align}
		\nrm{\pt^rR_{lj}[\rho^*U^*]}_{N}
		\lesssim&_{L,d}A \tmu \left(\max\bra{\varepsilon\tmu^{-1},\tlU,\tlu}\right)^r\left(\max\bra{\tmu^{-1},\tlU,\tlu}\right)^{N},&&0\leqslant N\leqslant L-r.\label{est on R rU^*}
	\end{align}
	Next, we calculate
	\begin{align*}
		\pa_t(R_{lj}[\rho^*U^*](t,y(t,x)))=&(\pt+\pt y\cdot\nabla) R_{lj}[\rho^*U^*](t,y(t,x)),\\
		\pa_t^2(R_{lj}[\rho^*U^*](t,y(t,x))) 	=&\pt^2R_{lj}[\rho^*U^*](t,y(t,x))+2(\pt y\cdot\nabla)\pt R_{lj}[\rho^*U^*](t,y(t,x))\\
		&+(\pt y)^{\top} \nabla^2R_{lj}[\rho^*U^*](\pt y)(t,y(t,x))+(\pt^2 y\cdot \nabla)R_{lj}[\rho^*U^*](t,y(t,x)),
	\end{align*}
	and hence
	\begin{align*}
		\nrm{\pt ( R_{lj}[\rho^*U^*](t,y(t,x)))}_{0}\lesssim&\nrm{\pt R_{lj}[\rho^*U^*]}_{0}+\nrm{\pt y}_{0}\nrm{\nabla R_{lj}[\rho^*U^*]}_{0}\\
		\lesssim&_{d}A\tmu\max\bra{\varepsilon\tmu^{-1},\tlU,\tlu} +A\varepsilon\tmu\max\bra{\tmu^{-1},\tlU,\tlu}\\
		\lesssim&_{d} A\tmu\max\bra{\varepsilon\tmu^{-1},\tlU,\tlu},\\
		\nrm{\pt^2 (R_{lj}[\rho^*U^*](t,y(t,x)))}_{0}\lesssim&\nrm{\pt^2 R_{lj}[\rho^*U^*]}_{0}+\nrm{\pt y}_{0}\nrm{\nabla\pt R_{lj}[\rho^*U^*]}_{0}
		+\nrm{\pt y}_{0}\nrm{\nabla^2R_{lj}[\rho^*U^*]}_{0}\nrm{\pt y}_{0}\\
		&+\nrm{\pt^2 y}_{0}\nrm{\nabla R_{lj}[\rho^*U^*]}_{0}\\
		\lesssim&_{d}A\tmu\left(\max\bra{\varepsilon\tmu^{-1},\tlU,\tlu}\right)^2+A\tmu\varepsilon\max\bra{\varepsilon\tmu^{-1},\tlU,\tlu}\max\bra{\tmu^{-1},\tlU,\tlu}\\
		&+A\tmu\varepsilon^2\left(\max\bra{\tmu^{-1},\tlU,\tlu}\right)^2+A\tmu\varepsilon\tlu\max\bra{\tmu^{-1},\tlU,\tlu}\\
		\lesssim&_{d} A\tmu\left(\max\bra{\varepsilon\tmu^{-1},\tlU,\tlu}\right)^2.
	\end{align*}
	By Lemma \ref{lm: ffhs} and \eqref{est on R rU^*}, for $0\leqslant r_1+r_2\leqslant 2$ and $1\leqslant N+r_1+r_2\leqslant L$, we have
	\begin{align*}
		\nrm{\pt^{r_1}\nabla^{r_2} R_{lj}[\rho^*U^*](t,y(t,x)))}_{N}
		\lesssim&_{L,d}\nrm{ \pt^{r_1}\nabla^{r_2}R_{lj}[\rho^*U^*]}_{1}\nrm{\nabla_xy}_{N-1}+\nrm{ \pt^{r_1}\nabla^{r_2}R_{lj}[\rho^*U^*]}_{N}\nrm{\nabla_xy}_{0}^N\\
		\lesssim&_{L,d}  A \tmu \left(\max\bra{\varepsilon\tmu^{-1},\tlU,\tlu}\right)^{r_1}\left(\max\bra{\tmu^{-1},\tlU,\tlu}\right)^{N+r_2},
	\end{align*}
	and consequently
	\begin{align*}
		\nrm{\pt (R_{lj}[\rho^*U^*](t,y(t,x)))}_{N}
		\lesssim&_{L,d}\nrm{\pt R_{lj}[\rho^*U^*](t,y(t,x)))}_{N}+\sum_{N_1+N_2=N}\nrm{\pt y}_{N_1}\nrm{\nabla R_{lj}[\rho^*U^*](t,y(t,x)))}_{N_2}\\
		\lesssim&_{L,d} A \tmu \left(\max\bra{\varepsilon\tmu^{-1},\tlU,\tlu}\right)\left(\max\bra{\tmu^{-1},\tlU,\tlu}\right)^{N},\\
		\nrm{\pt^2 (R_{lj}[\rho^*U^*](t,y(t,x)))}_{N}
		\lesssim&_{L,d}\nrm{\pt^2 R_{lj}[\rho^*U^*](t,y(t,x)))}_{N}+\sum_{N_1+N_2=N}\nrm{\pt y}_{N_1}\nrm{\nabla \pt R_{lj}[\rho^*U^*](t,y(t,x)))}_{N_2}\\
		&+\sum_{N_1+N_2+N_3=N}\nrm{\pt y}_{N_1}\nrm{\nabla^2R_{lj}[\rho^*U^*](t,y(t,x)))}_{N_2}\nrm{\pt y}_{N_3}\\
		&+\sum_{N_1+N_2=N}\nrm{\pt^2 y}_{N_1}\nrm{\nabla R_{lj}[\rho^*U^*](t,y(t,x)))}_{N_2}\\
		\lesssim&_{L,d} A \tmu \left(\max\bra{\varepsilon\tmu^{-1},\tlU,\tlu}\right)^2\left(\max\bra{\tmu^{-1},\tlU,\tlu}\right)^{N}.
	\end{align*}
	Finally, using the definition of $\tR[U]$, we obtain, for $r=0,1,2$,
	\begin{align*}
		\nrm{\pt^r \tR_{lj}[U]}_{N}&\lesssim\sum_{N_{1} + N_{2} + N_{3} + N_{4}= N}\sum_{r_{1} + r_{2} + r_{3} + r_{4}= r}\nrm{\pt^{r_1}J}_{N_1}\nrm{ \pt^{r_2}(F^{-1})}_{N_2}\nrm{\pt^{r_3} (R_{lj}[\rho^*U^*](t,y(t,x)))}_{N_3}\nrm{\pt^{r_4}(F^{-\top})}_{N_4}\\ 
		&\lesssim_{L,d} A \tmu \left(\max\bra{\varepsilon\tmu^{-1},\tlU,\tlu}\right)^r\left(\max\bra{\tmu^{-1},\tlU,\tlu}\right)^{N},\quad 0\leqslant N\leqslant L-r.\qedhere
	\end{align*}
\end{proof}

\section{Some notations}\label{Some notations}
Before defining the perturbation $(\tu_{q,i+1}, \tG_{q,i+1} = \nabla \tu_{q,i+1})$ and introducing the new Reynolds errors $R_{q,i+1}$, we first clarify some notations. 

For terms like $\tC_{q,i+1}^{(i)}, \tC_{q,i+1}^{(i),m}$, and $\tC_{q,i+1}^{(i),c} = \tC_{q,i+1}^{(i)} - \tC_{q,i+1}^{(i),m}$, the superscript $(i)$ indicates that the term is of magnitude approximately $\mathcal{O}(|\Gl|^k|\tG_{q,i+1}|^i)$. Here, $m$ denotes the main component, which is of magnitude approximately $\mathcal{O}(|\tG_{q,i+1}|^i)$, while $c$ represents the term of magnitude approximately $\mathcal{O}(|\Gl|^k|\tG_{q,i+1}|^i)$, with $k \geq 1$. For convenience, we use $\tC_{q,i+1}^{(0)}$ to refer to $\Cl$ itself. Similarly, this notation can be extended to other terms like $\tD, \tj$, and $\tL$. Below, we list some of these notations.
\begin{enumerate}
\item The perturbation of $\Cl$ and $\Dl$:
\begin{align}
	\tC_{q,i+1}^{(0)}&=\Cl,\quad\tD_{q,i+1}^{(0)}=\Dl,\label{def of tC 0}\\
	\tC_{q,i+1}&=\underbrace{\overbrace{\tG_{q,i+1}+\tG_{q,i+1}^{\top}}^{\tC_{q,i+1}^{(1),m}}+\Gl\tG_{q,i+1}^{\top}+\tG_{q,i+1}\Gl^{\top}}_{\tC_{q,i+1}^{(1)}}+\underbrace{\tG_{q,i+1}\tG_{q,i+1}^{\top}}_{\tC_{q,i+1}^{(2)}=\tC_{q,i+1}^{(2),m}},\label{def of tC}\\
	\tr\tC_{q,i+1}&=\underbrace{\overbrace{2\tr\tG_{q,i+1}}^{\tr\tC_{q,i+1}^{(1),m}}+2\tr(\Gl\tG_{q,i+1}^{\top})}_{\tr\tC_{q,i+1}^{(1)}}+\underbrace{\tr(\tG_{q,i+1}\tG_{q,i+1}^{\top})}_{\tr\tC_{q,i+1}^{(2)}=\tr\tC_{q,i+1}^{(2),m}},\label{def of tr tC 0}\\
	\tD_{q,i+1}&=\underbrace{\overbrace{\tG_{q,i+1}+\tG_{q,i+1}^{\top}}^{\tD_{q,i+1}^{(1),m}}+\Gl^{\top}\tG_{q,i+1}+\tG_{q,i+1}^{\top}\Gl}_{\tD_{q,i+1}^{(1)}}+\underbrace{\tG_{q,i+1}^{\top}\tG_{q,i+1}}_{\tD_{q,i+1}^{(2)}=\tD_{q,i+1}^{(2),m}}\label{def of tD}.
\end{align}
\begin{Rmk}
For convenience, we set
$
\tC_{q,i+1}^{(k)}=\tD_{q,i+1}^{(k)}=0$ for $k\geqslant3.
$
\end{Rmk}

\item The perturbation of the principal invariants $j_1$ and $j_2$ of the matrix $\Cl$:
\begin{align}
	\tj_1^{(0)}&=\tr\Cl,\quad\tj_2^{(0)}=\frac{1}{2}((\tr\Cl)^2-\tr\Cl^2),\label{def of tj 0}\\
	\tj_1&=\underbrace{\overbrace{2\tr(\tG_{q,i+1})}^{\tj_1^{(1),m}}+2\tr(\Gl\tG_{q,i+1}^{\top})}_{\tj_1^{(1)}}+\underbrace{\tr(\tG_{q,i+1}\tG_{q,i+1}^{\top})}_{\tj_1^{(2)}=\tj_1^{(2),m}}, \label{def of tj1}\\
	\tj_2&=\sum_{k=1}^{4}\underbrace{\sum_{r+s=k}\frac{1}{2}\left(\tj_1^{(r)}\tj_1^{(s)}-\tr(\tC_{q,i+1}^{(r)}\tC_{q,i+1}^{(s)})\right)}_{\tj_2^{(k)}},\label{def of tj2}\\
	\tj_2^{(2),m}&=\frac{1}{2}\left(\tj_1^{(1),m}\tj_1^{(1),m}-\tr(\tC_{q,i+1}^{(1),m}\tC_{q,i+1}^{(1),m})\right)=2(\tr\tG_{q,i+1})^2-\tr(\tG_{q,i+1}\tG_{q,i+1}^{\top}+\tG_{q,i+1}^2)\label{def of tj2 2m}.
\end{align}
\begin{Rmk}
	We set $\tj_1^{(k)}=0$ for $k\geqslant 3$, and $\tj_2^{(k)}=0$ for $k\geqslant 5.$
\end{Rmk}
\item The perturbation of coefficients $\Lambda_1$ and $\Lambda_2$:
\begin{align}
	\tL_1^{(0)}&=2\left(\sgm_{11}\tj_1^{(0)}+\sgm_{12}\tj_2^{(0)}+\frac{\sgm_{111}}{2}(\tj_1^{(0)})^2\right),\quad\tL_2^{(0)}=2\left(\sgm_2+\sgm_{12}\tj_1^{(0)}+\sgm_{22}\tj_2^{(0)}\right),\label{def of tL 012}\\
	\tL_1&=2\sgm_{11}\tj_1+2\sgm_{12}\tj_2+2\sgm_{111}\tj_1^{(0)}\tj_1+\sgm_{111}\tj_1^2=\sum_{k=1}^4\underbrace{\left(2\sgm_{11}\tj_1^{(k)}+2\sgm_{12}\tj_2^{(k)}+\sum_{r+s=k}\left(\sgm_{111}\tj_1^{(r)}\tj_1^{(s)}\right)\right)}_{\tL_1^{(k)}},\label{def of tL1}\\
	\tL_2&=2\sgm_{12}\tj_1+2\sgm_{22}\tj_2=\sum_{k=1}^4\underbrace{2\sgm_{12}\tj_1^{(k)}+2\sgm_{22}\tj_2^{(k)}}_{\tL_2^{(k)}},\label{def of tL2}
\end{align}
and their main part
\begin{align}
	\tL_1^{(1),m}&=2\sgm_{11}\tj_1^{(1),m}=4\sgm_{11}\tr\tG_{q,i+1},\label{def of tL1 1m}\\
	\tL_1^{(2),m}&=2\sgm_{11}\tj_1^{(2),m}+2\sgm_{12}\tj_2^{(2),m}+\sgm_{111}\tj_1^{(1),m}\tj_1^{(1),m}\nonumber\\
	&=2\sgm_{11}\tr(\tG_{q,i+1}\tG_{q,i+1}^{\top})+2\sgm_{12}(2(\tr\tG_{q,i+1})^2-\tr(\tG_{q,i+1}\tG_{q,i+1}^{\top}+\tG_{q,i+1}^2))+4\sgm_{111}(\tr\tG_{q,i+1})^2,\label{def of tL1 2m}\\
	\tL_2^{(1),m}&=2\sgm_{12}\tj_1^{(1),m}=4\sgm_{12}\tr\tG_{q,i+1},\label{def of tL2 1m}\\
	\tL_2^{(2),m}&=2\sgm_{12}\tj_1^{(2),m}+2\sgm_{22}\tj_2^{(2),m}=2\sgm_{12}\tr(\tG_{q,i+1}\tG_{q,i+1}^{\top})+2\sgm_{22}(2(\tr\tG_{q,i+1})^2-\tr(\tG_{q,i+1}\tG_{q,i+1}^{\top}+\tG_{q,i+1}^2)).\label{def of tL2 2m}
\end{align}
\begin{Rmk}
	We set $\tL_1^{(k)}=0$, and $\tL_2^{(k)}=0$ for $k\geqslant 5.$
\end{Rmk}
\item The perturbation of the second Piola--Kirchhoff stress tensor 	$\Sgm_{\Gl}$:
\begin{align}
	\tS_{q,i+1}=\Sgm_{G_{q,i+1}}-\Sgm_{\Gl}&=\sum_{k=1}^6\left(\tL_1^{(k)}\Id+\sum_{r+s=k}\tL_2^{(r)}(\tr\tC_{q,i+1}^{(s)}\Id-\tD_{q,i+1}^{(s)})\right)=\tS_{q,i+1}^{(1)}+\tS_{q,i+1}^{(2)}+\tS_{q,i+1}^{(\geqslant3)}\label{def of tS},
\end{align}
where
\begin{align}
	\tS_{q,i+1}^{(1)}&=\tL_1^{(1)}\Id+\tL_2^{(1)}(\tr \Cl\Id-\Dl)+\tL_2^{(0)}(\tr\tC_{q,i+1}^{(1)}\Id-\tD_{q,i+1}^{(1)}),\label{def of tS1}\\
	\tS_{q,i+1}^{(2)}&=\tL_1^{(2)}\Id+\sum_{r+s=2}\tL_2^{(r)}(\tr\tC_{q,i+1}^{(s)}\Id-\tD_{q,i+1}^{(s)}),\label{def of tS2}\\
	\tS_{q,i+1}^{(\geqslant3)}&=\sum_{k=3}^{6}\left(\tL_1^{(k)}\Id+\sum_{r+s=k}\tL_2^{(r)}(\tr\tC_{q,i+1}^{(s)}\Id-\tD_{q,i+1}^{(s)})\right). \label{def of tS3}
\end{align}

For the sake of the estimates below, we provide the full form of $\tS_{q,i+1}^{(1)}$ and  $\tS_{q,i+1}^{(2)}$ here.
\begin{equation}\label{def of tS_1}
	\begin{aligned}
		\tS_{q,i+1}^{(1)}=&4\left((\sgm_{11}+(\sgm_{12}+\sgm_{111})\tj_1^{(0)})\tr((\Id+\Gl)\tG_{q,i+1}^{\top})-\sgm_{12}\tr(\Cl(\Id+\Gl)\tG_{q,i+1}^{\top}))\right)\Id\\
		&+4\left((\sgm_{12}+\sgm_{22}\tj_1^{(0)})\tr((\Id+\Gl)\tG_{q,i+1}^{\top})-\sgm_{22}\tr(\Cl(\Id+\Gl)\tG_{q,i+1}^{\top})\right)(\tr \Cl\Id-\Dl)\\
		&+\tL_2^{(0)}\left(2\tr((\Id+\Gl)\tG_{q,i+1}^{\top})\Id-(\tG_{q,i+1}+\tG_{q,i+1}^{\top}+\Gl^{\top}\tG_{q,i+1}+\tG_{q,i+1}^{\top}\Gl)\right),
	\end{aligned}
\end{equation}
\begin{equation}\label{def of tS_2}
	\begin{aligned}
		\tS_{q,i+1}^{(2)}=&2\left((\sgm_{11}+(\sgm_{12}+\sgm_{111})\tj_1^{(0)})\tr(\tG_{q,i+1}\tG_{q,i+1}^{\top})-\sgm_{12}\tr(\Cl\tG_{q,i+1}\tG_{q,i+1}^{\top})\right)\Id\\
		&+\left(4(\sgm_{12}+\sgm_{111})(\tr((\Id+\Gl)\tG_{q,i+1}^{\top}))^2-\sgm_{12}\tr((\tG_{q,i+1}+\tG_{q,i+1}^{\top}+\Gl^{\top}\tG_{q,i+1}+\tG_{q,i+1}^{\top}\Gl)^2)\right)\Id\\
		&+\tL_2^{(0)}(\tr(\tG_{q,i+1}\tG_{q,i+1}^{\top})\Id-\tG_{q,i+1}^{\top}\tG_{q,i+1})\\
		&+4\left((\sgm_{12}+\sgm_{22}\tj_1^{(0)})\tr((\Id+\Gl)\tG_{q,i+1}^{\top})-\sgm_{22}\tr(\Cl(\Id+\Gl)\tG_{q,i+1}^{\top})\right)\\
		&\quad\cdot\left(2\tr((\Id+\Gl)\tG_{q,i+1}^{\top})\Id-(\tG_{q,i+1}+\tG_{q,i+1}^{\top}+\Gl^{\top}\tG_{q,i+1}+\tG_{q,i+1}^{\top}\Gl)\right)\\
		&+2((\sgm_{12}+\sgm_{22}\tj_1^{(0)})\tr(\tG_{q,i+1}\tG_{q,i+1}^{\top})-\sgm_{22}\tr(\Cl\tG_{q,i+1}\tG_{q,i+1}^{\top}))(\tr \Cl\Id-\Dl)\\
		&+\sgm_{22}\left(4(\tr((\Id+\Gl)\tG_{q,i+1}^{\top}))^2-\tr((\tG_{q,i+1}+\tG_{q,i+1}^{\top}+\Gl^{\top}\tG_{q,i+1}+\tG_{q,i+1}^{\top}\Gl)^2)\right)(\tr \Cl\Id-\Dl).
	\end{aligned}
\end{equation}
Moreover, we could give the main parts of $\tS_{q,i+1}^{(1)}$ and $\tS_{q,i+1}^{(2)}$:
\begin{align}
	\tS_{q,i+1}^{(1),m}&=\tL_1^{(1),m}\Id+2\sgm_2(\tr\tC_{q,i+1}^{(1),m}\Id-\tD_{q,i+1}^{(1),m})\nonumber\\
	&=4\sgm_{11}\tr\tG_{q,i+1}\Id + 2\sgm_2(2\tr(\tG_{q,i+1})\Id-\tG_{q,i+1}-\tG_{q,i+1}^{\top})
	=\lambda\tr\tG_{q,i+1}\Id+\mu(\tG_{q,i+1}+\tG_{q,i+1}^{\top}), \label{def of tS_1 m}\\
	\tS_{q,i+1}^{(2),m}&=\tL_1^{(2),m}\Id+\tL_2^{(1),m}(\tr\tC_{q,i+1}^{(1),m}\Id-\tD_{q,i+1}^{(1),m})+2\sgm_2(\tr\tC_{q,i+1}^{(2),m}\Id-\tD_{q,i+1}^{(2),m})\nonumber\\
	&=(2\sgm_{11}\tr(\tG_{q,i+1}\tG_{q,i+1}^{\top})+2\sgm_{12}(2(\tr\tG_{q,i+1})^2-\tr(\tG_{q,i+1}\tG_{q,i+1}^{\top}+\tG_{q,i+1}^2))+4\sgm_{111}(\tr\tG_{q,i+1})^2)\Id\nonumber\\
	&\quad+4\sgm_{12}\tr\tG_{q,i+1}(2\tr\tG_{q,i+1}\Id-\tG_{q,i+1}-\tG_{q,i+1}^{\top})\nonumber\\
	&\quad+2\sgm_2(\tr(\tG_{q,i+1}\tG_{q,i+1}^{\top})\Id-\tG_{q,i+1}^{\top}\tG_{q,i+1}))\nonumber\\
	&=\left(2(\sgm_{11}-\sgm_{12}+\sgm_2)\tr(\tG_{q,i+1}\tG_{q,i+1}^{\top})+4(3\sgm_{12}+\sgm_{111})(\tr\tG_{q,i+1})^2-2\sgm_{12}\tr(\tG_{q,i+1}^2)\right)\Id\nonumber\\
	&\quad-2\sgm_2\tG_{q,i+1}^{\top}\tG_{q,i+1}-4\sgm_{12}\tr\tG_{q,i+1}(\tG_{q,i+1}+\tG_{q,i+1}^{\top}).\label{def of tS_2 m}
\end{align}
\item The perturbation of the stored-energy function	$\sgm_{\Gl}$:
\begin{align}
\tsgm_{q,i+1}=&\sgm_{G_{q,i+1}}-\sgm_{\Gl}\nonumber\\
=&\sum_{k=1}^8\underbrace{\left(\sgm_2\tj_2^{(k)}+\sum_{r_1+r_2=k}\left(\frac{\sgm_{11}}{2}\tj_1^{(r_1)}\tj_1^{(r_2)}+\frac{\sgm_{22}}{2}\tj_2^{(r_1)}\tj_2^{(r_2)}+\sgm_{12}\tj_1^{(r_1)}\tj_2^{(r_2)}\right)+\frac{\sgm_{111}}{6}\sum_{r_1+r_2+r_3=k}\tj_1^{(r_1)}\tj_1^{(r_2)}\tj_1^{(r_3)}\right)}_{\tsgm_{q,i+1}^{(k)}}.
\end{align}
For convenience, we set
$
\tsgm_{q,i+1}^{(k)}=0,\ k\geqslant9,
$
and define
$
\tsgm_{q,i+1}^{(\geqslant3)}
:=\sum_{k=3}^{8}\tsgm_{q,i+1}^{(k)}.
$
Moreover, we provide the full form of $\tsgm_{q,i+1}^{(1)}$ and  $\tsgm_{q,i+1}^{(2)}$ here.
\begin{align*}
\tsgm_{q,i+1}^{(1)}=&\sgm_2\tj_2^{(1)}+\sgm_{11}\tj_1^{(0)}\tj_1^{(1)}+\sgm_{22}\tj_2^{(0)}\tj_2^{(1)}+\sgm_{12}(\tj_1^{(0)}\tj_2^{(1)}+\tj_1^{(1)}\tj_2^{(0)})+\frac{\sgm_{111}}{2}(\tj_1^{(0)})^2\tj_1^{(1)},\\
\tsgm_{q,i+1}^{(2)}=&\sgm_2\tj_2^{(2)}+\frac{\sgm_{11}}{2}(\tj_1^{(1)})^2+\sgm_{11}\tj_1^{(0)}\tj_1^{(2)}+\frac{\sgm_{22}}{2}(\tj_2^{(1)})^2+\sgm_{22}\tj_2^{(0)}\tj_2^{(2)}+\sgm_{12}\tj_1^{(1)}\tj_2^{(1)}+\sgm_{12}(\tj_1^{(0)}\tj_2^{(2)}+\tj_1^{(2)}\tj_2^{(0)})\\
&+\frac{\sgm_{111}}{2}(\tj_1^{(0)}(\tj_1^{(1)})^2+(\tj_1^{(0)})^2\tj_1^{(2)}).
\end{align*}
The main part of $\tsgm_{q,i+1}^{(2)}$ can be written as
\begin{align*}
	\tsgm_{q,i+1}^{(2),m}=&\sgm_2\tj_2^{(2),m}+\frac{\sgm_{11}}{2}(\tj_1^{(1),m})^2=2(\sgm_{2}+\sgm_{11})(\tr(\tG_{q,i+1}))^2-\sgm_2\tr(\tG_{q,i+1}\tG_{q,i+1}^{\top}+\tG_{q,i+1}^2).
\end{align*}
\end{enumerate}
\section{Definition of the perturbation }\label{Definition of the perturbation}  
In this section, we first introduce the building blocks which will be used in the construction of the perturbation.  Next, we give the construction of the perturbation. Based on it, we will  give the estimates on the perturbation. 
\subsection{Definition of building blocks}\label{Definition of Building blocks}
\begin{lm}\label{construction of building blocks}
	Given an integer $L\geqslant 2$, a nonzero vector $f\in\Z^2$, a positive integer $\tl$, and parameters $\lambda^*$ and $\mu^*$ satisfying $\mu^*>0$, $\lambda^*+2\mu^*>0$, and $\lambda^*+\mu^*\neq 0$, we consider the equation
	\begin{align*}
		\pa_{tt}w-\mu^*\Delta w-(\lambda^*+\mu^*)\nb(\Div w)+A(\nb^2w)=0,
	\end{align*} 
	where $w:(-\infty,\infty)\times\T^2\rightarrow\mathbb C^2$ and $(A(\nabla^2w))_k=A^{km}_{nr}\partial_{nr}w_m$, with $A^{km}_{nr}=A^{km}_{nr}(t)\in C^L(\R;\R)$.
	
	We seek a nontrivial complex-valued approximate solution of the form
	\begin{align}
		w_{A,f}=(f+a_A(t)f^{\perp})e^{i\xi_{A,f}},
	\end{align}
	where $f^{\perp}=(-f_2,f_1)\in\Z^2$ satisfies $f^{\perp}\perp f$, $|f^\perp|=|f|$, and $\xi_{A,f}=\tl(f\cdot x-c_A(t))$. Here, $a_A(t)$ and $c_A(t)$ are real-valued functions to be determined by the coefficients $A^{km}_{nr}$ and the vector $f$. Since $\tl\in\N$ and $f\in\Z^2$, the function $w_{A,f}$ is well-defined on $\T^2$.
	
	Suppose that, for all indices $1\leqslant k,m,n,r\leqslant2$, the coefficients $A^{km}_{nr}$ satisfy
	$
	\nrm{A^{km}_{nr}}_{C_t^0}\leqslant C\varepsilon
	$
	for some constant $C>0$, and
	$
	\nrm{A^{km}_{nr}}_{C_t^N}\leqslant\tA\tlA^N,\ 1\leqslant N\leqslant L,
	$
	for some constants $0<\tA\leqslant1$ and $\tlA\geqslant1$. Then, there exists $\varepsilon_2^*=\varepsilon_2^*(\lambda^*,\mu^*,C)>0$ such that, for any $0<\varepsilon<\varepsilon_2^*$, there exist real-valued functions $a_A\in C^L(\R)$ and $c_A\in C^{L+1}(\R)$, with $c_A(0)=0$, for which the following estimates hold:
	\begin{equation}\label{est on other dirct cof}
		\begin{aligned}
			&\nrm{a_A}_{C_t^0}\lesssim_{\lambda^*,\mu^*,C}\varepsilon\leqslant\frac{1}{10},&&
			\nrm{a_A}_{C_t^N}\lesssim_{\lambda^*,\mu^*,C,L}\tA\tlA^N,&&1\leqslant N\leqslant L,\\
			&\nrm{\pt c_A-(\lambda^*+2\mu^*)^{\frac{1}{2}}|f|}_{C_t^0}\lesssim_{\lambda^*,\mu^*,C}\varepsilon|f|,&&
			\nrm{\pt^N c_A}_{C_t^0}\lesssim_{\lambda^*,\mu^*,C,L}\tA\tlA^{N-1}|f|,&&2\leqslant N\leqslant L.
		\end{aligned}
	\end{equation}
	Moreover, the corresponding function $w_{A,f}$ satisfies
	\begin{equation}\label{eq of building block}
		\begin{aligned}
			&\pa_{tt}w_{A,f}-\mu^*\Delta w_{A,f}-(\lambda^*+\mu^*)\nb(\Div w_{A,f})+A(\nb^2w_{A,f})\\
			=&\left(-i\tl\partial_{tt}c_Af+\left(\partial_{tt}a_A-i\tl\partial_{tt}c_Aa_A-2i\tl\pt c_A\pt a_A\right)f^{\perp}\right)e^{i\xi_{A,f}}.
		\end{aligned}
	\end{equation}
	Since all the coefficients of the equation are real-valued, $\overline{w_{A,f}}$ satisfies the complex conjugate of \eqref{eq of building block}. Consequently,
	$
	w_{A,f}+\overline{w_{A,f}}=2\Re w_{A,f}
	$
	is a real-valued approximate solution.
\end{lm}
	\begin{proof}
		For convenience, we denote $f^{(1)} = f$ and $f^{(2)}=f^{\perp}$, and define  coefficients 
		$$\tilde{A}^{pd}_{ij} = A(f^{(i)}, f^{(j)}, f^{(d)}, f^{(p)})= A^{km}_{nr} f^{(i)}_{n} f^{(j)}_{r} f^{(d)}_{m} f^{(p)}_{k} |f|^{-2}.$$
		Next, we define $w_{A,1} := f^{(1)} e^{i\xi_{A,f}}$ and 	$w_{A,2}:=a_{A}f^{(2)}e^{i\xi_{A,f}}
		$, where $\xi_{A,f}=\tl(f\cdot x-c_A(t))$, and $a_A$ and $c_A$ will be chosen later. Noting
		\begin{align*}
		A(\nb^2w_{A,1})&=\sum_{p=1}^2(A^{km}_{nr}\pa_{nr}(w_{A,1})_mf^{(p)}_k)|f|^{-2}f^{(p)}=-\tl^2\sum_{p=1}^2\tilde{A}^{p1}_{11}f^{(p)}e^{i\xi_{A,f}},\\
		A(\nb^2w_{A,2})&=\sum_{p=1}^2(A^{km}_{nr}\pa_{nr}(w_{A,2})_mf^{(p)}_k)|f|^{-2}f^{(p)}=-\tl^2a_A\sum_{p=1}^2\tilde{A}^{p2}_{11}f^{(p)}e^{i\xi_{A,f}},
		\end{align*}
		we could calculate
		\begin{align*}
			&\quad\pa_{tt}w_{A,1}-\mu^*\Delta w_{A,1}-(\lambda^*+\mu^*)\nb(\Div w_{A,1})+A(\nb^2w_{A,1})\\
			&=\left(\tl^2\left(((\lambda^*+2\mu^*)|f|^2-(\pt c_A)^2)f-\sum_{p=1}^2\tilde{A}^{p1}_{11}f^{(p)}\right)-i\tl\partial_{tt}c_Af\right)e^{i\xi_{A,f}},\\
			&\quad\pa_{tt}w_{A,2}-\mu^*\Delta w_{A,2}-(\lambda^*+\mu^*)\nb\Div w_{A,2}+A(\nb^2 w_{A,2})\\
			&=\left(\tl^2a_A\left((\mu^*|f|^2-(\pt c_A)^2)f^{(2)}-\sum_{p=1}^2\tilde{A}^{p2}_{11}f^{(p)}\right)+(\partial_{tt}a_A-i\tl\partial_{tt}c_Aa_A-2i\tl\pt c_A\pt a_A)f^{(2)}\right)e^{i\xi_{A,f}}.
		\end{align*}
		So we  need the following equation
		\begin{equation}
			\left\{
			\begin{aligned}
				&(\lambda^*+2\mu^*)|f|^2-(\pt c_A)^2-\tilde{A}^{11}_{11}-a_{A}\tilde{A}^{12}_{11}=0, \\
				&a_{A}(\mu^*|f|^2-(\pt c_A)^2)-\tilde{A}^{21}_{11}-a_{A}\tilde{A}^{22}_{11}=0.
			\end{aligned}\right.
		\end{equation}
		It can be rewritten as 
		\begin{equation}
			\left\{
			\begin{aligned}
				&(\lambda^*+2\mu^*)|f|^2-\tilde{A}^{11}_{11}-a_{A}\tilde{A}^{12}_{11}=(\pt c_A)^2,\\
				&|f|^{-2}\tilde{A}^{12}_{11}a_{A}^2+\left(|f|^{-2}\tilde{A}^{11}_{11}-|f|^{-2}\tilde{A}^{22}_{11}-(\lambda^*+\mu^*)\right)a_{A}=|f|^{-2}\tilde{A}^{21}_{11},
			\end{aligned}\right.
		\end{equation}
		and we could choose $a_{A}$ and $c_A$ as
		\begin{align}
			a_{A}&=\left\{
			\begin{aligned}
			\frac{2|f|^{-2}\tilde{A}^{21}_{11}}{\left(|f|^{-2}\tilde{A}^{11}_{11}-|f|^{-2}\tilde{A}^{22}_{11}-(\lambda^*+\mu^*)\right)+\left(\left(|f|^{-2}\tilde{A}^{11}_{11}-|f|^{-2}\tilde{A}^{22}_{11}-(\lambda^*+\mu^*)\right)^2+4|f|^{-4}\tilde{A}^{12}_{11}\tilde{A}^{21}_{11}\right)^{\frac{1}{2}}},\quad \lambda^*+\mu^*<0,\\
			\frac{2|f|^{-2}\tilde{A}^{21}_{11}}{\left(|f|^{-2}\tilde{A}^{11}_{11}-|f|^{-2}\tilde{A}^{22}_{11}-(\lambda^*+\mu^*)\right)-\left(\left(|f|^{-2}\tilde{A}^{11}_{11}-|f|^{-2}\tilde{A}^{22}_{11}-(\lambda^*+\mu^*)\right)^2+4|f|^{-4}\tilde{A}^{12}_{11}\tilde{A}^{21}_{11}\right)^{\frac{1}{2}}},\quad \lambda^*+\mu^*>0,
			\end{aligned}\right.\nonumber\\
			c_A(t)&=\int_0^t\left((\lambda^*+2\mu^*)|f|^2-\tilde{A}^{11}_{11}(\tau)-a_A(\tau)\tilde{A}^{12}_{11}(\tau)\right)^{\frac{1}{2}}\rd\tau\nonumber\\
			&=(\lambda^*+2\mu^*)^{\frac{1}{2}}|f|t-\int_0^t\frac{\tilde{A}^{11}_{11}(\tau)+a_A(\tau)\tilde{A}^{12}_{11}(\tau)}{\left((\lambda^*+2\mu^*)|f|^2-\tilde{A}^{11}_{11}(\tau)-a_A(\tau)\tilde{A}^{12}_{11}(\tau)\right)^{\frac{1}{2}}+(\lambda^*+2\mu^*)^{\frac{1}{2}}|f|}\rd\tau.\nonumber
		\end{align}
		Let $w_{A,f}=w_{A,1}+w_{A,2}$. Under the assumptions of the lemma, the explicit formulas for $a_A$ and $c_A$ above, together with the standard composition estimates, yield \eqref{est on other dirct cof}. Combining the identities for $w_{A,1}$ and $w_{A,2},$ we obtain \eqref{eq of building block}, which completes the proof.
	\end{proof}
\begin{Rmk}
	Note that the construction of the building blocks here is based on the assumption that $\lambda^*$ and $\mu^*$ are constants. In general, we take them as the Lam{\'e} constants $\lambda$ and $\mu$. In the case where they are functions of $(t,x)$, we can take the constants as $\lambda(t(I),x(I))$ and $\mu(t(I),x(I))$ for each truncated region, respectively. This is because our building blocks will be used in different truncated regions, and the regions are small enough so that the difference between $\lambda(t,x)$, $\mu(t,x)$ and $\lambda(t(I),x(I))$, $\mu(t(I),x(I))$ is sufficiently small.
\end{Rmk}
\begin{Rmk}
	Note that $A^{km}_{nr}$ can be a time-independent constant. In this case, $\partial_t c_A$ can also be taken as a constant, and we could obtain a simpler version of the lemma.
\end{Rmk}
\subsection{Definition of the perturbation}
Up to now, we are ready to construct the perturbations by using the building blocks defined before. For convenience, we denote $ h^{\upsilon} = h(t, 2\pi\upsilon \mu_{q,i}) $ for any function $ h \in C^0(\mcI^{q,i}_{3\ell_{q,i}} \times \T^2) $. For the term $\tS_{q,i+1}^{(1)}$ mentioned in Section \ref{Some notations}, for its components, we replace $\Gl$, $\Cl$, and $\tG_{q,i+1}$ with $\Gl^{\upsilon}$, $\Cl^{\upsilon}$, and $\nabla w$, respectively, and denote it as $\tS_{q,i+1}^{(1),\upsilon,w}$. 

Notice that, for  $t\in \supp_t\te_{q,i}+3\tau_{q,i}$, we have $c_{q,i}=\sgm^*\te_{q,i}\Id^{<i+1>}+c_{q,i+1}$, and $c_{q,i+1}$ is a constant matrix, which leads to $c_{q,i}\in C^{\infty}((\supp_t\te_{q,i}+3\tau_{q,i})\times\T^2)$. 
Then, for each $I=(s, \upsilon)\in \mathscr{I}$ with $s\tau_{q,i}\in\supp_t\te_{q,i}+\tau_{q,i}$, by using Lemma \ref{construction of building blocks}, we could choose $\lambda^*=\lambda$, $\mu^*=\mu$, and define $(A_I)^{km}_{nr}$ as
\begin{equation}\label{def of A_I}
\begin{aligned}
	(A_I)^{km}_{nr}=&-(\Sgm_{\Gl}+\Rl+c_{q,i})_{nr}\delta_{km}+\lambda\delta_{kr}\delta_{mn}+\mu\delta_{km}\delta_{nr}+\mu\delta_{kn}\delta_{mr}\\
	&-4(\Id+G_{\ell,i})_{kr} \left((\sgm_{11}+(\sgm_{12}+\sgm_{111})\tj_1^{(0)})(\Id+\Gl)_{mn}-\sgm_{12}(\Cl(\Id+\Gl))_{mn}\right)\\
	&-4((\Id+G_{\ell,i})(\tr \Cl\Id-\Dl))_{kr} \left((\sgm_{12}+\sgm_{22}\tj_1^{(0)})(\Id+\Gl)_{mn}-\sgm_{22}(\Cl(\Id+\Gl))_{mn}\right)\\
	&-(\tL_2^{(0)})\left(2(\Id+\Gl)_{mn}(\Id+\Gl)_{kr}-((\Id+\Gl)(\Id+\Gl^{\top}))_{km}\delta_{nr}-(\Id+\Gl)_{kn}(\Id+\Gl)_{mr}\right).
\end{aligned}
\end{equation}
Next, we define $(A_I^{\upsilon})^{km}_{nr}$ which is obtained by replacing all $\Gl$ and $\Cl$ in $(A_I)^{km}_{nr}$ with $\Gl^{\upsilon}$ and $\Cl^{\upsilon}$, respectively. Then, we have
\begin{align*}
	&(A_I^{\upsilon})^{km}_{nr}\partial_{n}w_m\\
	=&-(\nabla w(\Sgm_{\Gl}+\Rl+c_{q,i})^{\upsilon})_{kr}+(\lambda\tr(\nabla w)\Id+\mu\nabla w+\mu(\nabla w)^{\top})_{kr}\\
	&-  4\left(\left((\sgm_{11}+(\sgm_{12}+\sgm_{111})\tj_1^{(0)})^{\upsilon}\tr((\Id+\Gl)^{\upsilon}(\nabla w)^{\top})-\sgm_{12}\tr((\Cl(\Id+\Gl))^{\upsilon}(\nabla w)^{\top})\right)(\Id+G_{\ell,i})^{\upsilon}\right)_{kr}\\
	&-4   \left(\left((\sgm_{12}+\sgm_{22}\tj_1^{(0)})^{\upsilon}\tr((\Id+\Gl)^{\upsilon}(\nabla w)^{\top})-\sgm_{22}\tr((\Cl(\Id+\Gl))^{\upsilon}(\nabla w)^{\top})\right)((\Id+G_{\ell,i})(\tr \Cl\Id-\Dl))^{\upsilon}\right)_{kr}\\
	&-(\tL_2^{(0)})^{\upsilon}\left(2\tr((\Id+\Gl)^{\upsilon}(\nabla w)^{\top})(\Id+\Gl)^{\upsilon}\Id-((\Id+\Gl)(\Id+\Gl^{\top}))^{\upsilon}\nabla w-(\Id+\Gl)^{\upsilon}(\nabla w)^{\top}(\Id+\Gl)^{\upsilon}\right)_{kr}\\
	=&\left(-\nabla w(\Sgm_{\Gl}+\Rl+c_{q,i})^{\upsilon}-(\Id+\Gl)^{\upsilon}\tS_{q,i+1}^{(1),\upsilon,w}+\lambda\tr(\nabla w)\Id+\mu\nabla w+\mu(\nabla w)^{\top}\right)_{kr},
\end{align*}
and  obtain
\begin{align*}
	(A_I^{\upsilon}(\nb^2 w))_k=&\partial_r((A_I^{\upsilon})^{km}_{nr}\pa_{n}w_m)
	=\left(-\Div\left(\nabla w(\Sgm_{\Gl}+\Rl+c_{q,i})^{\upsilon}+(\Id+\Gl)^{\upsilon}\tS_{q,i+1}^{(1),\upsilon,w}-(\lambda\tr(\nabla w)\Id+\mu\nabla w+\mu(\nabla w)^{\top})\right)\right)_k.
\end{align*}
Let $\xi_{A_I^{\upsilon},f_{i+1}}=[I]\left(f_{i+1}\cdot x-c_{A_I^{\upsilon}}\right)$ and 
\begin{align*}
	w_{A_I^{\upsilon},f_{i+1}}&=(f_{i+1}+a_{A_I^{\upsilon}}f_{i+1}^{\perp})e^{i\lambda_{q,i+1}\xi_{A_I^{\upsilon},f_{i+1}}}=(i\lambda_{q,i+1}[I]|f_{i+1}|^2)^{-1}\Div((f_{i+1}+a_{A_I^{\upsilon}}f_{i+1}^{\perp})\otimes f_{i+1}e^{i\lambda_{q,i+1}\xi_{A_I^{\upsilon},f_{i+1}}}),
\end{align*}
where $c_{A_I^{\upsilon}}$ and $a_{A_I^{\upsilon}}$ can be chosen by Lemma \ref{construction of building blocks}. Then, we could calculate
\begin{align}
	&\pa_{tt}w_{A_I^{\upsilon},f_{i+1}}-\mu\Delta w_{A_I^{\upsilon},f_{i+1}}-(\lambda+\mu)\nb(\Div w_{A_I^{\upsilon},f_{i+1}})+A_I^{\upsilon}(\nb^2 w_{A_I^{\upsilon},f_{i+1}})\nonumber\\
	=&\pa_{tt}w_{A_I^{\upsilon},f_{i+1}}-\mu\Delta w_{A_I^{\upsilon},f_{i+1}}-(\lambda+\mu)\nb(\Div w_{A_I^{\upsilon},f_{i+1}})\nonumber\\
	&-\Div\left(\nabla w_{A_I^{\upsilon},f_{i+1}}(\Sgm_{\Gl}+\Rl+c_{q,i})^{\upsilon}+(\Id+\Gl)^{\upsilon}\tS_{q,i+1}^{(1),\upsilon,w_{A_I^{\upsilon},f_{i+1}}}-(\lambda+\mu)\tr(\nabla w_{A_I^{\upsilon},f_{i+1}})\Id-\mu\nabla w_{A_I^{\upsilon},f_{i+1}}\right)\nonumber\\
	=&\pa_{tt}w_{A_I^{\upsilon},f_{i+1}}-\Div\left(\nabla w_{A_I^{\upsilon},f_{i+1}}(\Sgm_{\Gl}+\Rl+c_{q,i})^{\upsilon}+(\Id+\Gl)^{\upsilon}\tS_{q,i+1}^{(1),\upsilon,w_{A_I^{\upsilon},f_{i+1}}}\right).\nonumber
\end{align} 

By \eqref{pp of te_q,i}, \eqref{pp of supp l,i}, and
\eqref{est on R_li low}, we have
$
\te_{q,i}(t)=\delta_{q+1}
\
\text{on }\supp_tR_{\ell,i},
$
and
$
\nrm{
	(\sgm^*\delta_{q+1})^{-1}R_{\ell,i}
}_0
<
\tilde{c}_0.
$
Consequently, the quotient
$(\sgm^*\te_{q,i})^{-1}R_{\ell,i}$ admits a smooth extension
to $\mcI^{q,i}\times\T^2$, defined by
$$
(\sgm^*\te_{q,i})^{-1}R_{\ell,i}
:=
(\sgm^*\delta_{q+1})^{-1}R_{\ell,i}.
$$
With this convention, we define the weight coefficient
$d_{q,i+1}$ by
\begin{align}
	d_{q,i+1}(t,x)
	:=
	\te_{q,i}^{\frac{1}{2}}(t)
	\Gamma_{f_{i+1}}
	\left(
	\Id+
	(\sgm^*\te_{q,i})^{-1}R_{\ell,i}
	\right).
	\label{def of d_q,i+1}
\end{align}

For convenience, for each $I=(s,\upsilon)\in\mathscr{I}$, we set $\xi_I:=\xi_{A_I^{\upsilon},f_{i+1}}$ and use the following notations:
\begin{align}
	\Pi_I&=e^{i\lambda_{q,i+1}\xi_I}-e^{-i\lambda_{q,i+1}\xi_I},
	\qquad
	\tPi_I=e^{i\lambda_{q,i+1}\xi_I}+e^{-i\lambda_{q,i+1}\xi_I},
	\label{def of Pi_I}\\
	\Pi_{I,J}&=\Pi_I\Pi_J
	=e^{i\lambda_{q,i+1}(\xi_I+\xi_J)}
	+e^{-i\lambda_{q,i+1}(\xi_I+\xi_J)}
	-e^{i\lambda_{q,i+1}(\xi_I-\xi_J)}
	-e^{-i\lambda_{q,i+1}(\xi_I-\xi_J)},
	\label{def of Pi_IJ}\\
	\Pi_{I,J,K}&=\Pi_I\Pi_J\Pi_K
	=(e^{i\lambda_{q,i+1}\xi_I}-e^{-i\lambda_{q,i+1}\xi_I})
	(e^{i\lambda_{q,i+1}\xi_J}-e^{-i\lambda_{q,i+1}\xi_J})
	(e^{i\lambda_{q,i+1}\xi_K}-e^{-i\lambda_{q,i+1}\xi_K})
	\nonumber\\
	&=e^{i\lambda_{q,i+1}(\xi_I+\xi_J+\xi_K)}
	+e^{i\lambda_{q,i+1}(\xi_I-\xi_J-\xi_K)}
	+e^{i\lambda_{q,i+1}(-\xi_I+\xi_J-\xi_K)}
	+e^{i\lambda_{q,i+1}(-\xi_I-\xi_J+\xi_K)}
	\nonumber\\
	&\quad-e^{i\lambda_{q,i+1}(\xi_I+\xi_J-\xi_K)}
	-e^{i\lambda_{q,i+1}(\xi_I-\xi_J+\xi_K)}
	-e^{i\lambda_{q,i+1}(-\xi_I+\xi_J+\xi_K)}
	-e^{i\lambda_{q,i+1}(-\xi_I-\xi_J-\xi_K)}.
	\label{def of Pi_IJK}
\end{align}
Here $I,J,K\in\mathscr{I}$ are arbitrary. And we could divide the terms composed of products of $\Pi_{I}$ and $\Pi_{I'}$ into low-frequency parts and high-frequency parts, for example $\Pi_{I}\Pi_{I'}$ can be written as
\begin{align*}
\Pi_{I}\Pi_{I'}=\underbrace{\frac{1}{(2\pi)^2}\int_{\T^2}\Pi_{I}\Pi_{I'}\rd x}_{\mPO\Pi_{I}\Pi_{I'}}+\underbrace{\Pi_{I}\Pi_{I'}-\frac{1}{(2\pi)^2}\int_{\T^2}\Pi_{I}\Pi_{I'}\rd x}_{\mPG\Pi_{I}\Pi_{I'}}.
\end{align*}
Then, we could first define  $\tV_{q,i+1,p}$ as
\begin{align}
	\tV_{q,i+1,p}= \sum_I\tV_{q,i+1,p}^I&=\sum_I\frac{\theta_I\chi_Id_{q,i+1}\tilde{f}_{A_I^{\upsilon}}\otimes f_{i+1}}{i\sqrt{2}\lambda_{q,i+1}^2[I]^2|f_{i+1}|^4}\Pi_{I}=\sum_{I}\frac{V_{q,i+1,p}^{I,cof}}{i\lambda_{q,i+1}^2[I]^2|f_{i+1}|^2}\Pi_{I}, \label{def of tV_q k i+1,p}
\end{align}
where $\sum_I=\sum_{s}\sum_{\upsilon}$ and
\begin{equation}\label{def of cof of tV_p}
\begin{aligned}
	&\tilde{f}_{A_I^{\upsilon}}=f_{i+1}+a_{A_I^{\upsilon}}(t)f_{i+1}^{\perp},  \quad\gamma_{q,i+1}^{I}=\frac{\theta_I\chi_Id_{q,i+1}}{\sqrt{2}|f_{i+1}|^2},\quad V_{q,i+1,p}^{I,cof}=\gamma_{q,i+1}^{I}\tilde{f}_{A_I^{\upsilon}}\otimes f_{i+1}.
\end{aligned}
\end{equation}

Next, we could define the principal part of the perturbation $\tu_{q,i+1}$ as 
\begin{align}
	\tu_{q,i+1,p}=&\Div\tV_{q,i+1,p}= \sum_I\tu_{q,i+1,p}^I\nonumber\\
	=&\sum_I\frac{\theta_I\chi_Id_{q,i+1}\tilde{f}_{A_I^{\upsilon}}}{\sqrt{2}\lambda_{q,i+1}[I]|f_{i+1}|^2}\tPi_{I}+\sum_I\frac{(f_{i+1}\cdot\nabla)(\theta_I\chi_Id_{q,i+1})\tilde{f}_{A_I^{\upsilon}} }{i\sqrt{2}\lambda_{q,i+1}^2[I]^2|f_{i+1}|^4}\Pi_{I}\nonumber\\
	=&\sum_I\frac{1}{\lambda_{q,i+1}[I]}\left(\gamma_{q,i+1}^I\tPi_{I}+\gamma_{q,i+1,c}^I\Pi_{I}\right)\tilde{f}_{A_I^{\upsilon}}\label{def1 of tu_q i+1,p}\\
	=&\sum_{I}\frac{1}{\lambda_{q,i+1}[I]}(u_{q,i+1,p}^{I,cof}e^{i\lambda_{q,i+1}\xi_{A_I^{\upsilon},f_{i+1}}}+\ou_{q,i+1,p}^{I,cof}e^{-i\lambda_{q,i+1}\xi_{A_I^{\upsilon},f_{i+1}}})\tilde{f}_{A_I^{\upsilon}}, \label{def2 of tu_q i+1,p}
\end{align}
where
\begin{equation}
\label{def of u_q,i+1,I}
\begin{aligned}
	\gamma_{q,i+1,c}^I=\frac{(f_{i+1}\cdot\nabla)\gamma_{q,i+1}^{I}}{i\lambda_{q,i+1}[I]|f_{i+1}|^2},\quad u_{q,i+1,p}^{I,cof}=\gamma_{q,i+1}^{I}+\frac{(f_{i+1}\cdot\nabla)\gamma_{q,i+1}^{I}}{i\lambda_{q,i+1}[I]|f_{i+1}|^2},\quad
	\ou_{q,i+1,p}^{I,cof}=\gamma_{q,i+1}^{I}-\frac{(f_{i+1}\cdot\nabla)\gamma_{q,i+1}^{I}}{i\lambda_{q,i+1}[I]|f_{i+1}|^2}.
\end{aligned}
\end{equation}
Moreover, for any $I,J\in\mathscr I$ with $\nrm{I-J}>1$, we obtain
\begin{align}
	&\supp V_{q,i+1,p}^{I,cof}\bigcap\supp V_{q,i+1,p}^{J,cof}=\emptyset,\nonumber\\
	&\supp\gamma_{q,i+1}^{I}\bigcap\supp\gamma_{q,i+1}^{J}
	=\supp\gamma_{q,i+1,c}^{I}\bigcap\supp\gamma_{q,i+1,c}^{J}
	=\supp\gamma_{q,i+1}^{I}\bigcap\supp\gamma_{q,i+1,c}^{J}
	=\emptyset,\nonumber\\
	&\supp u_{q,i+1,p}^{I,cof}\bigcap\supp u_{q,i+1,p}^{J,cof}
	=\supp\ou_{q,i+1,p}^{I,cof}\bigcap\supp\ou_{q,i+1,p}^{J,cof}
	=\supp u_{q,i+1,p}^{I,cof}\bigcap\supp\ou_{q,i+1,p}^{J,cof}
	=\emptyset.
	\label{pp of supp cof}
\end{align}
Here, for $I=(s(I),\upsilon(I))$ and $J=(s(J),\upsilon(J))$, we define
$$
\nrm{I-J}:=\max\left\{|s(I)-s(J)|,\nrm{\upsilon(I)-\upsilon(J)}_{\ell^\infty}\right\},
\qquad
\nrm{\upsilon(I)-\upsilon(J)}_{\ell^\infty}
:=\max_{i=1,2}\left|\upsilon_i(I)-\upsilon_i(J)\right|.
$$
The main part of the perturbation can be also written as
\begin{align}
	\tu_{q,i+1,p}
	&=\sum_{I}\frac{1}{\lambda_{q,i+1}[I]}(u_{q,i+1,p}^{I,cof}w_{A_I^{\upsilon},f_{i+1}}+\ou_{q,i+1,p}^{I,cof}\ow_{A_I^{\upsilon},f_{i+1}}).\label{def 2 of tu_i+1p}
\end{align}

Then, we could calculate
\begin{align*}
	&\tG_{q,i+1,p}=\sum_I\tG_{q,i+1,p}^{I}=\sum_I\nb\tu_{q,i+1,p}^{I}\\
	=&\underbrace{i\sum_{I}\gamma_{q,i+1}^{I}f_{i+1}\otimes f_{i+1}\Pi_{I}}_{\tG_{q,i+1,m}}+\underbrace{\overbrace{i\sum_{I}\gamma_{q,i+1}^{I}a_{A_I^{\upsilon}}(t)f_{i+1}^{\perp}\otimes f_{i+1}\Pi_{I}}^{\tG_{q,i+1,s1}}+\overbrace{\sum_{I}\frac{1}{\lambda_{q,i+1}[I]}(\tilde{f}_{A_I^{\upsilon}}\otimes\nb \gamma_{q,i+1,c}^I\Pi_{I}+\tilde{f}_{A_I^{\upsilon}}\otimes\gamma_{q,i+1,s}^I\tPi_{I}}^{\tG_{q,i+1,s2}})}_{\tG_{q,i+1,s}},\\
	&\pt \tu_{q,i+1,p}=\sum_I\pt \tu_{q,i+1,p}^{I}\\
	=&\underbrace{-i\sum_{I}(\lambda+2\mu)^{\frac{1}{2}}|f_{i+1}|\gamma_{q,i+1}^{I}f_{i+1}\Pi_{I}}_{\pttu_{q,i+1,m}}\\
	&+\underbrace{\overbrace{i\sum_{I}\left(((\lambda+2\mu)^{\frac{1}{2}}|f_{i+1}|-\pt c_{A_I^\upsilon})f_{i+1}-a_{A_I^{\upsilon}}(t)\pt c_{A_I^\upsilon}f_{i+1}^{\perp}\right)\gamma_{q,i+1}^{I}\Pi_{I}}^{\pttu_{q,i+1,s1}}+\overbrace{\sum_{I}\frac{1}{\lambda_{q,i+1}[I]}(\pt (\gamma_{q,i+1,c}^I\tilde{f}_{A_I^{\upsilon}}\Pi_{I})+\pt( \gamma_{q,i+1}^{I}\tilde{f}_{A_I^{\upsilon}})\tPi_{I})}^{\pttu_{q,i+1,s2}}}_{\pttu_{q,i+1,s}}.
\end{align*}
where $\gamma_{q,i+1,s}^{I}:=\nb \gamma_{q,i+1}^{I}+i\lambda_{q,i+1}[I]\gamma_{q,i+1,c}^{I} f_{i+1}$.

For convenience, we use the following notations:
\begin{align}
	\tu_{q,i+1,k}^{(\upsilon)}&=\sum_{s:I=(s,\upsilon)}\tu_{q,i+1,k}^I,\ \tV_{q,i+1,k}^{(\upsilon)}=\sum_{s:I=(s,\upsilon)}\tV_{q,i+1,k}^I,\ \tG_{q,i+1,k}^{(\upsilon)}=\sum_{s:I=(s,\upsilon)}\tG_{q,i+1,k}^I,\ \tS_{q,i+1,k}^{(1),(\upsilon)}=\sum_{s:I=(s,\upsilon)}\tS_{q,i+1,k}^{(1),I},\label{def of tu_i upsilon}
\end{align}
where the subscript $k=p,m,s,s1,s2$, and $\tS_{q,i+1,p}^{(1),I}$, $\tS_{q,i+1,m}^{(1),I}$, $\tS_{q,i+1,s}^{(1),I}$, $\tS_{q,i+1,s1}^{(1),I}$, and $\tS_{q,i+1,s2}^{(1),I}$  are obtained by replacing $\tG_{q,i+1}$ in $\tS_{q,i+1}^{(1)}$ mentioned in Section \ref{Some notations}, with $\tG_{q,i+1,p}^I$, $\tG_{q,i+1,m}^I$, $\tG_{q,i+1,s}^I$, $\tG_{q,i+1,s1}^I$,  and $\tG_{q,i+1,s2}^I$. Here, we give the specific form of them.
 \begin{enumerate}
	\item  The specific form of  $\tS_{q,i+1,m}^{(1),(\upsilon)}=\tS_{q,i+1,m1}^{(1),(\upsilon)}+\tS_{q,i+1,m2}^{(1),(\upsilon)}$:
	\begin{equation}\label{def of tS_1m}
		\begin{aligned}
			\tS_{q,i+1,m}^{(1),(\upsilon)}=&\underbrace{\sum_{s:I=(s,\upsilon)}i\gamma_{q,i+1}^{I}\tS_{q,i+1,m}^{(1),cofm}\Pi_{I}}_{\tS_{q,i+1,m1}^{(1),(\upsilon)}}+\underbrace{\sum_{s:I=(s,\upsilon)}i\gamma_{q,i+1}^{I}\tS_{q,i+1,m}^{(1),cofc}\Pi_{I}}_{\tS_{q,i+1,m2}^{(1),(\upsilon)}},
		\end{aligned}
	\end{equation}
	where 
	\begin{align}
	\tS_{q,i+1,m}^{(1),cofm}
	=&4\left((\sgm_{11}+\sgm_{2})|f_{i+1}|^2\Id-\sgm_2(f_{i+1}\otimes f_{i+1})\right),\label{def of tS_1m cofm}\\
	\tS_{q,i+1,m}^{(1),cofc}
	=&4((2\sgm_{12}+\sgm_{111})\tj_1^{(0)}+\sgm_{22}\tj_2^{(0)})|f_{i+1}|^2\Id\nonumber\\
	&+4\left(f_{i+1}^{\top}\left((\sgm_{11}+\sgm_2+(2\sgm_{12}+\sgm_{111})\tj_1^{(0)}+\sgm_{22}\tj_2^{(0)})\Gl-\sgm_{12}\Cl(\Id+\Gl)\right)f_{i+1}\right)\Id\nonumber\\
	&+4\left(f_{i+1}^{\top}\left((\sgm_{12}+\sgm_{22}\tj_1^{(0)})(\Id+\Gl)-\sgm_{22}\Cl(\Id+\Gl)\right)f_{i+1}\right)(\tr \Cl\Id-\Dl)\label{def of tS_1m cofc}\\
	&-4(\sgm_{12}\tj_1^{(0)}+\sgm_{22}\tj_2^{(0)})(f_{i+1}\otimes f_{i+1})-2(\sgm_2+\sgm_{12}\tj_1^{(0)}+\sgm_{22}\tj_2^{(0)})\left((\Gl^{\top}f_{i+1})\otimes f_{i+1}+f_{i+1}\otimes (\Gl^{\top}f_{i+1})\right)\nonumber.
	\end{align}
	Next, we denote  $\tS_{q,i+1,m1}^{(1)}=\sum_{\upsilon}\tS_{q,i+1,m1}^{(1),(\upsilon)}$ and $\tS_{q,i+1,m2}^{(1)}=\sum_{\upsilon}\tS_{q,i+1,m2}^{(1),(\upsilon)}$.
	\item  The specific form of  $\tS_{q,i+1,s1}^{(1),(\upsilon)}$:
	\begin{equation}\label{def of tS_1s1}
		\begin{aligned}
			\tS_{q,i+1,s1}^{(1),(\upsilon)}=&\sum_{s:I=(s,\upsilon)}i\gamma_{q,i+1}^{I}a_{A_I^{\upsilon}}(t)\tS_{q,i+1,s1}^{(1),cof}\Pi_{I},
		\end{aligned}
	\end{equation}
	where
	\begin{equation} \label{def of tS_1s1 cof}
		\begin{aligned}
			\tS_{q,i+1,s1}^{(1),cof}
			=&4\left((f_{i+1}^{\perp})^{\top}\left((\sgm_{11}+\sgm_2+(2\sgm_{12}+\sgm_{111})\tj_1^{(0)}+\sgm_{22}\tj_2^{(0)})\Gl-\sgm_{12}\Cl(\Id+\Gl) \right)f_{i+1}\right)\Id\\
			&+4\left((f_{i+1}^{\perp})^{\top}\left((\sgm_{12}+\sgm_{22}\tj_1^{(0)})\Gl-\sgm_{22}\Cl(\Id+\Gl)\right) f_{i+1}\right)(\tr \Cl\Id-\Dl)\\
			&-2(\sgm_2+\sgm_{12}\tj_1^{(0)}+\sgm_{22}\tj_2^{(0)})\left(((\Id+\Gl^{\top})f_{i+1}^{\perp})\otimes f_{i+1}+f_{i+1}\otimes ((\Id+\Gl^{\top})f_{i+1}^{\perp})\right).
		\end{aligned}
	\end{equation}
	
	\item  The specific form of  $\tS_{q,i+1,s2}^{(1),(\upsilon)}$:
	\begin{equation}\label{def of tS_1s2}
		\begin{aligned}
			\tS_{q,i+1,s2}^{(1),(\upsilon)}=&\sum_{s:I=(s,\upsilon)}\frac{1}{\lambda_{q,i+1}[I]}\left(\tS_{q,i+1,s2}^{(1),I,cof1}\Pi_{I}+\tS_{q,i+1,s2}^{(1),I,cof2}\tPi_{I}\right),
		\end{aligned}
	\end{equation}
	where 
	\begin{equation} \label{def of tS_1s2 cof}
		\begin{aligned}
			\tS_{q,i+1,s2}^{(1),I,cof1}
			=&4\left(\tilde{f}_{A_I^{\upsilon}}^{\top}\left((\sgm_{11}+\sgm_{2}+(2\sgm_{12}+\sgm_{111})\tj_1^{(0)}+\sgm_{22}\tj_2^{(0)})(\Id+\Gl)-\sgm_{12}\Cl(\Id+\Gl)\right)\nb \gamma_{q,i+1,c}^{I}\right)\Id\\
			&+4\left(\tilde{f}_{A_I^{\upsilon}}^{\top}\left((\sgm_{12}+\sgm_{22}\tj_1^{(0)})(\Id+\Gl)-\sgm_{22}\Cl(\Id+\Gl)\right)\nb \gamma_{q,i+1,c}^{I}\right)(\tr \Cl\Id-\Dl)\\
			&-2(\sgm_2+\sgm_{12}\tj_1^{(0)}+\sgm_{22}\tj_2^{(0)})\left(((\Id+\Gl^{\top})\tilde{f}_{A_I^{\upsilon}})\otimes \nb \gamma_{q,i+1,c}^{I}+\nb \gamma_{q,i+1,c}^{I}\otimes ((\Id+\Gl^{\top})\tilde{f}_{A_I^{\upsilon}})\right),\\
			\tS_{q,i+1,s2}^{(1),I,cof2}
			=&4\left(\tilde{f}_{A_I^{\upsilon}}^{\top}\left((\sgm_{11}+\sgm_{2}+(2\sgm_{12}+\sgm_{111})\tj_1^{(0)}+\sgm_{22}\tj_2^{(0)})(\Id+\Gl)-\sgm_{12}\Cl(\Id+\Gl)\right)\gamma_{q,i+1,s}^{I}\right)\Id\\
			&+4\left(\tilde{f}_{A_I^{\upsilon}}^{\top}\left((\sgm_{12}+\sgm_{22}\tj_1^{(0)})(\Id+\Gl)-\sgm_{22}\Cl(\Id+\Gl)\right)\gamma_{q,i+1,s}^{I}\right)(\tr \Cl\Id-\Dl)\\
			&-2(\sgm_2+\sgm_{12}\tj_1^{(0)}+\sgm_{22}\tj_2^{(0)})\left(((\Id+\Gl^{\top})\tilde{f}_{A_I^{\upsilon}})\otimes \gamma_{q,i+1,s}^{I}+ \gamma_{q,i+1,s}^{I}\otimes ((\Id+\Gl^{\top})\tilde{f}_{A_I^{\upsilon}})\right).
		\end{aligned}
	\end{equation}

\end{enumerate}

In order to ensure that the divergence of the new Reynolds errors, defined later in Section~\ref{Definition of the new Reynolds error}, satisfies \eqref{eq of linear and angular momenta in L co} with $y=x+\ul+\tu_{q,i+1}$, we need two additional terms $\tu_{q,i+1,L}$ and $\tu_{q,i+1,M}$.  Moreover, we provide two parameters to determine the range where angular momentum correction is required.
\begin{align}
	S_{q,i}^{min}:=\lfloor(\min\supp_{t}\te_{q,i})/\tau_{q,i}\rfloor, \quad S_{q,i}^{max}:=\lceil T/\tau_{q,i} \rceil + 1.
\end{align}
Then, we could define a cutoff function $\theta_{q,i}^*(t)\in C^\infty(\mcI^{q,i})$ satisfying
	\begin{equation}\label{pp of theta^*}
	\theta_{q,i}^*=\left\lbrace
	\begin{aligned}
		&1,&&t\in (S_{q,i}^{min}\tau_{q,i},T+\tau_{q,i}],\\
		&\theta^*(\tau_{q,i}^{-1}(t-(S_{q,i}^{min}\tau_{q,i}-\tau_{q,i}))),&&t\in [S_{q,i}^{min}\tau_{q,i}-\tau_{q,i},S_{q,i}^{min}\tau_{q,i}],\\
		&0,&&t\in[-\tau_{q,i},S_{q,i}^{min}\tau_{q,i}-\tau_{q,i}),
	\end{aligned}\right.
\end{equation}
where $\theta^*$ is defined as \eqref{def of theta^*}, and
\begin{equation}
	\begin{aligned}
	&\nrm{\theta_{q,i}^*}_{C_t^r(\mcI^{q,i})}\lesssim_r\tau_{q,i}^{-r}.
	\end{aligned}
\end{equation} 
\begin{Rmk}
	In the proof of Proposition \ref{Proposition 1}, we usually have
	\begin{align}\label{time case 1}
		&0 < \min\supp_{t}\te_{q,i} < T + \tau_{q,i}, \quad 
		0 < S_{q,i}^{\text{min}} <  S_{q,i}^{max}.
	\end{align}
\end{Rmk}

We need another  time-cutoff function: $\Theta_{0}\in C_c^\infty(\R)$ satisfies $\Theta_{0}=1$ on $[0,1]$ and $\Theta_0=0$ on $(-1/4,5/4)^c$. Then, we could give the following cutoff function
\begin{equation}\label{def of Theta_q,i}
\Theta_{q,i}^{(S)}(t)=
\Theta_0(\tau_{q,i}^{-1}t-S),\quad S_{q,i}^{min}\leqslant S\leqslant S_{q,i}^{max},
\end{equation}
and $\Theta_S(t):=\Theta_{q,i}^{(S)}(\tau_{q,i}t)$.

In order to define $\tu_{q,i+1,L}$ and $\tu_{q,i+1,M}$, we first introduce compactly supported smooth functions $\psi,\Psi\in C_c^\infty(\R^2)$ satisfying
\begin{align}\label{pp of psi}
	\psi=\Delta\Psi,\quad\int_{\R^2}\psi(x)\ \rd x=0, \quad\int_{\R^2}\psi^2(x)\ \rd x=1, \quad\text{and}\quad \supp\psi\cup\supp\Psi\subseteq B\left(0,\frac1{16}\right).
\end{align}
In the following, for any function $f\in C_c^\infty(\R^2)$, we set 
\begin{align*}
f_{L}^{(\upsilon)}=f\left(\frac{x}{\mu_{q,i}}-\frac{\pi}{2}e_1-2\pi\upsilon\right),\quad f_{M}^{(\upsilon)}=f\left(\frac{x}{\mu_{q,i}}-\frac{\pi}{2}e_2-2\pi\upsilon\right),
\end{align*}
where $e_i,i=1,2$ denotes the standard unit vector. 

\begin{Rmk}
	For notational simplicity, if $D^\alpha$ is a spatial differential operator, we use $D^\alpha f_k^{(\upsilon)}$ to denote the function obtained by first applying $D^\alpha$ to $f$ and then performing the above translation and rescaling; namely,
	$
	D^\alpha f_k^{(\upsilon)}:=(D^\alpha f)_k^{(\upsilon)},\ k=L,M.
	$
	Thus, $D^\alpha f_k^{(\upsilon)}$ should not be confused with $D_x^\alpha(f_k^{(\upsilon)})$. More precisely,
	$
	D_x^\alpha(f_k^{(\upsilon)})=\mu_{q,i}^{-|\alpha|}D^\alpha f_k^{(\upsilon)},\ k=L,M.
	$
	In particular,
	$
	\nabla_x f_k^{(\upsilon)}=\mu_{q,i}^{-1}\nabla f_k^{(\upsilon)}
	$
	and
	$
	\Delta_x f_k^{(\upsilon)}=\mu_{q,i}^{-2}\Delta f_k^{(\upsilon)}.
	$
\end{Rmk}

And then, we define $\tV_{q,i+1,L}$ and $\tV_{q,i+1,M}$ as
\begin{align}
	\tV_{q,i+1,k} &=\sum_{\upsilon}\tV_{q,i+1,k}^{(\upsilon)}=\sum_{\upsilon} g_k^{(\upsilon)}(t)\otimes\nabla\Psi_{k}^{(\upsilon)},\quad k=L,M,\label{def of tV_q k}
\end{align}
where
$
	g_L^{(\upsilon)}(t)=\theta_{q,i}^*(t)(\varepsilon\mu_{q,i}^2\delta_{q+1},g_{L,2}^{(\upsilon)}(t))^{\top},\
	g_M^{(\upsilon)}(t)=\theta_{q,i}^*(t)(\varepsilon\mu_{q,i}^2\delta_{q+1},g_{M,2}^{(\upsilon)}(t))^{\top}
$
are smooth functions of time $t$, and
\begin{align}
	g_{L,2}^{(\upsilon)}(t)=\sum_{S=S_{q,i}^{min}}^{S_{q,i}^{max}}\Theta_{q,i}^{(S)}(t)g_{L,2}^{(S,\upsilon)}(\tau_{q,i}^{-1}t),\quad g_{M,2}^{(\upsilon)}(t)=\sum_{S=S_{q,i}^{min}}^{S_{q,i}^{max}}\Theta_{q,i}^{(S)}(t)g_{M,2}^{(S,\upsilon)}(\tau_{q,i}^{-1}t),\label{def of g_LM2}
\end{align}
which will be defined later. Then, $\tu_{q,i+1,L}$, $\tu_{q,i+1,M}$, $\tG_{q,i+1,L}$, and $\tG_{q,i+1,M}$ can be defined as
\begin{align}
	\tu_{q,i+1,k}&=\Div \tV_{q,i+1,k}=\sum_{\upsilon}\tu_{q,i+1,k}^{(\upsilon)}=\sum_{\upsilon}\mu_{q,i}^{-1}g_k^{(\upsilon)}(t)\psi_{k}^{(\upsilon)},&& k=L,M,\label{def of tu_q k}\\
	\tG_{q,i+1,k}&=\nb \tu_{q,i+1,k}=\sum_{\upsilon}\tG_{q,i+1,k}^{(\upsilon)}=\sum_{\upsilon}\mu_{q,i}^{-2}g_k^{(\upsilon)}(t)\otimes\nb\psi_{k}^{(\upsilon)},&& k=L,M.\label{def of tG_q k}
\end{align}

For convenience, we define the correction term as
\begin{align}
&\tu_{q,i+1,ac}^{(\upsilon)}=\tu_{q,i+1,L}^{(\upsilon)}+\tu_{q,i+1,M}^{(\upsilon)},&&\tV_{q,i+1,ac}^{(\upsilon)}=\tV_{q,i+1,L}^{(\upsilon)}+\tV_{q,i+1,M}^{(\upsilon)},&&\tG_{q,i+1,ac}^{(\upsilon)}=\tG_{q,i+1,L}^{(\upsilon)}+\tG_{q,i+1,M}^{(\upsilon)},\label{def of tuVG_ac up}\\
&\tu_{q,i+1,ac}=\tu_{q,i+1,L}+\tu_{q,i+1,M},
&&\tV_{q,i+1,ac}=\tV_{q,i+1,L}+\tV_{q,i+1,M},
&&\tG_{q,i+1,ac}=\tG_{q,i+1,L}+\tG_{q,i+1,M}.\label{def of tuVG_ac}
\end{align}
Similarly, we could define $\tS_{q,i+1,L}^{(1),(\upsilon)}$, $\tS_{q,i+1,M}^{(1),(\upsilon)}$, and  $\tS_{q,i+1,ac}^{(1),(\upsilon)}$ which correspond to the term $\tS_{q,i+1}^{(1)}$ mentioned in Section \ref{Some notations}. For its components, we replace $\tG_{q,i+1}$ with $\tG_{q,i+1,L}^{(\upsilon)}$, $\tG_{q,i+1,M}^{(\upsilon)}$, and $\tG_{q,i+1,ac}^{(\upsilon)}$ respectively. Then, we have
\begin{equation}\label{def of tS_1ac}
	\begin{aligned}
	\tS_{q,i+1,ac}^{(1),(\upsilon)}=&\tS_{q,i+1,L}^{(1),(\upsilon)}+\tS_{q,i+1,M}^{(1),(\upsilon)},\quad\tS_{q,i+1,ac}^{(1)}=\tS_{q,i+1,L}^{(1)}+\tS_{q,i+1,M}^{(1)}=\sum_{\upsilon}\left(\tS_{q,i+1,L}^{(1),(\upsilon)}+\tS_{q,i+1,M}^{(1),(\upsilon)}\right),
	\end{aligned}
\end{equation}
where 
\begin{equation} \label{def of tS_1ac dec}
	\begin{aligned}
		\tS_{q,i+1,L}^{(1),(\upsilon)}
		=&4\mu_{q,i}^{-2}\left((g_L^{(\upsilon)}(t))^\top\left((\sgm_{11}+\sgm_{2}+(2\sgm_{12}+\sgm_{111})\tj_1^{(0)}+\sgm_{22}\tj_2^{(0)})(\Id+\Gl)-\sgm_{12}\Cl(\Id+\Gl)\right)\nabla\psi_{L}^{(\upsilon)}\right)\Id\\
		&+4\mu_{q,i}^{-2}\left((g_L^{(\upsilon)}(t))^\top\left((\sgm_{12}+\sgm_{22}\tj_1^{(0)})(\Id+\Gl)-\sgm_{22}\Cl(\Id+\Gl)\right)\nabla \psi_{L}^{(\upsilon)}\right)(\tr \Cl\Id-\Dl)\\	&-2\mu_{q,i}^{-2}(\sgm_2+\sgm_{12}\tj_1^{(0)}+\sgm_{22}\tj_2^{(0)})\left(((\Id+\Gl^{\top})g_L^{(\upsilon)}(t))\otimes \nabla \psi_{L}^{(\upsilon)}+\nabla \psi_{L}^{(\upsilon)}\otimes ((\Id+\Gl^{\top})g_L^{(\upsilon)}(t))\right),\\
		\tS_{q,i+1,M}^{(1),(\upsilon)}
		=&4\mu_{q,i}^{-2}\left((g_M^{(\upsilon)}(t))^\top\left((\sgm_{11}+\sgm_{2}+(2\sgm_{12}+\sgm_{111})\tj_1^{(0)}+\sgm_{22}\tj_2^{(0)})(\Id+\Gl)-\sgm_{12}\Cl(\Id+\Gl)\right)\nabla\psi_{M}^{(\upsilon)}\right)\Id\\
		&+4\mu_{q,i}^{-2}\left((g_M^{(\upsilon)}(t))^\top\left((\sgm_{12}+\sgm_{22}\tj_1^{(0)})(\Id+\Gl)-\sgm_{22}\Cl(\Id+\Gl)\right)\nabla \psi_{M}^{(\upsilon)}\right)(\tr \Cl\Id-\Dl)\\
		&-2\mu_{q,i}^{-2}(\sgm_2+\sgm_{12}\tj_1^{(0)}+\sgm_{22}\tj_2^{(0)})\left(((\Id+\Gl^{\top})g_M^{(\upsilon)}(t))\otimes \nabla \psi_{M}^{(\upsilon)}+\nabla \psi_{M}^{(\upsilon)}\otimes ((\Id+\Gl^{\top})g_M^{(\upsilon)}(t))\right).
		\end{aligned}
\end{equation}

Notice that they satisfy the following properties: 
\begin{enumerate}
	\item The supports of  $\tV_{q,i+1,L}$, $\tV_{q,i+1,M}$, $\tu_{q,i+1,L}$, $\tu_{q,i+1,M}$, $\tG_{q,i+1,L}$, $\tG_{q,i+1,M}$,  $\tS_{q,i+1,L}^{(1)}$, and $\tS_{q,i+1,M}^{(1)}$ satisfy
	\begin{equation}\label{pp of supp tu_LMC 0}
	\begin{aligned}
		&\supp\tV_{q,i+1,L}\bigcap\supp\tV_{q,i+1,M}=\emptyset,
		&&\supp\tu_{q,i+1,L}\bigcap\supp\tu_{q,i+1,M}=\emptyset,\\
		&\supp\tG_{q,i+1,L}\bigcap\supp\tG_{q,i+1,M}=\emptyset,
		&&\supp\tS_{q,i+1,L}^{(1)}\bigcap\supp\tS_{q,i+1,M}^{(1)}=\emptyset.
	\end{aligned}
	\end{equation}
	 
		\item The supports of $\tV_{q,i+1,L}^{(\upsilon)}$, $\tV_{q,i+1,M}^{(\upsilon)}$,
	$\tu_{q,i+1,L}^{(\upsilon)}$, $\tu_{q,i+1,M}^{(\upsilon)}$, $\tG_{q,i+1,L}^{(\upsilon)}$, $\tG_{q,i+1,M}^{(\upsilon)}$, $\tS_{q,i+1,L}^{(1),(\upsilon)}$, and $\tS_{q,i+1,M}^{(1),(\upsilon)}$ satisfy, for $I=(s,\upsilon)\in\mathscr I$,
	\begin{equation} \label{pp of supp tu_LMC 1}
		\begin{aligned}
			&\supp_x\tV_{q,i+1,L}^{(\upsilon)},\supp_x\tV_{q,i+1,M}^{(\upsilon)}\subseteq \set{x\mid\chi_I(x)=\chi_{\upsilon}(\mu_{q,i}^{-1}x)=1},\\
			&\supp_x\tu_{q,i+1,L}^{(\upsilon)},\supp_x\tu_{q,i+1,M}^{(\upsilon)}\subseteq \set{x\mid\chi_I(x)=\chi_{\upsilon}(\mu_{q,i}^{-1}x)=1},\\
			&\supp_x\tG_{q,i+1,L}^{(\upsilon)},\supp_x\tG_{q,i+1,M}^{(\upsilon)}\subseteq \set{x\mid\chi_I(x)=\chi_{\upsilon}(\mu_{q,i}^{-1}x)=1},\\
			&\supp_x\tS_{q,i+1,L}^{(1),(\upsilon)},\supp_x\tS_{q,i+1,M}^{(1),(\upsilon)}\subseteq \set{x\mid\chi_I(x)=\chi_{\upsilon}(\mu_{q,i}^{-1}x)=1}.
		\end{aligned}
	\end{equation}
			Here and below, support relations involving a fixed spatial index are understood on the corresponding chosen lifts to $\R^2$. Let $e_L=e_1$ and $e_M=e_2$. By the definitions and \eqref{pp of psi}, if $x$ belongs to the support of any correction term indexed by $(k,\upsilon)$, where $k=L,M$, then
		$
		\left|\mu_{q,i}^{-1}x-\frac{\pi}{2}e_k-2\pi\upsilon\right|<\frac{1}{16}.
		$
		Consequently,
		$
		\left|\mu_{q,i}^{-1}x-2\pi\upsilon\right|_\infty\leqslant\frac{\pi}{2}+\frac{1}{16}<\frac{7\pi}{8},
		$
		and hence $\chi_\upsilon(\mu_{q,i}^{-1}x)=1$. If $\upsilon'\neq\upsilon$, then
		$
		\left|\mu_{q,i}^{-1}x-2\pi\upsilon'\right|_\infty\geqslant2\pi-\frac{\pi}{2}-\frac{1}{16}=\frac{3\pi}{2}-\frac{1}{16}>\frac{9\pi}{8},
		$
		which implies that $\chi_{\upsilon'}(\mu_{q,i}^{-1}x)=0$. Since every principal perturbation indexed by $I'=(s',\upsilon')$ is supported in $\supp_x\chi_{I'}$, the correction terms indexed by $\upsilon$ are disjoint from all principal perturbations whose spatial index is different from $\upsilon$. More precisely, for any $I'=(s',\upsilon')\in\mathscr I$ with $\upsilon'\neq\upsilon$, we have
		\begin{align*}
			&\left(\supp_x\tV_{q,i+1,L}^{(\upsilon)}\cup\supp_x\tV_{q,i+1,M}^{(\upsilon)}\right)\cap\supp_x\tV_{q,i+1,p}^{I'}=\emptyset,\\
			&\left(\supp_x\tu_{q,i+1,L}^{(\upsilon)}\cup\supp_x\tu_{q,i+1,M}^{(\upsilon)}\right)\cap\supp_x\tu_{q,i+1,p}^{I'}=\emptyset,\\
			&\left(\supp_x\tG_{q,i+1,L}^{(\upsilon)}\cup\supp_x\tG_{q,i+1,M}^{(\upsilon)}\right)\cap\supp_x\tG_{q,i+1,p}^{I'}=\emptyset,\\
			&\left(\supp_x\tS_{q,i+1,L}^{(1),(\upsilon)}\cup\supp_x\tS_{q,i+1,M}^{(1),(\upsilon)}\right)\cap\supp_x\tS_{q,i+1,p}^{(1),I'}=\emptyset.
		\end{align*}
\end{enumerate}

\begin{Rmk}
	The above support properties imply that the angular correction indexed by $\upsilon$ interacts only with the principal perturbations having the same spatial index $\upsilon$. This eliminates the cross terms involving different spatial cells and makes the construction of $g_L^{(\upsilon)}(t)$ and $g_M^{(\upsilon)}(t)$ more concise.
\end{Rmk}

Then, for each fixed $\upsilon$, we decompose $\tV_{q,i+1}^{(\upsilon)}$, $\tu_{q,i+1}^{(\upsilon)}$, $\tG_{q,i+1}^{(\upsilon)}$, and $\tS_{q,i+1}^{(1),(\upsilon)}$ as follows:
\begin{align*}
\tV_{q,i+1}^{(\upsilon)}=\tV_{q,i+1,p}^{(\upsilon)}+\tV_{q,i+1,ac}^{(\upsilon)},\quad \tu_{q,i+1}^{(\upsilon)}=\tu_{q,i+1,p}^{(\upsilon)}+\tu_{q,i+1,ac}^{(\upsilon)},\quad\tG_{q,i+1}^{(\upsilon)}=\tG_{q,i+1,p}^{(\upsilon)}+\tG_{q,i+1,ac}^{(\upsilon)},\quad\tS_{q,i+1}^{(1),(\upsilon)}=\tS_{q,i+1,p}^{(1),(\upsilon)}+\tS_{q,i+1,ac}^{(1),(\upsilon)}.
\end{align*}
Moreover, we could give the specific form of  $\tS_{q,i+1}^{(1),(\upsilon)}$ as
	\begin{equation}\label{specific def of tS_1}
	\begin{aligned}
	\tS_{q,i+1}^{(1),(\upsilon)}
	=&\underbrace{\sum_{s:I=(s,\upsilon)}i\gamma_{q,i+1}^{I}\tS_{q,i+1,m}^{(1),cofm}\Pi_{I}}_{\sum_{s:I=(s,\upsilon)}\tS_{q,i+1}^{(1),I,cofm}\Pi_{I}}+\underbrace{\sum_{s:I=(s,\upsilon)}\left(i\gamma_{q,i+1}^{I}(\tS_{q,i+1,m}^{(1),cofc}+a_{A_I^{\upsilon}}(t)\tS_{q,i+1,s1}^{(1),cof})+\frac{\tS_{q,i+1,s2}^{(1),I,cof1}}{\lambda_{q,i+1}[I]}\right)\Pi_{I}}_{\sum_{s:I=(s,\upsilon)}\tS_{q,i+1}^{(1),I,cof1}\Pi_{I}}\\
	&+\underbrace{\sum_{s:I=(s,\upsilon)}\frac{\tS_{q,i+1,s2}^{(1),I,cof2}}{\lambda_{q,i+1}[I]}\tPi_{I}}_{\sum_{s:I=(s,\upsilon)}\tS_{q,i+1}^{(1),I,cof2}\tPi_{I}}+\tS_{q,i+1,L}^{(1),(\upsilon)}+\tS_{q,i+1,M}^{(1),(\upsilon)}.
	\end{aligned}
\end{equation}
Next, we will  divide $\tS_{q,i+1}^{(2),m}$ into two parts $\tS_{q,i+1,p}^{(2),m}$ and $\tS_{q,i+1,s}^{(2),m}$, and give the specific form of $\tS_{q,i+1}^{(2),m}$:
	\begin{align}
		\tS_{q,i+1,p}^{(2),m}:=&\sum_{\upsilon}\tS_{q,i+1,p}^{(2),m,(\upsilon)}\nonumber\\
		=&\sum_{\upsilon}\sum_{\nrm{\tup-\upsilon}\leqslant1}\left(2(\sgm_{11}-\sgm_{12}+\sgm_2)\tr(\tG_{q,i+1,m}^{(\upsilon)}(\tG_{q,i+1,m}^{(\tup)})^{\top})+4(3\sgm_{12}+\sgm_{111})\tr\tG_{q,i+1,m}^{(\upsilon)}\tr\tG_{q,i+1,m}^{(\tup)}\right)\Id\nonumber\\
		&-\sum_{\upsilon}\sum_{\nrm{\tup-\upsilon}\leqslant1}\left(2\sgm_{12}\tr(\tG_{q,i+1,m}^{(\upsilon)}\tG_{q,i+1,m}^{(\tup)})\right)\Id\nonumber\\
		&-\sum_{\upsilon}\sum_{\nrm{\tup-\upsilon}\leqslant1}\left(2\sgm_2(\tG_{q,i+1,m}^{(\upsilon)})^{\top}\tG_{q,i+1,m}^{(\tup)}+4\sgm_{12}\tr\tG_{q,i+1,m}^{(\upsilon)}(\tG_{q,i+1,m}^{(\tup)}+(\tG_{q,i+1,m}^{(\tup)})^{\top})\right)\nonumber\\
		=&-\sum_{\upsilon}\sum_{s:I=(s,\upsilon)}\sum_{\substack{I'=(s',\tup)\in\mathscr I\\\nrm{I'-I}\leqslant1}}|f_{i+1}|^4(\gamma_{q,i+1}^I\gamma_{q,i+1}^{I'})\left(\sgm^*\Id-(2\sgm_2+8\sgm_{12})\frac{f_{i+1}}{|f_{i+1}|}\otimes \frac{f_{i+1}}{|f_{i+1}|}\right)\Pi_{I,I'},\label{def of tS_mm}\\
		\tS_{q,i+1,s}^{(2),m}
		=&2(\sgm_{11}-\sgm_{12}+\sgm_2)\tr\left(\tG_{q,i+1,m}(\tG_{q,i+1,s}+\tG_{q,i+1,ac})^{\top}+(\tG_{q,i+1,s}+\tG_{q,i+1,ac})(\tG_{q,i+1})^{\top}\right)\Id\nonumber\\
		&+4(3\sgm_{12}+\sgm_{111})\left(\tr(\tG_{q,i+1,m})\tr(\tG_{q,i+1,s}+\tG_{q,i+1,ac})+\tr(\tG_{q,i+1,s}+\tG_{q,i+1,ac})\tr(\tG_{q,i+1})\right)\Id\nonumber\\
		&-2\sgm_{12}\tr\left(\tG_{q,i+1,m}(\tG_{q,i+1,s}+\tG_{q,i+1,ac})+(\tG_{q,i+1,s}+\tG_{q,i+1,ac})\tG_{q,i+1}\right)\Id\label{def of tS_ms}\\
		&-2\sgm_2\left((\tG_{q,i+1,m})^{\top}(\tG_{q,i+1,s}+\tG_{q,i+1,ac})+(\tG_{q,i+1,s}+\tG_{q,i+1,ac})^{\top}\tG_{q,i+1}\right)\nonumber\\
		&-4\sgm_{12}\tr\tG_{q,i+1,m}(\tG_{q,i+1,s}+\tG_{q,i+1,ac}+(\tG_{q,i+1,s}+\tG_{q,i+1,ac})^{\top})\nonumber\\
		&-4\sgm_{12}\tr(\tG_{q,i+1,s}+\tG_{q,i+1,ac})(\tG_{q,i+1}+(\tG_{q,i+1})^{\top}).\nonumber
	\end{align}

Before we give the detailed construction of $g_L^{(\upsilon)}$ and $g_M^{(\upsilon)}$, we need to calculate the new Reynolds error $R_{q,i+1}$ after adding perturbation $\tu_{q,i+1}:=\tu_{q,i+1,p}+\tu_{q,i+1,ac}$ to $u_{\ell,i}$. We could divide $\Div((\Id+G_{q,i+1})R_{q,i+1})$ into three parts:
\begin{equation}\label{def of Reynold error 1}
\begin{aligned}
	&\Div((\Id+G_{q,i+1})R_{q,i+1})\\
	=&\underbrace{\pa_{tt}(\tu_{q,i+1,p}+\tu_{q,i+1,L})-\Div\left(\tG_{q,i+1}(\Sgm_{\Gl}+\Rl+c_{q,i})+(\Id+G_{q,i+1})\tS_{q,i+1}^{(1)}-(\Id+G_{q,i+1})\tG_{q,i+1,m}\tS_{q,i+1,m1}^{(1)}\right)}_{\Div ((\Id+G_{q,i+1})R_{P})}\\
	&+\underbrace{\Div\left((\Id+G_{q,i+1})(\sgm^*\te_{q,i}\Id^{<i+1>}+\Rl-\tG_{q,i+1,m}\tS_{q,i+1,m1}^{(1)}-\tS_{q,i+1}^{(2)}-\tS_{q,i+1}^{(\geqslant3)})\right)}_{\Div( (\Id+G_{q,i+1})R_O)=\Div( (\Id+G_{q,i+1})R_{O1})+\Div( (\Id+G_{q,i+1})R_{O2})+\Div( (\Id+G_{q,i+1})R_{O3})}+\underbrace{\pa_{tt}\tu_{q,i+1,M}+\Div\left(R_{m,i}\right)}_{\Div ((\Id+G_{q,i+1})R_{M})},
\end{aligned}
\end{equation} 
The second term naturally has the form $\Div( (\Id+G_{q,i+1})R_O)$ for some symmetric matrix $R_O$ due to its special form, so we only need to consider the other  terms.
Our goal is to ensure that the other terms above each can be divided into different compactly supported parts which satisfy \eqref{eq of linear and angular momenta in L co} individually. Noting the divergence form of the perturbation terms, the first integral condition of $ \eqref{eq of linear and angular momenta in L co} $ is naturally satisfied, so we mainly focus on the second condition, which is the conservation of angular momentum. This is the reason for constructing $ \tu_{q,i+1,L} $ and $ \tu_{q,i+1,M} $.

We first consider the following decomposition:
\begin{equation}\label{divide of Reynold error 1}
\begin{aligned}
&\pa_{tt}\tu_{q,i+1,M}+\Div R_{m,i}=\sum_{\upsilon}\left(\pa_{tt}\tu_{q,i+1,M}^{(\upsilon)}+\Div (\chi_\upsilon^2(\mu_{q,i}^{-1}x)R_{m,i})\right),\\
&\pa_{tt}(\tu_{q,i+1,p}+\tu_{q,i+1,L})-\Div\left(\tG_{q,i+1}(\Sgm_{\Gl}+\Rl+c_{q,i})+(\Id+G_{q,i+1})(\tS_{q,i+1}^{(1)}-\tG_{q,i+1,m}\tS_{q,i+1,m1}^{(1)})\right)\\
=&\sum_{\upsilon}\Big(\pa_{tt}\tu_{q,i+1,p}^{(\upsilon)}+\pa_{tt}\tu_{q,i+1,L}^{(\upsilon)}-\Div\Big(\tG_{q,i+1}^{(\upsilon)}(\Sgm_{\Gl}+\Rl+c_{q,i})+(\Id+G_{q,i+1})(\tS_{q,i+1}^{(1),(\upsilon)}-\tG_{q,i+1,m}\tS_{q,i+1,m1}^{(1),(\upsilon)})\Big)\Big),\\
&\Div\left((\Id+G_{q,i+1})(\sgm^*\te_{q,i}\Id^{<i+1>}+\Rl-\tG_{q,i+1,m}\tS_{q,i+1,m1}^{(1)}-\tS_{q,i+1}^{(2)}-\tS_{q,i+1}^{(\geqslant3)})\right)\\
=&\sum_{\upsilon}\Div\left((\Id+G_{q,i+1})\left(\sgm^*\chi_{\upsilon}^2(\mu_{q,i}^{-1}x)d_{q,i+1}^2\left(\Id-\frac{f_{i+1}}{|f_{i+1}|}\otimes \frac{f_{i+1}}{|f_{i+1}|}\right)-\tG_{q,i+1,m}\tS_{q,i+1,m1}^{(1),(\upsilon)}-\tS_{q,i+1,p}^{(2),m,(\upsilon)}\right)\right)\\
&\quad+\Div\left((\Id+G_{q,i+1})\left(\sgm^*\te_{q,i}\Id^{<i+1>}+\Rl-\sgm^*d_{q,i+1}^2\left(\Id-\frac{f_{i+1}}{|f_{i+1}|}\otimes \frac{f_{i+1}}{|f_{i+1}|}\right)-\tS_{q,i+1,s}^{(2),m}-\tS_{q,i+1}^{(2),c}-\tS_{q,i+1}^{(\geqslant3)}\right)\right).
\end{aligned}
\end{equation}
For each fixed $\upsilon$, all terms in \eqref{divide of Reynold error 1} indexed by $\upsilon$ are supported in $\supp\chi_\upsilon(\mu_{q,i}^{-1}\cdot)$. The final term in the last line is unlocalized and is treated separately through $R_{O2}$ and $R_{O3}$. 
\begin{Rmk}
	Although
	$
	\Div R_{m,i}
	=
	\sum_{\upsilon}\Div\left(\chi_\upsilon^2(\mu_{q,i}^{-1}x)R_{m,i}\right)
	=
	\sum_{\upsilon}\chi_\upsilon^2(\mu_{q,i}^{-1}x)\Div R_{m,i},
	$
	the two decompositions yield different local angular momentum estimates. We use the former, which gives better local control.
\end{Rmk}

So we need to choose proper $\tu_{q,i+1,ac}^{(\upsilon)}$ so that the following integrals
vanish for every $\upsilon$. Due to \eqref{pp of psi}, \eqref{pp of supp tu_LMC 0}, \eqref{pp of supp tu_LMC 1},  and  	for any smooth vector field $a=a(x)$
and any compactly supported matrix field $K$, 
\begin{align}
	\int_{\T^2}a\times \Div K\rd x&=\int_{\T^2}(a_1\partial_jK_{2j}-a_2\partial_jK_{1j})\rd x=\int_{\T^2}(K_{1j}\partial_ja_2-K_{2j}\partial_ja_1)\rd x:=\int_{\T^2}(K(\nabla a)^\top)_{12-21}\rd x,\label{int of a times divR}
\end{align}
we have
\begin{align*}
0=&\int_{\T^2} (y_{\ell,i}+\tu_{q,i+1})\times\left(\pa_{tt}\tu_{q,i+1,M}^{(\upsilon)}+\Div (\chi_\upsilon^2(\mu_{q,i}^{-1}x)R_{m,i})\right)\rd x\\
=&\int_{\T^2} y_{\ell,i}\times\Div (\chi_\upsilon^2(\mu_{q,i}^{-1}x)R_{m,i})\rd x+\int_{\T^2} \sum_{\|\tup-\upsilon\|\leqslant 1}\tu_{q,i+1,p}^{(\tup)}\times\Div (\chi_\upsilon^2(\mu_{q,i}^{-1}x)R_{m,i})\rd x+\int_{\T^2} \tu_{q,i+1,ac}^{(\upsilon)}\times\Div R_{m,i}\rd x\\
&+\int_{\T^2} y_{\ell,i}\times\pa_{tt}(\Div(\tV_{q,i+1,M}^{(\upsilon)}))\rd x+\int_{\T^2} \Div(\tV_{q,i+1,p}^{(\upsilon)})\times\pa_{tt}\tu_{q,i+1,M}^{(\upsilon)}\rd x+\int_{\T^2} \tu_{q,i+1,M}^{(\upsilon)}\times\pa_{tt}\tu_{q,i+1,M}^{(\upsilon)}\rd x\\
=&\int_{\T^2}  (\chi_\upsilon^2(\mu_{q,i}^{-1}x)R_{m,i}(\Id+\Gl^{\top}))_{12-21}\rd x+\int_{\T^2} \sum_{\|\tup-\upsilon\|\leqslant 1}\tu_{q,i+1,p}^{(\tup)}\times\Div (\chi_\upsilon^2(\mu_{q,i}^{-1}x)R_{m,i})\rd x\\
&+\mu_{q,i}^{-1}\left(g_M^{(\upsilon)}(t)\times\int_{\T^2} \psi_{M}^{(\upsilon)}\Div R_{m,i}\rd x+g_L^{(\upsilon)}(t)\times\int_{\T^2} \psi_{L}^{(\upsilon)}\Div R_{m,i}\rd x\right)\\
&+\int_{\T^2}  (\pa_{tt}\tV_{q,i+1,M}^{(\upsilon)}(\Id+\Gl^{\top}))_{12-21}\rd x+\int_{\T^2}  (\pa_{tt}\tG_{q,i+1,M}^{(\upsilon)}(\tV_{q,i+1,p}^{(\upsilon)})^{\top})_{12-21}\rd x+g_M^{(\upsilon)}(t)\times\pa_{tt}g_M^{(\upsilon)}(t),\\
0=&\int_{\T^2} (y_{\ell,i}+\tu_{q,i+1})\times\left(\pa_{tt}\tu_{q,i+1,p}^{(\upsilon)}+\pa_{tt}\tu_{q,i+1,L}^{(\upsilon)}-\Div(\tG_{q,i+1}^{(\upsilon)}(\Sgm_{\Gl}+\Rl+c_{q,i}))\right)\rd x\\
&-\int_{\T^2} (y_{\ell,i}+\tu_{q,i+1})\times\Div\left((\Id+G_{q,i+1})(\tS_{q,i+1}^{(1),(\upsilon)}-\tG_{q,i+1,m}\tS_{q,i+1,m1}^{(1),(\upsilon)})\right)\rd x\\
=&\int_{\T^2} (y_{\ell,i}+\sum_{\nrm{\tup-\upsilon}\leqslant1}\tu_{q,i+1,p}^{(\tup)}+\tu_{q,i+1,ac}^{(\upsilon)})\times(\pa_{tt}\Div(\tV_{q,i+1,p}^{(\upsilon)}))\rd x\\
&-\int_{\T^2} (y_{\ell,i}+\tu_{q,i+1})\times\Div(\tG_{q,i+1,p}^{(\upsilon)}(\Sgm_{\Gl}+\Rl+c_{q,i}))\rd x\\
&-\int_{\T^2} (y_{\ell,i}+\tu_{q,i+1,p}^{(\upsilon)}+\tu_{q,i+1,ac}^{(\upsilon)})\times\Div(\tG_{q,i+1,ac}^{(\upsilon)}(\Sgm_{\Gl}+\Rl+c_{q,i}))\rd x\\
&+\int_{\T^2} y_{\ell,i}\times\pa_{tt}(\Div(\tV_{q,i+1,L}^{(\upsilon)}))\rd x+\int_{\T^2} \Div(\tV_{q,i+1,p}^{(\upsilon)})\times\pa_{tt}\tu_{q,i+1,L}^{(\upsilon)}\rd x+\int_{\T^2} \tu_{q,i+1,L}^{(\upsilon)}\times\pa_{tt}\tu_{q,i+1,L}^{(\upsilon)}\rd x\\
=&\int_{\T^2} (\pa_{tt}\tV_{q,i+1,p}^{(\upsilon)}(\Id+\Gl+\sum_{\nrm{\tup-\upsilon}\leqslant1}\tG_{q,i+1,p}^{(\tup)})^{\top})_{12-21}\rd x\\
&-\int_{\T^2} (\tG_{q,i+1,p}^{(\upsilon)}(\Sgm_{\Gl}+\Rl+c_{q,i})(\Id+\Gl+\sum_{\nrm{\tup-\upsilon}=1}\tG_{q,i+1,p}^{(\tup)})^{\top})_{12-21}\rd x\\
&-\int_{\T^2} (\tG_{q,i+1,ac}^{(\upsilon)}(\Sgm_{\Gl}+\Rl+c_{q,i})(\Id+\Gl^{\top}))_{12-21}\rd x+\int_{\T^2} (\pa_{tt}\tV_{q,i+1,p}^{(\upsilon)}(\tG_{q,i+1,ac}^{(\upsilon)})^{\top})_{12-21}\rd x\\
&+\int_{\T^2} (\partial_{tt}\tV_{q,i+1,L}^{(\upsilon)}(\Id+\Gl^{\top}))_{12-21}\rd x+\int_{\T^2}  (\pa_{tt}\tG_{q,i+1,L}^{(\upsilon)}(\tV_{q,i+1,p}^{(\upsilon)})^{\top})_{12-21}\rd x+g_L^{(\upsilon)}(t)\times\pa_{tt}g_L^{(\upsilon)}(t),
\end{align*}
where $y_{\ell,i}=x+u_{\ell,i}$, and we have used  that
\begin{align*}
	\tG_{q,i+1,m}\tS_{q,i+1,m1}^{(1),(\upsilon)}=&-4\sgm_{11}\sum_{\|\tup-\upsilon\|\leqslant 1}\sum_{s:I=(s,\upsilon)}\sum_{s':I'=(s',\tup)}\gamma_{q,i+1}^{I}\gamma_{q,i+1}^{I'} |f_{i+1}|^2(f_{i+1}\otimes f_{i+1})\Pi_{I,I'},
\end{align*} 
is a symmetric matrix.

We first consider the case $\theta_{q,i}^*\equiv1$,  $t\in[S_{q,i}^{min}\tau_{q,i}, (S_{q,i}^{max}+1)\tau_{q,i}]$.
Then, we can rewrite the above equations as
\begin{align*}
	0=&\int_{\T^2}  (\chi_\upsilon^2(\mu_{q,i}^{-1}x)R_{m,i}(\Id+\Gl^{\top}))_{12-21}\rd x+\int_{\T^2} \sum_{\|\tup-\upsilon\|\leqslant 1}\tu_{q,i+1,p}^{(\tup)}\times\Div (\chi_\upsilon^2(\mu_{q,i}^{-1}x)R_{m,i})\rd x\\
	&+\varepsilon\mu_{q,i}\delta_{q+1}\int_{\T^2} (\psi_{M}^{(\upsilon)}+\psi_{L}^{(\upsilon)})(\Div R_{m,i})_2\rd x-\mu_{q,i}^{-1}\left(g_{M,2}^{(\upsilon)}(t)\int_{\T^2} \psi_{M}^{(\upsilon)}(\Div R_{m,i})_1\rd x+g_{L,2}^{(\upsilon)}(t)\int_{\T^2} \psi_{L}^{(\upsilon)}(\Div R_{m,i})_1\rd x\right)\\
	&+\left(\varepsilon\mu_{q,i}^2\delta_{q+1}-\int_{\T^2} ((\Id+\Gl)\nabla\Psi_{M}^{(\upsilon)})_{1}\rd x-\mu_{q,i}^{-2}\int_{\T^2} (\tV_{q,i+1,p}^{(\upsilon)}\nabla\psi_{M}^{(\upsilon)})_1\rd x\right)\pa_{tt}g_{M,2}^{(\upsilon)}(t),\\
	0=&\int_{\T^2} (\pa_{tt}\tV_{q,i+1,p}^{(\upsilon)}(\Id+\Gl+\sum_{\nrm{\tup-\upsilon}\leqslant1}\tG_{q,i+1,p}^{(\tup)})^{\top})_{12-21}\rd x\\
	&-\int_{\T^2} (\tG_{q,i+1,p}^{(\upsilon)}(\Sgm_{\Gl}+\Rl+c_{q,i})(\Id+\Gl+\sum_{\nrm{\tup-\upsilon}=1}\tG_{q,i+1,p}^{(\tup)})^{\top})_{12-21}\rd x\\
	&-\varepsilon\delta_{q+1}\int_{\T^2} ((\pa_{tt}\tV_{q,i+1,p}^{(\upsilon)}+(\Id+\Gl)(\Sgm_{\Gl}+\Rl+c_{q,i}))\nabla(\psi_{M}^{(\upsilon)}+\psi_{L}^{(\upsilon)}))_{2}\rd x\\
	&+\mu_{q,i}^{-2}g_{M,2}^{(\upsilon)}(t)\int_{\T^2}( (\pa_{tt}\tV_{q,i+1,p}^{(\upsilon)}+(\Id+\Gl)(\Sgm_{\Gl}+\Rl+c_{q,i}))\nabla\psi_{M}^{(\upsilon)})_{1}\rd x\\
	&+\mu_{q,i}^{-2}g_{L,2}^{(\upsilon)}(t)\int_{\T^2}( (\pa_{tt}\tV_{q,i+1,p}^{(\upsilon)}+(\Id+\Gl)(\Sgm_{\Gl}+\Rl+c_{q,i}))\nabla\psi_{L}^{(\upsilon)})_{1}\rd x\\
	&+\left(\varepsilon\mu_{q,i}^2\delta_{q+1}-\int_{\T^2} ((\Id+\Gl)\nabla\Psi_{L}^{(\upsilon)})_{1}\rd x-\mu_{q,i}^{-2}\int_{\T^2} (\tV_{q,i+1,p}^{(\upsilon)}\nabla\psi_{L}^{(\upsilon)})_1\rd x\right)\pa_{tt}g_{L,2}^{(\upsilon)}(t).
\end{align*}
For each fixed $\upsilon$, this is a linear ODE for $g_{L,2}^{(\upsilon)}$ and $g_{M,2}^{(\upsilon)}$, which can be written as
\begin{align}
	H^{(\upsilon)}\pa_{tt}U^{(\upsilon)}+B^{(\upsilon)}U^{(\upsilon)}=E^{(\upsilon)},\label{def of ODE of U}
\end{align}
where $U^{(\upsilon)}=(g_{L,2}^{(\upsilon)}(t),g_{M,2}^{(\upsilon)}(t))^{\top}$, $H^{(\upsilon)}=(H_{ij}^{(\upsilon)})\in\R^{2\times2}$, $B^{(\upsilon)}=(B_{ij}^{(\upsilon)})\in\R^{2\times2}$, and $E^{(\upsilon)}=(E_i^{(\upsilon)})\in\R^2$.
\begin{align*}
	H_{11}^{(\upsilon)}=&\varepsilon\mu_{q,i}^2\delta_{q+1}-\int_{\T^2} ((\Id+\Gl)\nabla\Psi_{L}^{(\upsilon)})_{1}\rd x-\mu_{q,i}^{-2}\int_{\T^2} (\tV_{q,i+1,p}^{(\upsilon)}\nabla\psi_{L}^{(\upsilon)})_1\rd x,\quad H_{12}^{(\upsilon)}=0,\\
	H_{22}^{(\upsilon)}=&\varepsilon\mu_{q,i}^2\delta_{q+1}-\int_{\T^2} ((\Id+\Gl)\nabla\Psi_{M}^{(\upsilon)})_{1}\rd x-\mu_{q,i}^{-2}\int_{\T^2} (\tV_{q,i+1,p}^{(\upsilon)}\nabla\psi_{M}^{(\upsilon)})_1\rd x,\quad H_{21}^{(\upsilon)}=0,\\
	B_{11}^{(\upsilon)}=&\mu_{q,i}^{-2}\int_{\T^2} ((\pa_{tt}\tV_{q,i+1,p}^{(\upsilon)}+(\Id+\Gl)(\Sgm_{\Gl}+\Rl+c_{q,i}))\nabla\psi_{L}^{(\upsilon)})_{1}\rd x,\quad B_{21}^{(\upsilon)}=-\mu_{q,i}^{-1}\int_{\T^2} \psi_{L}^{(\upsilon)}(\Div R_{m,i})_1\rd x,\\
	B_{12}^{(\upsilon)}=&\mu_{q,i}^{-2}\int_{\T^2} ((\pa_{tt}\tV_{q,i+1,p}^{(\upsilon)}+(\Id+\Gl)(\Sgm_{\Gl}+\Rl+c_{q,i}))\nabla\psi_{M}^{(\upsilon)})_{1}\rd x,\quad
	B_{22}^{(\upsilon)}=-\mu_{q,i}^{-1}\int_{\T^2} \psi_{M}^{(\upsilon)}(\Div R_{m,i})_1\rd x,\\
	E_1^{(\upsilon)}=&-\int_{\T^2} (\pa_{tt}\tV_{q,i+1,p}^{(\upsilon)}(\Id+\Gl+\sum_{\nrm{\tup-\upsilon}\leqslant1}\tG_{q,i+1,p}^{(\tup)})^{\top})_{12-21}\rd x\\
	&+\int_{\T^2} (\tG_{q,i+1,p}^{(\upsilon)}(\Sgm_{\Gl}+\Rl+c_{q,i})(\Id+\Gl+\sum_{\nrm{\tup-\upsilon}=1}\tG_{q,i+1,p}^{(\tup)})^{\top})_{12-21}\rd x\\
	&+\varepsilon\delta_{q+1}\int_{\T^2} ((\pa_{tt}\tV_{q,i+1,p}^{(\upsilon)}+(\Id+\Gl)(\Sgm_{\Gl}+\Rl+c_{q,i}))\nabla(\psi_{M}^{(\upsilon)}+\psi_{L}^{(\upsilon)}))_{2}\rd x,\\
	E_2^{(\upsilon)}=&-\int_{\T^2}  (\chi_\upsilon^2(\mu_{q,i}^{-1}x)R_{m,i}(\Id+\Gl^{\top}))_{12-21}\rd x-\int_{\T^2} \sum_{\|\tup-\upsilon\|\leqslant 1}\tu_{q,i+1,p}^{(\tup)}\times\Div (\chi_\upsilon^2(\mu_{q,i}^{-1}x)R_{m,i})\rd x\\
	&-\varepsilon\mu_{q,i}\delta_{q+1}\int_{\T^2} (\psi_{M}^{(\upsilon)}+\psi_{L}^{(\upsilon)})(\Div R_{m,i})_2\rd x.
\end{align*}
\begin{Rmk}
Even though all the functions here depend only on time $t$, for the sake of notational consistency and simplicity, we adopt the partial derivative symbol $\partial_t$ to represent derivatives with respect to $t$ instead of the standard derivative symbol $\frac{\rd}{\rd t}$. This convention not only simplifies the notation but also aligns with potential future analyses involving multiple variables.
\end{Rmk}
By \eqref{pp of te_q,i} and \eqref{pp of theta^*}, we have
\begin{align*}
	[-\tau_{q,i},S_{q,i}^{min}\tau_{q,i})
	\bigcap\supp_{t}\te_{q,i}
	=\emptyset.
\end{align*}
 Moreover, by the definition of $S_{q,i}^{min}$ and the continuity of
 $\te_{q,i}$, we also have
 $\te_{q,i}(S_{q,i}^{min}\tau_{q,i})=0$.
Consequently, for $t\in[-\tau_{q,i},S_{q,i}^{min}\tau_{q,i}]$, we have $\te_{q,i}=0$, $\tu_{q,i+1,p}=\ul=0$, $\tV_{q,i+1,p}=R_{m,i}=\Rl=0$, and $\Gl=0,$ which leads to
\begin{align}\label{special HBE}
	H_{11}^{(\upsilon)}&=H_{22}^{(\upsilon)}=\varepsilon\mu_{q,i}^2\delta_{q+1},\quad H_{12}^{(\upsilon)}=H_{21}^{(\upsilon)}=B_{11}^{(\upsilon)}=B_{12}^{(\upsilon)}=B_{21}^{(\upsilon)}=B_{22}^{(\upsilon)}=E_1^{(\upsilon)}=E_2^{(\upsilon)}=0.
\end{align}

Recall
\eqref{def of g_LM2}, for $S_{q,i}^{min}\leqslant S\leqslant S_{q,i}^{max}$, we need to define $g_{L,2}^{(S,\upsilon)}$ and $g_{M,2}^{(S,\upsilon)}$. In particular, the definitions of $g_{L,2}^{(S+1,\upsilon)}$ and $g_{M,2}^{(S+1,\upsilon)}$ depend on the definitions of $g_{L,2}^{(S,\upsilon)}$ and $g_{M,2}^{(S,\upsilon)}$.
Here, we first choose $g_{L,2}^{(S_{q,i}^{min},\upsilon)}$ and $g_{M,2}^{(S_{q,i}^{min},\upsilon)}$ to satisfy $U^{(S_{q,i}^{min},\upsilon)}=0, t\in(-\infty,S_{q,i}^{min}\tau_{q,i})$, and
\begin{equation}\left\lbrace\label{def of ODE of U^-1}
	\begin{aligned}
		&H^{(\upsilon)}\pa_{tt}U^{(S_{q,i}^{min},\upsilon)}+B^{(\upsilon)}U^{(S_{q,i}^{min},\upsilon)}=E^{(\upsilon)},\quad t\in[S_{q,i}^{min}\tau_{q,i},(S_{q,i}^{min}+5/4)\tau_{q,i}],\\
		&U^{(S_{q,i}^{min},\upsilon)}(S_{q,i}^{min}\tau_{q,i})=(0,0)^{\top},\quad\pa_tU^{(S_{q,i}^{min},\upsilon)}(S_{q,i}^{min}\tau_{q,i})=(0,0)^{\top},
	\end{aligned}\right.
\end{equation}
where $U^{(S_{q,i}^{min},\upsilon)}(t)=(g_{L,2}^{(S_{q,i}^{min},\upsilon)}(\tau_{q,i}^{-1}t),g_{M,2}^{(S_{q,i}^{min},\upsilon)}(\tau_{q,i}^{-1}t))^{\top}$. By a change of variables, we can rewrite the above equation as
\begin{equation}\left\lbrace\label{eq of ODE of U^-1}
	\begin{aligned}
		&\tau_{q,i}^{-2}\tH^{(\upsilon)}\pa_{tt}\tU^{(S_{q,i}^{min},\upsilon)}+\tB^{(\upsilon)}\tU^{(S_{q,i}^{min},\upsilon)}=\tE^{(\upsilon)},\quad t\in[S_{q,i}^{min},S_{q,i}^{min}+5/4],\\
		&\tU^{(S_{q,i}^{min},\upsilon)}(S_{q,i}^{min})=(0,0)^{\top},\quad\pa_t\tU^{(S_{q,i}^{min},\upsilon)}(S_{q,i}^{min})=(0,0)^{\top},
	\end{aligned}\right.
\end{equation}
where $\tU^{(S_{q,i}^{min},\upsilon)}(t)=(g_{L,2}^{(S_{q,i}^{min},\upsilon)}(t),g_{M,2}^{(S_{q,i}^{min},\upsilon)}(t))^{\top}$, $\tH^{(\upsilon)}(t)=H^{(\upsilon)}(\tau_{q,i}t)$, $\tB^{(\upsilon)}(t)=B^{(\upsilon)}(\tau_{q,i}t)$, and $\tE^{(\upsilon)}(t)=E^{(\upsilon)}(\tau_{q,i}t)$.
Assuming that $g_{L,2}^{(S-1,\upsilon)}$ and $g_{M,2}^{(S-1,\upsilon)}$ have already been defined, we now choose $g_{L,2}^{(S,\upsilon)}$ and $g_{M,2}^{(S,\upsilon)}$ to satisfy for $S_{q,i}^{min}+1\leqslant S\leqslant S_{q,i}^{max}$, $\tU^{(S,\upsilon)}=0, t\in(-\infty,S)$, and
\begin{equation}\left\lbrace\label{def of ODE of U^s+1}
	\begin{aligned}
		&\tau_{q,i}^{-2}\tH^{(\upsilon)}\pa_{tt}(\tU^{(S,\upsilon)}+\Theta_{S-1}\tU^{(S-1,\upsilon)})+\tB^{(\upsilon)}(\tU^{(S,\upsilon)}+\Theta_{S-1}\tU^{(S-1,\upsilon)})=\tE^{(\upsilon)},\quad t\in[S,S+5/4],\\
		&\tU^{(S,\upsilon)}(S)=(0,0)^{\top},\quad\pa_t\tU^{(S,\upsilon)}(S)=(0,0)^{\top},
	\end{aligned}\right.
\end{equation}
where $\tU^{(S,\upsilon)}(t)=(g_{L,2}^{(S,\upsilon)}(t),g_{M,2}^{(S,\upsilon)}(t))^{\top}$. Then, we could obtain
\begin{equation*}
	\tH^{(\upsilon)}\pa_{tt}\tU^{(S,\upsilon)}+\tau_{q,i}^{2}\tB^{(\upsilon)}\tU^{(S,\upsilon)}=\left\lbrace
	\begin{aligned}
		&\tau_{q,i}^{2}\tE^{(\upsilon)}(1-\Theta_{S-1})-2\pt\Theta_{S-1}\tH^{(\upsilon)}\pt\tU^{(S-1,\upsilon)}-\pa_{tt}\Theta_{S-1}\tH^{(\upsilon)}\tU^{(S-1,\upsilon)},\ t\in[S,S+1/4],\\
		&\tau_{q,i}^{2}\tE^{(\upsilon)},\hspace{208pt} t\in(S+1/4,S+5/4].
	\end{aligned}\right.
\end{equation*}

To sum up, we could obtain
\begin{align*}
	\tau_{q,i}^{-2}\tH^{(\upsilon)}\pa_{tt}\Big(\sum_{S=S_{q,i}^{min}}^{S_{q,i}^{max}}\Theta_{S}\tU^{(S,\upsilon)}\Big)+\tB^{(\upsilon)}\Big(\sum_{S=S_{q,i}^{min}}^{S_{q,i}^{max}}\Theta_{S}\tU^{(S,\upsilon)}\Big)=\tE^{(\upsilon)},\quad t\in [S_{q,i}^{min},S_{q,i}^{max}+1],
\end{align*}
which is equivalent to
\begin{align*}
	H^{(\upsilon)}\pa_{tt}\Big(\sum_{S=S_{q,i}^{min}}^{S_{q,i}^{max}}\Theta_{q,i}^{(S)}U^{(S,\upsilon)}\Big)+B^{(\upsilon)}\Big(\sum_{S=S_{q,i}^{min}}^{S_{q,i}^{max}}\Theta_{q,i}^{(S)} U^{(S,\upsilon)}\Big)=E^{(\upsilon)},\quad t\in [S_{q,i}^{min}\tau_{q,i},(S_{q,i}^{max}+1)\tau_{q,i}].
\end{align*}
Moreover, since $g_{L,2}^{(S,\upsilon)}$ and $g_{M,2}^{(S,\upsilon)}$ vanish on $(-\infty,S]$, the functions
$g_{L,2}^{(\upsilon)}(t)=\sum_{S=S_{q,i}^{min}}^{S_{q,i}^{max}}\Theta_{q,i}^{(S)}(t)g_{L,2}^{(S,\upsilon)}(\tau_{q,i}^{-1}t)$ and
$g_{M,2}^{(\upsilon)}(t)=\sum_{S=S_{q,i}^{min}}^{S_{q,i}^{max}}\Theta_{q,i}^{(S)}(t)g_{M,2}^{(S,\upsilon)}(\tau_{q,i}^{-1}t)$
vanish for $t\in[-\tau_{q,i},S_{q,i}^{min}\tau_{q,i}]$.
Then, we need to consider the angular momentum in the region 
$$
[S_{q,i}^{min} \tau_{q,i}-\tau_{q,i}, S_{q,i}^{min}\tau_{q,i}].
$$
The equation can be written as
\begin{align*}
	&\int_{\T^2} (x+\tu_{q,i+1,ac})\times(\pa_{tt}\tu_{q,i+1,ac}^{(\upsilon)}-\Div(\tG_{q,i+1,ac}^{(\upsilon)}c_{q,i}))\rd x\\
	=&\int_{\T^2}  (\pa_{tt}\tV_{q,i+1,M}^{(\upsilon)}+\pa_{tt}\tV_{q,i+1,L}^{(\upsilon)})_{12-21}\rd x-\int_{\T^2}  ((\tG_{q,i+1,L}^{(\upsilon)}+\tG_{q,i+1,M}^{(\upsilon)})c_{q,i})_{12-21}\rd x+g_M^{(\upsilon)}(t)\times\pa_{tt}g_M^{(\upsilon)}(t)+g_L^{(\upsilon)}(t)\times\pa_{tt}g_L^{(\upsilon)}(t)
	=0,
\end{align*}
where we have used the fact that $c_{q,i}=c_{q,i+1}$ is a constant matrix in $\T^2$ for $t\in[S_{q,i}^{min} \tau_{q,i}-\tau_{q,i}, S_{q,i}^{min}\tau_{q,i}]$, and
\begin{align*}
	\int_{\T^2}  (\pa_{tt}\tV_{q,i+1,L}^{(\upsilon)})_{12-21}\rd x=\int_{\T^2}  (\pa_{tt}g_{L}^{(\upsilon)}(t)\otimes\nb\Psi_{L}^{(\upsilon)})_{12-21}\rd x&=0,\\
	\int_{\T^2}  (\pa_{tt}\tV_{q,i+1,M}^{(\upsilon)})_{12-21}\rd x=\int_{\T^2}  (\pa_{tt}g_{M}^{(\upsilon)}(t)\otimes\nb\Psi_{M}^{(\upsilon)})_{12-21}\rd x&=0,\\
	\int_{\T^2}  (\tG_{q,i+1,L}^{(\upsilon)}c_{q,i})_{12-21}\rd x=\int_{\T^2}  (g_{L}^{(\upsilon)}(t)\otimes c_{q,i}\nabla\psi_L^{(\upsilon)})_{12-21}\rd x&=0,\\
	\int_{\T^2}  (\tG_{q,i+1,M}^{(\upsilon)}c_{q,i})_{12-21}\rd x=\int_{\T^2}  (g_{M}^{(\upsilon)}(t)\otimes c_{q,i}\nabla\psi_M^{(\upsilon)})_{12-21}\rd x&=0,\\
	g_L^{(\upsilon)}(t)\times\pa_{tt}g_L^{(\upsilon)}(t)=\theta_{q,i}^*\pa_{tt}\theta_{q,i}^*(\varepsilon\mu_{q,i}^2\delta_{q+1},0)^{\top}\times(\varepsilon\mu_{q,i}^2\delta_{q+1},0)^{\top}&=0,\\
	g_M^{(\upsilon)}(t)\times\pa_{tt}g_M^{(\upsilon)}(t)=\theta_{q,i}^*\pa_{tt}\theta_{q,i}^*(\varepsilon\mu_{q,i}^2\delta_{q+1},0)^{\top}\times(\varepsilon\mu_{q,i}^2\delta_{q+1},0)^{\top}&=0.
\end{align*}

At this point, we have ensured that the right-hand error terms in \eqref{divide of Reynold error 1} satisfy angular momentum conservation, i.e., \eqref{eq of linear and angular momenta in L co}, over the entire interval $\mcI^{q,i}$.

Finally, we will give the following estimates on the perturbation. The following estimates can be obtained directly from their definitions.
\begin{pp}\label{est on perturbation}
	For every $b>5$, there exists $\varepsilon_3^*=\varepsilon_3^*(b,\oM)>0$ such that, whenever $0<\varepsilon<\varepsilon_3^*$, the following properties hold:
	\begin{enumerate}
		\item The following time-support inclusions hold:
		\begin{align}
			\supp_{t}(\tV_{q,i+1,p},\tu_{q,i+1,p},\tG_{q,i+1,p})&\subseteq\supp_{t} \te_{q,i}\subseteq\supp_{t}(u_{q,i},R_{q,i})+\ell_{q,i}+\lambda_{q}^{-1},\label{pp of supp tu_q,i+1 p}\\
			\supp_{t}(\tV_{q,i+1,ac},\tu_{q,i+1,ac},\tG_{q,i+1,ac})&\subseteq\supp_{t}\te_{q,i}+2\tau_{q,i}\subseteq\supp_{t}(u_{q,i},R_{q,i})+3\tau_{q,i}+\ell_{q,i}+\lambda_{q}^{-1}.\label{pp of supp tu_q,i+1 ac}
		\end{align}
		\item For every fixed integer $N\geqslant0$ and $0\leqslant r\leqslant5$, we have the following estimates on these coefficients for $0\leqslant i\leqslant2$:
		\begin{align}
			\nrm{\pa_t^r\gamma_{q,i+1}^{I}}_N+\lambda_{q,i+1}\mu_{q,i}\nrm{\pa_t^r\gamma_{q,i+1,c}^{I}}_N\lesssim&_{\sgm,N,r} \tau_{q,i}^{-r}\mu_{q,i}^{-N}\dlt_{q+1}^{\frac{1}{2}},\label{est on gamma_I}\\
			\nrm{\pa_t^rV_{q,i+1,p}^{I,cof}}_N+\nrm{\pa_t^ru_{q,i+1,p}^{I,cof}}_N\lesssim&_{\sgm,N,r} \tau_{q,i}^{-r}\mu_{q,i}^{-N}\dlt_{q+1}^{\frac{1}{2}},\label{est on uV_I}\\
			\varepsilon^{-1}\nrm{\tS_{q,i+1,m}^{(1),cofc}}_0+\nrm{\tS_{q,i+1,s1}^{(1),cof}}_0+\delta_{q+1}^{-\frac{1}{2}}\mu_{q,i}\left(\lambda_{q,i+1}\mu_{q,i}\nrm{\tS_{q,i+1,s2}^{(1),I,cof1}}_0+\nrm{\tS_{q,i+1,s2}^{(1),I,cof2}}_0\right)\lesssim&_{\sgm}1,\label{est on cof of tS_p 0}\\
			\varepsilon^{-1}\nrm{\pt^r\tS_{q,i+1,m}^{(1),cofc}}_N+		\nrm{\pt^r\tS_{q,i+1,s1}^{(1),cof}}_N+\delta_{q+1}^{-\frac{1}{2}}\mu_{q,i}\left(\lambda_{q,i+1}\mu_{q,i}\nrm{\pt^r\tS_{q,i+1,s2}^{(1),I,cof1}}_N+\nrm{\pt^r\tS_{q,i+1,s2}^{(1),I,cof2}}_N\right)\lesssim&_{\sgm,N,r}\tau_{q,i}^{-r}\mu_{q,i}^{-N},\label{est on cof of tS_p N}\\
			\nrm{\tS_{q,i+1}^{(1),cofm}}_0+	\varepsilon^{-1}\nrm{\tS_{q,i+1}^{(1),cof1}}_0+	\lambda_{q,i+1}\mu_{q,i}\nrm{\tS_{q,i+1}^{(1),cof2}}_0\lesssim&_{\sgm,N,r}\delta_{q+1}^{\frac{1}{2}},\label{est on cof of tS}\\
			\nrm{\pt^r\tS_{q,i+1}^{(1),cofm}}_N+	\varepsilon^{-1}\nrm{\pt^r\tS_{q,i+1}^{(1),cof1}}_N+	\lambda_{q,i+1}\mu_{q,i}	\nrm{\pt^r\tS_{q,i+1}^{(1),cof2}}_N\lesssim&_{\sgm,N,r}\tau_{q,i}^{-r}\mu_{q,i}^{-N}\delta_{q+1}^{\frac{1}{2}},\label{est on cof of tS N}
		\end{align}
		where 
		$\nrm{\cdot}_N=\nrm{\cdot}_{C^0(\mcI^{q,i};C^N(\T^2))}$.
		\item For every fixed integer $N\geqslant0$ and $0\leqslant r\leqslant5$, we have the following estimates on $\tV_{q,i+1}$, $\tu_{q,i+1}$, $\tG_{q,i+1}$, and $\tS_{q,i+1}$ for $0\leqslant i\leqslant2$. In particular, the estimate \eqref{est on tS_p} holds for every spatial index $\upsilon$, with an implicit constant independent of $\upsilon$:
		\begin{align}
			\nrm{\pa_t^r\tV_{q,i+1,p}}_N+\mu_{q,i}\nrm{\pa_t^r(f_{i+1}^{\perp}\cdot\nb)\tV_{q,i+1,p}}_N&\lesssim_{\sgm,N,r}\lambda_{q,i+1}^{N+r-2} \dlt_{q+1}^{\frac{1}{2}}, \label{est on tV_p}\\
			\nrm{\pa_t^r\tu_{q,i+1,p}}_N+\mu_{q,i}\nrm{\pa_t^r(f_{i+1}^{\perp}\cdot\nb)\tu_{q,i+1,p}}_N&\lesssim_{\sgm,N,r}\lambda_{q,i+1}^{N+r-1} \dlt_{q+1}^{\frac{1}{2}}, \label{est on tu_p}\\
			\nrm{\pa_t^r\pttu_{q,i+1,m}}_N+\varepsilon^{-1}\nrm{\pa_t^r\pttu_{q,i+1,s}}_N&\lesssim_{\sgm,N,r}\lambda_{q,i+1}^{N+r-1} \dlt_{q+1}^{\frac{1}{2}}, \label{est on pt tu_p}\\
			\nrm{\pa_t^r\tG_{q,i+1,p}}_N+\varepsilon^{-1}\nrm{\pa_t^r\tG_{q,i+1,s}}_N&\lesssim_{\sgm,N,r}\lambda_{q,i+1}^{N+r} \dlt_{q+1}^{\frac{1}{2}}, \label{est on tG_p}\\
			\nrm{\pa_t^r(f_{i+1}^{\perp}\cdot\nb)\tG_{q,i+1,p}}_N+\varepsilon^{-1}\nrm{\pa_t^r(f_{i+1}^{\perp}\cdot\nb)\tG_{q,i+1,s}}_N
			&\lesssim_{\sgm,N,r}\mu_{q,i}^{-1}\lambda_{q,i+1}^{N+r} \dlt_{q+1}^{\frac{1}{2}}, \label{est on fp nb tG_p}\\
			\nrm{\pt^r\tS_{q,i+1,m1}^{(1),(\upsilon)}}_N
			+\varepsilon^{-1}
			\nrm{\pt^r\tS_{q,i+1,m2}^{(1),(\upsilon)}}_N
			+\varepsilon^{-1}
			\nrm{\pt^r\tS_{q,i+1,s1}^{(1),(\upsilon)}}_N
			+\lambda_{q,i+1}\mu_{q,i}
			\nrm{\pt^r\tS_{q,i+1,s2}^{(1),(\upsilon)}}_N
			&\lesssim_{\sgm,N,r}
			\lambda_{q,i+1}^{N+r}\delta_{q+1}^{\frac12}.
			\label{est on tS_p}\\
			\mu_{q,i}^{-2}\nrm{\pt^{r}\tV_{q,i+1,L}}_N+\mu_{q,i}^{-1}\nrm{\pt^{r}\tu_{q,i+1,L}}_N+\nrm{\pt^{r}\tG_{q,i+1,L}}_N&\lesssim_{\sgm,N,r}\mu_{q,i}^{-N-r}\varepsilon\delta_{q+1}, \label{est on tVuG L}\\
			\mu_{q,i}^{-2}\nrm{\pt^{r}\tV_{q,i+1,M}}_N+\mu_{q,i}^{-1}\nrm{\pt^{r}\tu_{q,i+1,M}}_N+\nrm{\pt^{r}\tG_{q,i+1,M}}_N&\lesssim_{\sgm,N,r}\mu_{q,i}^{-N-r}\varepsilon\delta_{q+1}, \label{est on tVuG M}\\
			\mu_{q,i}^{-2}\nrm{\pt^{r}\tV_{q,i+1,ac}}_N+\mu_{q,i}^{-1}\nrm{\pt^{r}\tu_{q,i+1,ac}}_N+\nrm{\pt^{r}\tG_{q,i+1,ac}}_N&\lesssim_{\sgm,N,r}\mu_{q,i}^{-N-r}\varepsilon\delta_{q+1}, \label{est on tVuG ac}\\
			\nrm{\pt^r\tS_{q,i+1,L}^{(1)}}_N+\nrm{\pt^r\tS_{q,i+1,M}^{(1)}}_N+\nrm{\pt^r\tS_{q,i+1,ac}^{(1)}}_N&\lesssim_{\sgm,N,r}\mu_{q,i}^{-N-r}\varepsilon\delta_{q+1},\label{est on tS_L}\\
			\nrm{\pa_t^r\tu_{q,i+1}}_N+\mu_{q,i}\nrm{\pa_t^r(f_{i+1}^{\perp}\cdot\nb)\tu_{q,i+1}}_N
			&\lesssim_{\sgm,N,r}\lambda_{q,i+1}^{N+r}\lambda_{q,i}^{-1}\dlt_{q+1}^{\frac{1}{2}}, \label{est on nb uq i+1}\\
			\nrm{\pa_t^r\tG_{q,i+1}}_N+\mu_{q,i}\nrm{\pa_t^r(f_{i+1}^{\perp}\cdot\nb)\tG_{q,i+1}}_N
			&\lesssim_{\sgm,N,r}\lambda_{q,i+1}^{N+r} \dlt_{q+1}^{\frac{1}{2}}, \label{est on nb Gq i+1}
		\end{align}
		where  
		$\nrm{\cdot}_N=\nrm{\cdot}_{C^0(\mcI^{q,i};C^N(\T^2))}$. 
	\end{enumerate}
	Moreover, the implicit constants in \eqref{est on gamma_I}--\eqref{est on nb Gq i+1} can be chosen to be independent of $\oM$ for $0\leqslant N+r\leqslant 3$. 
\end{pp}
\begin{proof}
	Notice that, for  $t\in \supp_t\te_{q,i}+3\tau_{q,i}$, we have $c_{q,i}=\sgm^*\te_{q,i}\Id^{<i+1>}+c_{q,i+1}$, and $c_{q,i+1}$ is a constant matrix, which leads to $c_{q,i}\in C^{\infty}((\supp_t\te_{q,i}+3\tau_{q,i})\times\T^2)$ and  $\nrm{c_{q,i}}_{C_t^r(\supp_t\te_{q,i}+3\tau_{q,i})}\lesssim_r\lambda_{q}^r\delta_{q+1}$.
	By the definition of $(A_I)^{km}_{nr}$ in \eqref{def of A_I}, \eqref{pp of te_q,i}, and \eqref{est on u_li 0}--\eqref{est on R_li high}, we obtain, for every $I=(s,\upsilon)\in\mathscr I$ satisfying $s\tau_{q,i}\in\supp_t\te_{q,i}+\tau_{q,i}$ and every $r\geqslant1$,
	
	\begin{align}
		\nrm{(A_I^{\upsilon})^{km}_{nr}}_{C_t^0}\lesssim&_{\sgm}\nrm{(\Sgm_{\Gl}+\Rl+c_{q,i})^{\upsilon}}_{C_t^0}+\nrm{(\Id+\Gl)^{\upsilon}}_{0}(\nrm{\Gl^{\upsilon}}_{C_t^0}+\nrm{\Cl^{\upsilon}(\Id+\Gl)^{\upsilon}}_{C_t^0})\nonumber\\
		&+\nrm{(\Id+\Gl)^{\upsilon}(\tr\Cl-\Dl)^{\upsilon}}_{C_t^0}(\nrm{(\Id+\Gl)^{\upsilon}}_{C_t^0}+\nrm{\Cl^{\upsilon}(\Id+\Gl)^{\upsilon}}_{C_t^0})\nonumber\\
		&+\nrm{(\tL_2^{(0)})^{\upsilon}-\sgm_2}_{C_t^0}\nrm{(\Id+\Gl)^{\upsilon}(\Id+\Gl)^{\upsilon}}_{C_t^0}+\nrm{(\Id+\Gl)^{\upsilon}(\Id+\Gl)^{\upsilon}-\Id}_{C_t^0}\nonumber\\
		\lesssim&_\sgm\varepsilon,\nonumber\\
		\nrm{(A_I^{\upsilon})^{km}_{nr}}_{C_t^r}\lesssim&_{\sgm}\nrm{(\Sgm_{\Gl}+\Rl+c_{q,i})^{\upsilon}}_{C_t^r}+\sum_{r_1+r_2=r}\nrm{(\Id+\Gl)^{\upsilon}}_{C_t^{r_1}}(\nrm{\Gl^{\upsilon}}_{C_t^{r_2}}+\nrm{\Cl^{\upsilon}(\Id+\Gl)^{\upsilon}}_{C_t^{r_2}})\nonumber\\
		&+\sum_{r_1+r_2=r}\nrm{(\Id+\Gl)^{\upsilon}(\tr\Cl-\Dl)^{\upsilon}}_{C_t^{r_1}}(\nrm{(\Id+\Gl)^{\upsilon}}_{C_t^{r_2}}+\nrm{\Cl^{\upsilon}(\Id+\Gl)^{\upsilon}}_{C_t^{r_2}})\nonumber\\
		&+\sum_{r_1+r_2=r}\nrm{(\tL_2^{(0)})^{\upsilon}-\sgm_2}_{C_t^{r_1}}\nrm{(\Id+\Gl)^{\upsilon}(\Id+\Gl)^{\upsilon}}_{C_t^{r_2}}+\nrm{(\Id+\Gl)^{\upsilon}(\Id+\Gl)^{\upsilon}-\Id}_{C_t^{r}}\nonumber\\
		\lesssim&_{\sgm,r}\ell_{q,i}^{-r}\delta_{q,i}^{\frac{1}{2}},\nonumber
	\end{align}
	where $\nrm{\cdot}_{C_t^r}=\nrm{\cdot}_{C^r(\supp_t\te_{q,i}+3\tau_{q,i})}$ and the implicit constants do not depend on $\oM$. Since the coefficients $(A_I^\upsilon)^{km}_{nr}$ are smooth, for each fixed integer $r$ we choose $L\geqslant\max\{2,r\}$. We apply Lemma \ref{construction of building blocks}. Thus, there exists $\varepsilon_{31}^*>0$ such that, for every $0<\varepsilon<\varepsilon_{31}^*$,
	\begin{align}
		&\nrm{a_{A_I^\upsilon}}_{C_t^0}\lesssim_{\sgm}\varepsilon\leqslant\frac{1}{10},\\
		&\nrm{\pt c_{A_I^\upsilon}-(\lambda+2\mu)^{\frac{1}{2}}|f_{i+1}|}_{C_t^0}\lesssim_{\sgm}\varepsilon|f_{i+1}|,\\
		&\nrm{a_{A_I^\upsilon}}_{C_t^r}\lesssim_{\sgm,r}\ell_{q,i}^{-r}\delta_{q,i}^{\frac{1}{2}},&&r\geqslant1,\\
		&\nrm{\pt^r c_{A_I^\upsilon}}_{C_t^0}\lesssim_{\sgm,r}\ell_{q,i}^{1-r}\delta_{q,i}^{\frac{1}{2}}|f_{i+1}|,&&r\geqslant2.
	\end{align}
	\eqref{est on uV_I} follows immediately from these estimates and the definitions in \eqref{def of cof of tV_p} and \eqref{def of u_q,i+1,I}. Moreover, \eqref{pp of supp tu_q,i+1 p} follows from \eqref{pp of te_q,i}, \eqref{pp of supp l,i},  \eqref{def of tV_q k i+1,p}, and \eqref{def of cof of tV_p}. By \eqref{def of tS_1m cofm}--\eqref{def of tS_1s2 cof}, \eqref{specific def of tS_1}, \eqref{est on gamma_I}, and \eqref{est on uV_I},  we could obtain \eqref{est on cof of tS_p 0}--\eqref{est on cof of tS N}.
	
	Next, we consider the estimates for $\tV_{q,i+1,p}$, $\tu_{q,i+1,p}$, $\tG_{q,i+1,p}$ and $\tS_{q,i+1,p}^{(1)}$. The local stress estimates below are uniform with respect to $\upsilon$, since the estimates for the corresponding coefficients are uniform in $I=(s,\upsilon)$. For $0\leqslant N+r\leqslant 3$, we obtain
	\begin{align*}
		\nrm{\pt^r\tV_{q,i+1,p}}_N\lesssim&\lambda_{q,i+1}^{-2}\sum_{N_1+N_2=N}\sum_{r_1+r_2=r}\nrm{\pt^{r_1}\tV_{q,i+1,p  }^{I,cof}}_{N_1}\nrm{\pt^{r_2}\Pi_{I}}_{N_2}\lesssim_{\sgm,N,r}\lambda_{q,i+1}^{N+r-2}\delta_{q+1}^{\frac{1}{2}},\\
		\nrm{\pt^r\tu_{q,i+1,p}}_N=&\nrm{\pt^r\Div\tV_{q,i+1,p}}_{N}\lesssim_{\sgm,N,r}\lambda_{q,i+1}^{N+r-1}\delta_{q+1}^{\frac{1}{2}},\\
		\nrm{\pt^r\tG_{q,i+1,p}}_N=&\nrm{\pt^r\nabla\tu_{q,i+1,p}}_{N}\lesssim_{\sgm,N,r}\lambda_{q,i+1}^{N+r}\delta_{q+1}^{\frac{1}{2}},\\
		\nrm{\pt^r\tG_{q,i+1,s1}}_N\lesssim&\sum_{N_1+N_2=N}\sum_{r_1+r_2=r}\nrm{\pt^{r_1}(a_{A_I^\upsilon}\gamma_{q,i+1}^{I})}_{N_1}\nrm{\pt^{r_2}\Pi_{I}}_{N_2}\lesssim_{\sgm,N,r}\lambda_{q,i+1}^{N+r}\varepsilon\delta_{q+1}^{\frac{1}{2}},\\
		\nrm{\pt^r\tG_{q,i+1,s2}}_N\lesssim&\lambda_{q,i+1}^{-1}\sum_{N_1+N_2=N}\sum_{r_1+r_2=r}\left(\nrm{\pt^{r_1}(\tilde{f}_{A_I^{\upsilon}}\otimes\nabla\gamma_{q,i+1,c})}_{N_1}\nrm{\pt^{r_2}\Pi_{I}}_{N_2}+\nrm{\pt^{r_1}(\tilde{f}_{A_I^{\upsilon}}\otimes\gamma_{q,i+1,s})}_{N_1}\nrm{\pt^{r_2}\tPi_{I}}_{N_2}\right)\\
		\lesssim&_{\sgm,N,r}(\lambda_{q,i+1}\mu_{q,i})^{-1}\lambda_{q,i+1}^{N+r}\delta_{q+1}^{\frac{1}{2}},\\
		\nrm{\pt^r\pttu_{q,i+1,p}}_N\lesssim&\sum_{N_1+N_2=N}\sum_{r_1+r_2=r}\nrm{\pt^{r_1}\gamma_{q,i+1}^{I}}_{N_1}\nrm{\pt^{r_2}\Pi_{I}}_{N_2}\lesssim_{\sgm,N,r}\lambda_{q,i+1}^{N+r}\delta_{q+1}^{\frac{1}{2}},\\
		\nrm{\pt^r\pttu_{q,i+1,s1}}_N\lesssim&\sum_{N_1+N_2=N}\sum_{r_1+r_2=r}\nrm{\pt^{r_1}(\gamma_{q,i+1}^{I}(((\lambda+2\mu)^{\frac{1}{2}}|f_{i+1}|-\pt c_{A_I^\upsilon})f_{i+1}-a_{A_I^\upsilon}\pt c_{A_I^\upsilon}f_{i+1}^{\perp}))}_{N_1}\nrm{\pt^{r_2}\Pi_{I}}_{N_2}
		\lesssim_{\sgm,N,r}\lambda_{q,i+1}^{N+r}\varepsilon\delta_{q+1}^{\frac{1}{2}},\\
		\nrm{\pt^r\pttu_{q,i+1,s2}}_N\lesssim&\lambda_{q,i+1}^{-1}\left(\sum_{N_1+N_2=N}\sum_{r_1+r_2=r+1}\nrm{\pt^{r_1}(\gamma_{q,i+1,c}\tilde{f}_{A_I^{\upsilon}})}_{N_1}\nrm{\pt^{r_2}\Pi_{I}}_{N_2}+\sum_{N_1+N_2=N}\sum_{r_1+r_2=r}\nrm{\pt^{r_1+1}(\gamma_{q,i+1}\tilde{f}_{A_I^{\upsilon}})}_{N_1}\nrm{\pt^{r_2}\tPi_{I}}_{N_2}\right)\\
		\lesssim&_{\sgm,N,r}(\lambda_{q,i+1}\mu_{q,i})^{-1}\lambda_{q,i+1}^{N+r}\delta_{q+1}^{\frac{1}{2}},\\
		\nrm{\pt^r\tS_{q,i+1,m2}^{(1),(\upsilon)}}_N\lesssim&\sum_{N_1+N_2+N_3=N}\sum_{r_1+r_2+r_3=r}\nrm{\pt^{r_1}\gamma_{q,i+1}^I}_{N_1}\nrm{\pt^{r_2}\tS_{q,i+1,m}^{(1),cofc}}_{N_2}\nrm{\pt^{r_3}\Pi_{I}}_{N_3}\lesssim_{\sgm,N,r}\lambda_{q,i+1}^{N+r}\varepsilon\delta_{q+1}^{\frac{1}{2}},\\
		\nrm{\pt^r\tS_{q,i+1,s1}^{(1),(\upsilon)}}_N\lesssim&\sum_{N_1+N_2+N_3=N}\sum_{r_1+r_2+r_3=r}\nrm{\pt^{r_1}(\gamma_{q,i+1}^Ia_{A_I^{\upsilon}})}_{N_1}\nrm{\pt^{r_2}\tS_{q,i+1,s1}^{(1),cofm}}_{N_2}\nrm{\pt^{r_3}\Pi_{I}}_{N_3}\lesssim_{\sgm,N,r}\lambda_{q,i+1}^{N+r}\varepsilon\delta_{q+1}^{\frac{1}{2}},\\
		\nrm{\pt^r\tS_{q,i+1,s2}^{(1),(\upsilon)}}_N\lesssim&\lambda_{q,i+1}^{-1}\hspace{-5pt}\sum_{N_1+N_2=N}\sum_{r_1+r_2=r}\left(\nrm{\pt^{r_1}\tS_{q,i+1,s1}^{(1),cof1}}_{N_1}\nrm{\pt^{r_2}\Pi_{I}}_{N_2}+\nrm{\pt^{r_1}\tS_{q,i+1,s1}^{(1),cof2}}_{N_1}\nrm{\pt^{r_2}\tPi_{I}}_{N_2}\right)
		\lesssim_{\sgm,N,r}(\lambda_{q,i+1}\mu_{q,i})^{-1}\lambda_{q,i+1}^{N+r}\delta_{q+1}^{\frac{1}{2}}.
	\end{align*}
	Moreover, by using
	\begin{align*}
		(f_{i+1}^{\perp}\cdot\nabla) e^{i\lambda_{q,i+1}\xi_{A_I^{\upsilon},f_{i+1}}}=(f_{i+1}^{\perp}\cdot\nabla) e^{-i\lambda_{q,i+1}\xi_{A_I^{\upsilon},f_{i+1}}}=(f_{i+1}^{\perp}\cdot\nabla)\Pi_{I}=(f_{i+1}^{\perp}\cdot\nabla)\tPi_{I}=0,
	\end{align*}	
	we could obtain
	\begin{align*}
		\nrm{\pt^r(f_{i+1}^{\perp}\cdot\nabla)\tV_{q,i+1,p}}_N\lesssim&\lambda_{q,i+1}^{-2}\sum_{N_1+N_2=N}\sum_{r_1+r_2=r}\nrm{\pt^{r_1}(f_{i+1}^{\perp}\cdot\nabla)\tV_{q,i+1,p  }^{I,cof}}_{N_1}\nrm{\pt^{r_2}\Pi_{I}}_{N_2}\lesssim_{\sgm,N,r}\mu_{q,i}^{-1}\lambda_{q,i+1}^{N+r-2}\delta_{q+1}^{\frac{1}{2}},\\
		\nrm{\pt^r(f_{i+1}^{\perp}\cdot\nabla)\tu_{q,i+1,p}}_N\lesssim&\lambda_{q,i+1}^{-1}\sum_{N_1+N_2=N}\sum_{r_1+r_2=r}\nrm{\pt^{r_1}(f_{i+1}^{\perp}\cdot\nabla)\tu_{q,i+1,p}^{I,cof}}_{N_1}\nrm{\pt^{r_2}e^{i\lambda_{q,i+1}\xi_{A_I^{\upsilon},f_{i+1}}}}_{N_2}\lesssim_{\sgm,N,r}\mu_{q,i}^{-1}\lambda_{q,i+1}^{N+r-1}\delta_{q+1}^{\frac{1}{2}},\\
		\nrm{\pt^r(f_{i+1}^{\perp}\cdot\nabla)\tG_{q,i+1,p}}_N\lesssim&\sum_{N_1+N_2=N}\sum_{r_1+r_2=r}\nrm{\pt^{r_1}(f_{i+1}^{\perp}\cdot\nabla)\gamma_{q,i+1}^{I}}_{N_1}\nrm{\pt^{r_2}\Pi_{I}}_{N_2}\lesssim_{\sgm,N,r}\mu_{q,i}^{-1}\lambda_{q,i+1}^{N+r}\delta_{q+1}^{\frac{1}{2}},\\
		\nrm{\pt^r(f_{i+1}^{\perp}\cdot\nabla)\tG_{q,i+1,s1}}_N\lesssim&\sum_{N_1+N_2=N}\sum_{r_1+r_2=r}\nrm{\pt^{r_1}(f_{i+1}^{\perp}\cdot\nabla)(a_{A_I^\upsilon}\gamma_{q,i+1}^{I})}_{N_1}\nrm{\pt^{r_2}\Pi_{I}}_{N_2}\lesssim_{\sgm,N,r}\mu_{q,i}^{-1}\lambda_{q,i+1}^{N+r}\varepsilon\delta_{q+1}^{\frac{1}{2}},\\
		\nrm{\pt^r(f_{i+1}^{\perp}\cdot\nabla)\tG_{q,i+1,s2}}_N\lesssim&\lambda_{q,i+1}^{-1}\sum_{N_1+N_2+N_3=N}\sum_{r_1+r_2+r_3=r}\nrm{\pt^{r_1}\tilde{f}_{A_I^{\upsilon}}}_{N_1}\nrm{\pt^{r_2}(f_{i+1}^{\perp}\cdot\nabla)\nabla\gamma_{q,i+1,c}}_{N_2}\nrm{\pt^{r_3}\Pi_{I}}_{N_3}\\
		&+\lambda_{q,i+1}^{-1}\sum_{N_1+N_2+N_3=N}\sum_{r_1+r_2+r_3=r}\nrm{\pt^{r_1}\tilde{f}_{A_I^{\upsilon}}}_{N_1}\nrm{\pt^{r_2}(f_{i+1}^{\perp}\cdot\nabla)\gamma_{q,i+1,s}}_{N_2}\nrm{\pt^{r_3}\tPi_{I}}_{N_3}\\
		\lesssim&_{\sgm,N,r}\mu_{q,i}^{-1}(\lambda_{q,i+1}\mu_{q,i})^{-1}\lambda_{q,i+1}^{N+r}\delta_{q+1}^{\frac{1}{2}}.
	\end{align*}
	Up to now, we have proved \eqref{est on tV_p}--\eqref{est on tS_p}.
	
	In order to get estimates for $\tV_{q,i+1,ac}$, $\tu_{q,i+1,ac}$, and $\tG_{q,i+1,ac}$, we need to get estimate for $g_{L,2}^{(\upsilon)}(t)$ and $g_{M,2}^{(\upsilon)}(t)$. Recall that 
	\begin{equation*}\left\lbrace
		\begin{aligned}
			&\pa_{tt}\tU^{(S_{q,i}^{min},\upsilon)}+\tau_{q,i}^{2}(\tH^{(\upsilon)})^{-1}\tB^{(\upsilon)}\tU^{(S_{q,i}^{min},\upsilon)}=\tau_{q,i}^{2}(\tH^{(\upsilon)})^{-1}\tE^{(\upsilon)},\quad t\in[S_{q,i}^{min},5/4+S_{q,i}^{min}],\\
			&\tU^{(S_{q,i}^{min},\upsilon)}(S_{q,i}^{min})=(0,0)^{\top},\quad\pa_t\tU^{(S_{q,i}^{min},\upsilon)}(S_{q,i}^{min})=(0,0)^{\top},
		\end{aligned}\right.
	\end{equation*}
	where $\tU^{(S_{q,i}^{min},\upsilon)}(t)=(g_{L,2}^{(S_{q,i}^{min},\upsilon)}(t),g_{M,2}^{(S_{q,i}^{min},\upsilon)}(t))^{\top},\ \tH^{(\upsilon)}(t)=H^{(\upsilon)}(\tau_{q,i}t),\ \tB^{(\upsilon)}(t)=B^{(\upsilon)}(\tau_{q,i}t),\ \tE^{(\upsilon)}(t)=E^{(\upsilon)}(\tau_{q,i}t),$  and
	\begin{equation*}
		\pa_{tt}\tU^{(S,\upsilon)}+\tau_{q,i}^{2}(\tH^{(\upsilon)})^{-1}\tB^{(\upsilon)}\tU^{(S,\upsilon)}=\left\lbrace
		\begin{aligned}
			&\tau_{q,i}^{2}(\tH^{(\upsilon)})^{-1}\tE^{(\upsilon)}(1-\Theta_{S-1})-2\pt\Theta_{S-1}\pt\tU^{(S-1,\upsilon)}-\pa_{tt}\Theta_{S-1}\tU^{(S-1,\upsilon)}, \ t\in[S,S+1/4],\\
			&\tau_{q,i}^{2}(\tH^{(\upsilon)})^{-1}\tE^{(\upsilon)},\hspace{174pt} t\in(S+1/4,S+5/4],
		\end{aligned}\right.
	\end{equation*}
	where $S_{q,i}^{min}+1\leqslant S\leqslant S_{q,i}^{max}$ and $\tU^{(S,\upsilon)}(t)=(g_{L,2}^{(S,\upsilon)}(t),g_{M,2}^{(S,\upsilon)}(t))^{\top}$. 
	
	So we should first give estimates for $H^{(\upsilon)}(t)$, $B^{(\upsilon)}(t)$, $\tH^{(\upsilon)}(t)$, and $\tB^{(\upsilon)}(t)$. 
	For $b>5$, we could find $\varepsilon^*_{32}(b,\oM)$ such that for $\varepsilon<\varepsilon^*_{32}$,
	\begin{align}
		\oM\mu_{q,i}\lambda_{q,i}\leqslant\varepsilon^2 \delta_{q+1}\delta_{q,i}^{-\frac{1}{2}}\leqslant\varepsilon^2\delta_{q+1}^{\frac{1}{2}},        \label{pp of Lambda2}
	\end{align} 
	and we have the following estimates for $0\leqslant r\leqslant 5$,
	\begin{enumerate}
		\item Estimates on $H^{(\upsilon)}(t)$ and $\tH^{(\upsilon)}(t)$.
		
		We could first calculate
		\begin{align*}
			H_{11}^{(\upsilon)}=&\varepsilon\mu_{q,i}^2\delta_{q+1}-\int_{\T^2} ((\Id+\Gl)\nabla\Psi_{L}^{(\upsilon)})_{1}\rd x-\mu_{q,i}^{-2}\int_{\T^2} (\tV_{q,i+1,p}^{(\upsilon)}\nabla\psi_{L}^{(\upsilon)})_1\rd x\\ =&\varepsilon\mu_{q,i}^2\delta_{q+1}+\mu_{q,i}\int_{\T^2} (\Psi_{L}^{(\upsilon)}\Div \Gl)_{1}\rd x-\sum_{s:I=(s,\upsilon)}\int_{\T^2} \frac{( V_{q,i+1,p}^{I,cof}\nabla\psi_{L}^{(\upsilon)}\Pi_{I})_1}{i(\lambda_{q,i+1}\mu_{q,i}[I]|f_{i+1}|)^2}\rd x.
		\end{align*}
		and then, by using  \eqref{pp of Lambda1}, \eqref{pp of supp tu_LMC 1},  \eqref{pp of Lambda2}, and Lemma \ref{est on int operator}, we have for $0\leqslant r\leqslant5$,
		\begin{align}
			\Nrm{\sum_{s:I=(s,\upsilon)}\int_{\T^2} \frac{( V_{q,i+1,p}^{I,cof}\nabla\psi_{L}^{(\upsilon)}\Pi_{I})_1}{i(\lambda_{q,i+1}\mu_{q,i}[I]|f_{i+1}|)^2}\rd x}_{C_{t}^r}
			\lesssim&\lambda_{q,i+1}^{r-2}\frac{\nrm{ V_{q,i+1,p}^{I,cof}\nabla\psi_{L}^{(\upsilon)}}_{10}}{(\lambda_{q,i+1}[I]|f_{i+1}|)^{10}}
			\lesssim(\lambda_{q,i+1}\mu_{q,i})^{-10}\lambda_{q,i+1}^{r-2}\delta_{q+1}^{\frac{1}{2}}\leqslant\mu_{q,i}^{-r}\frac{\varepsilon\mu_{q,i}^2\delta_{q+1}}{4},\nonumber\\
			\Nrm{\mu_{q,i}\int_{\T^2} (\Psi_{L}^{(\upsilon)}\Div \Gl)_{1}\rd x}_{C_{t}^r}
			\lesssim&\mu_{q,i}^3\Nrm{\pt^r(\Psi_{L}^{(\upsilon)}\Div \Gl)}_{0}
			\lesssim\mu_{q,i}^{3-r}\oM\lambda_{q,i}\delta_{q,i}^{\frac{1}{2}}\leqslant\mu_{q,i}^{-r}\frac{\varepsilon\mu_{q,i}^2\delta_{q+1}}{4},\nonumber
		\end{align}
		for sufficiently small $\varepsilon$.
		It is easy to obtain, for every $t$,
		\begin{align*}
			\frac{1}{2}\varepsilon\mu_{q,i}^2\delta_{q+1}
			\leqslant H_{11}^{(\upsilon)}(t)
			\leqslant
			\frac{3}{2}\varepsilon\mu_{q,i}^2\delta_{q+1},
		\end{align*}
		and, for $1\leqslant r\leqslant5$,
		\begin{align*}
			\nrm{\pt^rH_{11}^{(\upsilon)}}_{C_t^0}
			\leqslant&
			\Nrm{\sum_{s:I=(s,\upsilon)}
				\int_{\T^2}
				\frac{
					(V_{q,i+1,p}^{I,cof}
					\nabla\psi_L^{(\upsilon)}\Pi_I)_1
				}{
					i(\lambda_{q,i+1}\mu_{q,i}[I]|f_{i+1}|)^2
				}
				\rd x}_{C_t^r}
			+
			\Nrm{\mu_{q,i}
				\int_{\T^2}
				(\Psi_L^{(\upsilon)}\Div\Gl)_1
				\rd x}_{C_t^r}
			\leqslant
			\mu_{q,i}^{-r}
			\frac{\varepsilon\mu_{q,i}^2\delta_{q+1}}{2}.
		\end{align*}
		Similarly, for every $t$,
		\begin{align*}
			\frac{1}{2}\varepsilon\mu_{q,i}^2\delta_{q+1}
			\leqslant H_{22}^{(\upsilon)}(t)
			\leqslant
			\frac{3}{2}\varepsilon\mu_{q,i}^2\delta_{q+1},
		\end{align*}
		and, for $1\leqslant r\leqslant5$,
		\begin{align*}
			\nrm{\pt^rH_{22}^{(\upsilon)}}_{C_t^0}
			\leqslant
			\mu_{q,i}^{-r}
			\frac{\varepsilon\mu_{q,i}^2\delta_{q+1}}{2}.
		\end{align*}
		
		Since $H_{12}^{(\upsilon)}=H_{21}^{(\upsilon)}=0$, the matrix $H^{(\upsilon)}(t)$ is diagonal and
		invertible for every $t$, with
		$$
		(H^{(\upsilon)})^{-1}(t)
		=
		\begin{pmatrix}
			(H_{11}^{(\upsilon)})^{-1}(t)&0\\
			0&(H_{22}^{(\upsilon)})^{-1}(t)
		\end{pmatrix}.
		$$
		Consequently, we have the following estimates for $H^{(\upsilon)}(t)$,
		$\tH^{(\upsilon)}(t)$, and their inverses:
		\begin{equation}\label{est on H and tH}
			\begin{aligned}
				\nrm{H^{(\upsilon)}}_{C_t^0}
				\leqslant&
				\frac{3}{2}
				\varepsilon\mu_{q,i}^2\delta_{q+1},
				\qquad
				\nrm{\pt^rH^{(\upsilon)}}_{C_t^0}
				\lesssim_r
				\mu_{q,i}^{-r}
				\varepsilon\mu_{q,i}^2\delta_{q+1},
				&&1\leqslant r\leqslant5,
				\\
				\nrm{\tH^{(\upsilon)}}_{C_t^0}
				\leqslant&
				\frac{3}{2}
				\varepsilon\mu_{q,i}^2\delta_{q+1},
				\qquad
				\nrm{\pt^r\tH^{(\upsilon)}}_{C_t^0}
				\lesssim_r
				(\mu_{q,i}^{-1}\tau_{q,i})^r
				\varepsilon\mu_{q,i}^2\delta_{q+1},
				&&1\leqslant r\leqslant5,
				\\
				\nrm{(H^{(\upsilon)})^{-1}}_{C_t^0}
				\leqslant&
				2\varepsilon^{-1}
				\mu_{q,i}^{-2}\delta_{q+1}^{-1},
				\qquad
				\nrm{(\tH^{(\upsilon)})^{-1}}_{C_t^0}
				\leqslant
				2\varepsilon^{-1}
				\mu_{q,i}^{-2}\delta_{q+1}^{-1},
				\\
				\nrm{\pt^r(H^{(\upsilon)})^{-1}}_{C_t^0}
				\lesssim_r&
				\mu_{q,i}^{-r}
				\varepsilon^{-1}
				\mu_{q,i}^{-2}\delta_{q+1}^{-1},
				&&1\leqslant r\leqslant5,
				\\
				\nrm{\pt^r(\tH^{(\upsilon)})^{-1}}_{C_t^0}
				\lesssim_r&
				(\mu_{q,i}^{-1}\tau_{q,i})^r
				\varepsilon^{-1}
				\mu_{q,i}^{-2}\delta_{q+1}^{-1},
				&&1\leqslant r\leqslant5,
			\end{aligned}
		\end{equation}
		where  we have used
		$
		\pt (H^{(\upsilon)})^{-1}=-(H^{(\upsilon)})^{-1}\pt H^{(\upsilon)} (H^{(\upsilon)})^{-1}.
		$
		
		\item Estimates on $B^{(\upsilon)}(t)$ and $\tB^{(\upsilon)}(t)$.
		We could first calculate
		\begin{align*}
			B_{11}^{(\upsilon)}=&\mu_{q,i}^{-2}\int_{\T^2} ((\pa_{tt}\tV_{q,i+1,p}^{(\upsilon)}+(\Id+\Gl)(\Sgm_{\Gl}+\Rl+c_{q,i}))\nabla\psi_{L}^{(\upsilon)})_{1}\rd x\\
			=&\sum_{s:I=(s,\upsilon)}\int_{\T^2} \frac{(\pa_{tt}(V_{q,i+1,p}^{I,cof}\Pi_I)\nabla\psi_{L}^{(\upsilon)})_{1}}{i(\lambda_{q,i+1}\mu_{q,i}[I]|f_{i+1}|)^2}\rd x
			-\mu_{q,i}\int_{\T^2} (\Psi_{L}^{(\upsilon)}\Delta\Div((\Id+\Gl)(\Sgm_{\Gl}+\Rl+c_{q,i})))_{1}\rd x\\
			B_{21}^{(\upsilon)}=&-\mu_{q,i}^{-1}\int_{\T^2} \psi_{L}^{(\upsilon)}(\Div R_{m,i})_1\rd x
			=-\mu_{q,i}\int_{\T^2} \Psi_{L}^{(\upsilon)}\Delta(\Div R_{m,i})_1\rd x.
		\end{align*}
		and then, by using  \eqref{pp of Lambda1}, \eqref{pp of supp tu_LMC 1},  \eqref{pp of Lambda2}, and Lemma \ref{est on int operator},  we could similarly obtain
		\begin{align*}
			&\Nrm{\sum_{s:I=(s,\upsilon)}\int_{\T^2} \frac{(\pa_{tt}(V_{q,i+1,p}^{I,cof}\Pi_I)\nabla\psi_{L}^{(\upsilon)})_{1}}{i(\lambda_{q,i+1}\mu_{q,i}[I]|f_{i+1}|)^2}\rd x}_{C_{t}^r}\nonumber\\
			\lesssim&\lambda_{q,i+1}^{r-2}\frac{\nrm{\pa_{tt} V_{q,i+1,p}^{I,cof}\nabla\psi_{L}^{(\upsilon)}}_{10}}{(\lambda_{q,i+1}[I]|f_{i+1}|)^{10}}+\lambda_{q,i+1}^{r-1}\frac{\nrm{\pt(\pt c_{A_I^\upsilon} V_{q,i+1,p}^{I,cof})\nabla\psi_{L}^{(\upsilon)}}_{10}}{(\lambda_{q,i+1}[I]|f_{i+1}|)^{10}}+\lambda_{q,i+1}^{r}\frac{\nrm{(\pt c_{A_I^\upsilon})^2 V_{q,i+1,p}^{I,cof}\nabla\psi_{L}^{(\upsilon)}}_{10}}{(\lambda_{q,i+1}[I]|f_{i+1}|)^{10}}\nonumber\\
			\lesssim&(\lambda_{q,i+1}\mu_{q,i})^{-10}\lambda_{q,i+1}^{r}\delta_{q+1}^{\frac{1}{2}}\lesssim\mu_{q,i}^{-r}\varepsilon^4\delta_{q+1}^2,\\
			&\Nrm{\mu_{q,i}\int_{\T^2} (\Psi_{L}^{(\upsilon)}\Delta\Div((\Id+\Gl)(\Sgm_{\Gl}+\Rl+c_{q,i})))_{1}\rd x}_{C_{t}^r}\\
			\lesssim&\mu_{q,i}^3\Nrm{\pt^r(\Psi_{L}^{(\upsilon)}\Delta\Div((\Id+\Gl)(\Sgm_{\Gl}+\Rl+c_{q,i})))}_{0}\nonumber
			\lesssim_{\sgm}\mu_{q,i}^{2-r}\oM\lambda_{q,i}^2\delta_{q,i}^{\frac{1}{2}}\lesssim_{\sgm}\mu_{q,i}^{-r}\varepsilon^4\delta_{q+1}^{\frac{3}{2}},\\
			&\Nrm{\mu_{q,i}\int_{\T^2} \Psi_{L}^{(\upsilon)}\Delta(\Div R_{m,i})_1\rd x}_{C_{t}^r}\\
			\lesssim&\mu_{q,i}^3\Nrm{\pt^r(\Psi_{L}^{(\upsilon)}\Delta(\Div R_{m,i})_1)}_{0}
			\lesssim_{\sgm}\mu_{q,i}^{2-r}\oM^2\lambda_{q,i}^2\delta_{q,i}\lesssim_{\sgm}\mu_{q,i}^{-r}\varepsilon^4\delta_{q+1}^2,
		\end{align*}
		for sufficiently small $\varepsilon$.
		It is easy to obtain
		\begin{align*}
			\nrm{B_{11}^{(\upsilon)}}_{C_t^r}\leqslant&\Nrm{\sum_{s:I=(s,\upsilon)}\int_{\T^2} \frac{(\pa_{tt}(V_{q,i+1,p}^{I,cof}\Pi_I)\nabla\psi_{L}^{(\upsilon)})_{1}}{i(\lambda_{q,i+1}\mu_{q,i}[I]|f_{i+1}|)^2}\rd x}_{C_{t}^r}
			+\Nrm{\mu_{q,i}\int_{\T^2} (\Psi_{L}^{(\upsilon)}\Delta\Div((\Id+\Gl)(\Sgm_{\Gl}+\Rl+c_{q,i})))_{1}\rd x}_{C_{t}^r}\\
			\lesssim&_{\sgm}\mu_{q,i}^{-r}\varepsilon^4\delta_{q+1}^{\frac{3}{2}},\\
			\nrm{B_{21}^{(\upsilon)}}_{C_t^r}\leqslant&\Nrm{\mu_{q,i}\int_{\T^2} \Psi_{L}^{(\upsilon)}\Delta(\Div R_{m,i})_1\rd x}_{C_{t}^r}\lesssim_{\sgm}\mu_{q,i}^{-r}\varepsilon^4\delta_{q+1}^{\frac{3}{2}}.
		\end{align*}
		Similarly, we could obtain
		\begin{align*}
			\nrm{B_{12}^{(\upsilon)}}_{C_t^r}
			\lesssim_{\sgm}\mu_{q,i}^{-r}\varepsilon^4\delta_{q+1}^{\frac{3}{2}},\quad
			\nrm{B_{22}^{(\upsilon)}}_{C_t^r}\lesssim_{\sgm}\mu_{q,i}^{-r}\varepsilon^4\delta_{q+1}^{\frac{3}{2}}.
		\end{align*}
		Up to now, we could give the following estimates for $B^{(\upsilon)}(t)$ and $\tB^{(\upsilon)}(t)$:
		\begin{align}
			\nrm{B^{(\upsilon)}}_{C_t^r}\lesssim_{\sgm}\mu_{q,i}^{-r}\varepsilon^4\delta_{q+1}^{\frac{3}{2}},\quad\nrm{\tB^{(\upsilon)}}_{C_t^r}\lesssim_{\sgm}\mu_{q,i}^{-r}\tau_{q,i}^r\varepsilon^4\delta_{q+1}^{\frac{3}{2}},\quad 0\leqslant r\leqslant 5.\label{est on B and tB}
		\end{align}
		
		\item Estimates on $E^{(\upsilon)}(t)$ and $\tE^{(\upsilon)}(t)$.
		We could first calculate
		\begin{align*}
			E_1^{(\upsilon)}=&-\int_{\T^2} (\pa_{tt}\tV_{q,i+1,p}^{(\upsilon)}(\Id+\Gl+\sum_{\nrm{\tup-\upsilon}\leqslant1}\tG_{q,i+1,p}^{(\tup)})^{\top})_{12-21}\rd x\\
			&+\int_{\T^2} (\tG_{q,i+1,p}^{(\upsilon)}(\Sgm_{\Gl}+\Rl+c_{q,i})(\Id+\Gl+\sum_{\nrm{\tup-\upsilon}=1}\tG_{q,i+1,p}^{(\tup)})^{\top})_{12-21}\rd x\\
			&+\varepsilon\delta_{q+1}\int_{\T^2} ((\pa_{tt}\tV_{q,i+1,p}^{(\upsilon)}+(\Id+\Gl)(\Sgm_{\Gl}+\Rl+c_{q,i}))\nabla(\psi_{M}^{(\upsilon)}+\psi_{L}^{(\upsilon)}))_{2}\rd x\\
			=&-\sum_{\nrm{\tup-\upsilon}\leqslant1}\int_{\T^2} (\pa_{tt}\tV_{q,i+1,p}^{(\upsilon)}(\tG_{q,i+1,p}^{(\tup)})^{\top})_{12-21}\rd x
			-\sum_{s:I=(s,\upsilon)}\int_{\T^2} \frac{(\pa_{tt}(V_{q,i+1,p}^{I,cof}\Pi_I)(\Id+\Gl)^{\top})_{12-21}}{i(\lambda_{q,i+1}[I]|f_{i+1}|)^2}\rd x\\
			&+\sum_{s:I=(s,\upsilon)}\int_{\T^2} \frac{( V_{q,i+1,p}^{I,cof}(\nabla\Div((\Id+\Gl)(\Sgm_{\Gl}+\Rl+c_{q,i})))^{\top})_{12-21}\Pi_{I}}{i(\lambda_{q,i+1}[I]|f_{i+1}|)^2}\rd x\\
			&+\int_{\T^2} (\tG_{q,i+1,p}^{(\upsilon)}(\Sgm_{\Gl}+\Rl+c_{q,i})(\sum_{\nrm{\tup-\upsilon}=1}\tG_{q,i+1,p}^{(\tup)})^{\top})_{12-21}\rd x
			+\varepsilon\delta_{q+1}
			\sum_{s:I=(s,\upsilon)}\int_{\T^2} \frac{(\pa_{tt}(V_{q,i+1,p}^{I,cof}\Pi_I)\nabla(\psi_{M}^{(\upsilon)}+\psi_{L}^{(\upsilon)}))_{2}}{i(\lambda_{q,i+1}[I]|f_{i+1}|)^2}\rd x\\
			&+\varepsilon\mu_{q,i}^2\delta_{q+1}\int_{\T^2} (\nabla\Div((\Id+\Gl)(\Sgm_{\Gl}+\Rl+c_{q,i}))\nabla(\Psi_{M}^{(\upsilon)}+\Psi_{L}^{(\upsilon)}))_{2}\rd x,\\
			E_2^{(\upsilon)}=&-\int_{\T^2}  (\chi_\upsilon^2(\mu_{q,i}^{-1}x)R_{m,i}(\Id+\Gl^{\top}))_{12-21}\rd x-\int_{\T^2} \sum_{\|\tup-\upsilon\|\leqslant 1}\tu_{q,i+1,p}^{(\tup)}\times\Div (\chi_\upsilon^2(\mu_{q,i}^{-1}x)R_{m,i})\rd x\\
			&-\varepsilon\mu_{q,i}\delta_{q+1}\int_{\T^2} (\psi_{M}^{(\upsilon)}+\psi_{L}^{(\upsilon)})(\Div R_{m,i})_2\rd x\\
			=&-\int_{\T^2}  (\chi_\upsilon^2(\mu_{q,i}^{-1}x)R_{m,i}(\Id+\Gl^{\top}))_{12-21}\rd x+\varepsilon\mu_{q,i}^2\delta_{q+1}\int_{\T^2} \nabla(\Psi_{M}^{(\upsilon)}+\Psi_{L}^{(\upsilon)})\cdot\nabla(\Div R_{m,i})_2\rd x\\
			&+\sum_{\|\tup-\upsilon\|\leqslant 1}\int_{\T^2} (\tV_{q,i+1,p}^{(\tup)}(\nabla\Div (\chi_\upsilon^2(\mu_{q,i}^{-1}x)R_{m,i}))^{\top})_{12-21}\rd x\\
			=&-\int_{\T^2}  (\chi_\upsilon^2(\mu_{q,i}^{-1}x)R_{m,i}(\Id+\Gl^{\top}))_{12-21}\rd x+\varepsilon\mu_{q,i}^2\delta_{q+1}\int_{\T^2} \nabla(\Psi_{M}^{(\upsilon)}+\Psi_{L}^{(\upsilon)})\cdot\nabla(\Div R_{m,i})_2\rd x\\
			&+\sum_{\|\tup-\upsilon\|\leqslant 1}\sum_{s:I=(s,\tup)}\int_{\T^2} \frac{( V_{q,i+1,p}^{I,cof}(\nabla\Div (\chi_\upsilon^2(\mu_{q,i}^{-1}x)R_{m,i}))^{\top})_{12-21}\Pi_{I}}{i(\lambda_{q,i+1}[I]|f_{i+1}|)^2}\rd x,
		\end{align*}
		and then, by using  \eqref{pp of Lambda1}, \eqref{pp of supp tu_LMC 1},  \eqref{pp of Lambda2}, and Lemma \ref{est on int operator},  we could similarly obtain
		\begin{align*}
			&\Nrm{\sum_{s:I=(s,\upsilon)}\int_{\T^2} \frac{(\pa_{tt}(V_{q,i+1,p}^{I,cof}\Pi_I)(\Id+\Gl)^{\top})_{12-21}}{i(\lambda_{q,i+1}[I]|f_{i+1}|)^2}\rd x}_{C_{t}^r}\nonumber\\
			\lesssim&\mu_{q,i}^2\lambda_{q,i+1}^{r}\left(\frac{\nrm{\pa_{tt} V_{q,i+1,p}^{I,cof}(\Id+\Gl)^{\top}}_{10}}{\lambda_{q,i+1}^2(\lambda_{q,i+1}[I]|f_{i+1}|)^{10}}+\frac{\nrm{\pt(\pt c_{A_I^\upsilon} V_{q,i+1,p}^{I,cof})(\Id+\Gl)^{\top}}_{10}}{\lambda_{q,i+1}(\lambda_{q,i+1}[I]|f_{i+1}|)^{10}}+\frac{\nrm{(\pt c_{A_I^\upsilon})^2 V_{q,i+1,p}^{I,cof}(\Id+\Gl)^{\top}}_{10}}{(\lambda_{q,i+1}[I]|f_{i+1}|)^{10}}\right)\\
			\lesssim&(\lambda_{q,i+1}\mu_{q,i})^{-10}\mu_{q,i}^2\lambda_{q,i+1}^{r}\delta_{q+1}^{\frac{1}{2}}\lesssim\mu_{q,i}^{-r}\varepsilon^3\mu_{q,i}^2\delta_{q+1}^2,\\
			&\Nrm{\sum_{s:I=(s,\upsilon)}\int_{\T^2} \frac{( V_{q,i+1,p}^{I,cof}(\nabla\Div((\Id+\Gl)(\Sgm_{\Gl}+\Rl+c_{q,i})))^{\top})_{12-21}\Pi_{I}}{i(\lambda_{q,i+1}[I]|f_{i+1}|)^2}\rd x}_{C_{t}^r}\nonumber\\
			\lesssim&\mu_{q,i}^2\lambda_{q,i+1}^{r}\frac{\nrm{ V_{q,i+1,p}^{I,cof}(\nabla\Div((\Id+\Gl)(\Sgm_{\Gl}+\Rl+c_{q,i})))^{\top}}_{10}}{\lambda_{q,i+1}^2(\lambda_{q,i+1}[I]|f_{i+1}|)^{10}}\\
			\lesssim&_{\sgm}(\lambda_{q,i+1}\mu_{q,i})^{-10}\mu_{q,i}^2\lambda_{q,i+1}^{r-2}\oM\lambda_{q,i}^2\delta_{q,i}^{\frac{1}{2}}\delta_{q+1}^{\frac{1}{2}}\lesssim\mu_{q,i}^{-r}\varepsilon^3\mu_{q,i}^2\delta_{q+1}^2,\\
			&\Nrm{\varepsilon\delta_{q+1}
				\sum_{s:I=(s,\upsilon)}\int_{\T^2} \frac{(\pa_{tt}(V_{q,i+1,p}^{I,cof}\Pi_I)\nabla(\psi_{M}^{(\upsilon)}+\psi_{L}^{(\upsilon)}))_{2}}{i(\lambda_{q,i+1}[I]|f_{i+1}|)^2}\rd x}_{C_t^r}\\
			\lesssim&\mu_{q,i}^2\lambda_{q,i+1}^{r}\varepsilon\delta_{q+1}\left(\frac{\nrm{\pa_{tt} V_{q,i+1,p}^{I,cof}\nabla(\psi_{M}^{(\upsilon)}+\psi_{L}^{(\upsilon)})}_{10}}{\lambda_{q,i+1}^2(\lambda_{q,i+1}[I]|f_{i+1}|)^{10}}+\frac{\nrm{\pt(\pt c_{A_I^\upsilon} V_{q,i+1,p}^{I,cof})\nabla(\psi_{M}^{(\upsilon)}+\psi_{L}^{(\upsilon)})}_{10}}{\lambda_{q,i+1}(\lambda_{q,i+1}[I]|f_{i+1}|)^{10}}\right)\\
			&+\mu_{q,i}^2\lambda_{q,i+1}^{r}\varepsilon\delta_{q+1}\frac{\nrm{(\pt c_{A_I^\upsilon})^2 V_{q,i+1,p}^{I,cof}\nabla(\psi_{M}^{(\upsilon)}+\psi_{L}^{(\upsilon)})}_{10}}{(\lambda_{q,i+1}[I]|f_{i+1}|)^{10}}\\
			\lesssim&(\lambda_{q,i+1}\mu_{q,i})^{-10}\mu_{q,i}^2\lambda_{q,i+1}^{r}\varepsilon\delta_{q+1}^{\frac{3}{2}}\lesssim\mu_{q,i}^{-r}\varepsilon^3\mu_{q,i}^2\delta_{q+1}^2,\\
			&\Nrm{\varepsilon\mu_{q,i}^2\delta_{q+1}\int_{\T^2} (\nabla\Div((\Id+\Gl)(\Sgm_{\Gl}+\Rl+c_{q,i}))\nabla(\Psi_{M}^{(\upsilon)}+\Psi_{L}^{(\upsilon)}))_{2}\rd x}_{C_t^r}\\
			\lesssim&\varepsilon\mu_{q,i}^4\delta_{q+1}\Nrm{\nabla\Div((\Id+\Gl)(\Sgm_{\Gl}+\Rl+c_{q,i}))\nabla(\Psi_{M}^{(\upsilon)}+\Psi_{L}^{(\upsilon)})}_{C_t^r}\\
			\lesssim&_{\sgm}\varepsilon\mu_{q,i}^{4-r}\oM\lambda_{q,i}^2\delta_{q,i}^{\frac{1}{2}}\delta_{q+1}\lesssim_{\sgm}\mu_{q,i}^{-r}\varepsilon^3\mu_{q,i}^{2}\delta_{q+1}^2,\\
			&\Nrm{\int_{\T^2}  (\chi_\upsilon^2(\mu_{q,i}^{-1}x)R_{m,i}(\Id+\Gl^{\top}))_{12-21}\rd x}_{C_t^r}\\
			\lesssim&\mu_{q,i}^2\Nrm{ \pt^r(\chi_\upsilon^2(\mu_{q,i}^{-1}x)R_{m,i}(\Id+\Gl^{\top}))}_{0}\lesssim_{\sgm}\mu_{q,i}^{2-r}\ell_{q,i}^2\oM^2\lambda_{q,i}^2\delta_{q,i}\lesssim_{\sgm}\mu_{q,i}^{-r}\varepsilon^3\mu_{q,i}^2\delta_{q+1}^2,\\
			&\Nrm{\varepsilon\mu_{q,i}^2\delta_{q+1}\int_{\T^2} \nabla(\Psi_{M}^{(\upsilon)}+\Psi_{L}^{(\upsilon)})\cdot\nabla(\Div R_{m,i})_2\rd x}_{C_t^r}\\
			\lesssim&\varepsilon\mu_{q,i}^4\delta_{q+1}\Nrm{\pt^r( \nabla(\Psi_{M}^{(\upsilon)}+\Psi_{L}^{(\upsilon)})\cdot\nabla(\Div R_{m,i})_2)}_{0}\lesssim_{\sgm}\varepsilon\mu_{q,i}^{4-r}\delta_{q+1}\oM^2\lambda_{q,i}^2\delta_{q,i}\lesssim_{\sgm}\mu_{q,i}^{-r}\varepsilon^3\mu_{q,i}^{2}\delta_{q+1}^{2},\\
			&\Nrm{\sum_{\|\tup-\upsilon\|\leqslant 1}\sum_{s:I=(s,\tup)}\int_{\T^2} \frac{( V_{q,i+1,p}^{I,cof}(\nabla\Div (\chi_\upsilon^2(\mu_{q,i}^{-1}x)R_{m,i}))^{\top})_{12-21}\Pi_{I}}{i(\lambda_{q,i+1}[I]|f_{i+1}|)^2}\rd x}_{C_t^r}\\
			\lesssim&\mu_{q,i}^2\lambda_{q,i+1}^{r}\frac{\nrm{( V_{q,i+1,p}^{I,cof}\nabla\Div (\chi_\upsilon^2(\mu_{q,i}^{-1}x)R_{m,i}))^{\top}}_{10}}{\lambda_{q,i+1}^2(\lambda_{q,i+1}[I]|f_{i+1}|)^{10}}
			\lesssim_{\sgm}(\lambda_{q,i+1}\mu_{q,i})^{-10}\mu_{q,i}^2\lambda_{q,i+1}^{r-2}\oM^2\lambda_{q,i}^2\delta_{q,i}\delta_{q+1}^{\frac{1}{2}}\lesssim_{\sgm}\mu_{q,i}^{-r}\varepsilon^3\mu_{q,i}^2\delta_{q+1}^2,
		\end{align*}
		for sufficiently small $\varepsilon$. Noting 
		\begin{align*}
			&(\tG_{q,i+1,p}^{(\upsilon)}(\Sgm_{\Gl}+\Rl+c_{q,i})(\sum_{\nrm{\tup-\upsilon}=1}\tG_{q,i+1,p}^{(\tup)})^{\top})_{12-21}\\=&-\sum_{s:I=(s,\upsilon)}
		\sum_{\substack{I'=(s',\tup)\in\mathscr I\\
		\nrm{\tup-\upsilon}=1}}\gamma_{q,i+1}^I\gamma_{q,i+1}^{I'}((\tilde{f}_{A_I^{\upsilon}}\otimes f_{i+1})(\Sgm_{\Gl}+\Rl+c_{q,i})(f_{i+1}\otimes \tilde{f}_{A_{I'}^{\tup}}))_{12-21}\Pi_{I,I'}\\
			&+\sum_{s:I=(s,\upsilon)}
		\sum_{\substack{I'=(s',\tup)\in\mathscr I\\
		\nrm{\tup-\upsilon}=1}}\frac{i\gamma_{q,i+1}^{I'}}{\lambda_{q,i+1}[I]}((\tilde{f}_{A_I^{\upsilon}}\otimes \nabla \gamma_{q,i+1,c}^I)(\Sgm_{\Gl}+\Rl+c_{q,i})(f_{i+1}\otimes \tilde{f}_{A_{I'}^{\tup}}))_{12-21}\Pi_{I,I'}\\
			&+\sum_{s:I=(s,\upsilon)}
		\sum_{\substack{I'=(s',\tup)\in\mathscr I\\
		\nrm{\tup-\upsilon}=1}}\frac{i\gamma_{q,i+1}^{I'}}{\lambda_{q,i+1}[I]}((\tilde{f}_{A_I^{\upsilon}}\otimes \gamma_{q,i+1,s}^I)(\Sgm_{\Gl}+\Rl+c_{q,i})(f_{i+1}\otimes \tilde{f}_{A_{I'}^{\tup}}))_{12-21}\Pi_{I'}\tPi_{I}\\
			&+\sum_{s:I=(s,\upsilon)}
		\sum_{\substack{I'=(s',\tup)\in\mathscr I\\
		\nrm{\tup-\upsilon}=1}}\frac{i\gamma_{q,i+1}^{I}}{\lambda_{q,i+1}[I']}((\tilde{f}_{A_I^{\upsilon}}\otimes f_{i+1})(\Sgm_{\Gl}+\Rl+c_{q,i})(\nabla\gamma_{q,i+1,c}^{I'}\otimes \tilde{f}_{A_{I'}^{\tup}}))_{12-21}\Pi_{I,I'}\\
			&+\sum_{s:I=(s,\upsilon)}
		\sum_{\substack{I'=(s',\tup)\in\mathscr I\\
		\nrm{\tup-\upsilon}=1}}\frac{i\gamma_{q,i+1}^{I}}{\lambda_{q,i+1}[I']}((\tilde{f}_{A_I^{\upsilon}}\otimes f_{i+1})(\Sgm_{\Gl}+\Rl+c_{q,i})(\gamma_{q,i+1,s}^{I'}\otimes \tilde{f}_{A_{I'}^{\tup}}))_{12-21}\Pi_{I}\tPi_{I'}\\
			&+\sum_{s:I=(s,\upsilon)}
		\sum_{\substack{I'=(s',\tup)\in\mathscr I\\
		\nrm{\tup-\upsilon}=1}}\frac{1}{\lambda_{q,i+1}^2[I][I']}((\tilde{f}_{A_I^{\upsilon}}\otimes \nabla\gamma_{q,i+1,c}^I)(\Sgm_{\Gl}+\Rl+c_{q,i})(\nabla\gamma_{q,i+1,c}^{I'}\otimes \tilde{f}_{A_{I'}^{\tup}}))_{12-21}\Pi_{I,I'}\\
			&+\sum_{s:I=(s,\upsilon)}
		\sum_{\substack{I'=(s',\tup)\in\mathscr I\\
		\nrm{\tup-\upsilon}=1}}\frac{1}{\lambda_{q,i+1}^2[I][I']}((\tilde{f}_{A_I^{\upsilon}}\otimes \nabla\gamma_{q,i+1,c}^I)(\Sgm_{\Gl}+\Rl+c_{q,i})(\gamma_{q,i+1,s}^{I'}\otimes \tilde{f}_{A_{I'}^{\tup}}))_{12-21}\Pi_{I}\tPi_{I'}\\
			&+\sum_{s:I=(s,\upsilon)}
		\sum_{\substack{I'=(s',\tup)\in\mathscr I\\
		\nrm{\tup-\upsilon}=1}}\frac{1}{\lambda_{q,i+1}^2[I][I']}((\tilde{f}_{A_I^{\upsilon}}\otimes \gamma_{q,i+1,s}^{I})(\Sgm_{\Gl}+\Rl+c_{q,i})(\nabla\gamma_{q,i+1,c}^{I'}\otimes \tilde{f}_{A_{I'}^{\tup}}))_{12-21}\tPi_{I}\Pi_{I'}\\
			&+\sum_{s:I=(s,\upsilon)}
		\sum_{\substack{I'=(s',\tup)\in\mathscr I\\
		\nrm{\tup-\upsilon}=1}}\frac{1}{\lambda_{q,i+1}^2[I][I']}((\tilde{f}_{A_I^{\upsilon}}\otimes \gamma_{q,i+1,s}^I)(\Sgm_{\Gl}+\Rl+c_{q,i})(\gamma_{q,i+1,s}^{I'}\otimes \tilde{f}_{A_{I'}^{\tup}}))_{12-21}\tPi_{I}\tPi_{I'},
		\end{align*}
		due to the definition of $[I]$ in \eqref{def of [I]} and \eqref{pp of supp cof}, we should only consider $I = (s, \upsilon)$ and $I' = (s', \tup)$ with $\nrm{I - I'} \leqslant 1$, which can be divided into three cases:
		\begin{enumerate}
			\item $\nrm{\upsilon - \tup} = 1$, which is equivalent to $|s' - s| \leqslant 1$ and $\nrm{\upsilon - \tup} = 1$. 
			\begin{align}
				|[I]-[I']|=\left|[s]+\sum_{i=1}^22^i[\upsilon_i]-[s']-\sum_{i=1}^22^i[\tup_i]\right|\geqslant\left|\sum_{i=1}^22^i[\upsilon_i]-\sum_{i=1}^22^i[\tup_i]\right|-\left|[s]-[s']\right|\geqslant 1.\label{case 1}
			\end{align} 
			\item $\nrm{\upsilon - \tup} = 0,|s' - s|=1$. 
			\begin{align}
				|[I]-[I']|=\left|[s]-[s']\right|= 1.\label{case 2}
			\end{align} 
			\item $\nrm{I - I'} = 0$.
			\begin{align}
				[I]=[I'].\label{case 3}
			\end{align}
		\end{enumerate}
		Then, if $\nrm{I - I'} = 1$, terms like $\Pi_{I}\Pi_{I'}$, $\Pi_{I}\tPi_{I'}$, and $\tPi_{I}\tPi_{I'}$ do not produce low-frequency components, and their frequencies satisfy $\gtrsim \lambda_{q,i+1}$ i.e,  $\Pi_{I}\Pi_{I'}=\mPG\Pi_{I}\Pi_{I'}$.
		\begin{Rmk}
			The ones we commonly use are
			\begin{align*}
				&\mPO(\Pi_I\tPi_{I'})=\mPO(\Pi_{I,I'})=\mPO(\tPi_I\tPi_{I'})=0, 
				&& \nrm{I-I'}= 1,\\
				&\mPO(\Pi_I\tPi_{I'})=0,\quad\mPO(\Pi_{I,I'})=-2,\quad\mPO(\tPi_I\tPi_{I'})=2,&&I=I'.
			\end{align*}
		\end{Rmk}

		Then, by using  \eqref{pp of Lambda1}, \eqref{pp of supp tu_LMC 1},  \eqref{pp of Lambda2}, and Lemma \ref{est on int operator},  we could similarly obtain
		\begin{align*}
			\Nrm{\int_{\T^2} (\tG_{q,i+1,p}^{(\upsilon)}(\Sgm_{\Gl}+\Rl+c_{q,i})(\sum_{\nrm{\tup-\upsilon}=1}\tG_{q,i+1,p}^{(\tup)})^{\top})_{12-21}\rd x}_{C_t^r}
			\lesssim&_{\sgm}(\lambda_{q,i+1}\mu_{q,i})^{-10}\lambda_{q,i+1}^{r}\varepsilon\mu_{q,i}^2\delta_{q+1}\\
			\lesssim&_{\sgm}\mu_{q,i}^{-r}\varepsilon^3\mu_{q,i}^2\delta_{q+1}^2.
		\end{align*}
		Noting
		\begin{align*}
			&\sum_{\nrm{\tup-\upsilon}\leqslant1}(\pa_{tt}\tV_{q,i+1,p}^{(\upsilon)}(\tG_{q,i+1,p}^{(\tup)})^{\top})_{12-21}\\
			=&\sum_{s:I=(s,\upsilon)}\sum_{\nrm{\tup-\upsilon}\leqslant1}\frac{1}{i\lambda_{q,i+1}^2[I]^2|f_{i+1}|^2}(\pa_{tt}(\gamma_{q,i+1}^I\Pi_I\tilde{f}_{A_I^\upsilon})\otimes( \tG_{q,i+1,p}^{(\tup)}f_{i+1}))_{12-21}\\
			=&\sum_{s:I=(s,\upsilon)}
			\sum_{\substack{I'=(s',\tup)\in\mathscr I\\
			\nrm{\tup-\upsilon}\leqslant1}}\left(\frac{\gamma_{q,i+1}^{I'}\Pi_{I'}}{\lambda_{q,i+1}^2[I]^2}+\frac{f_{i+1}\cdot\gamma_{q,i+1,s}^{I'}\tPi_{I'}}{i\lambda_{q,i+1}^3[I]^2[I']|f_{i+1}|^2}+\frac{f_{i+1}\cdot\nb \gamma_{q,i+1,c}^{I'}\Pi_{I'}}{i\lambda_{q,i+1}^3[I]^2[I']|f_{i+1}|^2}\right)(\pa_{tt}(\gamma_{q,i+1}^I\Pi_I\tilde{f}_{A_I^\upsilon})\otimes \tilde{f}_{A_{I'}^{\tup}})_{12-21},
		\end{align*}
		we focus exclusively on the case $I=I'$, which leads to the production of low-frequency components. The analysis for cases that do not produce low-frequency components is similar to the preceding calculations. We could calculate
		\begin{align*}
			&\sum_{s:I=(s,\upsilon)}(\pa_{tt}\tV_{q,i+1,p}^{(\upsilon)}(\tG_{q,i+1,p}^{(\upsilon)})^{\top})_{12-21}\\
			=&\sum_{s:I=(s,\upsilon)}\left(\frac{\gamma_{q,i+1}^{I}\Pi_{I}}{\lambda_{q,i+1}^2[I]^2}+\frac{f_{i+1}\cdot\gamma_{q,i+1,s}^{I}\tPi_{I}}{i\lambda_{q,i+1}^3[I]^3|f_{i+1}|^2}+\frac{f_{i+1}\cdot\nb \gamma_{q,i+1,c}^{I}\Pi_{I}}{i\lambda_{q,i+1}^3[I]^3|f_{i+1}|^2}\right)(2\pa_{t}(\gamma_{q,i+1}^I\Pi_I)\pa_{t}a_{A_{I}^{\upsilon}}+\gamma_{q,i+1}^I\Pi_I\pa_{t}^2a_{A_{I}^{\upsilon}})(f_{i+1}^{\perp}\otimes f_{i+1})_{12-21},
		\end{align*}
		and
		\begin{align*}
			&\left(\frac{\gamma_{q,i+1}^{I}\Pi_{I}}{\lambda_{q,i+1}^2[I]^2}+\frac{f_{i+1}\cdot\gamma_{q,i+1,s}^{I}\tPi_{I}}{i\lambda_{q,i+1}^3[I]^3|f_{i+1}|^2}+\frac{f_{i+1}\cdot\nb \gamma_{q,i+1,c}^{I}\Pi_{I}}{i\lambda_{q,i+1}^3[I]^3|f_{i+1}|^2}\right)(2\pa_{t}(\gamma_{q,i+1}^I\Pi_I)\pa_{t}a_{A_{I}^{\upsilon}}+\gamma_{q,i+1}^I\Pi_I\pa_{t}^2a_{A_{I}^{\upsilon}})\\
			=&\left(\frac{\gamma_{q,i+1}^{I}\Pi_{I}}{\lambda_{q,i+1}^2[I]^2}+\frac{f_{i+1}\cdot\gamma_{q,i+1,s}^{I}\tPi_{I}}{i\lambda_{q,i+1}^3[I]^3|f_{i+1}|^2}+\frac{f_{i+1}\cdot\nb \gamma_{q,i+1,c}^{I}\Pi_{I}}{i\lambda_{q,i+1}^3[I]^3|f_{i+1}|^2}\right)(2\pa_{t}\gamma_{q,i+1}^I\Pi_I\pa_{t}a_{A_{I}^{\upsilon}}+i2\lambda_{q,i+1}[I]\gamma_{q,i+1}^I\tPi_I\pa_{t}a_{A_{I}^{\upsilon}}+\gamma_{q,i+1}^I\Pi_I\pa_{t}^2a_{A_{I}^{\upsilon}})\\
			=&\frac{4f_{i+1}\cdot \gamma_{q,i+1,s}^{I}\gamma_{q,i+1}^I\pa_{t}a_{A_{I}^{\upsilon}}}{\lambda_{q,i+1}^2[I]^2|f_{i+1}|^2}-\left(\frac{\gamma_{q,i+1}^{I}}{\lambda_{q,i+1}^2[I]^2}+\frac{f_{i+1}\cdot\nabla\gamma_{q,i+1,c}^{I}}{i\lambda_{q,i+1}^3[I]^3|f_{i+1}|^2}\right)(4\pa_{t}\gamma_{q,i+1}^I\pa_{t}a_{A_{I}^{\upsilon}}+2\gamma_{q,i+1}^I\pa_{t}^2a_{A_{I}^{\upsilon}})
			\\
			&+\left(\frac{\gamma_{q,i+1}^{I}}{\lambda_{q,i+1}^2[I]^2}+\frac{f_{i+1}\cdot\nabla\gamma_{q,i+1,c}^{I}}{i\lambda_{q,i+1}^3[I]^3|f_{i+1}|^2}\right)(2\pa_{t}\gamma_{q,i+1}^I\pa_{t}a_{A_{I}^{\upsilon}}+\gamma_{q,i+1}^I\pa_{t}^2a_{A_{I}^{\upsilon}})(\Pi_{I}^2+2)
			+\frac{2f_{i+1}\cdot \gamma_{q,i+1,s}^{I}\gamma_{q,i+1}^I\pa_{t}a_{A_{I}^{\upsilon}}}{\lambda_{q,i+1}^2[I]^2|f_{i+1}|^2}(\tPi_{I}^2-2)\\
			&+\left(\frac{2i\gamma_{q,i+1}^{I}}{\lambda_{q,i+1}[I]}+\frac{2f_{i+1}\cdot\nb\gamma_{q,i+1,c}^{I}}{\lambda_{q,i+1}^2[I]^2|f_{i+1}|^2}\right)(\gamma_{q,i+1}^I\pa_{t}a_{A_{I}^{\upsilon}})\Pi_{I}\tPi_I
			+\frac{f_{i+1}\cdot \gamma_{q,i+1,s}^{I}}{i\lambda_{q,i+1}^3[I]^3|f_{i+1}|^2}(2\pa_{t}\gamma_{q,i+1}^I\pa_{t}a_{A_{I}^{\upsilon}}+\gamma_{q,i+1}^I\pa_{t}^2a_{A_{I}^{\upsilon}})\Pi_{I}\tPi_{I}.
		\end{align*}
		Among these, $\Pi_{I}\tPi_{I}$, $(\tPi_{I}^2 - 2)$, and $(\Pi_{I}^2 + 2)$ do not contain low-frequency components. Therefore, the key lies in the estimation of the first term.
		\begin{align*}
			&\Nrm{\int_{\T^2}\frac{4f_{i+1}\cdot \gamma_{q,i+1,s}^{I}\gamma_{q,i+1}^I\pa_{t}a_{A_{I}^{\upsilon}}}{\lambda_{q,i+1}^2[I]^2|f_{i+1}|^2}-\left(\frac{\gamma_{q,i+1}^{I}}{\lambda_{q,i+1}^2[I]^2}+\frac{f_{i+1}\cdot\nb\gamma_{q,i+1,c}^{I}}{i\lambda_{q,i+1}^3[I]^3|f_{i+1}|^2}\right)(4\pa_{t}\gamma_{q,i+1}^I\pa_{t}a_{A_{I}^{\upsilon}}+2\gamma_{q,i+1}^I\pa_{t}^2a_{A_{I}^{\upsilon}})\rd x}_{C_t^r}\\
			\lesssim&\mu_{q,i}^2\lambda_{q,i+1}^{-2}\sum_{r_1+r_2=r}\nrm{\pt^{r_1} \gamma_{q,i+1,s}^I}_0\nrm{\pt^{r_2}(\gamma_{q,i+1}^I\pt a_{A_{I}^{\upsilon}})}_0\\
			&+\mu_{q,i}^2\sum_{r_1+r_2=r}(\lambda_{q,i+1}^{-2}\nrm{\pt^{r_1}\gamma_{q,i+1}^I}_0+\lambda_{q,i+1}^{-3}\nrm{\pt^{r_1}\nb \gamma_{q,i+1,c}^I}_0)(\nrm{\pt^{r_2}(\pt\gamma_{q,i+1}^I\pt a_{A_{I}^{\upsilon}})}_0+\nrm{\pt^{r_2}(\gamma_{q,i+1}^I\pt^2 a_{A_{I}^{\upsilon}})}_0)\\
			\lesssim&_{\sgm}\mu_{q,i}^2\tau_{q,i}^{-r}(\lambda_{q,i+1}\mu_{q,i})^{-2}\delta_{q,i}^{\frac{1}{2}}\delta_{q+1}\lesssim_{\sgm}\mu_{q,i}^{-r}\varepsilon^3\mu_{q,i}^2\delta_{q+1}^2.
		\end{align*}
		In summary, we obtain
		\begin{align*}
			\Nrm{\int_{\T^2} \sum_{\nrm{\tup-\upsilon}\leqslant1}(\pa_{tt}\tV_{q,i+1,p}^{(\upsilon)}(\tG_{q,i+1,p}^{(\tup)})^{\top})_{12-21}\rd x}_{C_t^r}
			\lesssim_{\sgm}\mu_{q,i}^{-r}\varepsilon^3\mu_{q,i}^2\delta_{q+1}^2.	
		\end{align*}
		Then, we could get
		\begin{align*}
			\nrm{E_{1}^{(\upsilon)}}_{C_t^r}
			\lesssim_{\sgm}\mu_{q,i}^{-r}\varepsilon^3\mu_{q,i}^2\delta_{q+1}^2,\quad
			\nrm{E_{2}^{(\upsilon)}}_{C_t^r}\lesssim_{\sgm}\mu_{q,i}^{-r}\varepsilon^3\mu_{q,i}^2\delta_{q+1}^2.
		\end{align*}
		Up to now, we could give the following estimates for $E^{(\upsilon)}(t)$ and $\tE^{(\upsilon)}(t)$:
		\begin{align}
			\nrm{E^{(\upsilon)}}_{C_t^r}\lesssim_{\sgm}\mu_{q,i}^{-r}\varepsilon^3\mu_{q,i}^2\delta_{q+1}^2,\quad\nrm{\tE^{(\upsilon)}}_{C_t^r}\lesssim_{\sgm}\mu_{q,i}^{-r}\tau_{q,i}^r\varepsilon^3\mu_{q,i}^2\delta_{q+1}^2,\quad 0\leqslant r\leqslant 5.\label{est on E and tE}
		\end{align}
	\end{enumerate}
	
	Next, we could give the estimates for $\tau_{q,i}^2(\tH^{(\upsilon)})^{-1}\tB^{(\upsilon)}$ and  $\tau_{q,i}^2(\tH^{(\upsilon)})^{-1}\tE^{(\upsilon)}$:
	\begin{align*}
		\nrm{\tau_{q,i}^2(\tH^{(\upsilon)})^{-1}\tB^{(\upsilon)}}_{C_t^0}\leqslant&\tau_{q,i}^2\nrm{(\tH^{(\upsilon)})^{-1}}_{C_t^0}\nrm{\tB^{(\upsilon)}}_{C_t^0}\lesssim_{\sgm}\tau_{q,i}^2\varepsilon^{3}\mu_{q,i}^{-2}\delta_{q+1}^{\frac{1}{2}}\lesssim_{\sgm}\varepsilon^3,\\
		\nrm{\tau_{q,i}^2(\tH^{(\upsilon)})^{-1}\tE^{(\upsilon)}}_{C_t^0}\leqslant&\tau_{q,i}^2\nrm{(\tH^{(\upsilon)})^{-1}}_{C_t^0}\nrm{\tE^{(\upsilon)}}_{C_t^0}\lesssim_{\sgm}\tau_{q,i}^2\varepsilon^2\delta_{q+1}\lesssim_{\sgm}\varepsilon^2\mu_{q,i}^2\delta_{q+1},\\
		\nrm{\tau_{q,i}^2(\tH^{(\upsilon)})^{-1}\tB^{(\upsilon)}}_{C_t^r}\lesssim&\tau_{q,i}^2\sum_{r_1+r_2=r}\nrm{(\tH^{(\upsilon)})^{-1}}_{C_t^{r_1}}\nrm{\tB^{(\upsilon)}}_{C_t^{r_2}}\lesssim_{\sgm}\mu_{q,i}^{-r}\tau_{q,i}^r\tau_{q,i}^2\varepsilon^{3}\mu_{q,i}^{-2}\delta_{q+1}^{\frac{1}{2}}\lesssim_{\sgm}\mu_{q,i}^{-r}\tau_{q,i}^r\varepsilon^3,\\
		\nrm{\tau_{q,i}^2(\tH^{(\upsilon)})^{-1}\tE^{(\upsilon)}}_{C_t^r}\lesssim&\tau_{q,i}^2\sum_{r_1+r_2=r}\nrm{(\tH^{(\upsilon)})^{-1}}_{C_t^{r_1}}\nrm{\tE^{(\upsilon)}}_{C_t^{r_2}}\lesssim_{\sgm}\mu_{q,i}^{-r}\tau_{q,i}^r\tau_{q,i}^2\varepsilon^2\delta_{q+1}\lesssim_{\sgm}\mu_{q,i}^{-r}\tau_{q,i}^r\varepsilon^2\mu_{q,i}^2\delta_{q+1}.
	\end{align*}
	
	We could calculate
	\begin{equation*}\left\lbrace
		\begin{aligned}
			&\pa_{tt}\pt^r\tU^{(S_{q,i}^{min},\upsilon)}+\tau_{q,i}^{2}(\tH^{(\upsilon)})^{-1}\tB^{(\upsilon)}\pt^r\tU^{(S_{q,i}^{min},\upsilon)}\\
			&\hspace{50pt}=-\sum_{k=1}^r\binom{r}{k}\left(\pt^k(\tau_{q,i}^{2}(\tH^{(\upsilon)})^{-1}\tB^{(\upsilon)})\pt^{r-k}\tU^{(S_{q,i}^{min},\upsilon)}\right)+\tau_{q,i}^{2}\pt^{r}((\tH^{(\upsilon)})^{-1}\tE^{(\upsilon)}),\quad t\in[S_{q,i}^{min},5/4+S_{q,i}^{min}],\\
			&\pt^r\tU^{(S_{q,i}^{min},\upsilon)}(S_{q,i}^{min})=(0,0)^{\top},\quad\pa_t^{r+1}\tU^{(S_{q,i}^{min},\upsilon)}(S_{q,i}^{min})=(0,0)^{\top},
		\end{aligned}\right.
	\end{equation*}
	and for $S_{q,i}^{min}+1\leqslant S\leqslant S_{q,i}^{max},$
	\begin{align}
		&\pa_{tt}\pt^r\tU^{(S,\upsilon)}+\tau_{q,i}^{2}(\tH^{(\upsilon)})^{-1}\tB^{(\upsilon)}\pt^r\tU^{(S,\upsilon)}\nonumber\\
		=&\left\lbrace
		\begin{aligned}
			&-\sum_{k=1}^r\binom{r}{k}\left(\pt^k(\tau_{q,i}^{2}(\tH^{(\upsilon)})^{-1}\tB^{(\upsilon)})\pt^{r-k}\tU^{(S,\upsilon)}\right)\\
			&\quad+\pt^r\left(\tau_{q,i}^{2}(\tH^{(\upsilon)})^{-1}\tE^{(\upsilon)}(1-\Theta_{S-1})-2\pt\Theta_{S-1}\pt\tU^{(S-1,\upsilon)}-\pa_{tt}\Theta_{S-1}\tU^{(S-1,\upsilon)}\right),&& t\in[S,S+1/4],\\
			&-\sum_{k=1}^r\binom{r}{k}\left(\pt^k(\tau_{q,i}^{2}(\tH^{(\upsilon)})^{-1}\tB^{(\upsilon)})\pt^{r-k}\tU^{(S,\upsilon)}\right)+\pt^r\left(\tau_{q,i}^{2}(\tH^{(\upsilon)})^{-1}\tE^{(\upsilon)}\right),&& t\in(S+1/4,S+5/4].
		\end{aligned}\right.\nonumber
	\end{align}
	
	Next, by using \eqref{ode ypp} in Appendix \ref{Estimate for nonautonomous linear differential systems}, we could obtain by induction
	\begin{align}
		\nrm{\pa_{tt}\tU^{(S_{q,i}^{min},\upsilon)}}_{C_t^0}\lesssim&_{\sgm}\varepsilon^2\mu_{q,i}^2\delta_{q+1}\leqslant\varepsilon\mu_{q,i}^2\delta_{q+1},\\ \nrm{\pa_{tt}\tU^{(S_{q,i}^{min},\upsilon)}}_{C_t^r}\lesssim&\sum_{k=1}^r\nrm{\tau_{q,i}^{2}(\tH^{(\upsilon)})^{-1}\tB^{(\upsilon)}}_{C_t^k}\nrm{\tU^{(S_{q,i}^{min},\upsilon)}}_{C_t^{r-k}}+\nrm{\tau_{q,i}^{2}(\tH^{(\upsilon)})^{-1}\tE^{(\upsilon)}}_{C_t^r}\nonumber\\
		\lesssim&_{\sgm}\mu_{q,i}^{-r}\tau_{q,i}^r\varepsilon^2\mu_{q,i}^2\delta_{q+1},&&1\leqslant r\leqslant 5,\\
		\nrm{\pa_{tt}\tU^{(S,\upsilon)}}_{C_t^0}\leqslant&\nrm{\tau_{q,i}^{2}(\tH^{(\upsilon)})^{-1}\tE^{(\upsilon)}}_{C_t^0}+\nrm{\pt\Theta_{S-1}\pt\tU^{(S-1,\upsilon)}}_{C_t^0}+\nrm{\pa_{tt}\Theta_{S-1}\tU^{(S-1,\upsilon)}}_{C_t^0}\nonumber\\
		\lesssim&_{\sgm}\varepsilon^2\mu_{q,i}^2\delta_{q+1}\leqslant\varepsilon\mu_{q,i}^2\delta_{q+1},\\
		\nrm{\pa_{tt}\tU^{(S,\upsilon)}}_{C_t^r}\lesssim&\sum_{k=1}^r\nrm{\tau_{q,i}^{2}(\tH^{(\upsilon)})^{-1}\tB^{(\upsilon)}}_{C_{t}^{k}}\nrm{\tU^{(S,\upsilon)}}_{C_{t}^{r-k}}+\nrm{\tau_{q,i}^{2}(\tH^{(\upsilon)})^{-1}\tE^{(\upsilon)}}_{C_t^r}\nonumber\\	&+\nrm{\tau_{q,i}^{2}(\tH^{(\upsilon)})^{-1}\tE^{(\upsilon)}(1-\Theta_{S-1})}_{C_t^r}+\nrm{\pt\Theta_{S-1}\pt\tU^{(S-1,\upsilon)}}_{C_t^r}+\nrm{\pa_{tt}\Theta_{S-1}\tU^{(S-1,\upsilon)}}_{C_t^r}\nonumber\\
		\lesssim&_{\sgm}\mu_{q,i}^{-r}\tau_{q,i}^r\varepsilon^2\mu_{q,i}^2\delta_{q+1},&&1\leqslant r\leqslant 5,
	\end{align}
	for sufficiently small $\varepsilon$. 
	And then, uniformly in $\upsilon$, we have for $k=L,M$ and $1\leqslant r\leqslant5$,
	\begin{align*}
		\nrm{g_k^{(\upsilon)}}_{C_t^0}\leqslant&(\nrm{g_{k,2}^{(\upsilon)}}_{C_t^0}+\varepsilon\mu_{q,i}^2\delta_{q+1})\nrm{\theta_{q,i}^*}_{C_t^0}\lesssim\varepsilon\mu_{q,i}^2\delta_{q+1},\\
		\nrm{g_k^{(\upsilon)}}_{C_t^r}\lesssim&\sum_{r_1+r_2=r}\nrm{g_{k,2}^{(\upsilon)}}_{C_t^{r_1}}\nrm{\theta_{q,i}^*}_{C_t^{r_2}}+\varepsilon\mu_{q,i}^2\delta_{q+1}\nrm{\theta_{q,i}^*}_{C_t^r}\\
		\lesssim&\sum_{r_1+r_2=r}\tau_{q,i}^{-r_1}\nrm{\sum_S\Theta_S\tU^{(S,\upsilon)}}_{C_t^{r_1}}\nrm{\theta_{q,i}^*}_{C_t^{r_2}}+\varepsilon\mu_{q,i}^2\delta_{q+1}\nrm{\theta_{q,i}^*}_{C_t^r}\\
		\lesssim_{\sgm}&\mu_{q,i}^{-r}\varepsilon\mu_{q,i}^2\delta_{q+1}.
	\end{align*}
	Using the finite overlap of the supports and the translation invariance of the localized bump functions, we are now ready to estimate $\tV_{q,i+1,ac}$, $\tu_{q,i+1,ac}$, $\tG_{q,i+1,ac}$, and $\tS_{q,i+1,ac}^{(1)}$:
	\begin{align*}
		\nrm{\pt^{r}\tV_{q,i+1,k}}_N\lesssim&\sup_{\upsilon}\nrm{\pt^{r}g_k^{(\upsilon)}}_0\sup_{\upsilon}\nrm{\nb\Psi_k^{(\upsilon)}}_N\lesssim_{\sgm,N,r}\mu_{q,i}^{-N-r}\varepsilon\mu_{q,i}^2\delta_{q+1},&&k=L,M,\\
		\nrm{\pt^{r}\tu_{q,i+1,k}}_N\lesssim&\mu_{q,i}^{-1}\sup_{\upsilon}\nrm{\pt^{r}g_k^{(\upsilon)}}_0\sup_{\upsilon}\nrm{\psi_k^{(\upsilon)}}_N\lesssim_{\sgm,N,r}\mu_{q,i}^{-N-r}\varepsilon\mu_{q,i}\delta_{q+1},&&k=L,M,\\
		\nrm{\pt^{r}\tG_{q,i+1,k}}_N\lesssim&\mu_{q,i}^{-2}\sup_{\upsilon}\nrm{\pt^{r}g_k^{(\upsilon)}}_0\sup_{\upsilon}\nrm{\nb\psi_k^{(\upsilon)}}_N\lesssim_{\sgm,N,r}\mu_{q,i}^{-N-r}\varepsilon\delta_{q+1},&&k=L,M,\\
		\nrm{\pt^{r}\tV_{q,i+1,ac}}_N\leqslant&\nrm{\pt^{r}\tV_{q,i+1,L}}_N+\nrm{\pt^{r}\tV_{q,i+1,M}}_N\lesssim_{\sgm,N,r}\mu_{q,i}^{-N-r}\varepsilon\mu_{q,i}^2\delta_{q+1},\\
		\nrm{\pt^{r}\tu_{q,i+1,ac}}_N\leqslant&\nrm{\pt^{r}\tu_{q,i+1,L}}_N+\nrm{\pt^{r}\tu_{q,i+1,M}}_N\lesssim_{\sgm,N,r}\mu_{q,i}^{-N-r}\varepsilon\mu_{q,i}\delta_{q+1},\\
		\nrm{\pt^{r}\tG_{q,i+1,ac}}_N\leqslant&\nrm{\pt^{r}\tG_{q,i+1,L}}_N+\nrm{\pt^{r}\tG_{q,i+1,M}}_N\lesssim_{\sgm,N,r}\mu_{q,i}^{-N-r}\varepsilon\delta_{q+1},\\
				\nrm{\pt^r\tS_{q,i+1,L}^{(1)}}_N\lesssim&\mu_{q,i}^{-2}\sum_{k=0}^6\sum_{N_1+\cdots+N_{k+1}=N}\sum_{r_1+\cdots+r_{k+1}=r}\prod_{l=1}^k\nrm{\pt^{r_l}\Gl}_{N_l}\sup_{\upsilon}\nrm{\pt^{r_{k+1}}g_L^{(\upsilon)}}_{0}\sup_{\upsilon}\nrm{\nb\psi_{L}^{(\upsilon)}}_{N_{k+1}}\lesssim_{\sgm,N,r}\mu_{q,i}^{-N-r}\varepsilon\delta_{q+1},\\
		\nrm{\pt^r\tS_{q,i+1,M}^{(1)}}_N\lesssim&\mu_{q,i}^{-2}\sum_{k=0}^6\sum_{N_1+\cdots+N_{k+1}=N}\sum_{r_1+\cdots+r_{k+1}=r}\prod_{l=1}^k\nrm{\pt^{r_l}\Gl}_{N_l}\sup_{\upsilon}\nrm{\pt^{r_{k+1}}g_M^{(\upsilon)}}_{0}\sup_{\upsilon}\nrm{\nb\psi_{M}^{(\upsilon)}}_{N_{k+1}}\lesssim_{\sgm,N,r}\mu_{q,i}^{-N-r}\varepsilon\delta_{q+1},\\
		\nrm{\pt^r\tS_{q,i+1,ac}^{(1)}}_N\leqslant&\nrm{\pt^r\tS_{q,i+1,L}^{(1)}}_N+\nrm{\pt^r\tS_{q,i+1,M}^{(1)}}_N\lesssim_{\sgm,N,r}\mu_{q,i}^{-N-r}\varepsilon\delta_{q+1}.
	\end{align*}
	Finally, we choose $\varepsilon_3^*>0$ to be no larger than $\varepsilon_{31}^*$ and all the additional smallness thresholds used in the proof. Then all the preceding estimates hold for every $0<\varepsilon<\varepsilon_3^*$.
	At this point, \eqref{est on tVuG L}--\eqref{est on tS_L} have been proved. Using \eqref{est on tu_p}, \eqref{est on tG_p}, and \eqref{est on tVuG L}--\eqref{est on tS_L}, we can derive \eqref{est on nb uq i+1}--\eqref{est on nb Gq i+1}. Finally, combining \eqref{pp of te_q,i}, \eqref{pp of supp l,i}, \eqref{pp of theta^*}, and \eqref{def of tV_q k}--\eqref{def of g_LM2}, we obtain \eqref{pp of supp tu_q,i+1 ac}.
\end{proof}

\section{Estimates on the new Reynolds error and energy}\label{Estimates on the new Reynolds error and energy}
\subsection{Estimates for the inverse divergence operator}
\label{Estimates for the inverse divergence operator}

By \eqref{est on u_qi low}, we have
$$
\nrm{u_{q,i+1}}_{0}\leqslant2\varepsilon,
\qquad
\nrm{\nabla u_{q,i+1}}_{0},\ \nrm{\pt u_{q,i+1}}_{0}\leqslant2\varepsilon.
$$
Moreover, by \eqref{est on u_li N}, \eqref{est on tu_p}, \eqref{est on tVuG ac}, and \eqref{pp of Lambda1}, after decreasing $\varepsilon$ if necessary, we obtain, for $2\leqslant N+r\leqslant4$,
\begin{align*}
	\nrm{\pt^ru_{q,i+1}}_N
	&\leqslant\nrm{\pt^r\ul}_N+\nrm{\pt^r\tu_{q,i+1,p}}_N+\nrm{\pt^r\tu_{q,i+1,ac}}_N\\
	&\lesssim_{N,r}\ell_{q,i}^{2-N-r}\oM\lambda_{q,i}\delta_{q,i}^{\frac12}
	+\lambda_{q,i+1}^{N+r-1}\delta_{q+1}^{\frac12}
	+\mu_{q,i}^{1-N-r}\varepsilon\delta_{q+1}
	\leqslant2\varepsilon\lambda_{q,i+1}^{N+r-1}.
\end{align*}
Therefore, $y_{q,i+1}=x+u_{q,i+1}$ satisfies the assumptions of Proposition \ref{lm: asym div eq} with $d=L=2$, $\tlu=\lambda_{q,i+1}$, and deformation size $2\varepsilon$. Throughout this subsection, we also assume that $2\varepsilon\leqslant\varepsilon_0(2)$.

Based on Proposition \ref{lm: asym div eq}, we construct compactly
supported solutions to the asymmetric divergence equation when the
vector field $U$ takes the form
$U=\sum_m a_m(t,x)e^{i\lambda_{q,i+1}\xi_m}+U^{low}$, namely,
\begin{align}\label{asym div eq for ae}
	\Div(F_{q,i+1}R) =\Div((\Id + G_{q,i+1})R) = U = \underbrace{\sum_{m}a_m(t,x)e^{i\lambda_{q,i+1}\xi_m}}_{U^{high}} + U^{low}.
\end{align}
Here and below, the sum over $m$ is finite. We assume that, for all
$t\in\mcI^{q,i}$,
\begin{align*}
	\supp_xa_m(t,\cdot)\bigcup\supp_xU^{low}(t,\cdot)
	\subseteq
	\supp_x\chi_\upsilon(\mu_{q,i}^{-1}\cdot)
	\subseteq
	B(2\pi\upsilon\mu_{q,i},2\pi\mu_{q,i}).
\end{align*}
We introduce the following proposition.

\begin{pp}[Compactly supported solutions to the asymmetric divergence equation with special $U$] \label{lm: asym div eq of ae}
	Consider a smooth vector field
	$
	U=\sum_m a_m(t,x)e^{i\lambda_{q,i+1}\xi_m}+U^{low},
	$
	where the sum over $m$ is finite and
	$
	\xi_m:\mcI^{q,i}\times\R^2\longrightarrow\R
	$
	is a smooth scalar function satisfying, uniformly in $m$,
	$
	f_{i+1}\cdot\nb\xi_m
	\text{ is constant in }(t,x),
	\
	f_{i+1}^{\perp}\cdot\nb\xi_m=0,
	\
	1\leqslant|f_{i+1}\cdot\nb\xi_m|\leqslant C.
	$
	and
	$
	\nrm{\pt^re^{i\lambda_{q,i+1}\xi_m}}_N
	\lesssim\lambda_{q,i+1}^{N+r}
	$
	for $N+r=0,\ldots,4$.
	Assume that there exists $\upsilon\in\Z^2$ such that
	\begin{align}
		\supp_{x} a_m(t,\cdot)\bigcup \supp_{x} U^{low}(t,\cdot)\subseteq \supp_x\chi_\upsilon(\mu_{q,i}^{-1}\cdot) \subseteq B(2\pi\upsilon\mu_{q,i},2\pi\mu_{q,i}),\quad\forall t \in \mcI^{q,i},\label{pp of supp a in L co}
	\end{align}
	and
	$
	\supp_ta_m\bigcup\supp_tU^{low}\subseteq\supp_tU
	$
	for every $m$. Suppose also that $U$ has vanishing linear and
	angular moments in Lagrangian coordinates, namely, for $j,l=1,2$
	and all $t\in\mcI^{q,i}$,
	\begin{align}
		&\int_{\R^2} U_{l}(t,x) \, \rd x =0,\quad \int_{\R^2} ((y_{q,i+1})_{j} U_{l} - (y_{q,i+1})_{l} U_{j})(t,x) \, \rd x =0. \nonumber
	\end{align}
	
	Let $\nrm{\cdot}_{N}=\nrm{\cdot}_{C^0(\mcI^{q,i};C^N(\R^2))}$ and assume that $a_m$ and $U^{low}$ satisfy
	\begin{equation}\label{est on ain L co}
		\begin{aligned}
			\nrm{\pt^r a_m}_{N}&\lesssim\nrm{a_m}_{0}\lambda_{q,i+1}^r\mu_{q,i}^{-N},&&N+r=0,\ldots,4,\\
			\nrm{\pt^r U^{low}}_{N}&\lesssim\nrm{U^{low}}_{0}\lambda_{q,i+1}^{N+r},&&N+r=0,\ldots,4.
		\end{aligned}
	\end{equation}
	
	Then there exists a solution $\uR[U]$ to the asymmetric divergence equation \eqref{asym div eq for ae}, depending linearly on $(a_m)_m$ and $U^{low}$ for the fixed phases $(\xi_m)_m$, with the following properties:
	\begin{enumerate}
		\item There exists an absolute constant $C>0$ such that the support of $\uR[U]$ is contained in the ball $B(2\pi\upsilon\mu_{q,i}, 2(1+C\varepsilon)\pi\mu_{q,i})$ for all $t\in\mcI^{q,i}$, i.e.,
		\begin{equation} \label{pp of supp rR in L co}
			\supp_{x} \uR[U](t,\cdot) \subseteq B(2\pi\upsilon\mu_{q,i}, 2(1+C\varepsilon)\pi\mu_{q,i}),\quad \forall t\in \mcI^{q,i}.
		\end{equation}
		
		\item The time support of $\uR[U]$ is contained in the time support of $U$, i.e.,
		\begin{equation} \label{pp of time supp uRlj}
			\supp_{t} \uR[U]\subseteq \supp_{t} U.
		\end{equation}
		
		\item For $N + r = 0, 1, 2$,
		\begin{equation} \label{est of nabla rR in L co}
			\nrm{\pt^r\uR[U]}_{N} \lesssim \lambda_{q,i+1}^{N+r}\left(\lambda_{q,i+1}^{-1}\sum_m\nrm{ a_m}_{0}+\mu_{q,i}\nrm{U^{low}}_{0}\right).
		\end{equation}
	\end{enumerate}
\end{pp}

\begin{proof}
	By the standing assumption, $\Id+G_{q,i+1}$ is invertible and
	$\nrm{(\Id+G_{q,i+1})^{-1}}_0\lesssim1$. For a vector
	$P\in\R^2$, we use the decomposition
	\begin{align*}
		P&=P_{f_{i+1},||}f_{i+1}+P_{f_{i+1},\perp}f_{i+1}^{\perp}=\frac{P\cdot f_{i+1}}{|f_{i+1}|^2}f_{i+1}+\frac{P\cdot f_{i+1}^{\perp}}{|f_{i+1}^{\perp}|^2}f_{i+1}^{\perp}.
	\end{align*}
	
	Let $g_m=(\Id+G_{q,i+1})^{-1}a_m$. We have
	\begin{align*}
		a_me^{i\lambda_{q,i+1}\xi_m}
		=&(\Id+G_{q,i+1})g_me^{i\lambda_{q,i+1}\xi_m}\\
		=&\Div\left(\frac{(\Id+G_{q,i+1})g_m\otimes f_{i+1}}{i\lambda_{q,i+1}(f_{i+1}\cdot\nb\xi_m)}e^{i\lambda_{q,i+1}\xi_m}\right)-(f_{i+1}\cdot\nabla)\left(\frac{(\Id+G_{q,i+1})g_m}{i\lambda_{q,i+1}(f_{i+1}\cdot\nb\xi_m)}\right)e^{i\lambda_{q,i+1}\xi_m}\\
		=&\Div\left((\Id+G_{q,i+1})\frac{(g_m)_{f_{i+1},||}f_{i+1}\otimes f_{i+1}+(g_m)_{f_{i+1},\perp}(f_{i+1}^{\perp}\otimes f_{i+1}+f_{i+1}\otimes f_{i+1}^{\perp})}{i\lambda_{q,i+1}(f_{i+1}\cdot\nb\xi_m)}e^{i\lambda_{q,i+1}\xi_m}\right)\\
		&- \left((f_{i+1}\cdot\nabla)\left(\frac{a_m}{i\lambda_{q,i+1}(f_{i+1}\cdot\nb\xi_m)}\right)+(f_{i+1}^{\perp}\cdot\nabla)\left(\frac{ (g_m)_{f_{i+1},\perp}(\Id+G_{q,i+1})f_{i+1}}{i\lambda_{q,i+1}(f_{i+1}\cdot\nb\xi_m)}\right)\right) e^{i\lambda_{q,i+1}\xi_m}.
	\end{align*}
	Define the symmetric matrix
	\begin{align*}
		Q_{a_m}:=&\frac{(g_m)_{f_{i+1},||}f_{i+1}\otimes f_{i+1}+(g_m)_{f_{i+1},\perp}(f_{i+1}^{\perp}\otimes f_{i+1}+f_{i+1}\otimes f_{i+1}^{\perp})}{i\lambda_{q,i+1}(f_{i+1}\cdot\nabla\xi_m)}e^{i\lambda_{q,i+1}\xi_m}.
	\end{align*}
	Then
	\begin{align*}
		\sum_{m}a_me^{i\lambda_{q,i+1}\xi_m}
		&=\sum_{m}\Div\left((\Id+G_{q,i+1})Q_{a_m}\right)+\sum_{m}a_{m}^{(1)} e^{i\lambda_{q,i+1}\xi_m},
	\end{align*}
	where
	\begin{align*}
		a_{m}^{(1)}=&-\left((f_{i+1}\cdot\nabla)\left(\frac{a_m}{i\lambda_{q,i+1}(f_{i+1}\cdot\nb\xi_m)}\right)+(f_{i+1}^{\perp}\cdot\nabla)\left(\frac{ (g_m)_{f_{i+1},\perp}(\Id+G_{q,i+1})f_{i+1}}{i\lambda_{q,i+1}(f_{i+1}\cdot\nb\xi_m)}\right)\right)\\
		=&-\frac{(f_{i+1}\cdot\nabla)a_m+(f_{i+1}^{\perp}\cdot\nabla)((g_m)_{f_{i+1},\perp}(\Id+G_{q,i+1}))f_{i+1}}{i\lambda_{q,i+1}(f_{i+1}\cdot\nb\xi_m)}.
	\end{align*}
	Set $Q=\sum_mQ_{a_m}$. Since $Q$ is symmetric and compactly supported and
	$F_{q,i+1}=\nabla y_{q,i+1}$, integration by parts shows that
	$\Div(F_{q,i+1}Q)$ has vanishing linear and angular moments in
	Lagrangian coordinates. 
	Therefore,
	$
	\sum_ma_m^{(1)}e^{i\lambda_{q,i+1}\xi_m}+U^{low}
	$
	has vanishing linear and angular moments in Lagrangian coordinates.
	
	Notice that
	\begin{align*}
		(f_{i+1}^{\perp}\cdot\nabla)(g_m)_{f_{i+1},\perp}
		=&\frac{\left(-(\Id+G_{q,i+1})^{-1}(f_{i+1}^{\perp}\cdot\nabla)G_{q,i+1}(\Id+G_{q,i+1})^{-1}a_m+(\Id+G_{q,i+1})^{-1}(f_{i+1}^{\perp}\cdot\nabla)a_m\right)\cdot f_{i+1}^{\perp}}{|f_{i+1}^{\perp}|^2}.
	\end{align*}
	Moreover, for $0\leqslant N+r\leqslant2$, we have
	\begin{align*}
		\nrm{\pt^rG_{q,i+1}}_{N+1}&\lesssim\nrm{\pt^rG_{\ell,i}}_{N+1}+\nrm{\pt^r\tG_{q,i+1}}_{N+1}
		\lesssim\ell_{q,i}^{-N-r}\oM\lambda_{q,i}\dlt_{q,i}^{\frac{1}{2}}+\lambda_{q,i+1}^{N+r+1}\dlt_{q+1}^{\frac{1}{2}}\lesssim\lambda_{q,i+1}^{N+r+1}\dlt_{q+1}^{\frac{1}{2}},\\
		\nrm{\pt^r(f_{i+1}^{\perp}\cdot\nabla)G_{q,i+1}}_N&\lesssim\nrm{\pt^r(f_{i+1}^{\perp}\cdot\nabla)G_{\ell,i}}_N+\nrm{\pt^r(f_{i+1}^{\perp}\cdot\nabla)\tG_{q,i+1}}_N\\
		&\lesssim\ell_{q,i}^{-N-r}\oM\lambda_{q,i}\dlt_{q,i}^{\frac{1}{2}}+\lambda_{q,i+1}^{N+r}\mu_{q,i}^{-1}\dlt_{q+1}^{\frac{1}{2}}\lesssim\mu_{q,i}^{-1}\lambda_{q,i+1}^{N+r}\dlt_{q+1}^{\frac{1}{2}}.
	\end{align*}
	we could use similar way as \eqref{der of F^-1} to obtain for $r\geqslant 0$,
	\begin{align*}
		\nrm{\nb (\Id+G_{q,i+1})^{-1}}_{N}\lesssim&\sum_{N_1+N_2+N_3=N}\nrm{(\Id+G_{q,i+1})^{-1}}_{N_1}\nrm{ (\Id+G_{q,i+1})^{-1}}_{N_2}\nrm{\nabla G_{q,i+1}}_{N_3}\lesssim\lambda_{q,i+1}^{N+1}\delta_{q+1}^{\frac{1}{2}},\\
		\nrm{\pt^{r+1} (\Id+G_{q,i+1})^{-1}}_{N}\lesssim&\sum_{N_1+N_2+N_3=N}\sum_{r_1+r_2+r_3=r}\nrm{\pt^{r_1}(\Id+G_{q,i+1})^{-1}}_{N_1}\nrm{\pt^{r_2}(\Id+G_{q,i+1})^{-1}}_{N_2}\nrm{\pt^{1+r_3} G_{q,i+1}}_{N_3}\\
		\lesssim&\lambda_{q,i+1}^{N+r+1}\delta_{q+1}^{\frac{1}{2}},\\
		\nrm{\pt^{r}(f_{i+1}^{\perp}\cdot\nb) (\Id+G_{q,i+1})^{-1}}_{N}\lesssim&\sum_{N_1+N_2+N_3=N}\sum_{r_1+r_2+r_3=r}\nrm{\pt^{r_1}(\Id+G_{q,i+1})^{-1}}_{N_1}\nrm{\pt^{r_2}(\Id+G_{q,i+1})^{-1}}_{N_2}\nrm{\pt^{r_3}((f_{i+1}^{\perp}\cdot\nb) G_{q,i+1})}_{N_3}\\
		\lesssim&\mu_{q,i}^{-1}\lambda_{q,i+1}^{N+r}\delta_{q+1}^{\frac{1}{2}},
	\end{align*}
	and, for $0\leqslant N+r\leqslant2$,
	\begin{align*}
		\nrm{\pt^r(g_m)_{f_{i+1},||}}_{N}
		+\nrm{\pt^r(g_m)_{f_{i+1},\perp}}_{N}
		\lesssim&
		\sum_{N_1+N_2=N}\sum_{r_1+r_2=r}
		\nrm{\pt^{r_1}(\Id+G_{q,i+1})^{-1}}_{N_1}
		\nrm{\pt^{r_2}a_m}_{N_2}\\
		\lesssim&
		\begin{cases}
			\lambda_{q,i+1}^{r}\nrm{a_m}_{0},
			&N=0,\\
			\lambda_{q,i+1}^{N+r}\delta_{q+1}^{\frac{1}{2}}\nrm{a_m}_{0},
			&N\geqslant1,
		\end{cases}\\
		\nrm{\pt^r(f_{i+1}^{\perp}\cdot\nabla)(g_m)_{f_{i+1},\perp}}_{N}
		\lesssim&
		\sum_{N_1+N_2=N}\sum_{r_1+r_2=r}
		\nrm{\pt^{r_1}(f_{i+1}^{\perp}\cdot\nb)(\Id+G_{q,i+1})^{-1}}_{N_1}
		\nrm{\pt^{r_2}a_m}_{N_2}\\
		&+\sum_{N_1+N_2=N}\sum_{r_1+r_2=r}
		\nrm{\pt^{r_1}(\Id+G_{q,i+1})^{-1}}_{N_1}
		\nrm{\pt^{r_2}(f_{i+1}^{\perp}\cdot\nb)a_m}_{N_2}\\
		\lesssim&
		\begin{cases}
			\mu_{q,i}^{-1}\lambda_{q,i+1}^{r}\nrm{a_m}_{0},
			&N=0,\\
			\mu_{q,i}^{-1}\lambda_{q,i+1}^{N+r}
			\delta_{q+1}^{\frac{1}{2}}\nrm{a_m}_{0},
			&N\geqslant1.
		\end{cases}
	\end{align*}
	In particular, since $\delta_{q+1}^{\frac{1}{2}}\leqslant1$, we have
	\begin{align*}
		\nrm{\pt^r(g_m)_{f_{i+1},||}}_{N}
		+\nrm{\pt^r(g_m)_{f_{i+1},\perp}}_{N}
		&\lesssim
		\lambda_{q,i+1}^{N+r}\nrm{a_m}_{0},\\
		\nrm{\pt^r(f_{i+1}^{\perp}\cdot\nabla)(g_m)_{f_{i+1},\perp}}_{N}
		&\lesssim
		\mu_{q,i}^{-1}\lambda_{q,i+1}^{N+r}\nrm{a_m}_{0},
		\qquad
		0\leqslant N+r\leqslant2.
	\end{align*}
	
	Since $f_{i+1}\cdot\nabla\xi_m$ is constant and bounded away from zero, the definition of $Q_{a_m}$ and the estimates for the phases imply that, for $0\leqslant N+r\leqslant2$,
	\begin{align*}
		\nrm{\pt^rQ_{a_m}}_{N}
		\lesssim&
		\lambda_{q,i+1}^{-1}
		\sum_{N_1+N_2=N}\sum_{r_1+r_2=r}
		\left(
		\nrm{\pt^{r_1}(g_m)_{f_{i+1},||}}_{N_1}
		+
		\nrm{\pt^{r_1}(g_m)_{f_{i+1},\perp}}_{N_1}
		\right)
		\nrm{\pt^{r_2}e^{i\lambda_{q,i+1}\xi_m}}_{N_2}
		\lesssim
		\lambda_{q,i+1}^{N+r-1}\nrm{a_m}_{0}.
	\end{align*}
	Moreover, by the definition of $a_m^{(1)}$,
	\begin{align*}
		\nrm{\pt^ra_m^{(1)}}_{N}
		\lesssim&
		\lambda_{q,i+1}^{-1}
		\nrm{\pt^r(f_{i+1}\cdot\nb)a_m}_{N}\\
		&+\lambda_{q,i+1}^{-1}
		\sum_{N_1+N_2=N}\sum_{r_1+r_2=r}
		\nrm{\pt^{r_1}(f_{i+1}^{\perp}\cdot\nb)(g_m)_{f_{i+1},\perp}}_{N_1}
		\nrm{\pt^{r_2}(\Id+G_{q,i+1})}_{N_2}\\
		&+\lambda_{q,i+1}^{-1}
		\sum_{N_1+N_2=N}\sum_{r_1+r_2=r}
		\nrm{\pt^{r_1}(g_m)_{f_{i+1},\perp}}_{N_1}
		\nrm{\pt^{r_2}(f_{i+1}^{\perp}\cdot\nb)G_{q,i+1}}_{N_2}\\
		\lesssim&
		\mu_{q,i}^{-1}\lambda_{q,i+1}^{N+r-1}\nrm{a_m}_{0},
		\qquad
		0\leqslant N+r\leqslant2.
	\end{align*}
	Consequently,
	\begin{align*}
		\nrm{\pt^r\left(
			\sum_ma_m^{(1)}e^{i\lambda_{q,i+1}\xi_m}
			+
			U^{low}
			\right)}_N
		\lesssim
		\lambda_{q,i+1}^{N+r}
		\left(
		(\lambda_{q,i+1}\mu_{q,i})^{-1}
		\sum_m\nrm{a_m}_{0}
		+
		\nrm{U^{low}}_{0}
		\right),
		\qquad
		0\leqslant N+r\leqslant2.
	\end{align*}
	
	We now define
	\begin{align}
		\uR[U]:=\sum_{m}Q_{a_m}
		+\tilde{R}^{(u_{q,i+1})}
		\left[
		\sum_{m}a_{m}^{(1)}e^{i\lambda_{q,i+1}\xi_m}
		+
		U^{low}
		\right].
	\end{align}
	We apply Proposition \ref{lm: asym div eq} to the second term with
	$\tlu=\tlU=\lambda_{q,i+1}$ and
	$\tmu=2\pi\mu_{q,i}$. By \eqref{pp of Lambda1},
	$\mu_{q,i}^{-1}\leqslant\lambda_{q,i+1}$, and hence all the
	hypotheses of Proposition \ref{lm: asym div eq} are satisfied.
	It follows that \eqref{pp of supp rR in L co} holds. Moreover,
	since the temporal supports of $Q_{a_m}$ and $a_m^{(1)}$ are
	contained in that of $a_m$, the assumed time-support inclusions and Proposition
	\ref{lm: asym div eq} imply \eqref{pp of time supp uRlj}. Finally,
	\begin{align*}
		\nrm{\pt^r\uR[U]}_{N}
		&\lesssim
		\sum_{m}\nrm{\pt^{r}Q_{a_m}}_{N}
		+\nrm{\pt^{r}\tilde{R}^{(u_{q,i+1})}
			\left[
			\sum_{m}a_m^{(1)}e^{i\lambda_{q,i+1}\xi_m}
			+
			U^{low}
			\right]}_{N}\\
		&\lesssim
		\lambda_{q,i+1}^{N+r-1}\sum_m\nrm{a_m}_{0}
		+\mu_{q,i}\lambda_{q,i+1}^{N+r}
		\left(
		(\lambda_{q,i+1}\mu_{q,i})^{-1}\sum_m\nrm{a_m}_{0}
		+
		\nrm{U^{low}}_{0}
		\right)\\
		&\lesssim
		\lambda_{q,i+1}^{N+r}
		\left(
		\lambda_{q,i+1}^{-1}\sum_m\nrm{a_m}_{0}
		+
		\mu_{q,i}\nrm{U^{low}}_{0}
		\right),
		\qquad
		N+r=0,1,2.\qedhere
	\end{align*}
\end{proof}

\begin{Rmk}
	The high-frequency components of products involving $\Pi_I$ and
	$\tPi_{I'}$, such as $\mPG\Pi_I\Pi_{I'}$, can be written as finite
	sums of oscillatory terms of the form in Proposition
	\ref{lm: asym div eq of ae}, whose phases satisfy the required
	assumptions. We denote the corresponding amplitudes by $a_m[U]$;
	their spatial and temporal supports are contained in those of the
corresponding localized vector field $U$,
	and all subsequent estimates for $a_m[U]$ are understood with
	respect to this representation.
\end{Rmk}

\begin{Rmk}
	Although Proposition \ref{lm: asym div eq of ae} is stated for compactly supported vector fields on $\R^2$, we apply it to spatially localized components of periodic vector fields on $\T^2$, with all periodic coefficients understood through their $2\pi$-periodic lifts to $\R^2$. Thus, each $U_{\upsilon}$ is regarded as a compactly supported smooth vector field on $\R^2$, independently of the choice of a fundamental domain.
	
	Since $\mu_{q,i}^{-1}\in2\Z$, we have
	$$
	U_{\upsilon+\mu_{q,i}^{-1}k}(t,x+2\pi k)
	=
	U_{\upsilon}(t,x),
	\qquad
	\upsilon,k\in\Z^2.
	$$
	Choosing one representative from each class of $\Z^2/\mu_{q,i}^{-1}\Z^2$ and defining the remaining local solutions by translation, we may also arrange that
	$$
	\uR[U_{\upsilon+\mu_{q,i}^{-1}k}](t,x+2\pi k)
	=
	\uR[U_{\upsilon}](t,x).
	$$
	Consequently, $\sum_{\upsilon\in\Z^2}U_{\upsilon}$ and $\sum_{\upsilon\in\Z^2}\uR[U_{\upsilon}]$ are $2\pi$-periodic in $x$ and define smooth fields on $\T^2$.
	
	Moreover, Proposition \ref{lm: asym div eq of ae} gives
	$$
	\supp_x\uR[U_{\upsilon}](t,\cdot)
	\subseteq
	B\left(2\pi\upsilon\mu_{q,i},2(1+C\varepsilon)\pi\mu_{q,i}\right).
	$$
	These balls have uniformly finite overlap, and hence
	$$
	\nrm{\pt^r\sum_{\upsilon\in\Z^2}\uR[U_{\upsilon}]}_{N}
	\lesssim
	\sup_{\upsilon\in\Z^2}\nrm{\pt^r\uR[U_{\upsilon}]}_{N}.
	$$
	Since the sums are locally finite, differentiation commutes with summation, and therefore
	$$
	\Div\left(F_{q,i+1}\sum_{\upsilon\in\Z^2}\uR[U_{\upsilon}]\right)
	=
	\sum_{\upsilon\in\Z^2}U_{\upsilon}
	\qquad\text{on }\T^2.
	$$
	Finally, since $\mu_{q,i}^{-1}\in2\Z$, the translation $\upsilon\mapsto\upsilon+\mu_{q,i}^{-1}k$ preserves the parity of each component of $\upsilon$ and hence the value of $[I]$.
\end{Rmk}
\subsection{Definition of the new Reynolds error}\label{Definition of the new Reynolds error}
In the previous section, we have constructed the perturbation.  Here, we  will add the perturbation $\tu_{q,i+1}$ to $u_{q,i}$, defining $u_{q,i+1}=\ul+\tu_{q,i+1}$, $G_{q,i+1}=\Gl+\tG_{q,i+1}$.  Based on the previous calculation \eqref{divide of Reynold error 1} and the definition of $\uR$ in Section \ref{Estimates for the inverse divergence operator}, we could define the new Reynolds error $R_{q,i+1}$ as 
\begin{align}
	R_{q,i+1}=\sgm^*\te_{q,i}\Id^{<i+1>}+\Rl-\sgm^*d_{q,i+1}^2\left(\Id-\frac{f_{i+1}}{|f_{i+1}|}\otimes \frac{f_{i+1}}{|f_{i+1}|}\right)+\delta R_{q,i+1},\label{def of R_q,i+1}
\end{align}
where
\begin{align}
	R_{P}:=&\sum_{\upsilon}\uR\Big[\pa_{tt}\tu_{q,i+1,p}^{(\upsilon)}+\pa_{tt}\tu_{q,i+1,L}^{(\upsilon)}-\Div\Big(\tG_{q,i+1}^{(\upsilon)}(\Sgm_{\Gl}+\Rl+c_{q,i})+(\Id+G_{q,i+1})\tS_{q,i+1}^{(1),(\upsilon)}\Big)\nonumber\\
	&\quad\quad\quad\quad+\Div\Big((\Id+G_{q,i+1})\tG_{q,i+1,m}\tS_{q,i+1,m1}^{(1),(\upsilon)}\Big)\Big],\label{def of R_P}\\
	R_{M}:=&\sum_{\upsilon}\uR\left[\pa_{tt}\tu_{q,i+1,M}^{(\upsilon)}+\Div (\chi_\upsilon^2(\mu_{q,i}^{-1}x)R_{m,i})\right],\label{def of R_M}\\
	R_{O1}:=&\sum_{\upsilon}\uR\left[\Div\left((\Id+G_{q,i+1})\left(\sgm^*\chi_{\upsilon}^2(\mu_{q,i}^{-1}x)d_{q,i+1}^2\left(\Id-\frac{f_{i+1}}{|f_{i+1}|}\otimes \frac{f_{i+1}}{|f_{i+1}|}\right)-\tG_{q,i+1,m}\tS_{q,i+1,m1}^{(1),(\upsilon)}-\tS_{q,i+1,p}^{(2),m,(\upsilon)}\right)\right)\right],\label{def of R_O1}\\
	R_{O2}:=&-\tS_{q,i+1,s}^{(2),m}-\tS_{q,i+1}^{(2),c}-\tS_{q,i+1}^{(\geqslant3)},\label{def of R_O2}\\
	R_{O3}:=&\sgm^*\te_{q,i}\Id^{<i+1>}+\Rl-\sgm^*d_{q,i+1}^2\left(\Id-\frac{f_{i+1}}{|f_{i+1}|}\otimes \frac{f_{i+1}}{|f_{i+1}|}\right),\label{def of R_O3}\\
	R_{O}:=&R_{O1}+R_{O2}+R_{O3},\label{def of R_O}\\
	\delta R_{q,i+1}:=&R_P+R_M+R_{O1}+R_{O2}\label{def of deltaR}.
\end{align}

\subsection{Estimates on the new Reynolds error}\label{Estimates on the new Reynolds error}
In this subsection, we estimate the new Reynolds stress $R_{q,i+1}$ and its derivatives. For convenience, if the implicit constants in the estimates in this section depend on $\sgm$ and $\oM$, we will not write it out.  For the remaining sections, we set $\nrm{\cdot}_N =\nrm{\cdot}_{C^0(\mcI^{q,i}; C^N(\T^2))}$
\begin{pp}\label{pp of Reynolds error}
	Let the parameters $\varepsilon^*_1$, $\varepsilon^*_2$, and
	$\varepsilon^*_3$ be as in the statements of
	Proposition \ref{pp of Mollification},
	Lemma \ref{construction of building blocks}, and
	Proposition \ref{est on perturbation}, respectively.
	Then, for $b>5$, we can find
	$$
	\varepsilon^*_0
	=
	\varepsilon^*_0(b,\oM)
	\leqslant
	\min
	\left\{
	\varepsilon^*_1,
	\varepsilon^*_2,
	\varepsilon^*_3,
	\frac{\varepsilon_0(2)}{2}
	\right\}
	$$
	such that, for every $0<\varepsilon<\varepsilon_0^*$, $0\leqslant i\leqslant2$, and
	$0\leqslant N+r\leqslant2$, we have
	\begin{align}
		\nrm{\pa_t^r\dlt R_{q,i+1}}_{N}&\leqslant C_{b,\sgm,\oM}\lambda_{q,i+1}^{N+r}\left(\varepsilon\delta_{q+1}+(\lambda_{q,i+1}\mu_{q,i})^{-1}\delta_{q+1}^{\frac{1}{2}}\right)
		\leqslant\frac{1}{128} \tilde{c}_0^4 \lambda_{q,i+1}^{N+r} \dlt_{q+2},\label{est on delta R_q i+1}
	\end{align}
	where $C_{b,\sgm,\oM}$ depends only upon $b$, $\sgm$,  and $\oM$ in Proposition \ref{Proposition 1}. Moreover, the support of $R_{q,i+1}$, $\dlt R_{q,i+1}$  and $R_{O3}$ satisfies 
	\begin{align}\label{pp of supp dR}
	\supp_{t}(R_{q,i+1}, \dlt R_{q,i+1}, R_{O3})\subseteq \supp_{t} R_{q,i}+ 3\tau_{q,i}+\ell_{q,i}+\lambda_{q}^{-1}.
	\end{align}
\end{pp}
We will consider \eqref{def of R_q,i+1} and  estimate the separate terms $ R_{P}$, $ R_{M}$, and $R_{O}$.  For $R_{O2}$, we use a direct estimate. The term $R_{O3}$ is retained explicitly in $R_{q,i+1}$ and is handled in the proof of the inductive proposition. For $R_P$, $R_{O1}$, and $R_M$, we apply Proposition \ref{lm: asym div eq of ae}.

Next, we present a fact that will be repeatedly used: For any $b > 5$ and any constant $\tilde{C}_{b,\sgm,\oM}$ depending only on $b$, $\sgm$, and $\oM$, there exists a sufficiently small $\varepsilon_1 = \varepsilon_1(b,\sgm,\oM)$ such that if $\varepsilon < \varepsilon_1$, we have
$$
\tilde{C}_{b,\sgm,\oM}\left( \varepsilon \delta_{q+1} + (\lambda_{q,i+1}\mu_{q,i})^{-1} \delta_{q+1}^{\frac{1}{2}} + \mu_{q,i}\lambda_{q,i} \delta_{q,i}^{\frac{1}{2}} \delta_{q+1}^{\frac{1}{2}} \right)
\leqslant \frac{1}{1024} \tilde{c}_0^4 \delta_{q+2}.
$$

\subsubsection{Estimates on the principal error} 
Recall 
\begin{align*}
	R_{P}:=&\sum_{\upsilon}\uR\Big[\pa_{tt}\tu_{q,i+1,p}^{(\upsilon)}+\pa_{tt}\tu_{q,i+1,L}^{(\upsilon)}-\Div\Big(\tG_{q,i+1}^{(\upsilon)}(\Sgm_{\Gl}+\Rl+c_{q,i})+(\Id+G_{q,i+1})\tS_{q,i+1}^{(1),(\upsilon)}\Big)\nonumber\\
	&\quad\quad\quad\quad+\Div\Big((\Id+G_{q,i+1})\tG_{q,i+1,m}\tS_{q,i+1,m1}^{(1),(\upsilon)}\Big)\Big],
\end{align*}
we could divide it into three parts:
\begin{align*}
	&\pa_{tt}\tu_{q,i+1,p}^{(\upsilon)}+\pa_{tt}\tu_{q,i+1,L}^{(\upsilon)}-\Div\Big(\tG_{q,i+1}^{(\upsilon)}(\Sgm_{\Gl}+\Rl+c_{q,i})+(\Id+G_{q,i+1})\tS_{q,i+1}^{(1),(\upsilon)}-(\Id+G_{q,i+1})\tG_{q,i+1,m}\tS_{q,i+1,m1}^{(1),(\upsilon)}\Big)\\
	=&\underbrace{\pa_{tt}\tu_{q,i+1,L}^{(\upsilon)}-\Div\Big(\tG_{q,i+1,ac}^{(\upsilon)}(\Sgm_{\Gl}+\Rl+c_{q,i})+(\Id+\Gl)\tS_{q,i+1,ac}^{(1),(\upsilon)}\Big)}_{\Xi_{ac}}\\
	&+\underbrace{\pa_{tt}\tu_{q,i+1,p}^{(\upsilon)}-\Div\Big(\tG_{q,i+1,p}^{(\upsilon)}(\Sgm_{\Gl}+\Rl+c_{q,i})+(\Id+\Gl)\tS_{q,i+1,p}^{(1),(\upsilon)}\Big)}_{\Xi_{linear}}\\
	&-\underbrace{\Div\Big(\tG_{q,i+1}\tS_{q,i+1}^{(1),(\upsilon)}-(\Id+G_{q,i+1})\tG_{q,i+1,m}\tS_{q,i+1,m1}^{(1),(\upsilon)}\Big)}_{\Xi_{nonlinear}}.
\end{align*}
For the first term $\Xi_{ac}$, we have
\begin{equation}\label{est on Xi_ac}
	\begin{aligned}
		&\Nrm{\pt^r\left(\pa_{tt}\tu_{q,i+1,L}^{(\upsilon)}-\Div\Big(\tG_{q,i+1,ac}^{(\upsilon)}(\Sgm_{\Gl}+\Rl+c_{q,i})+(\Id+\Gl)\tS_{q,i+1,ac}^{(1),(\upsilon)}\Big)\right)}_N\\
		\lesssim&\Nrm{\pt^{r+2}\tu_{q,i+1,L}^{(\upsilon)}}_N+\Nrm{\pt^r\Div\Big(\tG_{q,i+1,ac}^{(\upsilon)}(\Sgm_{\Gl}+\Rl+c_{q,i})+(\Id+\Gl)\tS_{q,i+1,ac}^{(1),(\upsilon)}\Big)}_N\\
		\lesssim&\mu_{q,i}^{-N-r-1}\varepsilon\delta_{q+1}.
	\end{aligned}
\end{equation}
The second part $\Xi_{linear}$ can be written as
\begin{align*}
	(\Xi_{linear})_k=&\left(\pa_{tt}\tu_{q,i+1,p}^{(\upsilon)}-\Div\Big(\tG_{q,i+1,p}^{(\upsilon)}(\Sgm_{\Gl}+\Rl+c_{q,i})+(\Id+\Gl)\tS_{q,i+1,p}^{(1),(\upsilon)}\Big)\right)_k\\
	=&\sum_{s:I=(s,\upsilon)}\frac{1}{\lambda_{q,i+1}[I]}(u_{q,i+1,p}^{I,cof}\pa_{tt}w_{A_I^{\upsilon},f_{i+1}}+\ou_{q,i+1,p}^{I,cof}\pa_{tt}\ow_{A_I^{\upsilon},f_{i+1}})_k\\
	&-\sum_{s:I=(s,\upsilon)}\frac{\lambda+\mu}{\lambda_{q,i+1}[I]}((u_{q,i+1,p}^{I,cof}\nb\Div w_{A_I^{\upsilon},f_{i+1}}+\ou_{q,i+1,p}^{I,cof}\nb\Div\ow_{A_I^{\upsilon},f_{i+1}}))_k\\
	&-\sum_{s:I=(s,\upsilon)}\frac{\mu}{\lambda_{q,i+1}[I]}( (u_{q,i+1,p}^{I,cof}\Delta w_{A_I^{\upsilon},f_{i+1}}+\ou_{q,i+1,p}^{I,cof}\Delta\ow_{A_I^{\upsilon},f_{i+1}}))_k\\
	&+\sum_{s:I=(s,\upsilon)}\frac{(A_I^{\upsilon})_{nr}^{km}}{\lambda_{q,i+1}[I]}(u_{q,i+1,p}^{I,cof}\partial_{nr}(w_{A_I^{\upsilon},f_{i+1}})_m+\ou_{q,i+1,p}^{I,cof}\partial_{nr}(\ow_{A_I^{\upsilon},f_{i+1}})_m)\\
	&+\sum_{s:I=(s,\upsilon)}\frac{1}{\lambda_{q,i+1}[I]}(\pa_{tt}u_{q,i+1,p}^{I,cof}w_{A_I^{\upsilon},f_{i+1}}+\pa_{tt}\ou_{q,i+1,p}^{I,cof}\ow_{A_I^{\upsilon},f_{i+1}}+2\pa_{t}u_{q,i+1,p}^{I,cof}\pa_{t}w_{A_I^{\upsilon},f_{i+1}}+2\pa_{t}\ou_{q,i+1,p}^{I,cof}\pa_{t}\ow_{A_I^{\upsilon},f_{i+1}})_k\\
	&-\sum_{s:I=(s,\upsilon)}\frac{\lambda+\mu}{\lambda_{q,i+1}[I]}(\Div w_{A_I^{\upsilon},f_{i+1}}\nb u_{q,i+1,p}^{I,cof}+\Div\ow_{A_I^{\upsilon},f_{i+1}}\nb\ou_{q,i+1,p}^{I,cof}+\nb(w_{A_I^{\upsilon},f_{i+1}}\cdot\nb u_{q,i+1,p}^{I,cof}+\ow_{A_I^{\upsilon},f_{i+1}}\cdot\nb \ou_{q,i+1,p}^{I,cof}))_k\\
	&-\sum_{s:I=(s,\upsilon)}\frac{\mu}{\lambda_{q,i+1}[I]}(2\nb w_{A_I^{\upsilon},f_{i+1}}\nb u_{q,i+1,p}^{I,cof}+2\nb\ow_{A_I^{\upsilon},f_{i+1}}\nb\ou_{q,i+1,p}^{I,cof}+\Delta(u_{q,i+1,p}^{I,cof})w_{A_I^{\upsilon},f_{i+1}}+\Delta(\ou_{q,i+1,p}^{I,cof})\ow_{A_I^{\upsilon},f_{i+1}}))_k\\
	&+\sum_{s:I=(s,\upsilon)}\frac{(A_I)_{nr}^{km}-(A_I^{\upsilon})_{nr}^{km}}{\lambda_{q,i+1}[I]}(u_{q,i+1,p}^{I,cof}\partial_{nr}(w_{A_I^{\upsilon},f_{i+1}})_m+\ou_{q,i+1,p}^{I,cof}\partial_{nr}(\ow_{A_I^{\upsilon},f_{i+1}})_m))\\
	&+\sum_{s:I=(s,\upsilon)}\frac{1}{\lambda_{q,i+1}[I]}\pa_r((A_I)_{nr}^{km})\partial_{n}(u_{q,i+1,p}^{I,cof}(w_{A_I^{\upsilon},f_{i+1}})_m+\ou_{q,i+1,p}^{I,cof}(\ow_{A_I^{\upsilon},f_{i+1}})_m)\\
	&+\sum_{s:I=(s,\upsilon)}\frac{1}{\lambda_{q,i+1}[I]}(A_I)_{nr}^{km}(\partial_{nr}(u_{q,i+1,p}^{I,cof})(w_{A_I^{\upsilon},f_{i+1}})_m+\partial_{nr}(\ou_{q,i+1,p}^{I,cof})(\ow_{A_I^{\upsilon},f_{i+1}})_m)\\
	&+\sum_{s:I=(s,\upsilon)}\frac{1}{\lambda_{q,i+1}[I]}(A_I)_{nr}^{km}(\partial_{n}(u_{q,i+1,p}^{I,cof})\partial_{r}(w_{A_I^{\upsilon},f_{i+1}})_m+\partial_{n}(\ou_{q,i+1,p}^{I,cof})\partial_{r}(\ow_{A_I^{\upsilon},f_{i+1}})_m)\\
	&+\sum_{s:I=(s,\upsilon)}\frac{1}{\lambda_{q,i+1}[I]}(A_I)_{nr}^{km}(\partial_{r}(u_{q,i+1,p}^{I,cof})\partial_{n}(w_{A_I^{\upsilon},f_{i+1}})_m+\partial_{r}(\ou_{q,i+1,p}^{I,cof})\partial_{n}(\ow_{A_I^{\upsilon},f_{i+1}})_m)\\
	=&\sum_{s:I=(s,\upsilon)}\left((W_{q,i+1,p}^{I,cof,||}f_{i+1}+W_{q,i+1,p}^{I,cof,\perp}f_{i+1}^{\perp})e^{i\lambda_{q,i+1}\xi_{A_I^{\upsilon},f_{i+1}}}+(\oW_{q,i+1,p}^{I,cof,||}f_{i+1}+\oW_{q,i+1,p}^{I,cof,\perp}f_{i+1}^{\perp})e^{-i\lambda_{q,i+1}\xi_{A_I^{\upsilon},f_{i+1}}}\right)_k,
\end{align*}
where
\begin{align*}
	W_{q,i+1,p}^{I,cof,||}=&-i\left( u_{q,i+1,p}^{I,cof}\partial_{tt}c_{A_I^{\upsilon}}+2\pt u_{q,i+1,p}^{I,cof}\pt c_{A_I^{\upsilon}}+(\lambda+\mu)((f_{i+1}\cdot\nabla u_{q,i+1,p}^{I,cof})+(\tilde{f}_{A_I^{\upsilon}}\cdot\nabla u_{q,i+1,p}^{I,cof}))+2\mu (f_{i+1}\cdot\nabla u_{q,i+1,p}^{I,cof})\right)\\
	&+\frac{1}{\lambda_{q,i+1}[I]}\left(\pa_{tt}u_{q,i+1,p}^{I,cof}-(\lambda+\mu)|f_{i+1}|^{-2}(f_{i+1}\cdot\nabla)(\tilde{f}_{A_I^{\upsilon}}\cdot\nabla u_{q,i+1,p}^{I,cof})-\mu\Delta u_{q,i+1,p}^{I,cof}\right)\\
	&+\lambda_{q,i+1}[I]|f_{i+1}|^{-2}\left((A_I^{\upsilon})_{nr}^{km}-(A_I)_{nr}^{km}\right)u_{q,i+1,p}^{I,cof}(f_{i+1})_n(f_{i+1})_r(\tilde{f}_{A_{I}^{\upsilon}})_m(f_{i+1})_k\\
	&+i|f_{i+1}|^{-2}\left(\pa_r((A_I)_{nr}^{km}u_{q,i+1,p}^{I,cof})(f_{i+1})_n+(A_I)_{nr}^{km}\partial_{n}(u_{q,i+1,p}^{I,cof})(f_{i+1})_r\right)(\tilde{f}_{A_{I}^{\upsilon}})_m(f_{i+1})_k\\
	&+\frac{1}{\lambda_{q,i+1}[I]|f_{i+1}|^{2}}\left(\pa_r((A_I)_{nr}^{km})\partial_{n}(u_{q,i+1,p}^{I,cof})+(A_I)_{nr}^{km}\partial_{nr}(u_{q,i+1,p}^{I,cof})\right)(\tilde{f}_{A_{I}^{\upsilon}})_m(f_{i+1})_k\\
	W_{q,i+1,p}^{I,cof,\perp}=&-i\left(u_{q,i+1,p}^{I,cof}(a_{A_I^{\upsilon}}\partial_{tt}c_{A_I^{\upsilon}}+2\pt a_{A_I^{\upsilon}}\pt c_{A_I^{\upsilon}})+2\pt u_{q,i+1,p}^{I,cof}a_{A_I^{\upsilon}}\pt c_{A_I^{\upsilon}}\right)\\
	&-i\left((\lambda+\mu)(f_{i+1}^{\perp}\cdot\nabla u_{q,i+1,p}^{I,cof})+2\mu  (f_{i+1}\cdot\nabla u_{q,i+1,p}^{I,cof})a_{A_I^{\upsilon}}\right)\\
	&+\frac{1}{\lambda_{q,i+1}[I]}\left(u_{q,i+1,p}^{I,cof}\partial_{tt}a_{A_I^{\upsilon}}+\pa_{tt}u_{q,i+1,p}^{I,cof}a_{A_I^{\upsilon}}+2\pt u_{q,i+1,p}^{I,cof}\pt a_{A_I^{\upsilon}}\right)\\
	&-\frac{1}{\lambda_{q,i+1}[I]}\left((\lambda+\mu)|f_{i+1}^{\perp}|^{-2}(f_{i+1}^{\perp}\cdot\nabla)(\tilde{f}_{A_I^{\upsilon}}\cdot\nabla u_{q,i+1,p}^{I,cof})+\mu\Delta u_{q,i+1,p}^{I,cof}a_{A_I^\upsilon}\right)\\
	&+\lambda_{q,i+1}[I]|f_{i+1}^{\perp}|^{-2}\left((A_I^{\upsilon})_{nr}^{km}-(A_I)_{nr}^{km}\right)u_{q,i+1,p}^{I,cof}(f_{i+1})_n(f_{i+1})_r(\tilde{f}_{A_{I}^{\upsilon}})_m(f_{i+1}^{\perp})_k\\
	&+i|f_{i+1}^{\perp}|^{-2}\left(\pa_r((A_I)_{nr}^{km}u_{q,i+1,p}^{I,cof})(f_{i+1})_n+(A_I)_{nr}^{km}\partial_{n}(u_{q,i+1,p}^{I,cof})(f_{i+1})_r\right)(\tilde{f}_{A_{I}^{\upsilon}})_m(f_{i+1}^{\perp})_k\\
	&+\frac{1}{\lambda_{q,i+1}[I]|f_{i+1}^{\perp}|^2}\left(\pa_r((A_I)_{nr}^{km})\partial_{n}(u_{q,i+1,p}^{I,cof})+(A_I)_{nr}^{km}\partial_{nr}(u_{q,i+1,p}^{I,cof})\right)(\tilde{f}_{A_{I}^{\upsilon}})_m(f_{i+1}^{\perp})_k.
\end{align*}
The following estimates immediately follow from Proposition \ref{est on perturbation}:
\begin{align*}
	\nrmrN{W_{q,i+1,p}^{I,cof,||}}\lesssim&\nrmrN{ u_{q,i+1,p}^{I,cof}\partial_{tt}c_{A_I^{\upsilon}}}+\nrmrN{\pt u_{q,i+1,p}^{I,cof}\pt c_{A_I^{\upsilon}}}+\nrmrN{f_{i+1}\cdot\nabla u_{q,i+1,p}^{I,cof}}+\nrmrN{\tilde{f}_{A_I^{\upsilon}}\cdot\nabla u_{q,i+1,p}^{I,cof}}\\
	&+\lambda_{q,i+1}^{-1}\left(\nrmrN{\pa_{tt}u_{q,i+1,p}^{I,cof}}+\nrmrN{(f_{i+1}\cdot\nabla)(\tilde{f}_{A_I^{\upsilon}}\cdot\nabla u_{q,i+1,p}^{I,cof})}+\nrmrN{\Delta u_{q,i+1,p}^{I,cof}}\right),\\
	&+\lambda_{q,i+1}\nrmrN{((A_I^{\upsilon})_{nr}^{km}-(A_I)_{nr}^{km})
		u_{q,i+1,p}^{I,cof}\tilde{f}_{A_I^{\upsilon}}}\\
	&+\sum_{r_1+r_2=r}\sum_{N_1+N_2=N}\left(\nRMr{\nabla((A_I)_{nr}^{km})u_{q,i+1,p}^{I,cof}}{N_1}{r_1}+\nRMr{(A_I)_{nr}^{km}\nabla u_{q,i+1,p}^{I,cof}}{N_1}{r_1}\right)\nRMr{\tilde{f}_{A_{I}^{\upsilon}}}{N_2}{r_2}\\
	&+\sum_{r_1+r_2=r}\sum_{N_1+N_2=N}\lambda_{q,i+1}^{-1}\left(\nRMr{\nb((A_I)_{nr}^{km})\nb u_{q,i+1,p}^{I,cof}}{N_1}{r_1}+\nRMr{(A_I)_{nr}^{km}\nabla^2(u_{q,i+1,p}^{I,cof})}{N_1}{r_1}\right)\nRMr{\tilde{f}_{A_{I}^{\upsilon}}}{N_2}{r_2},\\
	\lesssim& \tau_{q,i}^{-r}\mu_{q,i}^{-N}\left(\mu_{q,i}^{-1}\delta_{q,i}^{\frac{1}{2}}\delta_{q+1}^{\frac{1}{2}}+\mu_{q,i}^{-1}\delta_{q+1}^{\frac{1}{2}}+\lambda_{q,i+1}^{-1}(\tau_{q,i}^{-2}\delta_{q+1}^{\frac{1}{2}}+\mu_{q,i}^{-2}\delta_{q+1}^{\frac{1}{2}})+\lambda_{q,i+1}(\mu_{q,i}\lambda_{q,i}\delta_{q,i}^{\frac{1}{2}}\delta_{q+1}^{\frac{1}{2}})\right)\\
	&+\tau_{q,i}^{-r}\mu_{q,i}^{-N}\left(\lambda_{q,i}\delta_{q,i}^{\frac{1}{2}}\delta_{q+1}^{\frac{1}{2}}+\mu_{q,i}^{-1}\varepsilon\delta_{q+1}^{\frac{1}{2}}+\lambda_{q,i+1}^{-1}(\lambda_{q,i}\delta_{q,i}^{\frac{1}{2}}\mu_{q,i}^{-1}\delta_{q+1}^{\frac{1}{2}}+\mu_{q,i}^{-2}\varepsilon\delta_{q+1}^{\frac{1}{2}})\right)\\
	\lesssim&\tau_{q,i}^{-r}\mu_{q,i}^{-N-1}\delta_{q+1}^{\frac{1}{2}},\\
	\nrmrN{W_{q,i+1,p}^{I,cof,\perp}}\lesssim&\nrmrN{u_{q,i+1,p}^{I,cof}a_{A_I^{\upsilon}}\partial_{tt}c_{A_I^{\upsilon}}}+\nrmrN{u_{q,i+1,p}^{I,cof}\pt a_{A_I^{\upsilon}}\pt c_{A_I^{\upsilon}})}+\nrmrN{\pt u_{q,i+1,p}^{I,cof}a_{A_I^{\upsilon}}\pt c_{A_I^{\upsilon}}}\\
	&+\nrmrN{f_{i+1}^{\perp}\cdot\nabla u_{q,i+1,p}^{I,cof})}+\nrmrN{f_{i+1}\cdot\nabla u_{q,i+1,p}^{I,cof})a_{A_I^{\upsilon}}}\\
	&+\lambda_{q,i+1}^{-1}\left(\nrmrN{u_{q,i+1,p}^{I,cof}\partial_{tt}a_{A_I^{\upsilon}}}+\nrmrN{\pt u_{q,i+1,p}^{I,cof}\pt a_{A_I^{\upsilon}}}+\nrmrN{\pa_{tt}u_{q,i+1,p}^{I,cof}a_{A_I^{\upsilon}}}\right)\\
	&+\lambda_{q,i+1}^{-1}\left(\nrmrN{(f_{i+1}^{\perp}\cdot\nabla)(\tilde{f}_{A_I^{\upsilon}}\cdot\nabla u_{q,i+1,p}^{I,cof})}+\nrmrN{\Delta u_{q,i+1,p}^{I,cof}a_{A_I^\upsilon}}\right)\\
	&+\lambda_{q,i+1}\nrmrN{((A_I^{\upsilon})_{nr}^{km}-(A_I)_{nr}^{km})
		u_{q,i+1,p}^{I,cof}\tilde{f}_{A_I^{\upsilon}}}\\
	&+\sum_{r_1+r_2=r}\sum_{N_1+N_2=N}\left(\nRMr{\nb((A_I)_{nr}^{km})u_{q,i+1,p}^{I,cof}}{N_1}{r_1}+\nRMr{(A_I)_{nr}^{km}\nb u_{q,i+1,p}^{I,cof}}{N_1}{r_1}\right)\nRMr{\tilde{f}_{A_{I}^{\upsilon}}}{N_2}{r_2}\\
	&+\sum_{r_1+r_2=r}\sum_{N_1+N_2=N}\lambda_{q,i+1}^{-1}\left(\nRMr{\nb((A_I)_{nr}^{km})\nb u_{q,i+1,p}^{I,cof}}{N_1}{r_1}+\nRMr{(A_I)_{nr}^{km}\nb^2u_{q,i+1,p}^{I,cof}}{N_1}{r_1}\right)\nRMr{\tilde{f}_{A_{I}^{\upsilon}}}{N_2}{r_2}\\
	\lesssim& \tau_{q,i}^{-r}\mu_{q,i}^{-N}\left(\mu_{q,i}^{-1}\delta_{q,i}^{\frac{1}{2}}\delta_{q+1}^{\frac{1}{2}}+\mu_{q,i}^{-1}\delta_{q+1}^{\frac{1}{2}}+\lambda_{q,i+1}^{-1}(\tau_{q,i}^{-2}\delta_{q+1}^{\frac{1}{2}}+\mu_{q,i}^{-2}\delta_{q+1}^{\frac{1}{2}})+\lambda_{q,i+1}(\mu_{q,i}\lambda_{q,i}\delta_{q,i}^{\frac{1}{2}}\delta_{q+1}^{\frac{1}{2}})\right)\\
	&+\tau_{q,i}^{-r}\mu_{q,i}^{-N}\left(\lambda_{q,i}\delta_{q,i}^{\frac{1}{2}}\delta_{q+1}^{\frac{1}{2}}+\mu_{q,i}^{-1}\varepsilon\delta_{q+1}^{\frac{1}{2}}+\lambda_{q,i+1}^{-1}(\lambda_{q,i}\delta_{q,i}^{\frac{1}{2}}\mu_{q,i}^{-1}\delta_{q+1}^{\frac{1}{2}}+\mu_{q,i}^{-2}\varepsilon\delta_{q+1}^{\frac{1}{2}})\right)\\
	\lesssim&\tau_{q,i}^{-r}\mu_{q,i}^{-N-1}\delta_{q+1}^{\frac{1}{2}},
\end{align*}
where we have used
\begin{align*}
	\nrm{(A_I^{\upsilon})_{nr}^{km}-(A_I)_{nr}^{km}}_0\lesssim&\mu_{q,i}\nrm{(A_I)_{nr}^{km}}_1\lesssim\mu_{q,i}\lambda_{q,i}\delta_{q,i}^{\frac{1}{2}},\\
	\nrm{((A_I^{\upsilon})_{nr}^{km}-(A_I)_{nr}^{km})u_{q,i+1,p}^{I,cof}\tilde{f}_{A_I^{\upsilon}}}_0\lesssim&\nRM{(A_I^{\upsilon})_{nr}^{km}-(A_I)_{nr}^{km}}{0}\nRM{u_{q,i+1,p}^{I,cof}}{0}\nRM{\tilde{f}_{A_I^{\upsilon}}}{0}
	\lesssim\mu_{q,i}\lambda_{q,i}\delta_{q,i}^{\frac{1}{2}}\delta_{q+1}^{\frac{1}{2}},\\
	\nrmrN{((A_I^{\upsilon})_{nr}^{km}-(A_I)_{nr}^{km})u_{q,i+1,p}^{I,cof}\tilde{f}_{A_I^{\upsilon}}}\lesssim&\sum_{r_1+r_2+r_3=r}\sum_{N_1+N_2+N_3=N}\nRMr{(A_I^{\upsilon})_{nr}^{km}-(A_I)_{nr}^{km}}{N_1}{r_1}\nRMr{u_{q,i+1,p}^{I,cof}}{N_2}{r_2}\nRMr{\tilde{f}_{A_I^{\upsilon}}}{N_3}{r_3}\\
	\lesssim&\tau_{q,i}^{-r}\mu_{q,i}^{-N}\mu_{q,i}\lambda_{q,i}\delta_{q,i}^{\frac{1}{2}}\delta_{q+1}^{\frac{1}{2}}.
\end{align*}
Then, we have
\begin{align}\label{est on am linear}
\nrm{\pt^ra_m[\Xi_{linear}]}_{N}\lesssim\tau_{q,i}^{-r}\mu_{q,i}^{-N-1}\delta_{q+1}^{\frac{1}{2}}.
\end{align}
Notice that
\begin{align}
	\tG_{q,i+1,m}\tS_{q,i+1,m1}^{(1),(\upsilon)}
	=&\underbrace{4\sgm_{11}\chi^2_{\upsilon}(\mu_{q,i}^{-1}x)d_{q,i+1}^2\left(\frac{f_{i+1}}{|f_{i+1}|}\otimes \frac{f_{i+1}}{|f_{i+1}|}\right)}_{(\tG_{q,i+1,m}\tS_{q,i+1,m1}^{(1),(\upsilon)})^{low}}\label{def of GSG low}\\
	&-\underbrace{4\sgm_{11}|f_{i+1}|^2(f_{i+1}\otimes f_{i+1})\sum_{s:I=(s,\upsilon)}\big((\gamma_{q,i+1}^{I})^2 (\Pi_{I,I}+2)+\sum_{\substack{I'=(s',\tup)\in\mathscr I\\\nrm{I-I'}=1}}\gamma_{q,i+1}^{I}\gamma_{q,i+1}^{I'} \Pi_{I,I'}\big)}_{(\tG_{q,i+1,m}\tS_{q,i+1,m1}^{(1),(\upsilon)})^{high}},\label{def of GSG high}
\end{align}
we have
	\begin{align}
		G_{q,i+1}\tG_{q,i+1,m}\tS_{q,i+1,m1}^{(1),(\upsilon)}
		=&(G_{\ell,i}+\tG_{q,i+1,ac})\tG_{q,i+1,m}\tS_{q,i+1,m1}^{(1),(\upsilon)}+\tG_{q,i+1,p}\tG_{q,i+1,m}\tS_{q,i+1,m1}^{(1),(\upsilon)}\nonumber\\
		=&-4\sgm_{11}\sum_{s:I=(s,\upsilon)}\sum_{\substack{I'=(s',\tup)\in\mathscr I\\\nrm{I'-I}\leqslant1}}\gamma_{q,i+1}^{I}\gamma_{q,i+1}^{I'}|f_{i+1}|^2(((\Gl+\tG_{q,i+1,ac})f_{i+1})\otimes f_{i+1})\Pi_{I,I'}\nonumber\\
		&-4\sgm_{11}\sum_{s:I=(s,\upsilon)}\sum_{\substack{I'=(s',\tup)\in\mathscr I\\\nrm{I'-I}\leqslant1}}\sum_{\substack{I''=(s'',\ttup)\in\mathscr I\\\nrm{I''-I}\leqslant1}}i\gamma_{q,i+1}^{I}\gamma_{q,i+1}^{I'}\gamma_{q,i+1}^{I''}|f_{i+1}|^4(\tilde{f}_{A_{I''}^{\ttup}}\otimes f_{i+1})\Pi_{I,I',I''}\label{def of GtGtS}\\
		&-4\sgm_{11}\sum_{s:I=(s,\upsilon)}\sum_{\substack{I'=(s',\tup)\in\mathscr I\\\nrm{I'-I}\leqslant1}}\sum_{\substack{I''=(s'',\ttup)\in\mathscr I\\\nrm{I''-I}\leqslant1}}\frac{1}{\lambda_{q,i+1}[I'']}\Big(\gamma_{q,i+1}^{I}\gamma_{q,i+1}^{I'}|f_{i+1}|^2(\tilde{f}_{A_{I''}^{\ttup}}\otimes f_{i+1})\nonumber\\
		&\hspace{60pt}\cdot(f_{i+1}\cdot\gamma_{q,i+1,s}^{I''} \Pi_{I,I'}\tPi_{I''}+f_{i+1}\cdot\nb \gamma_{q,i+1,c}^{I''}\Pi_{I,I',I''})\Big),\nonumber
	\end{align}
and then, we could define
\begin{equation*}
	\begin{aligned}
		&(G_{q,i+1}\tG_{q,i+1,m}\tS_{q,i+1,m1}^{(1),(\upsilon)})^{low}\\
		=&-4\sgm_{11}\sum_{s:I=(s,\upsilon)}\sum_{\substack{I'=(s',\tup)\in\mathscr I\\\nrm{I'-I}\leqslant1}}
		\gamma_{q,i+1}^{I}\gamma_{q,i+1}^{I'}|f_{i+1}|^2
		(((\Gl+\tG_{q,i+1,ac})f_{i+1})\otimes f_{i+1})\mPO\Pi_{I,I'}\\
		&-4\sgm_{11}\sum_{s:I=(s,\upsilon)}\sum_{\substack{I'=(s',\tup)\in\mathscr I\\\nrm{I'-I}\leqslant1}}
		\sum_{\substack{I''=(s'',\ttup)\in\mathscr I\\\nrm{I''-I}\leqslant1}}
		i\gamma_{q,i+1}^{I}\gamma_{q,i+1}^{I'}\gamma_{q,i+1}^{I''}|f_{i+1}|^4
		(\tilde{f}_{A_{I''}^{\ttup}}\otimes f_{i+1})\mPO\Pi_{I,I',I''}\\
		&-4\sgm_{11}\sum_{s:I=(s,\upsilon)}\sum_{\substack{I'=(s',\tup)\in\mathscr I\\\nrm{I'-I}\leqslant1}}
		\sum_{\substack{I''=(s'',\ttup)\in\mathscr I\\\nrm{I''-I}\leqslant1}}
		\frac{1}{\lambda_{q,i+1}[I'']}\Big(\gamma_{q,i+1}^{I}\gamma_{q,i+1}^{I'}|f_{i+1}|^2
		(\tilde{f}_{A_{I''}^{\ttup}}\otimes f_{i+1})\\
		&\hspace{60pt}\cdot\big(f_{i+1}\cdot\gamma_{q,i+1,s}^{I''}
		\mPO(\Pi_{I,I'}\tPi_{I''})
		+f_{i+1}\cdot\nb\gamma_{q,i+1,c}^{I''}
		\mPO\Pi_{I,I',I''}\big)\Big).
	\end{aligned}
\end{equation*}
and
\begin{equation*}
	\begin{aligned}
		&(G_{q,i+1}\tG_{q,i+1,m}\tS_{q,i+1,m1}^{(1),(\upsilon)})^{high}\\
		=&-4\sgm_{11}\sum_{s:I=(s,\upsilon)}\sum_{\substack{I'=(s',\tup)\in\mathscr I\\\nrm{I'-I}\leqslant1}}
		\gamma_{q,i+1}^{I}\gamma_{q,i+1}^{I'}|f_{i+1}|^2
		(((\Gl+\tG_{q,i+1,ac})f_{i+1})\otimes f_{i+1})\mPG\Pi_{I,I'}\\
		&-4\sgm_{11}\sum_{s:I=(s,\upsilon)}\sum_{\substack{I'=(s',\tup)\in\mathscr I\\\nrm{I'-I}\leqslant1}}
		\sum_{\substack{I''=(s'',\ttup)\in\mathscr I\\\nrm{I''-I}\leqslant1}}
		i\gamma_{q,i+1}^{I}\gamma_{q,i+1}^{I'}\gamma_{q,i+1}^{I''}|f_{i+1}|^4
		(\tilde{f}_{A_{I''}^{\ttup}}\otimes f_{i+1})\mPG\Pi_{I,I',I''}\\
		&-4\sgm_{11}\sum_{s:I=(s,\upsilon)}\sum_{\substack{I'=(s',\tup)\in\mathscr I\\\nrm{I'-I}\leqslant1}}
		\sum_{\substack{I''=(s'',\ttup)\in\mathscr I\\\nrm{I''-I}\leqslant1}}
		\frac{1}{\lambda_{q,i+1}[I'']}\Big(\gamma_{q,i+1}^{I}\gamma_{q,i+1}^{I'}|f_{i+1}|^2
		(\tilde{f}_{A_{I''}^{\ttup}}\otimes f_{i+1})\\
		&\hspace{60pt}\cdot\big(f_{i+1}\cdot\gamma_{q,i+1,s}^{I''}
		\mPG(\Pi_{I,I'}\tPi_{I''})
		+f_{i+1}\cdot\nb\gamma_{q,i+1,c}^{I''}
		\mPG\Pi_{I,I',I''}\big)\Big).
	\end{aligned}
\end{equation*}
Notice that the spatial zero-frequency components may still oscillate in time. In particular,
\begin{align*}
	\nrm{\pt^r\mPO\Pi_{I,I',I''}}_0
	+\nrm{\pt^r\mPO(\Pi_{I,I'}\tPi_{I''})}_0
	\lesssim\lambda_{q,i+1}^{r}.
\end{align*}
Next, we have
\begin{equation}\label{est on am nolinear 1}
	\begin{aligned}
		&\nrm{\pt^r(G_{q,i+1}\tG_{q,i+1,m}
			\tS_{q,i+1,m1}^{(1),(\upsilon)})^{low}}_{N}\\
		\lesssim&\lambda_{q,i+1}^{r}\mu_{q,i}^{-N}
		\Big(\delta_{q+1}(\varepsilon+\varepsilon\delta_{q+1})
		+\delta_{q+1}^{\frac{3}{2}}
		+(\lambda_{q,i+1}\mu_{q,i})^{-1}\delta_{q+1}^{\frac{3}{2}}
		+(\lambda_{q,i+1}\mu_{q,i})^{-2}\delta_{q+1}^{\frac{3}{2}}\Big)
		\lesssim\lambda_{q,i+1}^{r}\mu_{q,i}^{-N}
		\varepsilon\delta_{q+1},\\
		&\nrm{\pt^ra_m[(G_{q,i+1}\tG_{q,i+1,m}
			\tS_{q,i+1,m1}^{(1),(\upsilon)})^{high}]}_{N}\\
		\lesssim&\mu_{q,i}^{-N-r}
		\Big(\delta_{q+1}(\varepsilon+\varepsilon\delta_{q+1})
		+\delta_{q+1}^{\frac{3}{2}}
		+(\lambda_{q,i+1}\mu_{q,i})^{-1}\delta_{q+1}^{\frac{3}{2}}
		+(\lambda_{q,i+1}\mu_{q,i})^{-2}\delta_{q+1}^{\frac{3}{2}}\Big)
		\lesssim\mu_{q,i}^{-N-r}\varepsilon\delta_{q+1}.
	\end{aligned}
\end{equation}
Similarly, we could calculate
\begin{align*}
	\tG_{q,i+1}\tS_{q,i+1}^{(1),(\upsilon)}\hspace{-3pt}-\tG_{q,i+1,m}\tS_{q,i+1,m1}^{(1),(\upsilon)}
	=&(\tG_{q,i+1,s1}+\tG_{q,i+1,s2}+\tG_{q,i+1,ac})\tS_{q,i+1}^{(1),(\upsilon)}
	\hspace{-3pt}+\tG_{q,i+1,m}(\tS_{q,i+1,m2}^{(1),(\upsilon)}+\tS_{q,i+1,s1}^{(1),(\upsilon)}+\tS_{q,i+1,s2}^{(1),(\upsilon)}+\tS_{q,i+1,ac}^{(1),(\upsilon)}),
\end{align*}
where
\begin{align*}
	&(\tG_{q,i+1,s1}+\tG_{q,i+1,s2}+\tG_{q,i+1,ac})\tS_{q,i+1}^{(1),(\upsilon)}\\
	=&\sum_{\nrm{\tup-\upsilon}\leqslant 1}(\tG_{q,i+1,s1}^{(\tup)}+\tG_{q,i+1,s2}^{(\tup)}+\tG_{q,i+1,ac}^{(\tup)})\tS_{q,i+1}^{(1),(\upsilon)}\\
	=& \sum_{s:I=(s,\upsilon)}\tG_{q,i+1,ac}^{(\upsilon)}((\tS_{q,i+1}^{(1),I,cofm}+\tS_{q,i+1}^{(1),I,cof1})\Pi_{I}+\tS_{q,i+1}^{(1),I,cof2}\tPi_{I})+\tG_{q,i+1,ac}^{(\upsilon)}\tS_{q,i+1,ac}^{(1),(\upsilon)}\\
	&+\sum_{s:I=(s,\upsilon)}\sum_{\substack{I'=(s',\tup)\in\mathscr I\\\nrm{I'-I}\leqslant1}}i\gamma_{q,i+1}^{I'}a_{A_{I'}^{\tup}}(t)f_{i+1}^{\perp}\otimes((\tS_{q,i+1}^{(1),I,cofm}+\tS_{q,i+1}^{(1),I,cof1})f_{i+1}\Pi_{I,I'}+\tS_{q,i+1}^{(1),I,cof2}f_{i+1}\tPi_{I}\Pi_{I'})\\
	&+\sum_{s:I=(s,\upsilon)}\sum_{\substack{I'=(s',\tup)\in\mathscr I\\\nrm{I'-I}\leqslant1}}\frac{1}{\lambda_{q,i+1}[I']}\tilde{f}_{A_{I'}^{\tup}}\otimes((\tS_{q,i+1}^{(1),I,cofm}+\tS_{q,i+1}^{(1),I,cof1})(\gamma_{q,i+1,s}^{I'}\tPi_{I'}\Pi_{I}+\nb \gamma_{q,i+1,c}^{I'}\Pi_{I,I'}))\\
	&+\sum_{s:I=(s,\upsilon)}\sum_{\substack{I'=(s',\tup)\in\mathscr I\\\nrm{I'-I}\leqslant1}}\frac{1}{\lambda_{q,i+1}[I']}\tilde{f}_{A_{I'}^{\tup}}\otimes(\tS_{q,i+1}^{(1),I,cof2}(\gamma_{q,i+1,s}^{I'}\tPi_{I'}\tPi_{I}+\nb \gamma_{q,i+1,c}^{I'}\Pi_{I'}\tPi_{I}))\\
	&+\sum_{s:I=(s,\upsilon)}i\gamma_{q,i+1}^{I}a_{A_I^{\upsilon}}(t)f_{i+1}^{\perp}\otimes(\tS_{q,i+1,ac}^{(1),(\upsilon)}f_{i+1})\Pi_{I}+\sum_{s:I=(s,\upsilon)}\frac{1}{\lambda_{q,i+1}[I]}\tilde{f}_{A_{I}^{\upsilon}}\otimes(\tS_{q,i+1,ac}^{(1),(\upsilon)}(\gamma_{q,i+1,s}^{I}\tPi_{I}+\nb \gamma_{q,i+1,c}^{I}\Pi_{I})),
\end{align*}
and
\begin{align*}
	&\tG_{q,i+1,m}(\tS_{q,i+1,m2}^{(1),(\upsilon)}+\tS_{q,i+1,s1}^{(1),(\upsilon)}+\tS_{q,i+1,s2}^{(1),(\upsilon)}+\tS_{q,i+1,ac}^{(1),(\upsilon)})\\
	=&\sum_{s:I=(s,\upsilon)}\sum_{\substack{I'=(s',\tup)\in\mathscr I\\\nrm{I'-I}\leqslant1}}i\gamma_{q,i+1}^{I'}f_{i+1}\otimes(\tS_{q,i+1}^{(1),I,cof1}f_{i+1})\Pi_{I',I}+\sum_{s:I=(s,\upsilon)}\sum_{\substack{I'=(s',\tup)\in\mathscr I\\\nrm{I'-I}\leqslant1}}i\gamma_{q,i+1}^{I'}f_{i+1}\otimes(\tS_{q,i+1}^{(1),I,cof2}f_{i+1})\Pi_{I'}\tPi_{I}\\
&+\sum_{s:I=(s,\upsilon)}i\gamma_{q,i+1}^{I}f_{i+1}\otimes(\tS_{q,i+1,ac}^{(1),(\upsilon)}f_{i+1})\Pi_{I}.
\end{align*}
Then, we could define
\begin{align*}
	&(\tG_{q,i+1}\tS_{q,i+1}^{(1),(\upsilon)}-\tG_{q,i+1,m}\tS_{q,i+1,m1}^{(1),(\upsilon)})^{low}\\
	:=& \tG_{q,i+1,ac}^{(\upsilon)}\tS_{q,i+1,ac}^{(1),(\upsilon)}+\sum_{s:I=(s,\upsilon)}\sum_{\substack{I'=(s',\tup)\in\mathscr I\\\nrm{I'-I}\leqslant1}}i\gamma_{q,i+1}^{I'}a_{A_{I'}^{\tup}}(t)f_{i+1}^{\perp}\otimes((\tS_{q,i+1}^{(1),I,cofm}+\tS_{q,i+1}^{(1),I,cof1})f_{i+1}\mPO(\Pi_{I,I'}))\\
	&+\sum_{s:I=(s,\upsilon)}\sum_{\substack{I'=(s',\tup)\in\mathscr I\\\nrm{I'-I}\leqslant1}}\frac{1}{\lambda_{q,i+1}[I']}\tilde{f}_{A_{I'}^{\tup}}\otimes((\tS_{q,i+1}^{(1),I,cofm}+\tS_{q,i+1}^{(1),I,cof1})\nb \gamma_{q,i+1,c}^{I'})\mPO(\Pi_{I',I})\\
	&+\sum_{s:I=(s,\upsilon)}\sum_{\substack{I'=(s',\tup)\in\mathscr I\\\nrm{I'-I}\leqslant1}}\frac{1}{\lambda_{q,i+1}[I']}\tilde{f}_{A_{I'}^{\tup}}\otimes(\tS_{q,i+1}^{(1),I,cof2}\gamma_{q,i+1,s}^{I'})\mPO(\tPi_{I'}\tPi_{I})\\
	&+\sum_{s:I=(s,\upsilon)}\sum_{\substack{I'=(s',\tup)\in\mathscr I\\\nrm{I'-I}\leqslant1}}i\gamma_{q,i+1}^{I'}f_{i+1}\otimes(\tS_{q,i+1}^{(1),I,cof1}f_{i+1})\mPO(\Pi_{I',I}),
\end{align*}
and
\begin{align*}
	&(\tG_{q,i+1}\tS_{q,i+1}^{(1),(\upsilon)}-\tG_{q,i+1,m}\tS_{q,i+1,m1}^{(1),(\upsilon)})^{high}\\
	:=& \sum_{s:I=(s,\upsilon)}\tG_{q,i+1,ac}^{(\upsilon)}((\tS_{q,i+1}^{(1),I,cofm}+\tS_{q,i+1}^{(1),I,cof1})\Pi_{I}+\tS_{q,i+1}^{(1),I,cof2}\tPi_{I})\\
	&+\sum_{s:I=(s,\upsilon)}\sum_{\substack{I'=(s',\tup)\in\mathscr I\\\nrm{I'-I}\leqslant1}}i\gamma_{q,i+1}^{I'}a_{A_{I'}^{\tup}}(t)f_{i+1}^{\perp}\otimes((\tS_{q,i+1}^{(1),I,cofm}+\tS_{q,i+1}^{(1),I,cof1})f_{i+1}\mPG(\Pi_{I,I'})+\tS_{q,i+1}^{(1),I,cof2}f_{i+1}\mPG(\tPi_{I}\Pi_{I'}))\\
	&+\sum_{s:I=(s,\upsilon)}\sum_{\substack{I'=(s',\tup)\in\mathscr I\\\nrm{I'-I}\leqslant1}}\frac{1}{\lambda_{q,i+1}[I']}\tilde{f}_{A_{I'}^{\tup}}\otimes((\tS_{q,i+1}^{(1),I,cofm}+\tS_{q,i+1}^{(1),I,cof1})(\gamma_{q,i+1,s}^{I'}\mPG(\tPi_{I'}\Pi_{I})+\nb \gamma_{q,i+1,c}^{I'}\mPG(\Pi_{I',I})))\\
	&+\sum_{s:I=(s,\upsilon)}\sum_{\substack{I'=(s',\tup)\in\mathscr I\\\nrm{I'-I}\leqslant1}}\frac{1}{\lambda_{q,i+1}[I']}\tilde{f}_{A_{I'}^{\tup}}\otimes(\tS_{q,i+1}^{(1),I,cof2}(\gamma_{q,i+1,s}^{I'}\mPG(\tPi_{I'}\tPi_{I})+\nb \gamma_{q,i+1,c}^{I'}\mPG(\Pi_{I'}\tPi_{I})))\\
	&+\sum_{s:I=(s,\upsilon)}i\gamma_{q,i+1}^{I}a_{A_{I}^{\upsilon}}(t)f_{i+1}^{\perp}\otimes(\tS_{q,i+1,ac}^{(1),(\upsilon)}f_{i+1})\Pi_{I}+\sum_{s:I=(s,\upsilon)}\frac{1}{\lambda_{q,i+1}[I]}\tilde{f}_{A_{I}^{\upsilon}}\otimes(\tS_{q,i+1,ac}^{(1),(\upsilon)}(\gamma_{q,i+1,s}^{I}\tPi_{I}+\nb \gamma_{q,i+1,c}^{I}\Pi_{I}))\\
	&+\sum_{s:I=(s,\upsilon)}\sum_{\substack{I'=(s',\tup)\in\mathscr I\\\nrm{I'-I}\leqslant1}}i\gamma_{q,i+1}^{I'}f_{i+1}\otimes(\tS_{q,i+1}^{(1),I,cof1}f_{i+1})\mPG(\Pi_{I',I})\\
	&+\sum_{s:I=(s,\upsilon)}\sum_{\substack{I'=(s',\tup)\in\mathscr I\\\nrm{I'-I}\leqslant1}}i\gamma_{q,i+1}^{I'}f_{i+1}\otimes(\tS_{q,i+1}^{(1),I,cof2}f_{i+1})\mPG(\Pi_{I'}\tPi_{I})+\sum_{s:I=(s,\upsilon)}i\gamma_{q,i+1}^{I}f_{i+1}\otimes(\tS_{q,i+1,ac}^{(1),(\upsilon)}f_{i+1})\Pi_{I}.
\end{align*}
Similarly, we could give
	\begin{align}
		&\nrm{\pt^r(\tG_{q,i+1}\tS_{q,i+1}^{(1),(\upsilon)}-\tG_{q,i+1,m}\tS_{q,i+1,m1}^{(1),(\upsilon)})^{low}}_{N}+\nrm{\pt^ra_m[(\tG_{q,i+1}\tS_{q,i+1}^{(1),(\upsilon)}-\tG_{q,i+1,m}\tS_{q,i+1,m1}^{(1),(\upsilon)})^{high}]}_{N}\nonumber\\
		\lesssim&\sum_{r_1+r_2=r}\sum_{N_1+N_2=N}\nrm{\pt^{r_1}(\tG_{q,i+1,ac}^{(\upsilon)})}_{N_1}\left(\nrm{\pt^{r_2}(\tS_{q,i+1}^{(1),I,cofm})}_{N_2}+\nrm{\pt^{r_2}(\tS_{q,i+1}^{(1),I,cof1})}_{N_2}+\nrm{\pt^{r_2}(\tS_{q,i+1}^{(1),I,cof2})}_{N_2}\right)\nonumber\\
		&+\sum_{r_1+r_2=r}\sum_{N_1+N_2=N}\nrm{\pt^{r_1}(a_{A_I^\upsilon}\gamma_{q,i+1}^{I})}_{N_1}\left(\nrm{\pt^{r_2}(\tS_{q,i+1}^{(1),I,cofm})}_{N_2}+\nrm{\pt^{r_2}(\tS_{q,i+1}^{(1),I,cof1})}_{N_2}+\nrm{\pt^{r_2}(\tS_{q,i+1}^{(1),I,cof2})}_{N_2}\right)\nonumber\\
		&+\lambda_{q,i+1}^{-1}\sum_{r_1+r_2=r}\sum_{N_1+N_2=N}\left(\nrm{\pt^{r_1}(\gamma_{q,i+1,s}^{I})}_{N_1}+\nrm{\pt^{r_1}(\nb\gamma_{q,i+1,c}^{I})}_{N_1}\right)\left(\nrm{\pt^{r_2}(\tS_{q,i+1}^{(1),I,cofm})}_{N_2}+\nrm{\pt^{r_2}(\tS_{q,i+1}^{(1),I,cof1})}_{N_2}\right)\label{est on am nolinear 2}\\
		&+\lambda_{q,i+1}^{-1}\sum_{r_1+r_2=r}\sum_{N_1+N_2=N}\left(\nrm{\pt^{r_1}(\gamma_{q,i+1,s}^{I})}_{N_1}+\nrm{\pt^{r_1}(\nb\gamma_{q,i+1,c}^{I})}_{N_1}\right)\nrm{\pt^{r_2}(\tS_{q,i+1}^{(1),I,cof2})}_{N_2}\nonumber\\
		&+\sum_{r_1+r_2=r}\sum_{N_1+N_2=N}\nrm{\pt^{r_1}(\gamma_{q,i+1}^{I})}_{N_1}\left(\nrm{\pt^{r_2}(\tS_{q,i+1}^{(1),I,cof1})}_{N_2}+\nrm{\pt^{r_2}(\tS_{q,i+1}^{(1),I,cof2})}_{N_2}+\nrm{\pt^{r_2}(\tS_{q,i+1,ac}^{(1),(\upsilon)})}_{N_2}\right)\nonumber\\
		\lesssim&\mu_{q,i}^{-N-r}\left(\varepsilon\delta_{q+1}^{\frac{3}{2}}+\varepsilon\delta_{q+1}+(\lambda_{q,i+1}\mu_{q,i})^{-1}\delta_{q+1}+(\lambda_{q,i+1}\mu_{q,i})^{-2}\delta_{q+1}+\varepsilon\delta_{q+1}\right)
		\lesssim\mu_{q,i}^{-N-r}\varepsilon\delta_{q+1}.\nonumber
	\end{align}
By using Proposition \ref{lm: asym div eq of ae} and \eqref{est on Xi_ac}--\eqref{est on am nolinear 2}, and choosing 
\begin{align*}
	U_P^{high}=&\Div\left(-(\tG_{q,i+1}\tS_{q,i+1}^{(1),(\upsilon)}-\tG_{q,i+1,m}\tS_{q,i+1,m1}^{(1),(\upsilon)})^{high}+(G_{q,i+1}\tG_{q,i+1,m}\tS_{q,i+1,m1}^{(1),(\upsilon)})^{high}\right)\\
	&+\sum_{s:I=(s,\upsilon)}(W_{q,i+1,p}^{I,cof,||}f_{i+1}+W_{q,i+1,p}^{I,cof,\perp}f_{i+1}^{\perp})e^{i\lambda_{q,i+1}\xi_{A_I^{\upsilon},f_{i+1}}}+\sum_{s:I=(s,\upsilon)}(\oW_{q,i+1,p}^{I,cof,||}f_{i+1}+\oW_{q,i+1,p}^{I,cof,\perp}f_{i+1}^{\perp})e^{-i\lambda_{q,i+1}\xi_{A_I^{\upsilon},f_{i+1}}},\\
	U_P^{low}=&\Div\left(-(\tG_{q,i+1}\tS_{q,i+1}^{(1),(\upsilon)}-\tG_{q,i+1,m}\tS_{q,i+1,m1}^{(1),(\upsilon)})^{low}+(G_{q,i+1}\tG_{q,i+1,m}\tS_{q,i+1,m1}^{(1),(\upsilon)})^{low}\right)\\
	&+\pa_{tt}\tu_{q,i+1,L}^{(\upsilon)}-\Div\Big(\tG_{q,i+1,ac}^{(\upsilon)}(\Sgm_{\Gl}+\Rl+c_{q,i})+(\Id+\Gl)\tS_{q,i+1,ac}^{(1),(\upsilon)}\Big),
\end{align*}
we obtain that the principal error $R_P$ satisfies
\begin{align}
	\nrm{\pt^rR_P}_N\lesssim&\lambda_{q,i+1}^{N+r}\left(\lambda_{q,i+1}^{-1}\nrm{U_P^{high}}_0+\mu_{q,i}\nrm{U_P^{low}}_0\right)
	\lesssim\lambda_{q,i+1}^{N+r}(\varepsilon\delta_{q+1}+(\lambda_{q,i+1}\mu_{q,i})^{-1}\delta_{q+1}^{\frac{1}{2}}).\label{est on R_P}
\end{align}
\subsubsection{Estimates on the mediation error}
Recall that 
\begin{align*}
	R_{M}=\sum_{\upsilon}\uR\left[\pa_{tt}\tu_{q,i+1,M}^{(\upsilon)}+\Div (\chi_\upsilon^2(\mu_{q,i}^{-1}x)R_{m,i})\right],
\end{align*}
we could calculate
\begin{equation}\label{est on R_M 1}
	\begin{aligned}
	\nrm{\pt^r(\pa_{tt}\tu_{q,i+1,M}^{(\upsilon)}+\Div (\chi_\upsilon^2(\mu_{q,i}^{-1}x)R_{m,i}))}_N
	\lesssim&\nrm{\pt^{r+2}\tu_{q,i+1,M}^{(\upsilon)}}_N+\nrm{\pt^r\Div (\chi_\upsilon^2(\mu_{q,i}^{-1}x)R_{m,i})}_N\\
	\lesssim&\mu_{q,i}^{-N-r-1}(\varepsilon\delta_{q+1}+\ell_{q,i}^2\lambda_{q,i}^2\delta_{q+1})\lesssim\mu_{q,i}^{-N-r-1}\varepsilon\delta_{q+1}.
	\end{aligned}
\end{equation}
By using Proposition \ref{lm: asym div eq of ae} and \eqref{est on R_M 1}, and choosing 
\begin{align*}
	U_M^{low}=&\pa_{tt}\tu_{q,i+1,M}^{(\upsilon)}+\Div (\chi_\upsilon^2(\mu_{q,i}^{-1}x)R_{m,i}),
\end{align*}
we obtain that the mediation error $R_M$ satisfies
\begin{align}
	\nrm{\pt^rR_M}_N\lesssim&\lambda_{q,i+1}^{N+r}\mu_{q,i}\nrm{U_M^{low}}_0
	\lesssim\lambda_{q,i+1}^{N+r}\varepsilon\delta_{q+1}.\label{est on R_M}
\end{align}
\subsubsection{Estimates on the oscillation error}
Recall that
\begin{align*}
	R_{O1}=&\sum_{\upsilon}\uR\left[\Div\left((\Id+G_{q,i+1})\left(\sgm^*\chi_{\upsilon}^2(\mu_{q,i}^{-1}x)d_{q,i+1}^2\left(\Id-\frac{f_{i+1}}{|f_{i+1}|}\otimes \frac{f_{i+1}}{|f_{i+1}|}\right)-\tG_{q,i+1,m}\tS_{q,i+1,m1}^{(1),(\upsilon)}-\tS_{q,i+1,p}^{(2),m,(\upsilon)}\right)\right)\right],\\
	R_{O2}=&-\tS_{q,i+1,s}^{(2),m}-\tS_{q,i+1}^{(2),c}-\tS_{q,i+1}^{(\geqslant3)},
\end{align*}
From \eqref{def of tS_mm}, we could define
\begin{align*}
	(\tS_{q,i+1,p}^{(2),m,(\upsilon)})^{low}
	&=-\sum_{s:I=(s,\upsilon)}\sum_{\substack{I'=(s',\tup)\in\mathscr I\\
\nrm{I'-I}\leqslant1}}|f_{i+1}|^4\gamma_{q,i+1}^I\gamma_{q,i+1}^{I'}\left(\sgm^*\Id-(2\sgm_2+8\sgm_{12})\frac{f_{i+1}}{|f_{i+1}|}\otimes \frac{f_{i+1}}{|f_{i+1}|}\right)\mPO\Pi_{I,I'}\\
	&=\sum_{s:I=(s,\upsilon)}2|f_{i+1}|^4(\gamma_{q,i+1}^I)^2\left(\sgm^*\Id-(2\sgm_2+8\sgm_{12})\frac{f_{i+1}}{|f_{i+1}|}\otimes \frac{f_{i+1}}{|f_{i+1}|}\right)\\
	&=\sum_{s:I=(s,\upsilon)}\theta_I^2\chi_I^2d_{q,i+1}^2\left(\sgm^*\Id-(2\sgm_2+8\sgm_{12})\frac{f_{i+1}}{|f_{i+1}|}\otimes \frac{f_{i+1}}{|f_{i+1}|}\right)\\
	&=\chi_\upsilon^2(\mu_{q,i}^{-1}x)d_{q,i+1}^2\left(\sgm^*\Id-(2\sgm_2+8\sgm_{12})\frac{f_{i+1}}{|f_{i+1}|}\otimes \frac{f_{i+1}}{|f_{i+1}|}\right),
\end{align*}
and
\begin{align*}
	(\tS_{q,i+1,p}^{(2),m,(\upsilon)})^{high}+(\tG_{q,i+1,m}\tS_{q,i+1,m1}^{(1),(\upsilon)})^{high}
	&=-\sum_{s:I=(s,\upsilon)}\sum_{\substack{I'=(s',\tup)\in\mathscr I\\\nrm{I'-I}\leqslant1}}\sgm^*|f_{i+1}|^4(\gamma_{q,i+1}^I\gamma_{q,i+1}^{I'})\left(\Id-\frac{f_{i+1}}{|f_{i+1}|}\otimes \frac{f_{i+1}}{|f_{i+1}|}\right)\mPG\Pi_{I,I'}.
\end{align*}
where we have used  \eqref{eq:structural-assumptions} and \eqref{def of GSG high}.
Similarly, we could use \eqref{eq:structural-assumptions}, \eqref{def of tS_mm}, and \eqref{def of GSG low} to obtain
\begin{align*}
(\tS_{q,i+1,p}^{(2),m,(\upsilon)})^{low}+(\tG_{q,i+1,m}\tS_{q,i+1,m1}^{(1),(\upsilon)})^{low}=&\sgm^*\chi_\upsilon^2(\mu_{q,i}^{-1}x)d_{q,i+1}^2\left(\Id-\frac{f_{i+1}}{|f_{i+1}|}\otimes \frac{f_{i+1}}{|f_{i+1}|}\right),
\end{align*}
and then,
\begin{align*}
	\sgm^*\chi_{\upsilon}^2(\mu_{q,i}^{-1}x)d_{q,i+1}^2\left(\Id-\frac{f_{i+1}}{|f_{i+1}|}\otimes \frac{f_{i+1}}{|f_{i+1}|}\right)-\tG_{q,i+1,m}\tS_{q,i+1,m1}^{(1),(\upsilon)}-\tS_{q,i+1,p}^{(2),m,(\upsilon)}=-(\tG_{q,i+1,m}\tS_{q,i+1,m1}^{(1),(\upsilon)})^{high}-(\tS_{q,i+1,p}^{(2),m,(\upsilon)})^{high}.
\end{align*}
Next, we calculate
\begin{align*}
	&(\Id+G_{q,i+1})((\tG_{q,i+1,m}\tS_{q,i+1,m1}^{(1),(\upsilon)})^{high}+(\tS_{q,i+1,p}^{(2),m,(\upsilon)})^{high})\\
	=&(\Id+\Gl)((\tG_{q,i+1,m}\tS_{q,i+1,m1}^{(1),(\upsilon)})^{high}+(\tS_{q,i+1,p}^{(2),m,(\upsilon)})^{high})+\tG_{q,i+1}((\tG_{q,i+1,m}\tS_{q,i+1,m1}^{(1),(\upsilon)})^{high}+(\tS_{q,i+1,p}^{(2),m,(\upsilon)})^{high}).
\end{align*}

For the first term, after simple calculation, we have
$$
|f_{i+1}|^2([I]\pm [I'])f_{i+1}=(f_{i+1}\cdot([I]\pm [I'])f_{i+1})f_{i+1},\quad \text{if}\ \nrm{I-I'}\leqslant1,
$$ 
which means
\begin{align*}
\Div\left(\left(\Id-\frac{f_{i+1}}{|f_{i+1}|}\otimes \frac{f_{i+1}}{|f_{i+1}|}\right)\mPG\Pi_{I,I'}\right)=0.
\end{align*}
Then, we have
\begin{align*}
	&\Div((\Id+\Gl)((\tG_{q,i+1,m}\tS_{q,i+1,m1}^{(1),(\upsilon)})^{high}+(\tS_{q,i+1,p}^{(2),m,(\upsilon)})^{high}))\\
    =&-\sum_{s:I=(s,\upsilon)}\sum_{\substack{I'=(s',\tup)\in\mathscr I\\\nrm{I'-I}\leqslant1}}\Div\left(\sgm^*|f_{i+1}|^4(\gamma_{q,i+1}^I\gamma_{q,i+1}^{I'})(\Id+\Gl)\left(\Id-\frac{f_{i+1}}{|f_{i+1}|}\otimes \frac{f_{i+1}}{|f_{i+1}|}\right)\right)\mPG\Pi_{I,I'},
\end{align*}
and 
\begin{align}
&\nrm{\pt^ra_m[\Div((\Id+\Gl)((\tG_{q,i+1,m}\tS_{q,i+1,m1}^{(1),(\upsilon)})^{high}+(\tS_{q,i+1,p}^{(2),m,(\upsilon)})^{high}))]}_N\nonumber\\
\lesssim&\nrm{\pt^{r}((\gamma_{q,i+1}^I\gamma_{q,i+1}^{I'})(\Id+\Gl))}_{N+1}\lesssim\mu_{q,i}^{-N-r-1}\delta_{q+1}.\label{est on R_O1 1}
\end{align}
For the second term, we could obtain
\begin{align*}
	&\tG_{q,i+1}((\tG_{q,i+1,m}\tS_{q,i+1,m1}^{(1),(\upsilon)})^{high}+(\tS_{q,i+1,p}^{(2),m,(\upsilon)})^{high})\\
	=&(\tG_{q,i+1,m}+\tG_{q,i+1,s1}+\tG_{q,i+1,s2}+\tG_{q,i+1,ac})((\tG_{q,i+1,m}\tS_{q,i+1,m1}^{(1),(\upsilon)})^{high}+(\tS_{q,i+1,p}^{(2),m,(\upsilon)})^{high})\\
	=&(\tG_{q,i+1,s2}+\tG_{q,i+1,ac})((\tG_{q,i+1,m}\tS_{q,i+1,m1}^{(1),(\upsilon)})^{high}+(\tS_{q,i+1,p}^{(2),m,(\upsilon)})^{high})\\
	=&-\sum_{s:I=(s,\upsilon)}\sum_{\substack{I'=(s',\tup)\in\mathscr I\\\nrm{I'-I}\leqslant1}}\sum_{\substack{I''=(s'',\ttup)\in\mathscr I\\\nrm{I''-I}\leqslant1}}\frac{1}{\lambda_{q,i+1}[I'']}\sgm^*|f_{i+1}|^4(\gamma_{q,i+1}^I\gamma_{q,i+1}^{I'})\tilde{f}_{A_{I''}^{\ttup}}\otimes\left(\left(\Id-\frac{f_{i+1}}{|f_{i+1}|}\otimes \frac{f_{i+1}}{|f_{i+1}|}\right)\nabla\gamma_{q,i+1,c}^{I''}\right)\Pi_{I''}\mPG\Pi_{I,I'}\\
	&-\sum_{s:I=(s,\upsilon)}\sum_{\substack{I'=(s',\tup)\in\mathscr I\\\nrm{I'-I}\leqslant1}}\sum_{\substack{I''=(s'',\ttup)\in\mathscr I\\\nrm{I''-I}\leqslant1}}\frac{1}{\lambda_{q,i+1}[I'']}\sgm^*|f_{i+1}|^4(\gamma_{q,i+1}^I\gamma_{q,i+1}^{I'})\tilde{f}_{A_{I''}^{\ttup}}\otimes\left(\left(\Id-\frac{f_{i+1}}{|f_{i+1}|}\otimes \frac{f_{i+1}}{|f_{i+1}|}\right)\gamma_{q,i+1,s}^{I''}\right)\tPi_{I''}\mPG\Pi_{I,I'},
\end{align*}
We denote
\begin{align*}
	\mathcal H_{O1}^{(\upsilon)}
	:=&(\tG_{q,i+1,m}\tS_{q,i+1,m1}^{(1),(\upsilon)})^{high}
	+(\tS_{q,i+1,p}^{(2),m,(\upsilon)})^{high}.
\end{align*}
The products in the two sums above may contain spatial zero-frequency components. We decompose
\begin{align*}
	\Pi_{I''}\mPG\Pi_{I,I'}
	=&\mPG\left(\Pi_{I''}\mPG\Pi_{I,I'}\right)
	+\mPO\left(\Pi_{I''}\mPG\Pi_{I,I'}\right),\\
	\tPi_{I''}\mPG\Pi_{I,I'}
	=&\mPG\left(\tPi_{I''}\mPG\Pi_{I,I'}\right)
	+\mPO\left(\tPi_{I''}\mPG\Pi_{I,I'}\right).
\end{align*}
Accordingly, we write
\begin{align*}
	\tG_{q,i+1}\mathcal H_{O1}^{(\upsilon)}
	=&(\tG_{q,i+1}\mathcal H_{O1}^{(\upsilon)})^{high}
	+(\tG_{q,i+1}\mathcal H_{O1}^{(\upsilon)})^{low},
\end{align*}
where the superscripts $high$ and $low$ denote the contributions obtained by replacing the two oscillatory products above by their $\mPG$ and $\mPO$ parts, respectively. The high-frequency part satisfies
\begin{align}
	&\nrm{\pt^ra_m\left[\Div\left((\tG_{q,i+1}\mathcal H_{O1}^{(\upsilon)})^{high}\right)\right]}_N\nonumber\\
	\lesssim&
	\sum_{N_1+N_2=N}\sum_{r_1+r_2=r}
	\nrm{\pt^{r_1}(\gamma_{q,i+1}^I\gamma_{q,i+1}^{I'})}_{N_1}
	\left(
	\nrm{\pt^{r_2}\nb\gamma_{q,i+1,c}^{I''}}_{N_2}
	+\nrm{\pt^{r_2}\gamma_{q,i+1,s}^{I''}}_{N_2}
	\right)
	\lesssim\mu_{q,i}^{-N-r-1}\delta_{q+1}^{\frac32}.
	\label{est on R_O1 2}
\end{align}
For the low-frequency part, the spatial derivatives fall only on the coefficients, while the time derivatives may also fall on the spatial zero-frequency oscillatory factors. Thus, for $0\leqslant N+r\leqslant4$,
\begin{align}
	\nrm{\pt^r\Div\left((\tG_{q,i+1}\mathcal H_{O1}^{(\upsilon)})^{low}\right)}_N
	\lesssim\lambda_{q,i+1}^{r}\mu_{q,i}^{-N}
	\lambda_{q,i+1}^{-1}\mu_{q,i}^{-2}
	\delta_{q+1}^{\frac32}.
	\label{est on R_O1 2 low}
\end{align}
We choose
\begin{align*}
	U_{O1}^{high}
	:=-\Div\left((\Id+G_{\ell,i})\mathcal H_{O1}^{(\upsilon)}
	+(\tG_{q,i+1}\mathcal H_{O1}^{(\upsilon)})^{high}\right),\quad
	U_{O1}^{low}
	:=-\Div\left((\tG_{q,i+1}\mathcal H_{O1}^{(\upsilon)})^{low}\right).
\end{align*}
By using Proposition \ref{lm: asym div eq of ae}, \eqref{est on R_O1 1},
\eqref{est on R_O1 2}, and \eqref{est on R_O1 2 low}, the oscillation error
$R_{O1}$ satisfies
\begin{align}
	\nrm{\pt^rR_{O1}}_N
	\lesssim&\lambda_{q,i+1}^{N+r}
	\left(
	\lambda_{q,i+1}^{-1}\nrm{U_{O1}^{high}}_0
	+\mu_{q,i}\nrm{U_{O1}^{low}}_0
	\right)\nonumber\\
	\lesssim&\lambda_{q,i+1}^{N+r}
	\left(
	(\lambda_{q,i+1}\mu_{q,i})^{-1}\delta_{q+1}
	+(\lambda_{q,i+1}\mu_{q,i})^{-1}\delta_{q+1}^{\frac32}
	\right)
	\lesssim\lambda_{q,i+1}^{N+r}
	(\lambda_{q,i+1}\mu_{q,i})^{-1}\delta_{q+1}.
	\label{est on R_O1}
\end{align}
As for $R_{O2}$, we could directly get
\begin{align*}
		\nrm{\pt^r\tS_{q,i+1,s}^{(2),m}}_N\lesssim \sum_{N_1+N_2=N}\sum_{r_1+r_2=r}\nrm{\pt^{r_1}\tG_{q,i+1}}_{N_1}\left(\nrm{\pt^{r_2}\tG_{q,i+1,s}}_{N_2}+\nrm{\pt^{r_2}\tG_{q,i+1,ac}}_{N_2}\right)\lesssim\lambda_{q,i+1}^{N+r}\varepsilon\delta_{q+1}.
\end{align*}
Moreover, due to the definition of $\tS_{q,i+1}^{(2),c}$ and $\tS_{q,i+1}^{(\geqslant 3)}$ in \eqref{def of tS3}, \eqref{def of tS_2},  and  \eqref{def of tS_2 m}, we have
\begin{align*}
	\nrm{\pt^r\tS_{q,i+1}^{(2),c}}_N\lesssim&\sum_{k=1}^6\sum_{N_1+\cdots+N_{k+2}=N}\sum_{r_1+\cdots+r_{k+2}=r}\prod_{l=1}^k\nrm{\pt^{r_l}\Gl}_{N_l} \nrm{\pt^{r_{k+1}}\tG_{q,i+1}}_{N_{k+1}}\nrm{\pt^{r_{k+2}}\tG_{q,i+1}}_{N_{k+2}}\lesssim\lambda_{q,i+1}^{N+r}\varepsilon\delta_{q+1},\\	\nrm{\pt^r\tS_{q,i+1}^{(\geqslant 3)}}_N\lesssim&\sum_{k=0}^6\sum_{N_1+\cdots+N_{k+3}=N}\sum_{r_1+\cdots+r_{k+3}=r}\prod_{l=1}^k\nrm{\pt^{r_l}\Gl}_{N_l} \nrm{\pt^{r_{k+1}}\tG_{q,i+1}}_{N_{k+1}}\nrm{\pt^{r_{k+2}}\tG_{q,i+1}}_{N_{k+2}}\nrm{\pt^{r_{k+3}}\tG_{q,i+1}}_{N_{k+3}}\lesssim\lambda_{q,i+1}^{N+r}\delta_{q+1}^{\frac{3}{2}}.
\end{align*}
Therefore, the oscillation error $R_{O2}$ satisfies
\begin{align}
	\nrm{\pt^rR_{O2}}_N
	\lesssim
	\nrm{\pt^r\tS_{q,i+1,s}^{(2),m}}_N
	+\nrm{\pt^r\tS_{q,i+1}^{(2),c}}_N
	+\nrm{\pt^r\tS_{q,i+1}^{(\geqslant3)}}_N
	\lesssim
	\lambda_{q,i+1}^{N+r}\varepsilon\delta_{q+1}.
	\label{est on R_O2}
\end{align}
To sum up, we have
\begin{align}
	\nrm{\pt^r\delta R_{q,i+1}}_N\lesssim\nrm{\pt^rR_{P}}_N+\nrm{\pt^rR_{M}}_N+\nrm{\pt^rR_{O1}}_N+\nrm{\pt^rR_{O2}}_N\lesssim\lambda_{q,i+1}^{N+r}\left(\varepsilon\delta_{q+1}+(\lambda_{q,i+1}\mu_{q,i})^{-1}\delta_{q+1}^{\frac{1}{2}}\right).
\end{align} 
\eqref{pp of supp dR} follows immediately from \eqref{pp of supp tu_q,i+1 p}, \eqref{pp of supp tu_q,i+1 ac}, and \eqref{pp of time supp uRlj}.
\subsection{Estimate on energy}
Here, we will give a proposition about the estimates on the total energy function.
\begin{pp}
	The total energy function
	$
	\int_{\T^2}\left(\frac{1}{2}|\pt u_{q,i+1}|^2+\sgm_{G_{q,i+1}}\right)\rd x
	$
	satisfies the following estimate for $t\in\supp_t\te_{q,i}$:
	\begin{align}\label{est on energy}
		\left|\int_{\T^2}\left[\left(\frac{1}{2}|\pt u_{q,i+1}|^2+\sgm_{G_{q,i+1}}\right)-\left(\frac{1}{2}|\pt u_{q,i}|^2+\sgm_{G_{q,i}}\right)-4\sgm_{11}d_{q,i+1}^2\right]\rd x\right|\lesssim_{\sgm}\varepsilon\delta_{q+1}.
	\end{align}
\end{pp}

			\begin{proof}
			For convenience, set
			$
				\pttu_{q,i+1,\mathrm{rem}}:=\pttu_{q,i+1,s}+\pt\tu_{q,i+1,ac},\ \tG_{q,i+1,\mathrm{rem}}:=\tG_{q,i+1,s}+\tG_{q,i+1,ac}.
			$
			By the definitions, we have 
			\begin{align*}
				\int_{\T^2}	\frac{1}{2}|\pt u_{q,i+1}|^2\rd x=&\int_{\T^2}\frac{1}{2}|\pt \ul|^2\rd x
				+\int_{\T^2}\pt \ul\cdot\pt\tu_{q,i+1}\rd x+\int_{\T^2}\frac{1}{2}|\pt \tu_{q,i+1}|^2\rd x,\\
				\int_{\T^2}\frac{1}{2}|\pt\tu_{q,i+1}|^2\rd x=&\int_{\T^2}\frac{1}{2}|\pttu_{q,i+1,m}|^2\rd x+\int_{\T^2}\pttu_{q,i+1,\mathrm{rem}}\cdot\pttu_{q,i+1,m}\rd x+\int_{\T^2}\frac{1}{2}|\pttu_{q,i+1,\mathrm{rem}}|^2\rd x,\\
				\int_{\T^2}\sgm_{G_{q,i+1}}\rd x=&\int_{\T^2}\sgm({j_1|_{G=\Gl},j_2|_{G=\Gl}})\rd x+\int_{\T^2}\tsgm_{q,i+1}^{(1)}\rd x+\int_{\T^2}\tsgm_{q,i+1}^{(2),m}\rd x+\int_{\T^2}\tsgm_{q,i+1}^{(2),c}\rd x+\int_{\T^2}\tsgm_{q,i+1}^{(\geqslant3)}\rd x.
			\end{align*}
			Among these terms, we calculate
			\begin{align*}	
				\int_{\T^2}\frac{1}{2}|\pt\tu_{q,i+1,m}|^2\rd x=&-\sum_{I}\sum_{\substack{I'\in\mathscr I\\\nrm{I'-I}\leqslant1}}\int_{\T^2}2\sgm_{11}|f_{i+1}|^4\gamma_{q,i+1}^I\gamma_{q,i+1}^{I'}\Pi_{I,I'}\rd x\\
				=&-\sum_{I}\sum_{\substack{I'\in\mathscr I\\\nrm{I'-I}\leqslant1}}\int_{\T^2}2\sgm_{11}|f_{i+1}|^4\gamma_{q,i+1}^I\gamma_{q,i+1}^{I'}\mPG\Pi_{I,I'}\rd x+\int_{\T^2}2\sgm_{11}d_{q,i+1}^2\rd x,\\
				\int_{\T^2}\tsgm_{q,i+1}^{(2),m}\rd x=&\int_{\T^2}\left(2(\sgm_{2}+\sgm_{11})(\tr(\tG_{q,i+1}))^2-\sgm_2\tr(\tG_{q,i+1}\tG_{q,i+1}^{\top}+\tG_{q,i+1}^2)\right)\rd x,\\
				=&\int_{\T^2}\left(4(\sgm_{2}+\sgm_{11})\tr(\tG_{q,i+1,m})\tr(\tG_{q,i+1,\mathrm{rem}})-2\sgm_2\tr(\tG_{q,i+1,m}\tG_{q,i+1,\mathrm{rem}}^{\top}+\tG_{q,i+1,m}\tG_{q,i+1,\mathrm{rem}})\right)\rd x\\
				&+\int_{\T^2}\left(2(\sgm_{2}+\sgm_{11})(\tr(\tG_{q,i+1,\mathrm{rem}}))^2-\sgm_2\tr(\tG_{q,i+1,\mathrm{rem}}\tG_{q,i+1,\mathrm{rem}}^{\top}+\tG_{q,i+1,\mathrm{rem}}^2)\right)\rd x\\
				&-\sum_{I}\sum_{\substack{I'\in\mathscr I\\\nrm{I'-I}\leqslant1}}\int_{\T^2}2\sgm_{11}|f_{i+1}|^4\gamma_{q,i+1}^I\gamma_{q,i+1}^{I'}\mPG\Pi_{I,I'}\rd x+\int_{\T^2}2\sgm_{11}d_{q,i+1}^2\rd x,
			\end{align*}
			\begin{align*}
				\int_{\T^2}\pt \ul\cdot\pt\tu_{q,i+1}\rd x
				=&\int_{\T^2}\left(i\sum_{I}\tilde{f}_{A_I^{\upsilon}}\cdot\pt \ul\gamma_{q,i+1}^{I}\Pi_{I}+\sum_{I}\frac{1}{\lambda_{q,i+1}[I]}\left(\pt(\gamma_{q,i+1,c}^I\tilde{f}_{A_I^{\upsilon}}\Pi_{I})+\pt(\gamma_{q,i+1}^{I}\tilde{f}_{A_I^{\upsilon}})\tPi_{I}\right)\cdot\pt \ul\right)\rd x\\
				&+\int_{\T^2}\pt \ul\cdot\pt\tu_{q,i+1,ac}\rd x,\\
				\int_{\T^2}\tsgm_{q,i+1}^{(1)}\rd x=&\int_{\T^2}\left(\sgm_2\tj_2^{(1)}+\sgm_{11}\tj_1^{(0)}\tj_1^{(1)}+\sgm_{22}\tj_2^{(0)}\tj_2^{(1)}+\sgm_{12}(\tj_1^{(0)}\tj_2^{(1)}+\tj_1^{(1)}\tj_2^{(0)})+\frac{\sgm_{111}}{2}(\tj_1^{(0)})^2\tj_1^{(1)}\right)\rd x.
			\end{align*}
			Using the finite overlap of the supports and Lemma \ref{est on int operator}, we obtain
			\begin{align*}
				&\left|\sum_{I}\sum_{\substack{I'\in\mathscr I\\\nrm{I'-I}\leqslant1}}\int_{\T^2}2\sgm_{11}|f_{i+1}|^4\gamma_{q,i+1}^I\gamma_{q,i+1}^{I'}\mPG\Pi_{I,I'}\rd x\right|\lesssim\frac{\nrm{\gamma_{q,i+1}^I\gamma_{q,i+1}^{I'}}_{10}}{\lambda_{q,i+1}^{10}}\lesssim(\lambda_{q,i+1}\mu_{q,i})^{-10}\delta_{q+1},
			\end{align*}
			and 
			\begin{align*}
				\left|\int_{\T^2}\pt \ul\cdot\pt\tu_{q,i+1}\rd x\right|\lesssim&\frac{\nrm{\gamma_{q,i+1}^I\pt\ul}_{10}}{\lambda_{q,i+1}^{10}}+\frac{\lambda_{q,i+1}\nrm{\gamma_{q,i+1,c}^I\tilde{f}_{A_I^{\upsilon}}\pt\ul}_{10}}{\lambda_{q,i+1}^{11}}\\
				&+\frac{\nrm{\pt(\gamma_{q,i+1,c}^I\tilde{f}_{A_I^{\upsilon}})\pt\ul}_{10}+\nrm{\pt(\gamma_{q,i+1}^I\tilde{f}_{A_I^{\upsilon}})\pt\ul}_{10}}{\lambda_{q,i+1}^{11}}
				+\nrm{\pt\ul}_0\nrm{\pt\tu_{q,i+1,ac}}_0\\
				\lesssim&(\lambda_{q,i+1}\mu_{q,i})^{-10}\varepsilon\delta_{q+1}^{\frac{1}{2}}+\varepsilon^2\delta_{q+1}\lesssim\varepsilon\delta_{q+1},\\
				\left|\int_{\T^2}\tsgm_{q,i+1}^{(1)}\rd x\right|\lesssim&\lambda_{q,i+1}^{-10}\sum_{k=1}^7\sum_{N_1+\cdots+N_{k}+N_{k+1}=10}\prod_{l=1}^k\nrm{\Gl}_{N_l}\nrm{\gamma_{q,i+1}^I}_{N_{k+1}}+\sum_{k=0}^7\prod_{l=1}^k\nrm{\Gl}_{0}\nrm{\tG_{q,i+1,ac}}_0\\
				\lesssim&(\lambda_{q,i+1}\mu_{q,i})^{-10}\varepsilon\delta_{q+1}^{\frac{1}{2}}+\varepsilon\delta_{q+1}\lesssim\varepsilon\delta_{q+1}.
			\end{align*}
			Moreover, by \eqref{est on pt tu_p}, \eqref{est on tG_p}, and \eqref{est on tVuG ac}, we directly obtain
			\begin{align*}
				&\left|\int_{\T^2}\pttu_{q,i+1,\mathrm{rem}}\cdot\pttu_{q,i+1,m}\rd x\right|+\left|\int_{\T^2}\frac{1}{2}|\pttu_{q,i+1,\mathrm{rem}}|^2\rd x\right|\lesssim\nrm{\pttu_{q,i+1,\mathrm{rem}}}_0\nrm{\pttu_{q,i+1,m}}_0+\nrm{\pttu_{q,i+1,\mathrm{rem}}}_0^2
				\lesssim\varepsilon\delta_{q+1},\\
				&\left|\int_{\T^2}\left(4(\sgm_{2}+\sgm_{11})\tr(\tG_{q,i+1,m})\tr(\tG_{q,i+1,\mathrm{rem}})-2\sgm_2\tr(\tG_{q,i+1,m}\tG_{q,i+1,\mathrm{rem}}^{\top}+\tG_{q,i+1,m}\tG_{q,i+1,\mathrm{rem}})\right)\rd x\right|
				\lesssim\varepsilon\delta_{q+1},\\
				&\left|\int_{\T^2}\left(2(\sgm_{2}+\sgm_{11})(\tr(\tG_{q,i+1,\mathrm{rem}}))^2-\sgm_2\tr(\tG_{q,i+1,\mathrm{rem}}\tG_{q,i+1,\mathrm{rem}}^{\top}+\tG_{q,i+1,\mathrm{rem}}^2)\right)\rd x\right|
				\lesssim\nrm{\tG_{q,i+1,\mathrm{rem}}}_0^2
				\lesssim\varepsilon^2\delta_{q+1},\\
				&\left|\int_{\T^2}\tsgm_{q,i+1}^{(2),c}\rd x\right|\lesssim\sum_{k=1}^6\prod_{l=1}^k\nrm{\Gl}_{0}\nrm{\tG_{q,i+1}}_{0}^2\lesssim\varepsilon\delta_{q+1},\\
				&\left|\int_{\T^2}\tsgm_{q,i+1}^{(\geqslant3)}\rd x\right|\lesssim\sum_{k=0}^6\prod_{l=1}^k\nrm{\Gl}_{0}\nrm{\tG_{q,i+1}}_{0}^3	\lesssim\delta_{q+1}^{\frac{3}{2}}.
			\end{align*}
			Combining the above estimates with \eqref{est on u_qi^2-u_li^2} and \eqref{est on sgm_qi-sgm_li}, we obtain
			\begin{align*}
				&\left|\int_{\T^2}\left(\frac{1}{2}|\pt u_{q,i+1}|^2+\sgm_{G_{q,i+1}}\right)\rd x-\int_{\T^2}\left(\frac{1}{2}|\pt u_{q,i}|^2+\sgm_{G_{q,i}}\right)\rd x-\int_{\T^2}4\sgm_{11}d_{q,i+1}^2\rd x\right|\\
				\leqslant&\left|\int_{\T^2}\frac{1}{2}(|\pt \ul|^2-|\pt u_{q,i}|^2)\rd x\right|
				+\left|\int_{\T^2}\pt \ul\cdot\pt\tu_{q,i+1}\rd x\right|+\left|\int_{\T^2}\pttu_{q,i+1,\mathrm{rem}}\cdot\pttu_{q,i+1,m}\rd x\right|+\left|\int_{\T^2}\frac{1}{2}|\pttu_{q,i+1,\mathrm{rem}}|^2\rd x\right|\\
				&+\left|\int_{\T^2}\left(\sgm({j_1|_{G=\Gl},j_2|_{G=\Gl}})-\sgm_{G_{q,i}}\right)\rd x\right|+\left|\int_{\T^2}\tsgm_{q,i+1}^{(1)}\rd x\right|+\left|\int_{\T^2}\tsgm_{q,i+1}^{(2),c}\rd x\right|\\
				&+\left|\int_{\T^2}\left(\frac{1}{2}|\pttu_{q,i+1,m}|^2-2\sgm_{11}d_{q,i+1}^2\right)\rd x\right|+\left|\int_{\T^2}\left(\tsgm_{q,i+1}^{(2),m}-2\sgm_{11}d_{q,i+1}^2\right)\rd x\right|+\left|
				\int_{\T^2}
				\sum_{k\geqslant3}
				\tsgm_{q,i+1}^{(k)}
				\rd x
				\right|\\
				\lesssim_{\sgm}&\varepsilon\delta_{q+1},
			\end{align*}
			which completes the proof.
		\end{proof}

\section{Proof of the main theorem}\label{Proof of the main theorem}
\subsection{Proof of Proposition \ref{Proposition 1}}
Let $\varepsilon_0^*$ be as in Proposition \ref{pp of Reynolds error} for $b>5$. Given an approximate solution $(u_q,c_q,R_q)$ on $\mcI^{q,-1}\times\T^2$, we construct $\tu_{q,i+1}$ as in Section \ref{Definition of the perturbation}, set
$
u_{q,i+1}=u_{\ell,i}+\tu_{q,i+1},\ i=0,1,2,
$
and define $u_{q+1}=u_{q,3}$. This produces a new Reynolds stress $R_{q+1}$ satisfying the estimates in Proposition \ref{pp of Reynolds error}. It remains to prove \eqref{Induction estimate}--\eqref{est on R_qi} and to verify that $(u_{q+1},c_{q+1},R_{q+1})$ satisfies \eqref{pp of supp q}--\eqref{est on R_q} at the $(q+1)$-st stage.

We denote the implicit constants in Proposition \ref{est on perturbation} for $0\leqslant N+r\leqslant3$, which do not depend on $\oM$, by $M_0$ and set $\oM = 60M_0$. Then, for sufficiently small  $\varepsilon<\varepsilon_0^*$, if we set $\nrm{\cdot}_N = \nrm{\cdot}_{C^0(\mcI^{q+1,-1}; C^N(\T^2))}$,  by \eqref{est on u_q low}, \eqref{est on u_q high}, \eqref{est on u_qi-u_li} and \eqref{est on nb uq i+1}, we have for $2\leqslant N+r\leqslant3$,
\begin{align*}
	\nrm{u_{q+1}}_0&\leqslant\nrm{u_{q}}_0+\sum_{i=0}^2\nrm{u_{q,i}-\ul}_0+\sum_{i=1}^3\nrm{\tu_{q,i}}_0\leqslant 2\varepsilon-\dlt_{q}^{\frac{1}{2}}+\sum_{i=0}^2\ell_{q,i}^{3}\oM\lambda_{q,i}^2\delta_{q,i}^{\frac{1}{2}}+2M_0\sum_{i=1}^3\lambda_{q}^{-1}\dlt_{q+1}^{\frac{1}{2}}\leqslant 2\varepsilon-\dlt_{q+1}^{\frac{1}{2}}, \\
	\nrm{\nb u_{q+1}}_0&\leqslant\nrm{\nb u_{q}}_0+\sum_{i=0}^2\nrm{\nabla(u_{q,i}-\ul)}_0+\sum_{i=1}^3\nrm{\nb  \tu_{q,i}}_0\leqslant 2\varepsilon-\dlt_{q}^{\frac{1}{2}}+\sum_{i=0}^2\ell_{q,i}^{2}\oM\lambda_{q,i}^2\delta_{q,i}^{\frac{1}{2}}+2M_0\sum_{i=1}^3\dlt_{q+1}^{\frac{1}{2}}\leqslant 2\varepsilon-\dlt_{q+1}^{\frac{1}{2}}, \\
	\nrm{\pa_t u_{q+1}}_0&\leqslant\nrm{\pa_t u_{q}}_0+\sum_{i=0}^2\nrm{\pt(u_{q,i}-\ul)}_0+\sum_{i=1}^3\nrm{\pa_t \tu_{q,i}}_0\leqslant 2\varepsilon-\dlt_{q}^{\frac{1}{2}}+\sum_{i=0}^2\ell_{q,i}^{2}\oM\lambda_{q,i}^2\delta_{q,i}^{\frac{1}{2}}+2M_0\sum_{i=1}^3\dlt_{q+1}^{\frac{1}{2}}\leqslant 2\varepsilon-\dlt_{q+1}^{\frac{1}{2}}, \\
	\nrm{\pa_{t}^ru_{q+1}}_N&\leqslant\nrm{\pa_{t}^ru_{q}}_N+\sum_{i=0}^2\nrm{\pt^r(u_{q,i}-\ul)}_N+\sum_{i=1}^3\nrm{\pa_{t}^r\tu_{q,i}}_N\leqslant \oM\lambda_{q}^{N+r-1}\dlt_{q}^{\frac{1}{2}}+\sum_{i=0}^2\ell_{q,i}^{3-N-r}\oM\lambda_{q,i}^2\delta_{q,i}^{\frac{1}{2}}+\sum_{i=1}^3\frac{1}{60}\oM\lambda_{q,i}^{N+r-1}\dlt_{q+1}^{\frac{1}{2}}\\
	&\hspace{195pt}\leqslant \oM\lambda_{q+1}^{N+r-1}\dlt_{q+1}^{\frac{1}{2}},
\end{align*}
and
\begin{equation}
	\begin{aligned}
		&\lambda_{q}\nrm{u_{q+1}-u_q}_0+\sum_{1\leqslant N+r\leqslant3}\lambda_{q+1}^{1-N-r}\nrm{\pa_{t}^r(u_{q+1}-u_q)}_N\\
		\leqslant&\sum_{i=1}^3\lambda_{q}(\nrm{\tu_{q,i}}_0+\nrm{u_{\ell,i-1}-u_{q,i-1}}_0)+\sum_{1\leqslant N+r\leqslant3}\sum_{i=1}^3\lambda_{q,i}^{1-N-r}(\nrm{\pa_{t}^r\tu_{q,i}}_N+\nrm{\pa_{t}^r(u_{\ell,i-1}-u_{q,i-1})}_N)\\
		\leqslant&30M_0\dlt_{q+1}^{\frac{1}{2}}\leqslant \oM\dlt_{q+1}^{\frac{1}{2}}.
	\end{aligned}
\end{equation}

Similarly, it is easy to check  \eqref{est on u_qi low}--\eqref{est on u_qi high} at the $i$-th substep.
The estimates on $R_{q+1}$ can be obtained from Proposition \ref{pp of Reynolds error}. We first need some calculation. From \eqref{Geometric lemma 1} and \eqref{def of R_q,i+1}, we can get
\begin{align*}
	R_{q,1}=&\sgm^*\te_{q,0}\Id^{<1>}+R_{\ell,0}-\sgm^*\te_{q,0}\Gamma_{f_{1}}^2(\Id+(\sgm^*\te_{q,0})^{-1}R_{\ell,0})\left(\Id-\frac{f_{1}}{|f_{1}|}\otimes \frac{f_{1}}{|f_{1}|}\right)+\delta R_{q,1}\\
	=&\sgm^*\te_{q,0}\Id^{<1>}+R_{\ell,0}-(\sgm^*\te_{q,0}\Id+R_{\ell,0})^{<1>}+\delta R_{q,1}\\
	=&\sum_{k=2}^3R_{\ell,0}^{<k>}+\delta R_{q,1},\\
	R_{\ell,1}=&\PLN{1}\ULN{1}\left(\sum_{k=2}^3R_{\ell,0}^{<k>}+\delta R_{q,1}\right)=\sum_{k=2}^3(\PLN{1}\ULN{1}R_{\ell,0})^{<k>}+\PLN{1}\ULN{1}\delta R_{q,1},
\end{align*}
and then
\begin{align*}
	R_{q,2}=&\sgm^*\te_{q,1}\Id^{<2>}+R_{\ell,1}-\sgm^*\te_{q,1}\Gamma_{f_{2}}^2(\Id+(\sgm^*\te_{q,1})^{-1}R_{\ell,1})\left(\Id-\frac{f_{2}}{|f_{2}|}\otimes \frac{f_{2}}{|f_{2}|}\right)+\delta R_{q,2}\\
	=&\sgm^*\te_{q,1}\Id^{<2>}+\sum_{k=2}^3(\PLN{1}\ULN{1}R_{\ell,0})^{<k>}+\PLN{1}\ULN{1}\delta R_{q,1}\\
	&-(\sgm^*\te_{q,1}\Id+\sum_{k=2}^3(\PLN{1}\ULN{1}R_{\ell,0})^{<k>}+\PLN{1}\ULN{1}\delta R_{q,1})^{<2>}+\delta R_{q,2}\\
	=&(\PLN{1}\ULN{1}R_{\ell,0})^{<3>}+(\PLN{1}\ULN{1}\delta R_{q,1})^{<1>}+(\PLN{1}\ULN{1}\delta R_{q,1})^{<3>}+\delta R_{q,2},\\
	R_{q,3}=&\sgm^*\te_{q,2}\Id^{<3>}+R_{\ell,2}-\sgm^*\te_{q,2}\Gamma_{f_{3}}^2(\Id+(\sgm^*\te_{q,2})^{-1}R_{\ell,2})\left(\Id-\frac{f_{3}}{|f_{3}|}\otimes \frac{f_{3}}{|f_{3}|}\right)+\delta R_{q,3}\\
	=&\sgm^*\te_{q,2}\Id^{<3>}+\PLN{2}\ULN{2}\left((\PLN{1}\ULN{1}R_{\ell,0})^{<3>}+(\PLN{1}\ULN{1}\delta R_{q,1})^{<1>}+(\PLN{1}\ULN{1}\delta R_{q,1})^{<3>}+\delta R_{q,2}\right)\\
	&-\left(\sgm^*\te_{q,2}\Id+\PLN{2}\ULN{2}\left((\PLN{1}\ULN{1}R_{\ell,0})^{<3>}+(\PLN{1}\ULN{1}\delta R_{q,1})^{<1>}+(\PLN{1}\ULN{1}\delta R_{q,1})^{<3>}+\delta R_{q,2}\right)\right)^{<3>}\\
	&+\delta R_{q,3}\\
	=&(\PLN{2}\ULN{2}\PLN{1}\ULN{1}\delta R_{q,1})^{<1>}+(\PLN{2}\ULN{2}\delta R_{q,2})^{<1>}+(\PLN{2}\ULN{2}\delta R_{q,2})^{<2>}+\delta R_{q,3},
\end{align*}
where we have used that for a symmetric matrix $K$, it holds $\PL\UL K^{<k>}=(\PL\UL K)^{<k>}$, and
\begin{equation}
	(K^{<i>})^{<k>}=\left\lbrace
	\begin{aligned}
		&K^{<i>},&& k=i,\\
		&0,&& k\neq i.
	\end{aligned}\right.
\end{equation}
\eqref{est on dR_qi} follows from \eqref{est on delta R_q i+1} by choosing $\varepsilon<\varepsilon_0^*$. The proof is similar for \eqref{est on R_q} and \eqref{est on R_qi}. Here, we only give the following estimates for $R_{q+1}=R_{q,3}$:
\begin{equation}\label{est on R_q+1}
	\begin{aligned}
		\nrm{\pt^rR_{q+1}}_N
		\leqslant{}&
		\nrm{\pt^r(\PLN{2}\ULN{2}\PLN{1}\ULN{1}\delta R_{q,1})^{<1>}}_N
		+\nrm{\pt^r(\PLN{2}\ULN{2}\delta R_{q,2})^{<1>}}_N\\
		&+\nrm{\pt^r(\PLN{2}\ULN{2}\delta R_{q,2})^{<2>}}_N
		+\nrm{\pt^r\delta R_{q,3}}_N\\
		\leqslant{}&
		3\nrm{\pt^r\PLN{2}\ULN{2}\PLN{1}\ULN{1}\delta R_{q,1}}_N
		+6\nrm{\pt^r\PLN{2}\ULN{2}\delta R_{q,2}}_N
		+3\nrm{\pt^r\delta R_{q,3}}_N\\
		\leqslant{}&
		12\cdot\frac{1}{24}\lambda_{q+1}^{N+r}\tilde{c}_0^4\delta_{q+2}
		\leqslant\lambda_{q+1}^{N+r}\tilde{c}_0^4\delta_{q+2}.
	\end{aligned}
\end{equation}
Moreover, from \eqref{pp of supp l,i}, \eqref{pp of supp tu_q,i+1 p}, \eqref{pp of supp tu_q,i+1 ac}, and \eqref{pp of supp dR}, we could obtain \eqref{pp of supp q} and \eqref{pp of supp q,i}.

Finally, we only need to prove \eqref{Induction energy estimate}.
Noting that 
\begin{align}\label{main part of energy difference}
\int_{\T^2}4\sgm_{11}d_{q,i+1}^2\rd x=\int_{\T^2}4\sgm_{11}\tr\left(d_{q,i+1}^2\left(\Id-\frac{f_{i+1}}{|f_{i+1}|}\otimes \frac{f_{i+1}}{|f_{i+1}|}\right)\right)\rd x
=\int_{\T^2}4\sgm_{11}\tr((\te_{q,i}\Id+(\sgm^{*})^{-1}R_{\ell,i})^{<i+1>})\rd x,
\end{align}
by using \eqref{Geometric lemma 1}  and \eqref{est on R_li low}, we have for $t\in \set{t|\ \te_{q,i}(t)=\delta_{q+1}}$,
\begin{align*}
	\int_{\T^2}4\sgm_{11}d_{q,i+1}^2\rd x&\leqslant16\pi^2\sgm_{11}\left(\delta_{q+1}+4|\sgm^*|^{-1}\nrm{\Rl}_0\right)\leqslant 32\pi^2\sgm_{11}\delta_{q+1},\\
	\int_{\T^2}4\sgm_{11}d_{q,i+1}^2\rd x&\geqslant16\pi^2\sgm_{11}\left(\frac{1}{2}\delta_{q+1}-4|\sgm^*|^{-1}\nrm{\Rl}_0\right)\geqslant 4\pi^2\sgm_{11}\delta_{q+1}.
\end{align*}
By \eqref{est on energy}, for sufficiently small $\varepsilon$, we obtain, for each $i=0,1,2$ and every $t$ such that $\te_{q,i}(t)=\delta_{q+1}$,
\begin{align*}
	\int_{\T^2}\left(\frac{1}{2}|\pt u_{q,i+1}|^2+\sgm_{G_{q,i+1}}\right)\rd x-\int_{\T^2}\left(\frac{1}{2}|\pt u_{q,i}|^2+\sgm_{G_{q,i}}\right)\rd x
	&\geqslant\int_{\T^2}4\sgm_{11}d_{q,i+1}^2\rd x-C_{\sgm}\varepsilon\delta_{q+1}
	\geqslant\pi^2\sgm_{11}\delta_{q+1},\\
	\int_{\T^2}\left(\frac{1}{2}|\pt u_{q,i+1}|^2+\sgm_{G_{q,i+1}}\right)\rd x-\int_{\T^2}\left(\frac{1}{2}|\pt u_{q,i}|^2+\sgm_{G_{q,i}}\right)\rd x
	&\leqslant\int_{\T^2}4\sgm_{11}d_{q,i+1}^2\rd x+C_{\sgm}\varepsilon\delta_{q+1}
	\leqslant40\pi^2\sgm_{11}\delta_{q+1}.
\end{align*}
Moreover, using the identity
\begin{align*}
	\int_{\T^2}\left(\frac{1}{2}|\pt u_{q+1}|^2+\sgm_{G_{q+1}}\right)\rd x-\int_{\T^2}\left(\frac{1}{2}|\pt u_q|^2+\sgm_{G_q}\right)\rd x=\sum_{i=0}^2\left[\int_{\T^2}\left(\frac{1}{2}|\pt u_{q,i+1}|^2+\sgm_{G_{q,i+1}}\right)\rd x-\int_{\T^2}\left(\frac{1}{2}|\pt u_{q,i}|^2+\sgm_{G_{q,i}}\right)\rd x\right],
\end{align*}
and noting that $\te_{q,i}(t)=\delta_{q+1}$ for every $i=0,1,2$ and every $t\in[T/2-T_{q,0},T+\tau_{q+1,-1}]$, we obtain
\begin{align*}
	3\pi^2\sgm_{11}\delta_{q+1}
	\leqslant\int_{\T^2}\left(\frac{1}{2}|\pt u_{q+1}|^2+\sgm_{G_{q+1}}\right)\rd x-\int_{\T^2}\left(\frac{1}{2}|\pt u_q|^2+\sgm_{G_q}\right)\rd x
	\leqslant120\pi^2\sgm_{11}\delta_{q+1}
\end{align*}
for every $t\in[T/2-T_{q,0},T+\tau_{q+1,-1}]$.

\subsection{Proof of Theorem \ref{Thm 1}}\label{Proof of the theorems}
Fix $T>0$. Since $b>5$, it follows from the definitions of
$\lambda_{q,k}$, $\delta_q$, $\tau_{q,k}$, and $\ell_{q,k}$ that
$$
\sum_{q=0}^{\infty}\sum_{k=0}^{2}
\left(
3\tau_{q,k}
+\ell_{q,k}
+\lambda_q^{-1}
\right)
\longrightarrow0
\qquad
\text{as }\varepsilon\to0.
$$
Therefore, by choosing $\varepsilon>0$ sufficiently small at the
beginning of the construction, depending also on $T$, we may assume
that
\begin{equation}\label{eq:total-time-buffer}
	\sum_{q=0}^{\infty}\sum_{k=0}^{2}
	\left(
	3\tau_{q,k}
	+\ell_{q,k}
	+\lambda_q^{-1}
	\right)
	\leqslant
	\frac{T}{4}.
\end{equation}
We choose $b$, $\oM$, $\tilde{c}_0$, and $\varepsilon_1$ based on Proposition \ref{Proposition 1}. We choose $(u_0,c_0,R_0)\equiv(0,c_0(t),0)$ on $\mcI^{0,-1}\times \T^2 $ which solves \eqref{approximate  elstro dynamic}  and satisfies \eqref{pp of supp q}--\eqref{est on R_q}.

We can apply Proposition \ref{Proposition 1} iteratively to produce a sequence of approximate solutions $(u_q, c_{q},R_q)$, which solve \eqref{approximate  elstro dynamic} and satisfy \eqref{pp of supp q}--\eqref{Induction energy estimate}.

First, we prove that $(u_q)$ is a Cauchy sequence in $C^0([0,T];C^1(\T^2))$. For any $q\leqslant q'$, we have
\begin{align*}
	\nrm{u_{q'} - u_q}_{C^0([0, T]; C^{1}(\T^2))}
	\leqslant \sum_{l=1}^{q'-q} \nrm{u_{q+l} - u_{q+l-1}}_{C^0([0, T]; C^{1}(\T^2))} \leqslant\oM\sum_{l=1}^{q'-q}\delta_{q+l}^{\frac12}.
\end{align*}
Since
$
\sum_{j=q+1}^{\infty}\delta_j^{\frac12}\longrightarrow0
$
as $q\to\infty$, the sequence $(u_q)$ is Cauchy in $C^0([0,T];C^1(\T^2))$. Hence, it converges to some
$
u\in C^0([0,T];C^1(\T^2)).
$
Likewise, \eqref{Induction estimate} implies that $(\pa_tu_q)$ is Cauchy in $C^0([0,T];C^0(\T^2))$, and hence
$
u\in C^1([0,T];C^0(\T^2)).
$
Therefore,
$
u\in C^1([0,T]\times\T^2).
$
Moreover, the preceding convergences imply that $u_q\to u$ in $C^1([0,T]\times\T^2)$. Consequently, we have $G_q\to G$ and $(\Id+G_q)\Sgm_{G_q}\to(\Id+G)\Sgm_G$ in $C^0([0,T]\times\T^2)$. On the other hand, by \eqref{est on R_q} and the definition of $c_q$, we have
$$
(\Id+G_q)(c_q+R_q)\longrightarrow0
\quad\text{in }C^0([0,T]\times\T^2).
$$
Moreover, by \eqref{pp of supp q} and
\eqref{eq:total-time-buffer}, we have
$$
\supp_tu_q\cap[0,T]
\subseteq
\left[
\frac{T}{2}
-
\sum_{j=0}^{\infty}\sum_{k=0}^{2}
\left(
3\tau_{j,k}
+\ell_{j,k}
+\lambda_j^{-1}
\right),
T
\right]
\subseteq
\left[\frac{T}{4},T\right].
$$
Thus, each $u_q$ vanishes in a neighborhood of $t=0$. In particular,
$
u_q(0,\cdot)=0,\
\pa_tu_q(0,\cdot)=0.
$
Passing to the limit, we also obtain
$
u(0,\cdot)=0,\
\pa_tu(0,\cdot)=0.
$

For any $\eta\in C_c^\infty([0,T)\times\T^2;\R^2)$, testing \eqref{approximate  elstro dynamic} against $\eta$ and using $\pa_tu_q(0,\cdot)=0$ gives
\begin{align*}
	\int_0^T\int_{\T^2}\left(\pa_tu_q\cdot\pa_t\eta-(\Id+G_q)\Sgm_{G_q}:\nb\eta\right)\rd x\rd t
	=\int_0^T\int_{\T^2}(\Id+G_q)(c_q+R_q):\nb\eta\rd x\rd t.
\end{align*}
Letting $q\to\infty$, we obtain
\begin{align*}
	\int_0^T\int_{\T^2}\left(\pa_tu\cdot\pa_t\eta-(\Id+G)\Sgm_G:\nb\eta\right)\rd x\rd t
	=0.
\end{align*}
Thus, $u$ satisfies \eqref{system 1} in the sense of distributions. Since $u\in C^1([0,T]\times\T^2)$, the tensor $(\Id+G)\Sgm_G$ belongs to $L^2(0,T;L^2(\T^2))$, and hence
$$
\pa_{tt}u=\Div((\Id+G)\Sgm_G)\in L^2(0,T;H^{-1}(\T^2)).
$$
Moreover, $u\in C([0,T];H^1(\T^2))\cap C^1([0,T];L^2(\T^2))$. Together with the zero initial conditions established above, this shows that $u$ is a weak solution of \eqref{system 1} in the sense of Definition \ref{Weak Solution}.

It remains to prove that $u$ is nontrivial. For $t\in[0,T]$, set
$$
E_q(t):=\int_{\T^2}\left(\frac{1}{2}|\pt u_q(t,x)|^2+\sgm_{G_q(t,x)}\right)\rd x,
\qquad
E(t):=\int_{\T^2}\left(\frac{1}{2}|\pt u(t,x)|^2+\sgm_{G(t,x)}\right)\rd x.
$$
Since $u_q\to u$ in $C^1([0,T]\times\T^2)$, we have $\pt u_q\to\pt u$ and $G_q\to G$ uniformly on $[0,T]\times\T^2$. By the continuity of the stored-energy function, it follows that $E_q(t)\to E(t)$ uniformly for $t\in[0,T]$. Moreover, since $u_0\equiv0$ and $\sgm(0,0)=0$, we have $E_0(t)=0$. Therefore, for every $Q\in\N$ and $t\in[T/2,T]$, the telescoping identity gives
\begin{align*}
	E_{Q+1}(t)=E_{Q+1}(t)-E_0(t)=\sum_{q=0}^{Q}\left(E_{q+1}(t)-E_q(t)\right).
\end{align*}
Using \eqref{Induction energy estimate} and then letting $Q\to\infty$, we obtain
\begin{align}
	3\pi^2\sgm_{11}\sum_{q=0}^{\infty}\delta_{q+1}
	\leqslant\int_{\T^2}\left(\frac{1}{2}|\pt u|^2+\sgm_G\right)\rd x
	\leqslant120\pi^2\sgm_{11}\sum_{q=0}^{\infty}\delta_{q+1}
\end{align}
for every $t\in[T/2,T]$. Since $\sgm_{11}>0$ and $\sum_{q=0}^{\infty}\delta_{q+1}>0$, the lower bound shows that $u$ is nontrivial. Thus, we have constructed a nontrivial $C^1$ weak solution to \eqref{system 1} emanating from zero initial data. This completes the proof.

\appendix
\section{H\"{o}lder spaces}
In this section, we introduce the notations we would use for H{\"o}lder spaces. For some time interval $\mcI\subset\R,$ we denote the supremum norm as  $\|f\|_0=\left\|f\right\|_{C^0(\mcI;C^0(\T^2))}:=\underset{(t,x)\in\mcI\times\T^2}{\sup}|f(t,x)|$. For $N\in\N,$ a multi-index $k=(k_1,k_2)\in\N^2$ and $\alpha\in(0,1],$ we denote the H\"{o}lder seminorms as
$$
\begin{aligned}
	&[f]_{N}=\underset{|k|=N}{\max}\left\|D^k f\right\|_{0},&&
	[f]_{N+\alpha}=\underset{|k|=N}{\max}\underset{t}{\sup}\underset{x\neq y}{\sup}\frac{|D^k f(t,x)-D^k f(t,y)|}{|x-y|^\alpha},
\end{aligned}
$$
where $D^k$ are spatial derivatives. Then we can denote the H\"{o}lder norms as
$$
\left\|f\right\|_{N}=\sum_{j=0}^{N}[f]_{j}, \qquad\left\|f\right\|_{N+\alpha}=\left\|f\right\|_{N}+[f]_{N+\alpha}.
$$
	\section{Some technical lemmas}
In this section, we introduce some lemmas given in \cite{BDSV19,DK22,GK22}. The proof for the following lemma can be found in \cite[Appendix]{BDSV19}.
\begin{lm}\cite[Proposition A.1]{BDSV19}\label{lm: ffhs}
	Suppose $F:\Omega\rightarrow\R$ and $\Psi:\R^n\rightarrow\Omega$ are smooth functions for some $\Omega\subset\R^m$. Then, for each $N\in\Z_+$, we have 
	\begin{equation}
		\begin{aligned}
			&\nrm{\nb^N(F\circ\Psi)}_0\lesssim\nrm{\nb F}_0\nrm{\nb\Psi}_{N-1}+\nrm{\nb F}_{N-1}\nrm{\Psi}_0^{N-1}\nrm{\Psi}_N,\\
			&\nrm{\nb^N(F\circ\Psi)}_0\lesssim\nrm{\nb F}_0\nrm{\nb\Psi}_{N-1}+\nrm{\nb F}_{N-1}\nrm{\nb\Psi}_0^{N},
		\end{aligned}\label{est on Holider norm}
	\end{equation}
	where the implicit constants in the inequalities depend only on $n, m$, and $N$.\label{Holder norm of composition}
\end{lm}
\begin{lm}\cite[Proposition C.2]{BDSV19}.\label{est on int operator}
Let $N\geqslant1$. Suppose that $a\in C^\infty(\T^2)$ and
$\xi\in C^\infty(\T^2;\R^2)$ satisfies
\begin{equation}
	\frac{1}{C}\leqslant|\nabla\xi|\leqslant C
\end{equation}
for some constant $C>1$. Then
\begin{equation}\label{est on Rae^ikxi}
	\Bigg|\int_{\T^2}a(x)e^{ik\cdot\xi(x)}\rd x\Bigg|
	\lesssim_{C,N}
	\frac{\|a\|_N+\|a\|_0\|\nabla\xi\|_N}{|k|^N}.
\end{equation}
Moreover, if there exist $x_0\in\T^2$ and a sufficiently small
parameter $\mu>0$ such that
$
\supp a\subseteq B(x_0,\mu),
$
then
\begin{equation}\label{est on Rae^ikxi 1}
	\Bigg|\int_{\T^2}a(x)e^{ik\cdot x}\rd x\Bigg|
	\lesssim_N
	\mu^2\frac{\|a\|_N}{|k|^N}.
\end{equation}
\end{lm}

\begin{proof}
Estimate \eqref{est on Rae^ikxi} is the two-dimensional analogue of
\cite[Proposition C.2]{BDSV19} and follows from the stationary-phase
argument in \cite[Lemma 2.2]{DS17}. Since
	$$
	e^{ik\cdot x}
	=
	\frac{(k\cdot\nabla)^Ne^{ik\cdot x}}
	{i^N|k|^{2N}},
	$$
	integration by parts $N$ times yields
	$$
	\int_{\T^2}a(x)e^{ik\cdot x}\rd x
	=
	\frac{(-1)^N}{i^N|k|^{2N}}
	\int_{\T^2}(k\cdot\nabla)^Na(x)e^{ik\cdot x}\rd x.
	$$
	Therefore, using $\supp a\subseteq B(x_0,\mu)$, we obtain
	$$
	\begin{aligned}
		\Bigg|
		\int_{\T^2}a(x)e^{ik\cdot x}\rd x
		\Bigg|
		\lesssim_N
		\frac{1}{|k|^N}
		\int_{\supp a}|\nabla^Na(x)|\rd x\lesssim_N
		\mu^2\frac{\|a\|_N}{|k|^N}.
	\end{aligned}
	$$
	This proves \eqref{est on Rae^ikxi 1}.
\end{proof}

To obtain the commutator estimate, we present the following lemma in  \cite{DK22,GK22,MQ24}. In the following lemmas, we use the notation $\mcI$ and $\mcI_\ell$ to represent the interval $\mcI=[c,d],\ \mcI_\ell=[c-\ell,d+\ell]$ and denote $\nrm{\cdot}_N=\nrm{\cdot}_{C^0(\mcI;C^N(\T^2))}$.
\begin{lm}\label{lm of commutator}
	Let $f$ and $g$ be in $C^\infty(\mcI_\ell\times\T^2)$. Then, for each $N,r\geqslant0$, the following holds,
	\begin{align}
		\|\PL f \PL g-\PL(fg)\|_N&\lesssim_N\ell^{2-N}\|f\|_1\|g\|_1,\label{est on commutator 0}\\
		\|\partial_{t}^r(\UL f \UL g-\UL(fg))\|_0&\lesssim_r\ell^{2-r}\|\partial_{t}f\|_{C^0(\mcI_\ell;C^0(\T^2))}\|\partial_tg\|_{C^0(\mcI_\ell;C^0(\T^2))}.\label{est on time commutator 0}
	\end{align} 
	If we set $f_\ell=\UL\PL f$, $g_\ell=\UL\PL g$, and $(fg)_\ell=\UL\PL (fg)$, we have for each $N,r\geqslant 0$,
	\begin{equation}\label{est on mollification commutator 0}
		\begin{aligned}
			\|\partial_{t}^r(f_\ell g_\ell-(fg)_\ell)\|_{N}&\lesssim_{N,r} \ell^{2-N-r}\|\partial_{t}f\|_{C^0(\mcI_\ell;C^0(\T^2))}\|\partial_{t}g\|_{C^0(\mcI_\ell;C^0(\T^2))}+\ell^{2-N-r}\|f\|_{C^0(\mcI_\ell;C^1(\T^2))}\|g\|_{C^0(\mcI_\ell;C^1(\T^2))}.
		\end{aligned}
	\end{equation} 
	For a more general case involving the product of multiple functions, the commutator estimate between the mollification operator and the product is given by:\par 
	Let $N_0\in\N$ and let
	$\{f_n\}_{n=1}^{N_0}\subset
	C^\infty(\mcI_\ell\times\T^2)$.
	Then, for each $N,r\geqslant0$, the following holds,
		\begin{align}
		&\Nrm{\prod_{n=1}^{N_0}(\PL f_n)  -\PL\left(\prod_{n=1}^{N_0}f_n\right)}_{N}\lesssim_N\ell^{2-N}\sum_{\substack{s_1+\cdots+s_{N_0}=2\\s_n<2,n=1,\cdots,N_0} }\prod_{n=1}^{N_0}\nrm{f_n}_{s_n},\label{est on N commutator 0}\\
		&\Nrm{\pa_{t}^r\left(\prod_{n=1}^{N_0}( \UL f_n ) -\UL\left(\prod_{n=1}^{N_0}f_n\right)\right)}_{0}\lesssim_r\ell^{2-r}\sum_{\substack{s_1+\cdots+s_{N_0}=2\\s_n<2,n=1,\cdots,N_0}}\prod_{n=1}^{N_0}\nrm{\pa_{t}^{s_n}f_n}_{C^0(\mcI_\ell;C^{0}(\T^2))},\label{est on N time commutator 0}\\
		&\Nrm{\pa_{t}^r\left(\prod_{n=1}^{N_0}( \PL\UL f_n ) -\PL\UL\left(\prod_{n=1}^{N_0}f_n\right)\right)}_{N}\nonumber\\
		&\hspace{135pt}\lesssim_{N,r}\ell^{2-N-r}\sum_{\substack{s_1+\cdots+s_{N_0}=2\\s_n<2,n=1,\cdots,N_0}}\left(\prod_{n=1}^{N_0}\nrm{f_n}_{C^0(\mcI_\ell;C^{s_n}(\T^2))}+\prod_{n=1}^{N_0}\nrm{\pt^{s_n}f_n}_{C^0(\mcI_\ell;C^{0}(\T^2))}\right).\label{est on N space and time commutator 0}
	\end{align} 
\end{lm}
\begin{proof}
Recall the definition of $\PL$ and $\UL$, we could calculate for $(t,x)\in \mcI\times\T^2$,
	\begin{align*}
		(\PL f\PL g-\PL(fg))(t,x)&=(f-\PL f)(g-\PL g)-\int_{\R^2}(f(t,x)-f(t,x-z))(g(t,x)-g(t,x-z))\eta_\ell(z)\rd z,
	\end{align*}
	Then, \eqref{est on commutator 0} follows  from  Lemma \ref{lm: Mollification}, $\supp\eta_\ell\subseteq B(0,\ell)$, and
	$$
	\begin{aligned}
		&|f(t,x)-f(t,x-z)|\leqslant|z|\nrm{f}_{C^0(\mcI_\ell;C^1(\T^2))}, \quad|g(t,x)-g(t,x-z)|\leqslant|z|\nrm{g}_{C^0(\mcI_\ell;C^1(\T^2))}.
	\end{aligned}
	$$
	Similarly, we could calculate
	\begin{align*}
		(\UL f\UL g-\UL(fg))(t,x)&=(f-\UL f)(g-\UL g)-\int_{\R}(f(t,x)-f(t-\tau,x))(g(t,x)-g(t-\tau,x)){\eta}^t_\ell(\tau)\rd\tau,
	\end{align*}
	and \eqref{est on time commutator 0} follows from Lemma \ref{lm: Mollification}, $\supp\eta_\ell^t\subseteq (-\ell,\ell)$, and
	$$
	\begin{aligned}
		&|f(t,x)-f(t-\tau,x)|\leqslant|\tau|\nrm{\pa_{t}f}_{C^0(\mcI_\ell;C^0(\T^2))}, \quad|g(t,x)-g(t-\tau,x)|\leqslant|\tau|\nrm{\pa_{t}g}_{C^0(\mcI_\ell;C^0(\T^2))}.
	\end{aligned}
	$$
	Next, we use \eqref{est on commutator 0} and \eqref{est on time commutator 0} to obtain
	\begin{align*}
		&\quad\nrm{\pa_{t}^r(f_\ell g_\ell-(fg)_\ell)}_{N}\\
		&\leqslant\nrm{\pa_{t}^r(f_\ell g_\ell-\UL(\PL f\PL g))}_{N}+\nrm{\pa_{t}^r(\UL(\PL f\PL g)-\UL\PL(fg))}_{N}\\
		&\lesssim_{N,r}\sum_{|\alpha_0|+|\alpha_1|=N}\nrm{\pa_{t}^r( (\UL\nb^{\alpha_0}(\PL f)) (\UL\nb^{\alpha_1}(\PL g)) -\UL(\nb^{\alpha_0}(\PL f) \nb^{\alpha_1}(\PL g) ))}_{0}\\
		&\quad+\ell^{-r}\nrm{\PL f\PL g-\PL(fg)}_{C^0(\mcI_\ell;C^N(\T^2))}\\
		&\lesssim_{N,r} \ell^{2-N-r}\nrm{\pa_{t}f}_{C^0(\mcI_\ell;C^0(\T^2))}\nrm{\pa_{t}g}_{C^0(\mcI_\ell;C^0(\T^2))}+\ell^{2-N-r}\nrm{f}_{C^0(\mcI_\ell;C^1(\T^2))}\nrm{g}_{C^0(\mcI_\ell;C^1(\T^2))}.
	\end{align*}
	For the case involving the product of multiple functions, we can use mathematical induction. The case of $N_0=2$ has been proved. Assume that \eqref{est on N commutator 0} holds for $N_0=k$,  we compute 
	\begin{align*}
	&\Nrm{\prod_{n=1}^{k+1}(\PL f_n)  -\PL\left(\prod_{n=1}^{k+1}f_n\right)}_{N}\\
	\leqslant&\Nrm{\prod_{n=1}^{k+1}(\PL f_n) - \prod_{n=1}^{k-1}(\PL f_n)\PL (f_{k}f_{k+1})}_{N} + \Nrm{\prod_{n=1}^{k-1}(\PL f_n)\PL (f_{k}f_{k+1})  -\PL\left(\prod_{n=1}^{k+1}f_n\right)}_{N}\\
	\lesssim&_N  \sum_{N_1+N_2=N}\Nrm{\prod_{n=1}^{k-1}(\PL f_n)}_{N_1}\Nrm{(\PL f_{k})(\PL f_{k+1}) - \PL (f_{k}f_{k+1})}_{N_2}+\ell^{2-N}\sum_{ \substack{s_1+\cdots+s_{k}=2\\s_n<2,n=1,\cdots,k}}\prod_{n=1}^{k-1}\Nrm{f_n}_{s_n}\Nrm{f_{k}f_{k+1}}_{s_{k}}\\
	\lesssim&_N\ell^{2-N}\sum_{ \substack{s_1+\cdots+s_{k+1}=2\\s_n<2,n=1,\cdots,k+1}}\prod_{n=1}^{k+1}\Nrm{f_n}_{s_n}.
	\end{align*}
	Similarly, we could obtain \eqref{est on N time commutator 0}. By using  \eqref{est on N commutator 0} and \eqref{est on N time commutator 0}, we could get \eqref{est on N space and time commutator 0}.
	\begin{align*}
		&\quad\Nrm{\pt^{r}\left(\prod_{n=1}^{k+1}(\PL\UL f_n)  -\PL\UL\left(\prod_{n=1}^{k+1}f_n\right)\right)}_{N}\\
		&\leqslant\Nrm{\pt^{r}\left(\prod_{n=1}^{k+1}(\PL\UL f_n)  -\PL\left(\prod_{n=1}^{k+1}\UL f_n\right)\right)}_{N}+\Nrm{\pt^{r}\left(\PL\left(\prod_{n=1}^{k+1}\UL f_n\right) -\PL\UL\left(\prod_{n=1}^{k+1}f_n\right)\right)}_{N}\\
		&\lesssim_{N,r}\ell^{-r}\Nrm{\prod_{n=1}^{k+1}(\PL\UL f_n)  -\PL\left(\prod_{n=1}^{k+1}\UL f_n\right)}_{N}+\ell^{-N}\Nrm{\pt^{r}\left(\prod_{n=1}^{k+1}\UL f_n -\UL\left(\prod_{n=1}^{k+1}f_n\right)\right)}_{0}\\
		&\lesssim_{N,r}\ell^{2-N-r}\sum_{\substack{s_1+\cdots+s_{k+1}=2\\s_n<2,n=1,\ldots,k+1}}\left(\prod_{n=1}^{k+1}\nrm{f_n}_{C^0(\mcI_\ell;C^{s_n}(\T^2))}+\prod_{n=1}^{k+1}\nrm{\pt^{s_n}f_n}_{C^0(\mcI_\ell;C^0(\T^2))}\right).\qedhere
	\end{align*}
\end{proof}
\section{Estimate for nonautonomous linear differential systems}\label{Estimate for nonautonomous linear differential systems}
Here, we give an estimate for nonautonomous linear differential systems. Let $m\geqslant 1$, $T>0$, $A,B\in C^0([0,T];\R^{m\times m})$, and $R\in C^0([0,T];\R^m)$. If we consider the second-order ordinary differential system
\begin{equation}\label{second order ode}
	\left\lbrace 
	\begin{aligned}
		&\frac{\rd^2y}{\rd t^2}+A(t)\frac{\rd y}{\rd t}+B(t)y=R(t),\\ 
		&y(0)=0,\quad \frac{\rd y}{\rd t}(0)=0,
	\end{aligned}
	\right.
\end{equation}
where $y:[0,T]\to\R^m$, then there exists a unique solution $y\in C^2([0,T];\R^m)$. Moreover, there exists a constant $C=C(T,K)>0$ such that
\begin{align*}
	\|y\|_{C_t^0}
	+\left\|\frac{\rd y}{\rd t}\right\|_{C_t^0}
	\leqslant
	C(T,K)\|R\|_{C_t^0},
\end{align*}
where
\begin{align*}
	K:=\sup_{t\in[0,T]}\max\{\|A(t)\|,\|B(t)\|\},
\end{align*}
and the matrix norms are the operator norms induced by the Euclidean norm on $\R^m$. Consequently,
\begin{equation}\label{ode ypp}
	\left\|\frac{\rd^2y}{\rd t^2}\right\|_{C_t^0}
	\leqslant
	C(T,K)\|R\|_{C_t^0}.
\end{equation}

Indeed, setting
\begin{align*}
	Y(t):=
	\begin{pmatrix}
		y(t)\\
		\frac{\rd y}{\rd t}(t)
	\end{pmatrix},
	\qquad
	\mathcal A(t):=
	\begin{pmatrix}
		0&\Id\\
		-B(t)&-A(t)
	\end{pmatrix},
	\qquad
	\mathcal R(t):=
	\begin{pmatrix}
		0\\
		R(t)
	\end{pmatrix},
\end{align*}
we rewrite \eqref{second order ode} as
\begin{align*}
	\frac{\rd Y}{\rd t}=\mathcal A(t)Y+\mathcal R(t),
	\qquad
	Y(0)=0.
\end{align*}
Since $\|\mathcal A\|_{C_t^0}\lesssim 1+K$, Gronwall's inequality yields
\begin{align*}
	\|Y\|_{C_t^0}\leqslant C(T,K)\|\mathcal R\|_{C_t^0}
	=C(T,K)\|R\|_{C_t^0},
\end{align*}
which gives the first estimate. Finally, \eqref{ode ypp} follows directly from \eqref{second order ode}.\\

\noindent{\bf Acknowledgment.}
P.~Qu acknowledges support from the NSFC  Grants No.~62588101 and No.~12431007. S.~Mao acknowledges support from the Fundamental Research Funds for the Central Universities  (No. 2232026D46).\\

\noindent{\bf Conflict of interest.} The authors declare that they have no conflict of interest.\ \\

\noindent{\bf Data availability statement.} Data sharing is not applicable to this article as no datasets were generated or analysed during the current study.

\end{document}